\documentclass[10pt,twoside]{article}
\def\YEAR{\year}\newcount\VOL\VOL=\YEAR\advance\VOL by-1995
\def\firstpage{73}\def\lastpage{176}
\def\received{May 13, 2009}\def\revised{}
\def\communicated{Peter Schneider}

\makeatletter
\def\magnification{\afterassignment\m@g\count@}
\def\m@g{\mag=\count@\hsize6.5truein\vsize8.9truein\dimen\footins8truein}
\makeatother

\font\eightrm=cmr8
\font\caps=cmcsc10                    
\font\Caps=cmcsc10 scaled \magstep1   
\font\scaps=cmcsc8

\usepackage{amsmath,mathrsfs,amssymb,amsthm,amscd}
\usepackage[]{fontenc}
\usepackage{xy} 
\usepackage[hyperindex]{hyperref}
\usepackage{graphicx,color}
\usepackage{ifpdf}
\xyoption{all}
\hypersetup{breaklinks=true}
\makeatletter
\@specialpagefalse\headheight=8.5pt
\def\DocMath{}
\def\DocMath{{\def\th{\thinspace}\scaps Documenta Math.}}
\renewcommand{\@oddfoot}{\hfill\scaps Documenta Mathematica 
    \number\VOL\  (\number\YEAR) \number\firstpage--\lastpage\hfill}
\renewcommand{\@evenfoot}{\ifnum\thepage>\lastpage\hfill\scaps
    Documenta Mathematica \number\VOL\  (\number\YEAR)\hfill\else\@oddfoot\fi}%

\renewcommand{\@evenhead}{%
    \ifnum\thepage>\lastpage\rlap{\thepage}\hfill%
    \else\rlap{\thepage}\slshape\leftmark\hfill{\caps\SAuthor}\hfill\fi}%
\renewcommand{\@oddhead}{%
    \ifnum\thepage=\firstpage{\DocMath\hfill\llap{\thepage}}%
    \else{\slshape\rightmark}\hfill{\caps\STitle}\hfill\llap{\thepage}\fi}%
\makeatother

\def\TSkip{\bigskip}
\newbox\TheTitle{\obeylines\gdef\GetTitle #1
\ShortTitle  #2
\SubTitle    #3
\Author      #4
\ShortAuthor #5
\EndTitle
{\setbox\TheTitle=\vbox{\baselineskip=20pt\let\par=\cr\obeylines%
\halign{\centerline{\Caps##}\cr\noalign{\medskip}\cr#1\cr}}%
	\copy\TheTitle\TSkip\TSkip%
\def\next{#2}\ifx\next\empty\gdef\STitle{#1}\else\gdef\STitle{#2}\fi%
\def\next{#3}\ifx\next\empty%
    \else\setbox\TheTitle=\vbox{\baselineskip=20pt\let\par=\cr\obeylines%
    \halign{\centerline{\caps##} #3\cr}}\copy\TheTitle\TSkip\TSkip\fi%
\centerline{\caps #4}\TSkip\TSkip%
\def\next{#5}\ifx\next\empty\gdef\SAuthor{#4}\else\gdef\SAuthor{#5}\fi%
\ifx\received\empty\relax
    \else\centerline{\eightrm Received: \received}\fi%
\ifx\revised\empty\TSkip%
    \else\centerline{\eightrm Revised: \revised}\TSkip\fi%
\ifx\communicated\empty\relax
    \else\centerline{\eightrm Communicated by \communicated}\fi\TSkip\TSkip%
\catcode'015=5}}\def\Title{\obeylines\GetTitle}
\def\Abstract{\begingroup\narrower
    \parskip=\medskipamount\parindent=0pt{\caps Abstract. }}
\def\EndAbstract{\par\endgroup\TSkip}

\long\def\MSC#1\EndMSC{\def\arg{#1}\ifx\arg\empty\relax\else
     {\par\narrower\noindent%
     2010 Mathematics Subject Classification: #1\par}\fi}

\long\def\KEY#1\EndKEY{\def\arg{#1}\ifx\arg\empty\relax\else
	{\par\narrower\noindent Keywords and Phrases: #1\par}\fi\TSkip}

\newbox\TheAdd\def\Addresses{\vfill\copy\TheAdd\vfill
    \ifodd\number\lastpage\vfill\eject\phantom{.}\vfill\eject\fi}
{\obeylines\gdef\GetAddress #1
\Address #2 
\Address #3
\Address #4
\EndAddress
{\def\xs{5.3truecm}\parindent=0pt
\setbox0=\vtop{{\obeylines\hsize=\xs#1\par}}\def\next{#2}
\ifx\next\empty 
     \setbox\TheAdd=\hbox to\hsize{\hfill\copy0\hfill}
\else\setbox1=\vtop{{\obeylines\hsize=\xs#2\par}}\def\next{#3}
\ifx\next\empty 
     \setbox\TheAdd=\hbox to\hsize{\hfill\copy0\hfill\copy1\hfill}
\else\setbox2=\vtop{{\obeylines\hsize=\xs#3\par}}\def\next{#4}
\ifx\next\empty\ 
     \setbox\TheAdd=\vtop{\hbox to\hsize{\hfill\copy0\hfill\copy1\hfill}
                \vskip20pt\hbox to\hsize{\hfill\copy2\hfill}}
\else\setbox3=\vtop{{\obeylines\hsize=\xs#4\par}}
     \setbox\TheAdd=\vtop{\hbox to\hsize{\hfill\copy0\hfill\copy1\hfill}
	        \vskip20pt\hbox to\hsize{\hfill\copy2\hfill\copy3\hfill}}
\fi\fi\fi\catcode'015=5}}\gdef\Address{\obeylines\GetAddress}

\numberwithin{equation}{section}
\nopagebreak%
\theoremstyle{plain}
\newtheorem{theorem}[equation]{Theorem}
\newtheorem{corollary}[equation]{Corollary}
\newtheorem{proposition}[equation]{Proposition}
\newtheorem{lemma}[equation]{Lemma}

\theoremstyle{definition}
\newtheorem{definition}[equation]{Definition}
\newtheorem{notation}[equation]{Notation}

\newtheorem{remark}[equation]{Remark}

\newtheorem{claim}{Claim}

\DeclareMathOperator{\Det}{Det}
\DeclareMathOperator{\Tr}{Tr}
\DeclareMathOperator{\WF}{WF}

\DeclareMathOperator{\tr}{tr}
\DeclareMathOperator{\dv}{div}

\DeclareMathOperator{\cha}{\widehat{CH}}

\DeclareMathOperator{\CH}{CH}

\DeclareMathOperator{\Td}{Td}
\DeclareMathOperator{\cht}{\widetilde{ch}}
\DeclareMathOperator{\chh}{\widehat{ch}}
\DeclareMathOperator{\Tot}{Tot}

\DeclareMathOperator{\dd}{d}
\DeclareMathOperator{\Img}{Im}

\DeclareMathOperator{\ch}{ch}
\DeclareMathOperator{\cl}{cl}
\DeclareMathOperator{\rk}{rk}

\DeclareMathOperator{\ad}{ad}

\DeclareMathOperator{\Ker}{Ker}

\DeclareMathOperator{\Id}{id}
\DeclareMathOperator{\Pic}{Pic}

\DeclareMathOperator{\Coker}{Coker}
\DeclareMathOperator{\Spec}{Spec}

\DeclareMathOperator{\supp}{supp}

\DeclareMathOperator{\amap}{a}

\newcommand{\an}{\text{{\rm an}}}

\newcommand{\CC}{{\mathbb C}}
\newcommand{\QQ}{{\mathbb Q}}

\newcommand{\cc}{{\mathcal{C}}}

\newcommand{\las}{{\text{\rm l,a}}}

\newcommand{\D}{\text{{\rm cur}}}

\def\?{\ ???\ \immediate\write16{}%
\immediate\write16{Warning: There was still a question mark . . . }%
\immediate\write16{}}

\begin{document}
\Title
Singular Bott-Chern Classes 
and the Arithmetic Grothendieck
Riemann Roch Theorem for Closed Immersions
\ShortTitle 
Singular Bott-Chern Classes
\SubTitle   
\Author Jos\'e I.~Burgos Gil\footnote{Partially supported by Grant %
DGI MTM2006-14234-C02-01.}
and R\u azvan Li\c tcanu\footnote{Partially supported by CNCSIS Grant %
  1338/2007 and PN II Grant ID\_2228 (502/2009)}
\ShortAuthor
Jos\'e I.~Burgos Gil and  R\u azvan Li\c tcanu
\EndTitle
\Abstract 
We study the singular Bott-Chern classes introduced by
Bismut, Gillet and Soul\'e. Singular Bott-Chern classes are the main
ingredient to define direct images for closed immersions in arithmetic
$K$-theory. In this paper we give an axiomatic definition of a theory
of singular Bott-Chern classes, study their properties, and
classify all possible
theories of this kind. We identify the theory defined by Bismut,
Gillet and Soul\'e 
as the only one that satisfies the additional condition of being
homogeneous. We include a proof of the arithmetic
Grothendieck-Riemann-Roch theorem for closed immersions that
generalizes a result of Bismut, Gillet and Soul\'e and was already
proved by Zha. This result can be combined with the arithmetic
Grothendieck-Riemann-Roch theorem for submersions to extend this
theorem to arbitrary projective morphisms. As a byproduct of this study we
obtain two results of independent interest. First, we prove a
Poincar\'e lemma for the complex of currents with fixed wave front set,
and second we prove that certain direct images of Bott-Chern classes
are closed.
\EndAbstract
\MSC 
14G40 32U40 
\EndMSC
\KEY 
Arakelov Geometry, Closed immersions, Bott-Chern classes, Arithmetic
Riemann-Roch theorem, currents, wave front sets. 
\EndKEY
\Address 
Jos\'e I.~Burgos Gil
Instituto de Ciencias Matem\'aticas 
(CSIC-UAM-UC3M-UCM)
burgos@icmat.es, 
\ \ jiburgosgil@gmail.com
Temporary Address:
Centre de Recerca Matem\'atica CRM
UAB Science Faculty
08193 Bellaterra
Barcelona, Spain
\Address
R\u azvan Li\c tcanu
University Al. I. Cuza
Faculty of Mathematics
Bd. Carol I, 11
700506 Ia\c si
Romania
litcanu@uaic.ro
\Address
\Address
\EndAddress
\setcounter{tocdepth}{2}
\setcounter{section}{-1}
\date{}
\vspace*{2cm} 



\newpage

\tableofcontents

\section{Introduction}
\label{sec:introduction}

Chern-Weil theory associates to each hermitian vector
bundle a family of closed characteristic forms that represent the
characteristic classes of the vector bundle. The characteristic
classes are compatible with exact sequences. But this is not true for
the characteristic forms. The Bott-Chern classes measure the lack of
compatibility of the characteristic forms with exact sequences.

The Grothendieck-Riemann-Roch theorem gives a formula that relates
direct images and characteristic classes. In general this formula is
not valid for the characteristic forms. The singular Bott-Chern
classes measure, in a functorial way, the failure of an exact
Grothendieck-Riemann-Roch theorem for
closed immersions at the level of characteristic forms. In the same
spirit, the analytic torsion forms 
measure the failure of an exact
Grothendieck-Riemann-Roch theorem for
submersions at the level of characteristic forms. Hence singular
Bott-Chern classes and analytic torsion 
forms are analogous objects, the first for closed immersions and the
second for submersions. 

Let us give a more precise description of Bott-Chern
classes and singular Bott-Chern
classes.
Let $X$ be a complex manifold and let $\varphi$ be a symmetric power
series in $r$ variables with real coefficients. Let $\overline
E=(E,h)$ be a rank $r$ holomorphic vector bundle provided with a
hermitian 
metric. Using Chern-Weil theory, we can associate to $\overline E$ a
differential form $\varphi(\overline E)=\varphi(-K)$, where $K$ is the
curvature tensor of $E$ viewed as a matrix of 2-forms. The
differential form $\varphi(\overline E)$ is closed and is a sum of
components of bidegree $(p,p)$ for $p\ge 0$.

If
$$\overline {\xi}\colon 0\longrightarrow \overline{E}'\longrightarrow
\overline{E}\longrightarrow 
\overline{E}''\longrightarrow 0$$ 
is a short exact sequence of holomorphic vector bundles provided with
hermitian metrics, then the differential forms $\varphi(\overline E)$
and $\varphi(\overline {E}'\oplus \overline {E}')$ may be different,
but they represent the same cohomology class. 

The Bott-Chern form associated to $\overline \xi$ is a solution of the
differential equation         
\begin{equation}\label{eq:1}
  -2\partial\bar \partial \varphi(\overline \xi)=
\varphi(\overline {E}'\oplus \overline {E}')
- \varphi(\overline E)
\end{equation}
obtained in a functorial way. The class of a Bott-Chern form modulo
the image of $\partial$ and $\overline{\partial}$ is called a
Bott-Chern class and is denoted by $\widetilde
{\varphi}(\overline {\xi})$.

There are three ways of defining the Bott-Chern classes. The
first one is the original definition of Bott and Chern
\cite{BottChern:hvb}. It is based on a deformation
between the connection associated to $\overline E$ and the connection
associated to $\overline E'\oplus \overline E''$. This deformation is
parameterized by a real variable. 

In \cite{GilletSoule:MR854556} Gillet and
Soul\'e introduced a second definition of Bott-Chern classes that
is 
based on a deformation between $\overline E$ and $\overline E'\oplus
\overline E''$ parameterized by a projective line. This second
definition is used in \cite{BismutGilletSoule:at} to prove that the
Bott-Chern classes are characterized by three properties
\begin{enumerate}
\item \label{item:1} The differential equation \eqref{eq:1}.
\item \label{item:2} Functoriality (i.e. compatibility with pull-backs
  via holomorphic maps). 
\item \label{item:3} The vanishing of the Bott-Chern class of a
  orthogonally split
  exact sequence. 
\end{enumerate}

In \cite{BismutGilletSoule:at} Bismut, Gillet and
Soul\'e have a third definition of Bott-Chern classes based
on the theory of superconnections. This definition is useful to link
Bott-Chern classes with analytic torsion forms.

The definition of Bott-Chern classes can be generalized to any
bounded exact sequence of hermitian vector bundles (see section
\ref{sec:bott-chern-forms} for details). Let
\begin{displaymath}
\overline \xi\colon0\longrightarrow(E_{n},h_{n})\longrightarrow
\dots \longrightarrow (E_{1},h_{1}) 
\longrightarrow(E_{0},h_{0})\longrightarrow 0
\end{displaymath}
be a bounded acyclic complex of hermitian vector bundles; by this we
mean 
a bounded acyclic complex of vector bundles, where each vector bundle
is 
equipped with an arbitrarily chosen hermitian metric. Let
\begin{displaymath}
  r=\sum_{i\text{ even}} \rk(E_{i})=\sum_{i\text{ odd}} \rk(E_{i}). 
\end{displaymath}
As before, let $\varphi$ be a
symmetric power series in $r$ variables. A Bott-Chern class
associated to $\overline \xi$ satisfies the differential equation 
\begin{displaymath}
    -2\partial\bar \partial \widetilde{\varphi}(\overline \xi)=
\varphi(\bigoplus_{k}\overline E_{2k})
- \varphi(\bigoplus_{k}\overline E_{2k+1}).
\end{displaymath}
In particular, let ``$\ch$'' denote the power series associated to the
Chern character class. The Chern character class
has the advantage of being additive for direct sums. Then, the
Bott-Chern class associated to the long exact sequence $\overline
{\xi}$ 
and to the Chern character class satisfies the differential equation 
\begin{displaymath}
  -2\partial\bar \partial \widetilde{\ch}(\overline \xi)=
  -\sum_{k=0}^{n}(-1)^{i}\ch(\overline E_{k}). 
\end{displaymath}

Let now $i\colon Y\longrightarrow X$ be a closed immersion of complex
manifolds. Let $\overline {F}$ be a holomorphic vector bundle on $Y$
provided with a hermitian metric. Let $\overline N$ be the normal
bundle to $Y$ in $X$ provided also with a hermitian metric. Let
\begin{displaymath}
 0\longrightarrow \overline E_{n}\longrightarrow
 \overline E_{n-1}
  \longrightarrow \dots \longrightarrow \overline E_{0}
  \longrightarrow i_{\ast} F\longrightarrow 0
\end{displaymath}
be a resolution of the coherent sheaf $i_{\ast} F$ by locally free
sheaves, provided with hermitian metrics (following Zha
\cite{zha99:_rieman_roch} we shall call such a
sequence a metric on the coherent sheaf $i_{\ast} F$). Let $\Td$
denote the Todd characteristic class. Then the
Grothendieck-Riemann-Roch theorem for the closed immersion $i$
implies that the current $i_{\ast}(\Td(\overline
N)^{-1}\ch(\overline F))$ and the differential form
$\sum_{k}(-1)^{k}\ch(\overline E_{k})$ represent the same class in
cohomology. We denote $\overline {\xi}$ the data consisting in the
closed embedding $i$, the hermitian bundle $\overline N$, the
hermitian bundle $\overline {F}$ and the resolution $\overline
E_{\ast} \longrightarrow i_{\ast} F$.

In the paper \cite{BismutGilletSoule:MR1047123}, Bismut, Gillet and
Soul\'e introduced a current associated to the above
situation. These currents are called singular Bott-Chern currents
and denoted in \cite{BismutGilletSoule:MR1047123} by
$T(\overline \xi)$. When the hermitian metrics satisfy a
certain technical condition (condition A of Bismut) then the
singular Bott-Chern current $T(\overline \xi)$ satisfies the
differential equation
\begin{displaymath}
  -2\partial\bar \partial T(\overline \xi)=
  i_{\ast}(\Td(\overline N)^{-1}\ch(\overline F))
  -\sum_{i=0}^{n}(-1)^{i}\ch(\overline E_{i}).
\end{displaymath}

These singular Bott-Chern currents are among the main ingredients
of the proof of Gillet and
Soul\'e's arithmetic Riemann-Roch theorem. In fact it is the main
ingredient of the arithmetic 
Riemann-Roch theorem for closed immersions
\cite{BismutGilletSoule:MR1086887}. This definition of singular
Bott-Chern classes is based on the formalism of superconnections,
like the third definition of ordinary Bott-Chern classes.

In his thesis \cite{zha99:_rieman_roch}, Zha gave another
definition of singular Bott-Chern currents and used it to give a
proof of a different version of the arithmetic Riemann-Roch theorem.
This second definition is analogous to Bott and Chern's original
definition. Nevertheless there is no explicit comparison between 
the two definitions of singular Bott-Chern currents.

One of the purposes of this note is to give a third construction of singular
Bott-Chern currents, in fact of their classes modulo the image of
$\partial$ and $\overline \partial$, which could be seen as analogous
to the second 
definition of Bott-Chern classes. Moreover we will use this third
construction to give an axiomatic definition of a theory of singular
Bott-Chern classes. A theory of singular Bott-Chern classes 
is an assignment that, to each  data $\overline{\xi}$
as above, associates a class of currents $T(\overline{\xi})$, that
satisfies the analogue of conditions \ref{item:1}, \ref{item:2} and
\ref{item:3}.
The main technical point of this axiomatic
definition is that the conditions analogous to \ref{item:1},
\ref{item:2} and \ref{item:3} above are not enough to characterize
the singular Bott-Chern classes. Thus we are led to the problem of
classifying the possible theories of Bott-Chern classes, which is the
other purpose of this paper.

We fix a theory $T$ of singular Bott-Chern
classes. 
Let $Y$ be a complex manifold and let $\overline N$ and $\overline
F$ be two hermitian holomorphic vector bundles on $Y$. We write
$P=\mathbb{P}(N\oplus 1)$ for the projective completion of $N$. Let
$s\colon Y\longrightarrow P$ be the inclusion as the zero section and let
$\pi_P\colon P\longrightarrow Y$ be the projection. Let
$\overline{K}_{\ast}$ be the Koszul resolution of
$s_{\ast}\mathcal{O}_Y$ endowed with the metric induced by
$\overline{N}$. Then we have a resolution by hermitian vector
bundles
\begin{displaymath}
 K(\overline F,\overline N)\colon \overline{K}_{\ast}\otimes
 \pi_P^{\ast}\overline{F}\longrightarrow s_{\ast}F. 
\end{displaymath}
To these data we associate a singular Bott-Chern class $T(K(\overline
F,\overline N))$. It turns out that the current
\begin{displaymath}
\frac{1}{(2\pi i)^{\rk N}}  
\int_{\pi_P}  T(K(\overline F,\overline
  N))=(\pi_P)_{\ast}T(K(\overline F,\overline N)) 
\end{displaymath}
is closed (see section \ref{sec:direct-images-bott} for general
properties of the Bott-Chern classes that imply this property) and
determines a characteristic class $C_T(F,N)$ on $Y$ for the vector
bundles $N$ and $F$. Conversely, any arbitrary characteristic class
for pairs of vector bundles can be obtained in this way. This allows
us to classify the possible
theories of singular Bott-Chern classes:

\begin{claim} [theorem \ref{thm:6}] The assignment that sends a
singular Bott-Chern class $T$ to the characteristic class $C_{T}$ is a
bijection between the set of theories of singular Bott-Chern classes and the set
of characteristic classes.   
\end{claim}

The next objective of this note is to study the properties of the
different theories of singular 
Bott-Chern classes and of the corresponding characteristic classes.
We mention, in the first place, that for the functoriality condition
to make sense, we have to study the wave front sets of the currents
representing the singular Bott-Chern classes. In particular we use a
Poincar\'e Lemma for currents with fixed wave front set. This result
implies that, in each singular Bott-Chern class, we can find a
representative with controlled wave front set that can be pulled back
with respect certain morphisms.

We also  investigate how different properties of the singular Bott-Chern
classes $T$ are reflected in properties of the characteristic
classes $C_T$. We thus characterize the compatibility of the
singular Bott-Chern classes with the projection formula, by the
property of $C_T$ of being compatible with the
  projection formula. We
also relate the compatibility of the singular Bott-Chern classes
with the composition of successive closed immersions to an
additivity property of the associated characteristic class.

Furthermore, we show that we can add a 
natural fourth axiom to the conditions analogue to \ref{item:1},
\ref{item:2} and \ref{item:3}, namely the condition of being
homogeneous (see section \ref{sec:bismut-gillet-soule} for the precise
definition). 

\begin{claim} [theorem \ref{thm:12}] There exists a unique homogeneous
  theory of singular Bott-Chern classes.
\end{claim}

Thanks to this axiomatic characterization, we prove that this theory
agrees with the theories 
of singular Bott-Chern classes introduced by  Bismut, Gillet and
Soul\'e \cite{BismutGilletSoule:MR1086887}, and by Zha
\cite{zha99:_rieman_roch}. In particular
this provides us a comparison
between the two definitions. We will also characterize 
the characteristic class $C_{T^h}$ for the theory
of homogeneous singular Bott-Chern classes.

The last objective of this paper is to give a proof of the arithmetic
Riemann-Roch theorem for closed immersions. A version of this theorem
was proved by Bismut, Gillet and Soul\'e and by Zha.

Next we will discuss the contents of the different sections of this paper. 
In section \S 1 we recall the properties of characteristic classes in
analytic Deligne cohomology. A characteristic class is just a
functorial assignment that associates a
cohomology class to each vector bundle. The main result of this
section is that any 
characteristic class is given by a power series on the Chern classes,
with appropriate 
coefficients.

In section \S 2 we recall the theory of Bott-Chern forms and its
main properties. The contents of this section are standard although
the presentation is slightly different to the ones published in the
literature. 

In section \S 3 we study certain direct images of Bott-Chern forms. The
main result of this section is that, even if the Bott-Chern classes are not
closed, certain direct images of Bott-Chern classes are closed. This
result generalizes previous 
results of Bismut, Gillet and Soul\'e and of Mourougane. This result
is used to prove that the class $C_{T}$ mentioned previously is indeed a
cohomology class, but it can be of independent interest because it
implies that several identities in characteristic classes are valid at
the level of differential forms.

In section \S 4 we study the cohomology of the complex of currents with
a fixed wave front set. The main result of this section is a
Poincar\'e lemma for currents of this kind. This implies in particular
a $\partial\bar \partial$-lemma. The results of this section are
necessary to state the functorial properties of singular Bott-Chern
classes.     

In section \S 5 we recall the deformation of resolutions, that is a
generalization of the deformation to the normal cone, and we also recall
the construction of the Koszul resolution. These are the main
geometric tools used to study singular Bott-Chern classes. 

Sections \S 6 to \S 9 are devoted to the definition and study of the
theories of singular Bott-Chern classes. Section \S 6 contains the
definition and first properties. Section \S 7 is devoted to the
classification theorem of such theories. In section \S 8 we study how
properties of the theory of singular Bott-Chern classes and of the
associated characteristic class are related. And in section \S 9 we
define the theory of homogeneous singular Bott-Chern classes and we
prove that it agrees with the theories defined by Bismut, Gillet and 
Soul\'e and by Zha.   

Finally in section \S 10 we define arithmetic $K$-groups associated to 
a 
$\mathcal{D}_{\log}$-arithmetic variety $(\mathcal{X}, \mathcal{C})$
(in the sense of \cite{BurgosKramerKuehn:cacg}) and push-forward
maps for closed immersions of metrized arithmetic varieties, at the
level of the arithmetic $K$-groups. After studying the compatibility
of these maps with the projection formula and with the push-forward
map at the level of currents, we prove a general Riemann-Roch
theorem for closed immersions (theorem \ref{thm:15}) that compares
the direct images in the arithmetic $K$-groups with the direct
images in the arithmetic Chow groups. This theorem is compatible, if
we 
choose the theory of homogeneous singular Bott-Chern classes, with
the arithmetic Riemann-Roch theorem for closed immersions proved by
Bismut, Gillet and Soul\'e \cite{BismutGilletSoule:MR1086887} and it
agrees with the theorem proved by
Zha  \cite{zha99:_rieman_roch}. Theorem \ref{thm:15}, together with the
arithmetic Grothendieck-Riemann-Roch theorem for submersions 
proved in
\cite{GilletRoesslerSoule:_arith_rieman_roch_theor_in_higher_degrees},
can be used 
to obtain an arithmetic Grothendieck-Riemann-Roch theorem
for projective morphisms of regular arithmetic varieties.

\emph{Acknowledgements}: This project was started during the Special
Year on Arakelov Theory and Shimura Varieties held at the CRM
(Bellaterra, Spain). We would like to thank the CRM for his
hospitality during that year. We would also like to thank the
University of Barcelona and the University Alexandru Ioan Cuza of Ia\c si for
their hospitality during several visits that allowed us to finish the
project. We would also like to thank K. K\"ohler,
J. Kramer, U. K\"uhn, V. Maillot, D. Rossler, and J. Wildeshaus  with
whom we have had many discussions on the subject of this paper. Our
special thanks to G. Freixas and Shun Tang for their careful reading of
the paper and for suggesting some simplifications of the
proofs. Finally we would like to thank the referee for his excellent
work.      

\section{Characteristic classes in analytic Deligne cohomology}
\label{sec:char-class}

A characteristic class for complex vector bundles is a functorial
assignment which, to each 
complex continuous vector bundle on a paracompact topological space $X$,
assigns a class in a suitable cohomology theory of $X$. For example,
if the cohomology theory is
singular cohomology, it is 
well known that each characteristic class can be expressed as a power
series in the Chern classes. This can be seen for instance, showing
that continuous complex vector bundles on a paracompact space $X$ can
be classified by homotopy 
classes of maps from $X$ to the classifying space
$BGL_{\infty}(\mathbb{C})$ and that the cohomology of
$BGL_{\infty}(\mathbb{C})$ is generated by the Chern classes (see for
instance \cite{MilnorStasheff:cc}).

The aim of this section is to show that a similar result is true if we
restrict the class of spaces to the class of quasi-projective smooth
complex manifolds, the class of maps to the class of algebraic maps
and the class of vector bundles to the class of algebraic vector
bundles and we choose analytic Deligne
cohomology as our cohomology theory.  

This result and the techniques used to prove it are standard. We will
use the splitting principle to reduce to 
the case of line bundles and will then use the projective spaces as
a model of the 
classifying space $BGL_{1}(\mathbb{C})$. In this section we also
recall the definition of Chern classes in analytic Deligne cohomology
and 
we fix some notations that will be used through the
paper.

\begin{definition}
  Let $X$ be a complex manifold. For each integer $p$, \emph{the
    analytic real Deligne  complex} of $X$ is
  \begin{multline*}
    \mathbb{R}_{X,\mathcal{D}}(p)= (\underline
    {\mathbb{R}}(p)\longrightarrow 
    \mathcal{O}_{X}\longrightarrow \Omega ^{1}_{X}\longrightarrow \dots 
    \longrightarrow \Omega ^{p-1}_{X})\\
    \cong s(\underline
    {\mathbb{R}}(p)\oplus F^{p}\Omega _{X}^{\ast}\longrightarrow \Omega
    ^{\ast}_{X}), 
  \end{multline*}
  where $\underline
    {\mathbb{R}}(p)$ is the constant sheaf $(2\pi i)^{p}\underline
    {\mathbb{R}}\subseteq \underline {\mathbb{C}}$.
  The \emph{analytic real Deligne cohomology  of $X$}, denoted
  $H^{\ast}_{\mathcal{D}^{\an}}(X,\mathbb{R}(p))$, is the 
  hyper-cohomology of the above complex. 
\end{definition}

Analytic Deligne cohomology satisfies the following result.

\begin{theorem} \label{thm:13}
  The assignment $X\longmapsto
  H^{\ast}_{\mathcal{D}^{\an}}(X,\mathbb{R}(\ast))= 
  \bigoplus_{p}H^{\ast}_{\mathcal{D}^{\an}}(X,\mathbb{R}(p))$ is a
  contravariant functor between the category of complex manifolds and
  holomorphic maps and the category of unitary bigraded rings that are graded
  commutative (with respect to the first degree) and
  associative. Moreover
  there exists a functorial map
    \begin{displaymath}
      c\colon\Pic(X)=H^{1}(X,\mathcal{O}^{\ast}_{X})\longrightarrow 
      H^{2}_{\mathcal{D}^{\an}}(X,\mathbb{R}(1))
    \end{displaymath}
    and, for each closed immersion of complex manifolds
    $i\colon Y\longrightarrow X$ of codimension $p$, there exists a morphism
    \begin{displaymath}
      i_{\ast}\colon H^{*}_{\mathcal{D}^{\an}}(Y,\mathbb{R}(*))\longrightarrow 
      H^{*+2p}_{\mathcal{D}^{\an}}(X,\mathbb{R}(*+p))
    \end{displaymath}
    satisfying the properties
    \begin{enumerate}
    \item [A1] \label{item:20} Let $X$ be a complex manifold and let
      $E$ be a holomorphic 
      vector bundle of rank $r$. Let $\mathbb{P}(E)$ be the associated
      projective 
      bundle and let $\mathcal{O}(-1)$ the tautological line
      bundle. The map
      \begin{displaymath}
        \pi ^{\ast}\colon H^{\ast}_{\mathcal{D}^{\an}}(X,\mathbb{R}(\ast))
        \longrightarrow
        H^{\ast}_{\mathcal{D}^{\an}}(\mathbb{P}(E),\mathbb{R}(\ast)) 
      \end{displaymath}
      induced by the projection $\pi \colon\mathbb{P}(E)\longrightarrow X$
      gives to the second ring a structure of left
      module  over the first. Then
      the elements $c(\cl(\mathcal{O}(-1)))^{i}$, $i=0,\dots ,r-1$
      form a basis of this module.  
  \item [A2] \label{item:21} If $X$ is a complex manifold, $L$  a line
    bundle, $s$ a 
    holomorphic section of $L$ that is transverse to the zero section,
    $Y$ is the zero locus of $s$ and $i\colon Y\longrightarrow X$ the
    inclusion, then 
    \begin{displaymath}
      c(\cl(L))=i _{\ast}(1_{Y}).
    \end{displaymath}
  \item [A3] \label{item:22} If $j\colon Z\longrightarrow Y$ and
    $i\colon Y\longrightarrow X$ are 
    closed immersions of complex manifolds then
    $(ij)_{\ast}=i_{\ast}j_{\ast}$.
  \item [A4] \label{item:23} If $i\colon Y\longrightarrow X$ is a closed
    immersion of 
    complex manifolds then, for every $a\in
    H^{\ast}_{\mathcal{D}^{\an}}(X,\mathbb{R}(\ast))$ and $b\in
    H^{\ast}_{\mathcal{D}^{\an}}(Y,\mathbb{R}(\ast))$ 
    \begin{displaymath}
      i_{\ast}(b i^{\ast}a)=(i_{\ast} b)a.
    \end{displaymath}
  \end{enumerate}
\end{theorem}
\begin{proof}
  The functoriality is clear. The product structure is described, for
  instance, in \cite{EsnaultViehweg:DBc}. The morphism $c$ is defined
  by the morphism in the derived category
  \begin{displaymath}
    \mathcal{O}^{\ast}_{X}[1]\overset{\cong}{\longleftarrow}
    s(\underline {\mathbb{Z}}(1)\rightarrow \mathcal{O}_{X})
    \longrightarrow s(\underline {\mathbb{R}}(1)\rightarrow
    \mathcal{O}_{X}) = \mathbb{R}_{\mathcal{D}}(1).
  \end{displaymath}
  The morphism $i_{\ast}$ can be constructed by resolving the sheaves
  $\mathbb{R}_{\mathcal{D}}(p)$ by means of currents (see
  \cite{Jannsen:DcHD} for a related construction). Properties 
  A3 and A4 follow easily from this
  construction. 

  By abuse of
  notation, we will denote by $c_{1}(\mathcal{O}(-1))$ the first Chern
  class of $\mathcal{O}(-1)$ with the algebro-geometric twist, in any
  of the groups 
  $H^{2}(\mathbb{P}(E),\underline {\mathbb{R}}(1))$,
  $H^{2}(\mathbb{P}(E),\underline {\mathbb{C}})$,
  $H^{1}(\mathbb{P}(E),\Omega ^{1}_{\mathbb{P}(E)})$. Then, we have
  sheaf isomorphisms (see for instance \cite{Gros:MR844488} for a related
  result),
  \begin{align*}
    \bigoplus_{i=0}^{r-1}\underline {\mathbb{R}}_{X}(p-i)[-2i]
    &
    \longrightarrow R\pi _{\ast}\underline {\mathbb{R}}_{\mathbb{P}(E)}(p)\\
    \bigoplus_{i=0}^{r-1}\Omega ^{\ast}_{X}[-2i]
    &
    \longrightarrow R\pi _{\ast}\Omega ^{\ast}_{\mathbb{P}(E)}\\
    \bigoplus_{i=0}^{r-1}F^{p-i}\Omega ^{\ast}_{X}[-2i]
    &
    \longrightarrow R\pi _{\ast}F^{p}\Omega ^{\ast}_{\mathbb{P}(E)}
  \end{align*}
  given, all of them, by $(a_{0},\dots ,a_{r-1})\longmapsto \sum
  a_{i}c_{1}(\mathcal{O}(-1))^{i}$.  
  Hence we obtain a sheaf isomorphism
  \begin{displaymath}
        \bigoplus_{i=0}^{r-1}\mathbb{R}_{X,\mathcal{D}}(p-i)[-2i]
    \longrightarrow R\pi _{\ast}\mathbb{R}_{\mathbb{P}(E),\mathcal{D}}(p)
  \end{displaymath}
  from which property A1 follows. Finally property
  A2 in this context is given by the Poincare-Lelong
  formula (see \cite{BurgosKramerKuehn:cacg} proposition 5.64). 
\end{proof}

\begin{notation}\label{def:19}
  For the convenience of the reader, we gather here together several
  notations and conventions regarding the differential forms, currents
  and Deligne cohomology that will be used through the paper.

  Throughout this paper we will use consistently the algebro-geometric
  twist. In particular the Chern classes $c_{i}$, $i=0,\dots$ in Betti
  cohomology will live in $c_{i}\in H^{2i}(X,\mathbb{R}(i))$; hence
  our normalizations differ from the ones in \cite{GilletSoule:ait}
  where real forms and currents are used.

  Moreover we will use the following notations.  We will denote by
  $\mathscr{E}^{\ast}_{X}$ the sheaf of Dolbeault algebras of
  differential forms on $X$ and by $\mathscr{D}^{\ast}_{X}$ the sheaf
  of Dolbeault complexes of currents on $X$ (see
  \cite{BurgosKramerKuehn:cacg} \S 5.4 for the structure of Dolbeault
  complex of $\mathscr{D}^{\ast}_{X}$).  We will denote by
  $E^{\ast}(X)$ and by $D^{\ast}(X)$ the complexes of global sections
  of $\mathscr{E}^{\ast}_{X}$ and $\mathscr{D}^{\ast}_{X}$
  respectively.  Following \cite{Burgos:CDB} and
  \cite{BurgosKramerKuehn:cacg} definition 5.10, we denote by
  $(\mathcal{D}^{\ast}(\underline{\phantom{A}},\ast),\dd_{\mathcal{D}})$
  the functor that associates to a Dolbeault complex its corresponding
  Deligne complex. For shorthand, we will denote
  \begin{align*}
    \mathcal{D}^{\ast}(X,p)&=\mathcal{D}^{\ast}(E^{\ast }(X),p),\\
    \mathcal{D}^{\ast}_{D}(X,p)&=\mathcal{D}^{\ast}(D^{\ast }(X),p).
  \end{align*}
  To keep track of the algebro-geometric twist we will use the
  conventions of 
  \cite{BurgosKramerKuehn:cacg} \S 5.4 regarding the current
  associated to a locally integrable differential form
  \begin{displaymath}
    [\omega ](\eta)=\frac{1}{(2\pi i)^{\dim X}}\int_{X}\eta\land \omega
  \end{displaymath}
  and the current associated with a subvariety $Y$
  \begin{displaymath}
    \delta _{Y}(\eta)=\frac{1}{(2\pi i)^{\dim Y}}\int_{Y}\eta.
  \end{displaymath}
  With these conventions, we have a bigraded morphism 
  $\mathcal{D}^{\ast}(X,\ast)\to \mathcal{D}^{\ast}_{D}(X,\ast)$ and,
  if $Y$ has codimension $p$, the current $\delta _{Y}$ belongs to 
  $\mathcal{D}^{2p}_{D}(X,p)$.
  Then $\mathcal{D}^{\ast}(X,p)$ and $\mathcal{D}_{D}^{\ast}(X,p)$ are
  the complex of global sections of 
  an acyclic resolution of $\mathbb{R}_{X,\mathcal{D}}(p)$. Therefore
  \begin{displaymath}
    H^{\ast}_{\mathcal{D}^{\an}}(X,\mathbb{R}(p))=
    H^{\ast}(\mathcal{D}(X,p))=H^{\ast}(\mathcal{D}_{D}(X,p)).
  \end{displaymath}
  If $f:X\to Y$ is a proper smooth morphism of complex manifolds of
  relative dimension $e$, then the integral along the fibre morphism
  \begin{displaymath}
    f_{\ast}:\mathcal{D}^{k}(X,p)\longrightarrow
    \mathcal{D}^{k-2e}(X,p-e)
  \end{displaymath}
  is given by
  \begin{equation}\label{eq:36}
    f_{\ast} \omega =\frac{1}{(2\pi i)^{e}}\int_{f}\omega .
  \end{equation}

  If $(\mathcal{D}^{\ast}(\ast),\dd_{\mathcal{D}})$ is a
  Deligne complex associated to a Dolbeault complex, we will write 
  \begin{displaymath}
    \widetilde{\mathcal{D}}^{k}(X,p):=
    \mathcal{D}^{k}(X,p)/\dd_{\mathcal{D}}\mathcal{D}^{k-1}(X,p).
  \end{displaymath}

  Finally, following \cite{BurgosKramerKuehn:cacg} 5.14 we denote by
  $\bullet$ the product in the Deligne complex that induces the usual
  product in Deligne cohomology. Note that, if $\omega \in \bigoplus
  _{p}\mathcal{D}^{2p}(X,p)$, then for any $\eta\in
  \mathcal{D}^{\ast}(X,\ast)$ we have $\omega \bullet \eta=\eta\bullet
  \omega =\eta\land \omega $. Sometimes, in this case we will just
  write $\eta\omega :=\eta\bullet\omega $.
\end{notation}

We denote by $\ast$ the complex  manifold
consisting on one single point. Then
\begin{displaymath}
  H^{n}_{\mathcal{D^{\an}}}(\ast,p)=
  \begin{cases}
    \mathbb{R}(p):=(2\pi i)^{p}\mathbb{R}, & \text{ if }n=0,\ p\le 0,\\
    \mathbb{R}(p-1):=(2\pi i)^{p-1}\mathbb{R}, & \text{ if }n=1,\ p>
    0.\\
    \{0\}, & \text{otherwise.}
  \end{cases}
\end{displaymath}
The product structure in this case is the bigraded
product that is given by complex number multiplication when the
degrees allow the product to be non zero. We will denote by
$\mathbb{D}$ this ring. This is the base ring for analytic Deligne
cohomology. Note that, in particular, $
H^{1}_{\mathcal{D^{\an}}}(\ast,1)=\mathbb{R}=\mathbb{C}/\mathbb{R}(1)$. We
will denote by  
${\bf 1}_{1}$ the image of $1$ in $H^{1}_{\mathcal{D^{\an}}}(\ast,1)$.

Following \cite{Grothendieck:tcc}, theorem \ref{thm:13} implies the
existence of a theory of Chern classes for holomorphic vector bundles
in analytic Deligne cohomology. That is, to every vector bundle $E$,
we can associate a 
collection of Chern classes $c_{i}(E)\in
H^{2i}_{\mathcal{D}^{\an}}(X,\mathbb{R}(i))$, $i\ge 1$ in a functorial
way.
 
We want to see that all possible characteristic classes in analytic
Deligne cohomology can be derived from the Chern classes. 

\begin{definition} \label{def:6}
  Let $n\ge 1$ be an integer and let $r_1\ge 1,\dots ,r_n\ge 1$ be a
  collection of integers. A \emph{theory of characteristic classes
    for $n$-tuples of vector bundles of rank $r_{1},\dots ,r_{n}$} is
  an assignment that, to each 
  $n$-tuple of isomorphism classes of vector bundles $(E_{1},\dots
  ,E_{n})$ over a complex 
  manifold $X$, with
  $\rk(E_{i})=r_{i}$, assigns a class  
  $$\cl(E_{1},\dots ,E_{n})\in \bigoplus _{k,p}
  H^{k}_{\mathcal{D}^{\an}}(X,\mathbb{R}(p))$$ 
  in a functorial way. That is, for every morphism
    $f\colon X\longrightarrow Y$ of 
    complex manifolds, the equality
    \begin{displaymath}
      f^{\ast}(\cl(E_{1},\dots ,E_{n}))=
      \cl(f^{\ast} E_{1},\dots ,f^{\ast} E_{n})
    \end{displaymath}
    holds
\end{definition}

The first consequence of the functoriality and certain homotopy property
of analytic Deligne cohomology classes is the following.

\begin{proposition} \label{prop:12}
  Let $\cl$ be a
  theory of characteristic classes
  for $n$-tuples of vector bundles of rank $r_{1},\dots ,r_{n}$. Let
  $X$ be a complex manifold and let 
  $(E_{1},\dots ,E_{n})$ be a $n$-tuple of vector bundles over $X$ with
  $\rk(E_{i})=r_{i}$ for all $i$. Let $1\le j\le n$ and let
  \begin{displaymath}
    0\longrightarrow E'_{j}\longrightarrow  E_{j}
    \longrightarrow E''_{j}\longrightarrow 0,
  \end{displaymath}
  be a short exact sequence. Then 
  the equality
    \begin{displaymath}
      \cl(E_{1},\dots, E_{j},\dots ,E_{n})=
      \cl(E_{1},\dots, E'_{j}\oplus E''_{j},\dots  ,E_{n}) 
    \end{displaymath}
    holds.
\end{proposition}
\begin{proof}
  Let $\iota_{0},\iota_{\infty}\colon X\longrightarrow X\times
  \mathbb{P}^{1}$ be the inclusion as the fiber over $0$ and the fiber
  over $\infty$ respectively. Then there exists a vector bundle
  $\widetilde E_{j}$ on $X\times \mathbb{P}^{1}$ (see for instance
  \cite{GilletSoule:vbhm}  (1.2.3.1) or definition \ref{def:12} below)
  such that 
  $\iota^{\ast}_{0}\widetilde E_{j}\cong E_{j}$ and
  $\iota^{\ast}_{\infty}\widetilde E_{j}\cong E'_{j}\oplus E''_{j}$. Let
  $p_{1}\colon X\times 
  \mathbb{P}^{1}\longrightarrow X$ be the first projection. Let
  $\omega\in \bigoplus _{k,p}\mathcal{D}^{k}(X,p)$ be any
  $\dd_{\mathcal{D}}$-closed form
  that represents  $\cl(p_{1}^{\ast}E_{1},\dots,
  \widetilde E_{j},\dots ,p_{1}^{\ast} E_{n})$. Then, by functoriality
  we know that $\iota_{0}^{\ast}\omega $ represents $ \cl(E_{1},\dots,
  E_{j},\dots ,E_{n}) $ and  $\iota_{\infty}^{\ast}\omega $ represents
  $\cl(E_{1},\dots, E'_{j}\oplus E''_{j},\dots  ,E_{n})$. We write
  \begin{displaymath}
    \beta =\frac{1}{2\pi i}\int_{\mathbb{P}^{1}}
    \frac{-1}{2}\log t\bar t\bullet \omega,
  \end{displaymath}
  where $t$ is the absolute coordinate of $\mathbb{P}^{1}$. Then
  \begin{displaymath}
    \dd_{\mathcal{D}}\beta 
    = \iota_{\infty}^{\ast}\omega -\iota^{\ast}_{0}\omega 
  \end{displaymath}
  which implies the result.
\end{proof}

A standard method to produce characteristic classes for vector bundles
is to choose hermitian metrics on the vector bundles and to construct
closed differential forms out of them. The following result shows that
functoriality implies that the cohomology classes represented by these
forms are independent from the hermitian metrics and therefore are
characteristic classes. When working with hermitian vector bundles we
will use the convention that, if $E$ denotes the vector bundle, then
$\overline E=(E,h)$ will denote the vector bundle together with the
hermitian metric.

\begin{proposition} \label{prop:22}
  Let $n\ge 1$ be an integer and let $r_1\ge 1,\dots ,r_n\ge 1$ be a
  collection of integers.
  Let $\cl$ be an assignment that, to each $n$-tuple
  $(\overline E_{1},\dots,\overline
  E_{n})=((E_{1},h_{1}),\dots,(E_{n},h_{n}))$ of isometry classes of
  hermitian   
  vector bundles of rank $r_{1},\dots,r_{n}$ over a complex manifold
  $X$, associates a cohomology class 
  \begin{displaymath}
    \cl (\overline
  E_{1},\dots,\overline E_{n})\in \bigoplus_{k,p}
  H_{\mathcal{D}}^{k}(X,\mathbb{R}(p))
  \end{displaymath}
  such that, for each morphism $f:Y\to X$,
  \begin{displaymath}
    \cl (f^{\ast}\overline
  E_{1},\dots,f^{\ast} \overline E_{n}) = f^{\ast}\cl (\overline
  E_{1},\dots,\overline E_{n}).
  \end{displaymath}
Then the cohomology class $\cl (\overline
  E_{1},\dots,\overline E_{n})$
  is independent from the hermitian metrics. Therefore it is a well
  defined characteristic class.
\end{proposition}
\begin{proof}
  Let $1\le j \le n$ be an integer and let $\overline
  E'_{j}=(E_{j},h'_{j})$ be the 
  vector bundle underlying $\overline E_{j}$ with a different choice of
  metric. Let $\iota_{0}$, $\iota_{\infty}$ and $p_{1}$ be as in the
  proof of proposition 
  \ref{prop:12}. Then we can choose a hermitian metric $h$ on
  $p_{1}^{\ast}E_{j}$,  such that $\iota
  _{0}^{\ast}(p_{1}^{\ast}E_{j},h)=\overline E_{j}$ and $\iota
  _{\infty}^{\ast}(p_{1}^{\ast}E_{j},h)=\overline E_{j}'$. Let $\omega
  $ be any smooth closed differential form on $X\times \mathbb{P}^{1}$
  that represents  
  $
    \cl (p_{1}^{\ast}\overline
    E_{1},\dots,(p_{1}^{\ast}E_{1},h),\dots,p_{1}^{\ast}\overline
    E_{n}).
  $
  Then,
  \begin{displaymath}
    \beta =\frac{1}{2\pi i}\int_{\mathbb{P}^{1}}
    \frac{-1}{2}\log t\bar t\bullet \omega
  \end{displaymath}
  satisfies
  \begin{displaymath}
    \dd_{\mathcal{D}}\beta 
    = \iota_{\infty}^{\ast}\omega -\iota^{\ast}_{0}\omega 
  \end{displaymath}
  which implies the result.
\end{proof}

We are interested in vector bundles that can be extended to a
projective variety. Therefore we will restrict ourselves to the
algebraic category. So,
by a complex algebraic manifold we will mean the complex manifold
associated to a smooth quasi-projective variety over
$\mathbb{C}$. When working with an algebraic manifold,
by a vector bundle  we will mean
the holomorphic vector bundle associated to an algebraic
vector bundle.

We will denote
by $\mathbb{D}[[x_{1},\dots ,x_{r}]]$ the ring of commutative formal
power series. That is, the unknowns $x_{1},\dots ,x_{r}$ commute with
each other and with $\mathbb{D}$. We turn it into a commutative
bigraded ring by declaring that the unknowns $x_{i}$ have bidegree
$(2,1)$.  The symmetric group in $r$ elements, 
$\mathfrak{S}_{r}$ 
acts on $\mathbb{D}[[x_{1},\dots ,x_{r}]]$. The subalgebra of
invariant elements is generated over $\mathbb{D}$ by the elementary
symmetric functions.   
The main result of this section is the following
\begin{theorem} \label{thm:14}
  Let $\cl$ be a theory of characteristic classes for $n$-tuples of
  vector bundles of rank $r_{1},\dots ,r_{n}$. Then, there is a power
  series $\varphi \in \mathbb{D}[[x_{1},\dots ,x_{r}]]$ in $r=r_{1}+\dots
  +r_{n}$ variables with 
  coefficients in the ring $\mathbb{D}$,
  such that, for each complex algebraic manifold $X$ and each
  $n$-tuple of
  algebraic vector bundles $(E_{1},\dots 
  ,E_{n})$ over $X$ with $\rk(E_{i})=r_{i}$ this equality holds:  
  \begin{equation}\label{eq:2}
    \cl(E_{1},\dots ,E_{n})=
    \varphi(c_{1}(E_{1}),\dots ,c_{r_{1}}(E_{1}),\dots ,
    c_{1}(E_{n}),\dots ,c_{r_{n}}(E_{n})).
  \end{equation}
  Conversely, any power series $\varphi$ as before
  determines a theory of characteristic classes for 
  $n$-tuples  of vector bundles of rank $r_{1},\dots ,r_{n}$, by
  equation \eqref{eq:2}.
\end{theorem}

\begin{proof}
  The second statement is obvious from the properties of Chern
  classes. 

  Since we are assuming $X$ quasi-projective, given $n$ algebraic
  vector bundles $E_{1}, \dots ,E_{n}$ on $X$, there is a 
  smooth projective compactification $\widetilde X$ and vector bundles
  $\widetilde E_{1},\dots ,\widetilde E_{n}$ on $\widetilde X$, such
  that $E_{i}=\widetilde E_{i}|_{X}$ (see for instance
  \cite{BurgosWang:hBC} proposition 2.2), we are reduced to the case when
  $X$ is projective. In this case, analytic
  Deligne cohomology agrees with ordinary Deligne cohomology.

  Let us assume first that $r_{1}=\dots =r_{n}=1$ and that we have a
  characteristic class $\cl$ for $n$ line bundles. Then, for each
  $n$-tuple of positive integers $m_{1},\dots ,m_{n}$ we consider the
  space $\mathbb{P}^{m_{1},\dots
    ,m_{n}}=\mathbb{P}^{m_{1}}_{\mathbb{C}}\times \dots \times 
  \mathbb{P}^{m_{n}}_{\mathbb{C}}$ and we denote by $p_{i}$ the projection over
  the $i$-th factor. Then 
  $$\left.
  \bigoplus
  _{k,p}H^{k}_{\mathcal{D}}(\mathbb{P}^{m_{1},\dots
    ,m_{n}},\mathbb{R}(p))=\mathbb{D}[x_{1},\dots ,x_{n}]
  \right/
  (x_{1}^{m_{1}},\dots ,x_{n}^{m_{n}})
  $$ 
  is a quotient of the 
  polynomial ring generated by the classes
  $x_{i}=c_{1}(p_{i}^{\ast}\mathcal{O}(1))$ with coefficients in the
  ring $\mathbb{D}$.
  Therefore, there is a
  polynomial $\varphi_{m_{1},\dots ,m_{n}}$ in $n$ variables such that 
  \begin{displaymath}
    \cl(p_{1}^{\ast}\mathcal{O}(1),\dots ,p_{1}^{\ast}\mathcal{O}(1))=
    \varphi_{m_{1},\dots ,m_{n}}(x_{1},\dots ,x_{n}).
  \end{displaymath}
  If $m_{1}\le m'_{1}$, \dots , $m_{n}\le m_{n}'$ then, by
  functoriality, the polynomial $\varphi_{m_{1},\dots ,m_{n}}$
  is the truncation of the polynomial $\varphi_{m'_{1},\dots
    ,m'_{n}}$. Therefore there is a power series in $n$ variables,
  $\varphi $ such that $\varphi_{m_{1},\dots ,m_{n}}$ is the truncation
  of $\varphi $ in the appropriate quotient of the polynomial ring. 

  Let
  $L_{1}, \dots ,L_{n}$ be line bundles on a 
  projective algebraic manifold that are generated by global
  sections. Then they determine a morphism $f\colon X\longrightarrow
  \mathbb{P}^{m_{1},\dots ,m_{n}}$ such that
  $L_{i}=f^{\ast}p^{\ast}_{i}\mathcal{O}(1)$. Therefore, again by
  functoriality, we obtain
  \begin{displaymath}
    \cl(L_{1},\dots ,L_{n})=
    \varphi(c_{1}(L_{1}),\dots ,c_{1}(L_{n})).
  \end{displaymath}
 
  From the class $\cl$ we can define a new characteristic class for $n+1$ line
  bundles by the formula
  \begin{displaymath}
    \cl'(L_{1},\dots ,L_{n},M)=\cl(L_{1}\otimes M^{\vee},\dots
    ,L_{n}\otimes M^{\vee}). 
  \end{displaymath}
  When $L_{1},\dots ,L_{n}$ and $M$ are generated by global sections
  we have that there is a power series $\psi $ such that
  \begin{displaymath}
    \cl'(L_{1},\dots ,L_{n},M)=
    \psi (c_{1}(L_{1}),\dots ,c_{1}(L_{n}),c_{1}(M)).
  \end{displaymath}
  Moreover, when the line bundles $L_{i}\otimes M^{\vee}$ are also
  generated by global 
  sections the following holds
  \begin{align*}
    \psi (c_{1}(L_{1}),\dots
    ,c_{1}(L_{n}),c_{1}(M))&=\varphi(c_{1}(L_{1}\otimes M^{\vee}),\dots
    ,c_{1}(L_{n}\otimes M^{\vee}))\\
    &=\varphi(c_{1}(L_{1})-c_1(M),\dots
    ,c_{1}(L_{n})-c_{1}(M)). 
  \end{align*}
  Considering the system of spaces $\mathbb{P}^{m_{1},\dots
    ,m_{n},m_{n+1}}$ with line bundles
  \begin{displaymath}
    L_{i}=p_{i}^{\ast}\mathcal{O}(1)\otimes
    p_{n+1}^{\ast}\mathcal{O}(1),\ i=1,\dots ,n,\quad
    M=p_{n+1}^{\ast}\mathcal{O}(1),
  \end{displaymath}
  we see that there is an identity of power series
  \begin{displaymath}
    \varphi(x_{1}-y,\dots ,x_{n}-y)=\psi (x_{1},\dots, x_{n},y).
  \end{displaymath}
  Now let $X$ be a projective complex manifold and let $L_{1}, \dots ,
  L_{n}$ be arbitrary line bundles. Then there is a line bundle $M$
  such that $M$ and $L_{i}'=L_{i}\otimes M$, $i=1,\dots ,n$ are
  generated by global sections. Then we have
  \begin{align*}
    \cl(L_{1},\dots ,L_{n})&=\cl(L'_{1}\otimes M^{\vee},\dots
    ,L'_{n}\otimes M^{\vee})\\
    &=\cl'(L'_{1},\dots ,L'_{n},M)\\
    &=\psi (c_{1}(L'_{1}),\dots ,c_{1}(L'_{n}),c_{1}(M))\\
    &=\varphi((c_{1}(L'_{1})-c_{1}(M),\dots
    ,c_{1}(L'_{n})-c_{1}(M)))\\
    &=\varphi(c_{1}(L_{1}),\dots ,c_{1}(L_{n})).
  \end{align*}

  The case  of arbitrary rank vector bundles follows from the case of rank
  one vector bundles by proposition \ref{prop:12} and the splitting
  principle. We next recall the argument. Given a projective complex
  manifold $X$ and vector bundles $E_{1},\dots ,E_{n}$ of rank
  $r_{1},\dots ,r_{n}$, we can find a proper morphism
  $\pi \colon\widetilde X\longrightarrow X$, with $\widetilde X$ a complex
  projective manifold, and such that the induced morphism
  \begin{displaymath}
    \pi
    ^{\ast}\colon H^{\ast}_{\mathcal{D}}(X,\mathbb{R}(\ast))\longrightarrow 
    H^{\ast}_{\mathcal{D}}(\widetilde X,\mathbb{R}(\ast))
  \end{displaymath}
  is injective and every bundle $\pi ^{\ast}(E_{i})$ admits a
  holomorphic filtration
  \begin{displaymath}
    0= K_{i,0}\subset K_{i,1}\subset \dots \subset
    K_{i,r_{i}-1}\subset K_{i,r_{i}}=\pi ^{\ast}(E_{i}),
  \end{displaymath}
  with $L_{i,j}=K_{i,j}/K_{i,j-1}$ a line bundle. If $\cl$ is a
  characteristic class for $n$-tuples of vector bundles of rank
  $r_{1},\dots ,r_{n}$, we define a characteristic class for
  $r_{1}+\dots +r_{n}$-tuples of line bundles by the formula
  \begin{multline*}
    \cl'(L_{1,1},\dots ,L_{1,r_{1}},\dots ,L_{n,1},\dots
    ,L_{n,r_{n}})=\\
    \cl(L_{1,1}\oplus \dots \oplus L_{1,r_{1}},\dots ,L_{n,1}\oplus
    \dots \oplus, L_{n,r_{n}}).
  \end{multline*}
  By the case of line bundles we know that there is a power series in
  $r_{1}+\dots +r_{n}$ variables $\psi$ such that
  \begin{displaymath}
    \cl'(L_{1,1},\dots ,L_{1,r_{1}},\dots ,L_{n,1},\dots
    ,L_{n,r_{n}})=\psi(c_{1}(L_{1,1}),\dots ,c_{1}(L_{n,r_{n}})).
  \end{displaymath}
  Since the class $\cl'$ is symmetric under the group
  $\mathfrak{S}_{r_{1}}\times \dots \times \mathfrak{S}_{r_{n}}$, the
  same is true for the power series $\psi$. Therefore $\psi$ can
  be written in terms of symmetric elementary functions. That is,
  there is another power series in $r_{1}+\dots +r_{n}$ variables
  $\varphi $, such that
  \begin{multline*}
    \psi(x_{1,1},\dots ,x_{n,r_{n}})=
    \varphi(s_{1}(x_{1,1},\dots ,x_{1,r_{1}}),\dots
    ,s_{r_{1}}(x_{1,1},\dots ,x_{1,r_{1}}),\dots\\
    \dots ,s_{1}(x_{n,1},\dots ,x_{n,r_{n}}),\dots
    ,s_{r_{n}}(x_{n,1},\dots ,x_{n,r_{n}})),
  \end{multline*}
  where $s_{i}$ is the $i$-th elementary symmetric function of the
  appropriate number of variables.
  Then
  \begin{align*}
    \pi ^{\ast}(\cl(E_{1},\dots , E_{n}))&=
    \cl(\pi ^{\ast}E_{1},\dots , \pi ^{\ast} E_{n}))\\
    &=\cl'(L_{1,1},\dots ,L_{n,r_{n}})\\
    &= \psi (c_{1}(L_{1,1}),\dots ,c_{1}(L_{n,r_{n}}))\\
    &= \varphi(c_{1}(\pi ^{\ast}E_{1}),\dots ,c_{r_{1}}(\pi
    ^{\ast}E_{1}),\dots ,c_{1}(\pi ^{\ast}E_{n}),\dots ,c_{r_{n}}(\pi
    ^{\ast}E_{n}))\\
    &=\pi ^{\ast} \varphi(c_{1}(E_{1}),\dots ,c_{r_{1}}(E_{1}),\dots
    ,c_{1}(E_{n}),\dots ,c_{r_{n}}(E_{n})).
  \end{align*}
  Therefore, the result follows from the injectivity of $\pi ^{\ast}$.
\end{proof}

\begin{remark}\label{rem:3}
  It would be interesting to know if the functoriality of a
  characteristic class in enough to imply that it is a power series in
  the Chern classes for arbitrary complex manifolds and holomorphic
  vector bundles. 
\end{remark}

\section{Bott-Chern classes}
\label{sec:bott-chern-forms}

The aim of this section is to recall the theory of Bott-Chern
classes. For more details we 
refer the reader to \cite{BottChern:hvb}, \cite{BismutGilletSoule:at},
\cite{GilletSoule:vbhm}, \cite{Soule:lag}, \cite{BurgosWang:hBC},
\cite{Burgos:hvbcc} 
and \cite{BurgosKramerKuehn:accavb}. Note however that the theory we
present here is equivalent, although not identical, to the different
versions that appear in the literature.

Let $X$ be a complex manifold and let $\overline {E}=(E,h)$ be a rank
$r$ holomorphic vector bundle provided with a hermitian metric. Let
$\phi \in \mathbb{D}[[x_{1},\dots ,x_{r}]]$ be a formal power series
in $r$ variables that is symmetric under the action of
$\mathfrak{S}_{r}$. Let $s_{i}$, 
$i=1,\dots ,r$ be the elementary symmetric functions in $r$
variables. Then $\phi (x_{1},\dots ,x_{r})=\varphi(s_{1},\dots
,s_{r})$ for certain power series $\varphi$. By Chern-Weil theory we
can obtain a representative of the 
class 
\begin{displaymath}
  \phi (E):=\varphi(c_{1}(E),\dots ,c_{r}(E))\in \bigoplus
  _{k,p}H^{k}_{\mathcal{D}^{\an}}(X,\mathbb{R}(p))   
\end{displaymath}
as follows.

We denote also by $\phi $ the invariant power series in $r\times r$
matrices defined by $\phi $. Let $K$ be the curvature matrix of the
hermitian holomorphic connection of $(E,h)$. The entries of $K$ in a
particular trivialization of $E$ are local sections 
of $\mathcal{D}^{2}(X,1)$. Then we write
\begin{displaymath}
  \phi (E,h)=\phi (-K)\in  \bigoplus _{k,p}\mathcal{D}^{k}(X,p).
\end{displaymath}
The form $\phi (E,h)$ is well defined, closed, and it represents the class
$\phi(E)$. 
  
Now let 
\begin{displaymath}
\overline E_{\ast}=(\dots  \overset {f_{n+1}}{\longrightarrow }\overline
    E_{n}\overset {f_{n}}{\longrightarrow }\overline E_{n-1}
\overset {f_{n-1}}{\longrightarrow }\dots )
\end{displaymath}
be a bounded acyclic complex of hermitian vector bundles; by this we
mean 
a bounded acyclic complex of vector bundles, where each vector bundle is
equipped with an arbitrarily chosen hermitian metric. 

Write
\begin{displaymath}
  r=\sum_{i\text{ even}} \rk(E_{i})=\sum_{i\text{ odd}} \rk(E_{i}). 
\end{displaymath}
and let $\phi$ be a
symmetric power series in $r$ variables.

As before, we can define the Chern forms $$\phi
(\bigoplus_{i\text{ even}}(E_{i},h_{i})) \text{ and } 
\phi
(\bigoplus_{i\text{ odd}}(E_{i},h_{i})),$$ that represent the Chern
classes $\phi 
(\bigoplus_{i\text{ even}}E_{i})$ and
$\phi
(\bigoplus_{i\text{ odd}}E_{i})$. 
The Chern classes are compatible with respect to exact sequences, that is,
\begin{displaymath}
\phi
(\bigoplus_{i\text{ even}}E_{i})=\phi 
(\bigoplus_{i\text{ odd}}E_{i}).
\end{displaymath}
But, in general, this is not true for the Chern forms. This lack of
compatibility with exact sequences on the level of Chern forms is
measured by the Bott-Chern 
classes. 

\begin{definition}
  Let 
$$\overline E_{\ast}=(\dots  \overset {f_{n+1}}{\longrightarrow }\overline
    E_{n}\overset {f_{n}}{\longrightarrow }\overline E_{n-1}
\overset {f_{n-1}}{\longrightarrow }\dots )
$$ 
be an acyclic complex of hermitian vector 
bundles, we
will say that $\overline E_{\ast}$ is an \emph{orthogonally 
  split complex} of vector bundles if, for any integer $n$, the exact sequence
\begin{displaymath}
  0 \longrightarrow \Ker f_{n} \
  \longrightarrow \overline E_{n}
  \longrightarrow \Ker f_{n-1} \longrightarrow 0
\end{displaymath}
is split, there is a splitting section $s_{n}\colon\Ker
f_{n-1} \to  E_{n}$ 
such that $\overline E_{n}$ is the orthogonal direct sum of $\Ker
f_{n}$ and $\Img s_{n}$ and the metrics induced in the subbundle
$\Ker f_{n-1} $ by the inclusion $\Ker f_{n-1}\subset \overline E_{n-1}$ and
by the section $s_{n}$ agree. 
\end{definition}

\begin{notation}
  \label{def:3}
  Let $(x:y)$ be homogeneous coordinates of $\mathbb{P}^{1}$ and let
  $t=x/y$ be the absolute coordinate. In order to make certain choices
  of metrics in a functorial way, we fix once and for all a partition of
  unity $\{\sigma _{0},\sigma _{\infty}\}$, over $\mathbb{P}^{1}$
  subordinated to the open cover of $\mathbb{P}^{1}$ given by the open
  subsets $\left\{\{|y|>1/2|x|\},
    \{|x|>1/2|y|\}\right\}$. As usual we will write $\infty =(1:0)$, $0=(0:1)$.
\end{notation}

The fundamental result of the theory of Bott-Chern classes is the
following theorem (see \cite{BottChern:hvb}, \cite{BismutGilletSoule:at},
\cite{GilletSoule:vbhm}).

\begin{theorem}
\label{thm:3}
There is a unique way to attach to each bounded exact complex
$\overline{E}_{\ast}$ as above, a class $\widetilde{\phi}(\overline{E}_{\ast})$ in
\begin{displaymath}
\bigoplus_{k}\widetilde{\mathcal{D}}^{2k-1}(X,k)=\bigoplus_{k}
\mathcal{D}^{2k-1}(X,k)/\Img(\dd_{\mathcal{D}})
\end{displaymath}
satisfying the following properties
\begin{enumerate}
\item \label{item:12} (Differential equation)
  \begin{equation}
    \label{eq:13}
    \dd_{\mathcal{D}}\widetilde{\phi}(\overline{E}_{\ast})=
    \phi
    (\bigoplus_{i\text{ even}}(E_{i},h_{i}))- 
    \phi
    (\bigoplus_{i\text{ odd}}(E_{i},h_{i})).   
  \end{equation}
\item \label{item:13} (Functoriality)
$f^{\ast}\widetilde{\phi}(\overline{E}_{\ast})=\widetilde{\phi}
(f^{\ast}\overline{E}_{\ast})$, for every holomorphic map $f\colon X'
\longrightarrow X$. 
\item \label{item:14} (Normalization)
If $\overline {E}_{\ast}$ is orthogonally split, then $\widetilde 
{\phi}(\overline{E}_{\ast})=0$.
\end{enumerate}
\end{theorem}

\begin{proof}
We first recall how to prove the uniqueness. 

Let $\overline K_{i}=(K_{i},g_{i})$, where $K_{i}=\Ker f_{i}$
and $g_{i}$ is the metric induced by the inclusion $K_{i}\subset
E_{i}$. Consider the complex manifold $X\times
\mathbb{P}^{1}$ with projections $p_{1}$ and $p_{2}$. For every vector
bundle $F$ on $X$ we will denote $F(i)=p_{1}^{\ast}F\otimes
p_{2}^{\ast}\mathcal{O}_{\mathbb{P}^{1}}(i)$. Let $\widetilde 
C_{\ast}=\widetilde C(E_{\ast})_{\ast}$ be the complex
of vector bundles on $X\times \mathbb{P}^{1}$ given by
$\widetilde C_{i}=E_{i}(i)\oplus E_{i-1}(i-1)$ with differential
$d(s,t)=(t,0)$. Let $\widetilde D_{\ast}=\widetilde D(E_{\ast})_{\ast}$ be the complex
of vector bundles with 
$\widetilde D_{i}=E_{i-1}(i)\oplus E_{i-2}(i-1)$ and differential
$d(s,t)=(t,0)$. Using notation
\ref{def:3} we define the map $\psi 
\colon\widetilde C(E_{\ast})_{i}\longrightarrow \widetilde
D(E_{\ast})_{i}$ given by $\psi
(s,t)=(f_{i}(s)-t\otimes y,f_{i-1}(t))$. It is a morphism of complexes.
\begin{definition}\label{def:12}
  The \emph{first transgression exact sequence} of $E_{\ast}$ is given by
  \begin{displaymath}
    \tr_{1}(E_{\ast})_{\ast}=\Ker \psi.
  \end{displaymath}
\end{definition}

On $X\times \mathbb{A}^{1}$, the
map $p_{1}^{\ast}E_{i}\longrightarrow \widetilde C(E_{\ast})_{i}$ given by
$s\longmapsto
(s\otimes y^{i},f_{i}(s)\otimes y^{i-1})$ induces an isomorphism of
complexes
\begin{equation}
  \label{eq:59}
  p_{1}^{\ast}E_{\ast} \longrightarrow \tr_{1}(E_{\ast})_{\ast}|_{X\times
  \mathbb{A}^{1}}, 
\end{equation}
 and in particular isomorphisms
\begin{equation}
  \label{eq:57}
  \tr_{1}(E_{\ast})_{i}|_{X\times
  \{0\}}\cong E_{i}.
\end{equation}
Moreover, we have isomorphisms
\begin{equation}
  \label{eq:58}
  \tr_{1}(E_{\ast})_{i}|_{X\times
  \{\infty \}}\cong K_{i}\oplus K_{i-1}.
\end{equation}

\begin{definition}\label{def:13} We will denote by $\tr_{1}(\overline
  E_{\ast})_{\ast}$ the complex $\tr_{1}(E_{\ast})_{\ast}$ provided with any
  hermitian metric such that the isomorphisms \eqref{eq:57} and
  \eqref{eq:58} are isometries. If we need a functorial choice of
  metric, we proceed as follows. On $X\times
  (\mathbb{P}^{1}\setminus \{0\})$ we consider the metric induced by
  $\widetilde C$  on
  $\tr_{1}(E_{\ast})_{\ast}$. On $X\times
  (\mathbb{P}^{1}\setminus \{\infty\})$ we consider the metric induced
  by the isomorphism \eqref{eq:59}. We glue both metrics by means of
  the partition of unity of notation \ref{def:3}.
\end{definition}
In particular, we have that $\tr_{1}(\overline E_{\ast})|_{X\times
  \{\infty \}}$ is orthogonally split.  
We assume that there
exists a theory of Bott-Chern classes satisfying the above
properties. 
 Thus, there exists a class of
differential forms
$\widetilde{\phi}(\tr_{1}(\overline{E}_{\ast})_{\ast})$ with the
following properties. By
\ref{item:12} this class satisfies
\begin{displaymath}
  \dd_{\mathcal{D}}\widetilde{\phi}(\tr_{1}(\overline{E}_{\ast})_{\ast})= 
  \phi(\bigoplus_{i\text{ even}}\tr_{1}(\overline{E}_{\ast})_{i}))- 
  \phi
  (\bigoplus_{i\text{ odd}}\tr_{1}(\overline{E}_{\ast})_{i}).  
\end{displaymath}
By \ref{item:13}, it satisfies
\begin{displaymath}
  \widetilde{\phi}(\tr_{1}(\overline{E}_{\ast})_{\ast})\mid_{X\times \{0\}}=
  \widetilde{\phi}(\tr_{1}(\overline{E}_{\ast})_{\ast}\mid_{X\times
    \{0\}})=\widetilde{\phi}(\overline{E}_{\ast}).
\end{displaymath}
Finally, by \ref{item:13} and~\ref{item:14} it satisfies
\begin{displaymath}
  \widetilde{\phi}(\tr_{1}(\overline{E}_{\ast})_{\ast})\mid_{X\times 
    \{\infty\}  } = 
  \widetilde{\phi}(\tr_{1}(\overline{E}_{\ast})_{\ast}\mid_{X\times
    \{\infty\}})=0.
\end{displaymath}

Let $\phi(\tr_{1}(\overline{E}_{\ast})_{\ast})$ be any
representative of the class
$\widetilde{\phi}(\tr_{1}(\overline{E}_{\ast})_{\ast})$. 

Then, in the group
$\bigoplus_{k}\widetilde{\mathcal{D}}^{2k-1}(X,k)$, we have
\begin{align*}
  0
  &=\dd_{\mathcal{D}} \frac{1}
  {2\pi
    i}\int_{\mathbb{P}^{1}}\frac{-1}{2}\log(t\bar{t}) \bullet
  \phi(\tr_{1}(\overline{E}_{\ast})_{\ast}) 
  \\
  &=\frac{1}
  {2\pi
    i}\int_{\mathbb{P}^{1}}\left(\dd_{\mathcal{D}}\frac{-1}{2}\log(t\bar{t})
  \bullet \phi(\tr_{1}(\overline{E}_{\ast})_{\ast}) -
  \frac{-1}{2}\log(t\bar{t})\bullet
  \dd_{\mathcal{D}}\phi(\tr_{1}(\overline{E}_{\ast})_{\ast})\right)\\ 
  &= \widetilde{\phi}(\tr_{1}(\overline{E}_{\ast})_{\ast})|_{X\times
    \{\infty\}}
  -\widetilde{\phi}(\tr_{1}(\overline{E}_{\ast})_{\ast})|_{X\times
    \{0\}}\\
  &\phantom{\ }-\frac{1}
  {2\pi
    i}\int_{\mathbb{P}^{1}}\frac{-1}{2}\log(t\bar{t}) \bullet 
  (\phi(\bigoplus_{i\text{ even}}\tr_{1}(\overline{E}_{\ast})_{i})- 
  \phi
  (\bigoplus_{i\text{ odd}}\tr_{1}(\overline{E}_{\ast})_{i}))\\
  &=  -\widetilde{\phi}(\overline{E}_{\ast}) -\frac{1}
  {2\pi
    i}\int_{\mathbb{P}^{1}}\frac{-1}{2}\log(t\bar{t}) \bullet 
  (\phi(\bigoplus_{i\text{ even}}\tr_{1}(\overline{E}_{\ast})_{i})- 
  \phi
  (\bigoplus_{i\text{ odd}}\tr_{1}(\overline{E}_{\ast})_{i})
  ).
\end{align*}

Hence, if such a theory exists, it should satisfy the formula
\begin{equation}
  \label{eq:7}
  \widetilde{\phi}(\overline{E}_{\ast})=
 \frac{1}
  {2\pi
    i}\int_{\mathbb{P}^{1}}\frac{-1}{2}\log(t\bar{t}) \bullet 
  (\phi(\bigoplus_{i\text{ odd}}\tr_{1}(\overline{E}_{\ast})_{i})- 
  \phi
  (\bigoplus_{i\text{ even}}\tr_{1}(\overline{E}_{\ast})_{i})
  ).
\end{equation}
Therefore $ \widetilde{\phi}(\overline{E}_{\ast})$ is determined by properties (i), (ii) and
(iii). 

In order to prove the existence of a theory of functorial Bott-Chern
forms, we have to see that the right hand 
side of equation
\eqref{eq:7} is independent from the choice of the metric on
$\tr_{1}(\overline{E}_{\ast})_{\ast}$  and that it satisfies
the properties (i), (ii) and
(iii). For this the reader can follow the proof of
\cite{BismutGilletSoule:at} theorem 
1.29. 

\end{proof}

In view of the proof of theorem \ref{thm:3}, we can define the
Bott-Chern classes as follows.

\begin{definition}
  \label{def:1}
  Let 
  \begin{displaymath}
    \overline{E}_{\ast}\colon 0\longrightarrow(E_{n},h_{n})\longrightarrow
    \dots \longrightarrow (E_{1},h_{1}) 
    \longrightarrow(E_{0},h_{0})\longrightarrow 0
  \end{displaymath}
  be a bounded acyclic complex of hermitian vector bundles. Let
  \begin{displaymath}
    r=\sum_{i\text{ even}} \rk(E_{i})=\sum_{i\text{ odd}} \rk(E_{i}). 
  \end{displaymath}
  Let $\phi\in \mathbb{D}[[x_{1},\dots ,x_{r}]]^{\mathfrak{S}_{r}}$ be
  a symmetric power series in $r$ variables. Then the 
  \emph{Bott-Chern class} associated to $\phi $ and
  $\overline{E}_{\ast}$ is the element of
  $\bigoplus_{k,p}\widetilde{\mathcal{D}}^{k}(E_{X},p)$ given by
  \begin{displaymath}
    \widetilde{\phi}(\overline{E_{\ast}})=
    \frac{1}
    {2\pi
      i}\int_{\mathbb{P}^{1}}\frac{-1}{2}\log(t\bar{t}) \bullet 
    (\phi(\bigoplus_{i\text{ odd}}\tr_{1}(\overline{E}_{\ast})_{i})- 
    \phi
    (\bigoplus_{i\text{ even}}\tr_{1}(\overline{E}_{\ast})_{i})
    ).
  \end{displaymath}
\end{definition}

The following property is obvious from the definition.
\begin{lemma}
  Let $\overline E_{\ast}$ be an acyclic complex of hermitian
  vector bundles. Then, for any integer $k$,
  \begin{displaymath}
    \widetilde{\phi}(\overline{E}_{\ast
    }[k])=(-1)^{k}\widetilde{\phi}(\overline{E}_{\ast }). 
  \end{displaymath}
\end{lemma}
\hfill$\square$

Particular cases of Bott-Chern classes are obtained when we consider a
single vector bundle with two different hermitian metrics or a short
exact sequence of vector bundles. Note however that, in order to fix
the sign of the Bott-Chern classes on these cases, one has to
choose the degree of the vector bundles involved, for instance as in
the next definition.  
\begin{definition} \label{def:11} 
Let $E$ be
a holomorphic vector bundle of rank $r$, let $h_{0}$ and  $h_{1}$
be two hermitian metrics and let $\phi $ be an invariant power series of
$r$ variables. We will denote by $\widetilde \phi (E,h_{0},h_{1})$ the
Bott-Chern class associated to the complex
\begin{displaymath}
  \overline {\xi}\colon 0\longrightarrow (E,h_{1})\longrightarrow
  (E,h_{0})\longrightarrow 0, 
\end{displaymath}
where $(E,h_{0})$ sits in degree zero. 
\end{definition}
Therefore, this class satisfies
\begin{displaymath}
  \dd_{\mathcal{D}}\widetilde \phi (E,h_{0},h_{1})=\phi (E,h_{0})-\phi
  (E,h_{1}). 
\end{displaymath}

In fact we can characterize $\widetilde \phi (E,h_{0},h_{1})$
axiomatically as follows.

\begin{proposition} \label{prop:18}
  Given $\phi$, a symmetric power series in $r$
  variables, 
  there is a unique way to attach, to each rank $r$ vector bundle $E$
  on a complex manifold $X$ and
  metrics $h_{0}$ and $h_{1}$, a class $\widetilde
  {\phi}(E,h_{0},h_{1})$ satisfying
  \begin{enumerate}
  \item $\dd_{\mathcal{D}}\widetilde \phi (E,h_{0},h_{1})=\phi (E,h_{0})-\phi
    (E,h_{1})$.
  \item $f^{\ast}\widetilde \phi (E,h_{0},h_{1})=\widetilde \phi
    (f^{\ast}(E,h_{0},h_{1}))$ for every holomorphic map
    $f\colon Y\longrightarrow X$.
  \item $\widetilde \phi (E,h,h)=0$.
  \end{enumerate}
  Moreover, if we denote $\widetilde E:=\tr_{1}(\overline 
  \xi)_{1}$, then it satisfies
  \begin{displaymath}
    \widetilde E|_{X\times\{\infty\}}\cong (E,h_{0}),\quad
    \widetilde E|_{X\times\{0\}}\cong (E,h_{1})
  \end{displaymath}
  and
  \begin{equation}\label{eq:106}
    \widetilde \phi (E,h_{0},h_{1})=
 \frac{1}
  {2\pi
    i}\int_{\mathbb{P}^{1}}\frac{-1}{2}\log(t\bar{t}) \bullet 
  \phi(\widetilde E).
  \end{equation}
\end{proposition}
\begin{proof}
  The axiomatic characterization is proved as in theorem
  \ref{thm:3}. In order to prove equation \eqref{eq:106},
  if we follow the notations of the proof of theorem \ref{thm:3} we
  have $K_{0}=(E,h_{0})$ and $K_{1}=0$. Therefore $\tr_{1}(\overline
  \xi)_{0}=p_{1}^{\ast}(E,h_{0})$, while $\widetilde
  E:=\tr_{1}(\overline 
  \xi)_{1}$ satisfies $\widetilde E|_{X\times \{0\}}=(E,h_{1})$ and
  $\widetilde E|_{X\times \{\infty\}}=(E,h_{0})$. Using the antisymmetry
  of $\log t\bar t$ under the involution $t\mapsto 1/t$ we obtain
  \begin{displaymath}
    \widetilde \phi (E,h_{0},h_{1})=\widetilde {\phi }(\overline{\xi})=
 \frac{1}
  {2\pi
    i}\int_{\mathbb{P}^{1}}\frac{-1}{2}\log(t\bar{t}) \bullet 
  \phi(\widetilde E).
  \end{displaymath}  
\end{proof}

We can also treat the case of short exact
sequences. If
 $$\overline{\varepsilon}\colon 0 \longrightarrow
\overline{E}_{2} \longrightarrow \overline{E}_{1}
\longrightarrow \overline{E}_{0} \longrightarrow 0$$
is a short exact sequence of hermitian vector bundles, by convention, we
will assume that $\overline{E}_{0}$ sits in degree zero.
This fixs the sign of $\widetilde \phi(\overline{\varepsilon})$. 

\begin{proposition} \label{prop:20}
  Given $\phi$, a symmetric power series in $r$
  variables, 
  there is a unique way to attach, to each short exact sequence of
  hermitian vector bundles 
  on a complex manifold $X$
  $$\overline{\varepsilon}\colon 0 \longrightarrow
  \overline{E}_{2} \longrightarrow \overline{E}_{1}
  \longrightarrow \overline{E}_{0} \longrightarrow 0,$$
  where $\overline E_{1}$ has rank $r$, a class
  $\widetilde
  {\phi}(\overline{\varepsilon })$ satisfying
  \begin{enumerate}
  \item $\dd_{\mathcal{D}}\widetilde \phi (\overline{\varepsilon })
    =\phi (\overline E_{0}\oplus \overline E_{2})-\phi
    (\overline E_{1})$.
  \item $f^{\ast}\widetilde \phi (\overline{\varepsilon })=\widetilde \phi
    (f^{\ast}(\overline{\epsilon }))$ for every holomorphic map
    $f\colon Y\longrightarrow X$.
  \item $\widetilde \phi (\overline {\varepsilon })=0$ whenever
    $\overline{\varepsilon }$ is orthogonally split. 
  \end{enumerate}
\ \hfill $\square$
\end{proposition}

The following additivity result of Bott-Chern classes will be useful
later. 

\begin{lemma} \label{lemm:1}
  Let $\overline A_{\ast,\ast}$ be a bounded exact sequence of bounded
  exact 
  sequences of 
  hermitian vector bundles. Let 
  $$r=\sum_{i, j \text{ even} }\rk(A_{i,j})
  =\sum_{i, j \text{ odd} }\rk(A_{i,j})
  =\sum_{\substack{i \text{ odd}\\ j \text{ even}} }\rk(A_{i,j})
  =\sum_{\substack{i\text{ even}\\j  \text{ odd}} }\rk(A_{i,j}).
  $$
  Let $\phi $ be a symmetric power series in $r$ variables. Then 
  \begin{displaymath}
    \widetilde {\phi }(\bigoplus_{k \text{ even} }
    \overline A_{k,\ast})-
    \widetilde {\phi }(\bigoplus_{k \text{ odd} }\overline 
    A_{k,\ast})=
    \widetilde {\phi }(\bigoplus_{k \text{ even} }\overline 
    A_{\ast,k})-
    \widetilde {\phi }(\bigoplus_{k \text{ odd} }\overline 
    A_{\ast,k}).
  \end{displaymath}
\end{lemma}
\begin{proof}
  The proof is analogous to the proof of proposition \ref{prop:5} and
  is left to the reader.
\end{proof}

\begin{corollary}\label{cor:12}
  Let $\overline A_{\ast,\ast}$ be a bounded double complex of
  hermitian vector bundles with exact rows, let
  $$r=\sum_{i+j \text{ even} }\rk(A_{i,j})=\sum_{i+j \text{ odd}
  }\rk(A_{i,j})
  $$
  and
  let $\phi $ be a symmetric power series in $r$ variables. Then
  \begin{displaymath}
        \widetilde {\phi }(\Tot \overline A_{\ast,\ast})
        =\widetilde
        {\phi}(\bigoplus_{k}\overline A_{\ast,k}[-k]).
  \end{displaymath}
\end{corollary}
\begin{proof}
  Let $k_{0}$ be an integer such that $\overline A_{k,l}=0$ for
  $k<k_{0}$. For 
  any integer $n$ we denote by $\Tot_{n}=\Tot((\overline
  A_{k,l})_{k\ge n})$ the 
  total complex of the exact complex formed by the rows with index
  greater or equal than $n$. Then
  $\Tot_{k_{0}}=\Tot(\overline A_{\ast,\ast})$. For each $k$ there is
  an exact 
  sequence of complexes 
  \begin{displaymath}
    0\longrightarrow 
    \Tot_{k+1}\longrightarrow \Tot_{k}\oplus \bigoplus_{l<k}
    \overline {A}_{l,\ast}[-l]\longrightarrow 
    \bigoplus_{l\le k}
    \overline {A}_{l,\ast}[-l]\longrightarrow 0,
  \end{displaymath}
  which is orthogonally split in each degree. Therefore by lemma
  \ref{lemm:1} we obtain
  \begin{displaymath}
        \widetilde {\phi }(\Tot_{k} \oplus \bigoplus_{l<k}
    \overline {A}_{l,\ast}[-l])
        =\widetilde {\phi }(\Tot_{k-1} \oplus \bigoplus_{l\le k}
    \overline {A}_{l,\ast}[-l]).    
  \end{displaymath}
  Hence the result follows by induction.
\end{proof}

A particularly important characteristic class is the Chern
character. This class is additive for exact sequences. Specializing
lemma 
\ref{lemm:1} and corollary \ref{cor:12} to the Chern character we
obtain 

\begin{corollary} \label{cor:9} With the hypothesis of lemma
  \ref{lemm:1}, the following equality holds: 
  \begin{displaymath}
    \sum_{k}(-1)^{k}\widetilde{\ch}(\overline A_{k,\ast})=        
    \sum_{k}(-1)^{k}\widetilde{\ch}(\overline A_{\ast,k})=
    \widetilde{\ch}(\Tot \overline A_{\ast,\ast}).
  \end{displaymath}
\end{corollary}
\hfill $\square$

Our next aim is to extend the Bott-Chern classes associated to the
Chern character to metrized coherent sheaves. This extension is due to
Zha \cite{zha99:_rieman_roch}, although it is still unpublished.

\begin{definition}\label{metrcohsh}
A \textit{metrized coherent sheaf} $\overline{\mathcal{F}}$ on
$X$ is a pair $(\mathcal{F}, \overline{E}_{\ast}\to
\mathcal{F})$ where $\mathcal{F}$ is a coherent sheaf on
$X$ and 
\begin{displaymath}
 0\to \overline{E}_{n}\to \overline{E}_{n-1}\to
\dots \to \overline{E}_{0}\to \mathcal{F} \to 0
\end{displaymath}
is
a finite resolution by hermitian vector bundles of the
coherent sheaf $\mathcal{F}$. This resolution is also called the metric of
$\overline{\mathcal{F}}$.

If $\overline E$ is a hermitian vector bundle, we will also denote by
$\overline E$ the metrized coherent sheaf $(E,\overline
E\overset{\Id}{\longrightarrow} E)$.
\end{definition}

Note that the coherent sheaf $0$ may have non trivial metrics. In fact,
any exact sequence of hermitian vector bundles
\begin{displaymath}
  0\to \overline {A}_{n}\to \dots \to \overline{A}_{0}\to 0\to 0
\end{displaymath}
can be seen as a metric on $0$. It will be denoted
$\overline {0}_{A_{\ast}}$. A metric on $0$ is said to be \emph{orthogonally split}
if the exact sequence is orthogonally split. 

A morphism of metrized coherent sheaves $\overline{\mathcal{F}}_1\to
\overline{\mathcal{F}}_2$ is just a morphism of sheaves
$\mathcal{F}_1\to \mathcal{F}_2$. A sequence of metrized coherent
sheaves
$$\overline{\varepsilon}\colon\qquad \ldots \longrightarrow
\overline{\mathcal{F}}_{n+1} \longrightarrow
\overline{\mathcal{F}}_n \longrightarrow
\overline{\mathcal{F}}_{n-1} \longrightarrow \ldots$$ 
is said to be
exact if it is exact as a sequence of coherent sheaves.

\begin{definition}
  Let  $\overline{\mathcal{F}}=(\mathcal{F}, \overline{E}_{\ast}\to
\mathcal{F})$ be a metrized coherent sheaf. Then the \emph{Chern character
form} associated to $\overline{\mathcal{F}}$ is given by
\begin{displaymath}
  \ch(\overline{\mathcal{F}}) =\sum_{i}(-1)^{i}\ch(\overline E_{i}).
\end{displaymath}
\end{definition}

\begin{definition}
  \emph{An exact sequence of metrized coherent sheaves with compatible
    metrics}  
  is a commutative diagram 
  \begin{equation}
    \label{eq:100}
    \begin{array}{ccccccccc}
      && \vdots && \vdots && \vdots && \\
      && \downarrow && \downarrow && \downarrow && \\
      0 & \rightarrow & \overline{E}_{n,1} & \rightarrow &
      \ldots & \rightarrow &
      \overline{E}_{0,1} & \rightarrow & 0\\
      && \downarrow && \downarrow && \downarrow && \\
      0 & \rightarrow & \overline{E}_{n,0} & \rightarrow &
      \ldots & \rightarrow &
      \overline{E}_{0,0} & \rightarrow & 0 \\
      && \downarrow && \downarrow && \downarrow && \\
      0 & \rightarrow & \mathcal{F}_{n} & \rightarrow &
      \ldots & \rightarrow &
      \mathcal{F}_{0} & \rightarrow & 0 \\
      && \downarrow && \downarrow && \downarrow && \\
      && 0 && 0 && 0 &&
    \end{array}  
  \end{equation}
  where all the rows and columns are exact. The columns of this
  diagram are the individual metrics of each coherent sheaf. We will
  say that an exact sequence with compatible metrics is
  \emph{orthogonally split} if each row of vector bundles is an
  orthogonally split exact 
  sequence of hermitian vector bundles.
\end{definition}

As in the case of exact sequences of hermitian vector bundles, the
Chern character form is not compatible with exact sequences of
metrized coherent sheaves and we can define a secondary Bott-Chern
character which measures the lack of compatibility between the
metrics.
\begin{theorem}\label{zhabc}
  \begin{enumerate}
  \item[1)] There is a unique way to attach to every finite exact sequence
    of metrized coherent sheaves with compatible metrics
    $$\overline{\varepsilon}\colon\qquad 0\to 
    \overline{\mathcal{F}}_{n} \to \dots
    \to \overline{\mathcal{F}}_0 \to 0
    $$ on a complex manifold $X$
    a Bott-Chern secondary character
    $$\widetilde{\ch}(\overline{\varepsilon})\in
    \bigoplus_{p}\widetilde{\mathcal{D}}^{2p-1}(X,p)$$  
    such that the
    following axioms are 
    satisfied:
    \begin{enumerate}
    \item[(i)]\label{item:24} (Differential equation) 
      \begin{displaymath}
        \dd_{\mathcal{D}}\widetilde{\ch}(\overline{\varepsilon })=
        \sum_{k}(-1)^{k}\ch(\overline{\mathcal{F}_{k}}).  
      \end{displaymath}
    \item[(ii)]\label{item:25} (Functoriality) If $f\colon X'\longrightarrow
      X$ is a morphism of complex 
      manifolds, that is tor-independent from the coherent sheaves
      $\mathcal{F}_{k}$, then
      \begin{displaymath}
        f^{\ast}(\widetilde{\ch})(\overline{\varepsilon })=\widetilde
        \ch(f^{\ast}\overline{\varepsilon }), 
      \end{displaymath}
      where the exact sequence $f^{\ast}\overline{\varepsilon }$
      exists thanks to the tor-independence.
    \item [(iii)]\label{item:26} (Horizontal normalization) If
      $\overline{\varepsilon }$ is orthogonally split then
      \begin{displaymath}
        \widetilde{\ch}(\overline{\varepsilon })=0.
      \end{displaymath}
    \end{enumerate}
  \item[2)] There is a unique way to attach to every finite exact sequence
    of metrized coherent sheaves
    $$\overline{\varepsilon}\colon\qquad 0\to 
    \overline{\mathcal{F}}_{n} \to \dots
    \to \overline{\mathcal{F}}_0 \to 0
    $$ on a complex manifold $X$
    a Bott-Chern secondary character
    $$\widetilde{\ch}(\overline{\varepsilon})\in
    \bigoplus_{p}\widetilde{\mathcal{D}}^{2p-1}(X,p)$$   
    such that the axioms (i), (ii) and
    (iii) above and the 
    axiom (iv) below are 
    satisfied:
    \begin{enumerate}
    \item [(iv)] (Vertical normalization) For every bounded complex of
      hermitian vector bundles 
      \begin{displaymath}
        \dots \rightarrow \overline A_{k}\rightarrow
        \dots \rightarrow \overline A_{0}\rightarrow 0
      \end{displaymath}
      that is orthogonally split,
      and every bounded complex of metrized coherent sheaves
    $$\overline{\varepsilon}\colon\qquad 0\to 
    \overline{\mathcal{F}}_{n} \to \dots
    \to \overline{\mathcal{F}}_0 \to 0
    $$
    where the metrics are given by $\overline E_{i,\ast}\rightarrow
    \mathcal{F}_{i}$, if, for some $i_{0}$ we denote
    \begin{displaymath}
      \overline {\mathcal{F}}_{i_{0}}'=(\mathcal{F}_{i_{0}},
      \overline E_{i_{0},\ast}\oplus \overline A_{\ast}\rightarrow
      \mathcal{F}_{i_{0}})
    \end{displaymath}
    and
    $$\overline{\varepsilon}'\colon\qquad 0\to 
    \overline{\mathcal{F}}_{n} \to \dots \to \overline
    {\mathcal{F}}_{i_{0}}' \to \dots
    \to \overline{\mathcal{F}}_0 \to 0,
    $$
    then $\widetilde{\ch}(\overline{\varepsilon }')=
    \widetilde{\ch}(\overline{\varepsilon })$.
    \end{enumerate}
  \end{enumerate}
\end{theorem}
\begin{proof}
  \emph{1)} The uniqueness is proved using the standard
  deformation argument. By definition, the metrics of the coherent
  sheaves form a diagram like (\ref{eq:100}). On $X\times
  \mathbb{P}^{1}$, for each $j\ge 0$ we consider the exact sequences
  $\widetilde E_{\ast,j}=\tr_{1}(E_{\ast,j})$ associated to the rows
  of the diagram with the hermitian metrics
  of definition \ref{def:13}. Then, for each $i,j$ there are maps
  $\dd\colon\widetilde E_{i,j}\to \widetilde E_{i-1,j}$, and $\delta
  \colon\widetilde E_{i,j}\to \widetilde E_{i,j-1}$. We denote 
  $$\widetilde {\mathcal{F}}_{i}=
  \Coker(\delta \colon\widetilde E_{i,1}\to \widetilde E_{i,0}).$$
  Using the definition of $\tr_{1}$ and diagram chasing one can prove
  that there is a commutative diagram  
  \begin{equation} \label{eq:101}
    \begin{array}{ccccccccc}
      && \vdots && \vdots && \vdots && \\
      && \downarrow && \downarrow && \downarrow && \\
      0 & \rightarrow & \widetilde{E}_{n,1} & \rightarrow &
      \ldots & \rightarrow &
      \widetilde{E}_{0,1} & \rightarrow & 0\\
      && \downarrow && \downarrow && \downarrow && \\
      0 & \rightarrow & \widetilde{E}_{n,0} & \rightarrow &
      \ldots & \rightarrow &
      \widetilde{E}_{0,0} & \rightarrow & 0 \\
      && \downarrow && \downarrow && \downarrow && \\
      0 & \rightarrow & \widetilde{\mathcal{F}}_{n} & \rightarrow &
      \ldots & \rightarrow &
      \widetilde{\mathcal{F}}_{0} & \rightarrow & 0 \\
      && \downarrow && \downarrow && \downarrow && \\
      && 0 && 0 && 0 &&
    \end{array}  
  \end{equation}
  where all the rows and columns are exact. In particular this implies
  that the inclusions $i_{0}\colon X\to X\times \{0\}\to X\times
  \mathbb{P}^{1}$ and $i_{\infty}\colon X\to X\times \{\infty\}\to X\times
  \mathbb{P}^{1}$ are tor-independent from the sheaves $\widetilde
  {\mathcal{F}}_{i}$. But $i_{0}^{\ast}\widetilde
  {\mathcal{F}}_{\ast}$ is isometric with $\overline
  {\mathcal{F}}_{\ast}$ and $i_{\infty}^{\ast}\widetilde
  {\mathcal{F}}_{\ast}$ is orthogonally split. Hence, by the standard
  argument, axioms (i), (ii) and (iii) imply that
  \begin{equation}
    \label{eq:102}
    \widetilde{\ch}(\overline{\varepsilon})=
    \sum_{j}(-1)^{j}\widetilde{\ch}(\overline E_{\ast,j}).
  \end{equation}
  To prove the existence we use equation (\ref{eq:102}) as
  definition. Then the properties of the Bott-Chern classes of exact
  sequences of hermitian vector bundles imply that axioms (i), (ii)
  and (iii) are satisfied.

  \emph{Proof of 2)}. We first assume that such theory exists. Let
      \begin{displaymath}
        \dots \rightarrow \overline A_{k}\rightarrow
        \dots \rightarrow \overline A_{0}\rightarrow 0
      \end{displaymath}
      be a bounded complex of hermitian vector bundles, non
      necessarily orthogonally split,  
      and
    $$\overline{\varepsilon}\colon\qquad 0\to 
    \overline{\mathcal{F}}_{n} \to \dots
    \to \overline{\mathcal{F}}_0 \to 0
    $$
     a bounded complex of metrized coherent sheaves
    where the metrics are given by $\overline E_{i,\ast}\rightarrow
    \mathcal{F}_{i}$. As in axiom (iv), for some $i_{0}$ we denote
    \begin{displaymath}
      \overline {\mathcal{F}}_{i_{0}}'=(\mathcal{F}_{i_{0}},
      \overline E_{i,\ast}\oplus \overline A_{\ast}\rightarrow
      \mathcal{F}_{i_{0}})
    \end{displaymath}
    and
    $$\overline{\varepsilon}'\colon\qquad 0\to 
    \overline{\mathcal{F}}_{n} \to \dots \to \overline
    {\mathcal{F}}_{i_{0}}' \to \dots
    \to \overline{\mathcal{F}}_0 \to 0.
    $$
    
    By axioms (i), (ii) and (iv), the class
    $(-1)^{i_{0}}(\widetilde{\ch}(\overline{\varepsilon }')- 
    \widetilde{\ch}(\overline{\varepsilon
    }))$ satisfies the properties that characterize
    $\widetilde{\ch}(A_{\ast})$. Therefore
    $\widetilde{\ch}(\overline{\varepsilon }')=
    \widetilde{\ch}(\overline{\varepsilon
    })+(-1)^{i_{0}}\widetilde{\ch}(A_{\ast})$.  
    
    Fix again a number $i_{0}$ and assume that there is an exact
    sequence of resolutions
     \begin{equation}\label{eq:110}
   \xymatrix{
     0\ar[r]&
     \overline A_{\bullet}\ar[d]\ar[r]&
     \overline E_{i_{0},\ast}'\ar[d]\ar[r]&
     \overline E_{i_{0},\ast}\ar[d]\ar[r]&
     0\\
     & 0\ar[r]&
      \mathcal{F}_{i_{0}} \ar@{=}[r] &
       \mathcal{F}_{i_{0}} &
}
 \end{equation}

Let now $\overline{\varepsilon} '$ denote the exact sequence
$\overline{\varepsilon }$ but with the metric $\overline
E_{i_{0},\ast}'$ in the position $i_{0}$. Let $\overline{\eta}_{j}$
denote the $j$-th row of the diagram \eqref{eq:110}. Again using a
deformation 
argument one sees that
\begin{equation}
  \label{eq:107}
  \cht(\overline{\varepsilon }')-\cht(\overline{\varepsilon })=
  (-1)^{i_{0}}\left(\cht(\overline A_{\ast})-
    \sum_{j}(-1)^{j}\cht(\eta_{j})
  \right).
\end{equation}

Choose now a compatible system of metrics
  \begin{equation}
    \label{eq:103}
    \begin{array}{ccccccccc}
      && \vdots && \vdots && \vdots && \\
      && \downarrow && \downarrow && \downarrow && \\
     0 & \rightarrow& \overline{D}_{n,1} & \rightarrow &
      \ldots & \rightarrow &
      \overline{D}_{0,1} & \rightarrow & 0\\
      && \downarrow && \downarrow && \downarrow && \\
      0&\rightarrow & \overline{D}_{n,0} & \rightarrow &
      \ldots & \rightarrow &
      \overline{D}_{0,0} & \rightarrow & 0 \\
      && \downarrow && \downarrow && \downarrow && \\
      0 & \rightarrow & \mathcal{F}_{n} & \rightarrow &
      \ldots & \rightarrow &
      \mathcal{F}_{0} & \rightarrow & 0 \\
      && \downarrow && \downarrow && \downarrow && \\
      && 0 && 0 && 0 &&
    \end{array}    
  \end{equation}
we denote by $\overline{\lambda}
_{j}$ each row of the above diagram.
For each $i$, choose a resolution $\overline E'_{i,\ast}\longrightarrow
\mathcal{F}_{i}$ such that there exist exact sequences of resolutions 
     \begin{equation} \label{eq:108}
   \xymatrix{
     0\ar[r]&
     \overline A_{i,\ast}\ar[d]\ar[r]&
     \overline E_{i,\ast}'\ar[d]\ar[r]&
     \overline E_{i,\ast}\ar[d]\ar[r]&
     0\\
     & 0\ar[r]&
      \mathcal{F}_{i} \ar@{=}[r] &
       \mathcal{F}_{i} &
}
 \end{equation}
and 
     \begin{equation} \label{eq:109}
   \xymatrix{
     0\ar[r]&
     \overline B_{i,\ast}\ar[d]\ar[r]&
     \overline E_{i,\ast}'\ar[d]\ar[r]&
     \overline D_{i,\ast}\ar[d]\ar[r]&
     0\\
     & 0\ar[r]&
      \mathcal{F}_{i} \ar@{=}[r] &
       \mathcal{F}_{i} &
}
 \end{equation}
We denote by $\overline{\eta}_{i,j}$ each row of the diagram
\eqref{eq:108} and by $\overline{\mu}_{i,j}$ each row of the diagram
\eqref{eq:109}.  
Then, by \eqref{eq:107} and \eqref{eq:102}, we have
\begin{multline}
  \label{eq:111}
  \cht(\overline{\varepsilon })=\sum_{j}(-1)^{j}\cht(\overline{\lambda
  }_{j})
  + \sum _{i}(-1)^{i}(\cht(\overline B_{i,\ast})-\cht(\overline A_{i,\ast}))\\
  +\sum_{i,j}(-1)^{i+j}(\cht(\overline{\eta}_{i,j})-\cht(\overline{\mu}_{i,j})) 
\end{multline}
Thus, $\cht(\overline{\varepsilon })$ is uniquely determined by axioms
(i) to (iv).
  To prove the existence we use equation (\ref{eq:111}) as
  definition.   
  We have to show that this definition is independent of the choices of
  the new resolutions. This
  independence follows from corollary \ref{cor:9}. Once we know that
  the Bott-Chern classes are well defined, it is clear that they
  satisfy axioms (i), (ii), (iii) and (iv).
\end{proof}

\begin{proposition} \label{prop:21}
    (Compatibility with exact squares) If
      $$\begin{array}{ccccccccc}
        && \vdots && \vdots && \vdots && \\
        && \downarrow && \downarrow && \downarrow && \\
        \dots & \rightarrow & \overline{\mathcal{F}}_{n+1,m+1} & \rightarrow &
        \overline{\mathcal{F}}_{n+1,m} & \rightarrow &
        \overline{\mathcal{F}}_{n+1,m-1} & \rightarrow & \dots \\
        && \downarrow && \downarrow && \downarrow && \\
        \dots & \rightarrow & \overline{\mathcal{F}}_{n,m+1} & \rightarrow &
        \overline{\mathcal{F}}_{n,m} & \rightarrow &
        \overline{\mathcal{F}}_{n,m-1} & \rightarrow & \dots \\
        && \downarrow && \downarrow && \downarrow && \\
        \dots & \rightarrow & \overline{\mathcal{F}}_{n-1,m+1} & \rightarrow &
        \overline{\mathcal{F}}_{n-1,m} & \rightarrow &
        \overline{\mathcal{F}}_{n-1,m-1} & \rightarrow & \dots \\
        && \downarrow && \downarrow && \downarrow && \\
        && \vdots && \vdots && \vdots &&
      \end{array}$$
      is a bounded commutative diagram of metrized coherent sheaves,
      where all the 
      rows \dots $(\overline{\varepsilon}_{n-1})$,
      $(\overline{\varepsilon}_{n})$, 
      $(\overline{\varepsilon}_{n+1})$, \dots and all the columns
      $(\overline{\eta}_{m-1})$, $(\overline{\eta}_{m})$,
      $(\overline{\eta}_{m+1})$ 
      are exact, then
      $$\sum_{n}(-1)^{n}\widetilde{\ch}(\overline{\varepsilon}_{n}) =
      \sum_{m}(-1)^{m}\widetilde{\ch}(\overline{\eta}_m)
      .$$  
\end{proposition}
\begin{proof}
  This follows from equation \eqref{eq:111} and corollary \ref{cor:9}.
\end{proof}

We will use the notation of definition \ref{def:11} also in the case
of metrized coherent sheaves.

It is easy to verify the following result.

\begin{proposition}\label{prop:16} Let 
$$(\overline{\varepsilon})\qquad \ldots \longrightarrow
\overline{E}_{n+1} \longrightarrow
\overline{E}_n \longrightarrow
\overline{E}_{n-1} \longrightarrow \ldots$$ 
be a finite exact sequence of hermitian vector bundles. Then the
Bott-Chern classes obtained by theorem \ref{zhabc} and by theorem
\ref{thm:3} agree.
\hfill$\square$
\end{proposition}

\begin{proposition}\label{prop:19} Let $\overline
  {\mathcal{F}}=(\mathcal{F},\overline E_{\ast}\to\mathcal{F})$ be a
  metrized coherent sheaf. We consider the exact sequence of metrized
  coherent sheaves
  \begin{displaymath}
    \overline{\varepsilon }\colon\qquad 0\longrightarrow 
    \overline E_{n}\to\dots\to\overline E_{0}\to 
    \overline{\mathcal{F}}\to 0,
  \end{displaymath}
  where, by abuse of notation, $\overline E_{i}=(E_{i},\overline
  E_{i}\overset {=}{\to} E_{i})$.
  Then $\widetilde {\ch}(\overline{\varepsilon })=0$.
\end{proposition}
\begin{proof} Define $\mathcal{K}_{i}=\Ker(E_{i}\to E_{i-1})$,
  $i=1,\dots,n$ and $\mathcal{K}_{0}=\Ker(E_{0}\to
  \mathcal{F})$. Write
  \begin{displaymath}
    \overline {\mathcal{K}_{i}}=(\mathcal{K}_{i}
    ,0\to\overline E_{n}\to\dots\to \overline E_{i+1}\to
    \mathcal{K}_{i}),\  i=0,\dots,n,
  \end{displaymath}
  and $\overline{\mathcal{K}}_{-1}=\overline{\mathcal{F}}$. If we
  prove that
  \begin{equation}\label{eq:105}
    \widetilde{\ch}(0\to \overline {\mathcal{K}_{i}}\to
    \overline E_{i}\to \overline {\mathcal{K}}_{i-1}\to 0)=0,
  \end{equation}
  then we obtain the result by induction using proposition
  \ref{prop:21}.  
  In order to prove equation (\ref{eq:105}) we apply equation
  \eqref{eq:111}. To this end consider resolutions 
  \begin{alignat*}{2}
    \overline D_{0,\ast}&\longrightarrow \mathcal{K}_{i-1},&\qquad
    \overline D_{0,k}&=\overline E_{k+i}\\
    \overline D_{1,\ast}&\longrightarrow E_{i},&\qquad
    \overline D_{1,k}&=\overline E_{k+i+1}\oplus\overline E_{k+i}\\
    \overline D_{2,\ast}&\longrightarrow \mathcal{K}_{i},&\qquad
    \overline D_{2,k}&=\overline E_{k+i+1}
  \end{alignat*}
  with the map $D_{2,k}\overset{\Delta }{\to} D_{1,k}$ given by
  $s\mapsto (s,\dd s)$ and 
  the map $D_{1,k}\overset{\nabla}{\to} D_{0,k}$ given by
  $(s,t)\mapsto t-\dd s$. The 
  differential of the complex $D_{1,k}$ is given by $(s,t)\mapsto
  (t,0)$. Using equations (\ref{eq:111}) and (\ref{eq:102}) we write
  the left hand side of equation (\ref{eq:105}) in terms of Bott-Chern
  classes of vector bundles. All the exact sequences involved are
  orthogonally split except maybe the sequences
  \begin{displaymath}
    \overline{\lambda} _{k}\colon\qquad 0\to \overline D_{2,k}\to  \overline
    D_{1,k}\to  
    \overline D_{0,k}\to 0.
  \end{displaymath}
  But now we consider the diagrams
  \begin{displaymath}
    \xymatrix{ \overline E_{k+i+1} \ar[r]^-{i_{1}} \ar[d]^{\Id}&
      \overline E_{k+i+1}\oplus \overline E_{k+i} \ar[r]^-{p_{2}}
      \ar[d]^{f}&
      \overline E_{k+i} \ar[d]^{\Id}\\
      \overline E_{k+i+1} \ar[r]^-{\Delta }&
      \overline E_{k+i+1}\oplus \overline E_{k+i} \ar[r]^-{\nabla}
      & \overline E_{k+i}
    }
  \end{displaymath}
  and 
  \begin{displaymath}
    \xymatrix{ \overline E_{k+i} \ar[r]^-{i_{2}} \ar[d]^{\Id}&
      \overline E_{k+i+1}\oplus \overline E_{k+i} \ar[r]^-{p_{1}}
      \ar[d]^{f}&
      \overline E_{k+i+1} \ar[d]^{\Id}\\
      \overline E_{k+i} \ar[r]^-{i_{2} } &
      \overline E_{k+i+1}\oplus \overline E_{k+i} \ar[r]^-{p_{1}}
      & \overline E_{k+i+1}
    },
  \end{displaymath}
  where $i_{i}$, $i_{2}$ are the natural inclusions, $p_{1}$ and
  $p_{2}$ are the projections and $f(s,t)=(s,t+f(s))$. These diagrams
  and corollary \ref{cor:9} imply that $\widetilde{\ch}(\overline
  {\lambda}_{k} )=0$. 
\end{proof}

\begin{remark} In \cite{zha99:_rieman_roch}, Zha shows that the
  Bott-Chern classes associated to exact sequences of metrized
  coherent sheaves are characterized by proposition \ref{prop:16},
  proposition \ref{prop:19} and proposition \ref{prop:21}. We prefer the
  characterization in terms of the differential equation, the
  functoriality and the normalization, because it relies on natural
  extensions of the corresponding axioms that define the Bott-Chern
  classes for exact sequences of hermitian vector bundles. Moreover,
  this approach will be used in a subsequent paper where we will study
  singular Bott-Chern classes 
  associated to arbitrary proper morphisms.
\end{remark}

The following generalization of proposition \ref{prop:19} will be
useful later. 
Let
\begin{displaymath}
 \varepsilon \colon 0\rightarrow \mathcal{G}_{n}\rightarrow
 \mathcal{G}_{n-1}\rightarrow 
\dots \rightarrow
  \mathcal{G}_{0} \rightarrow \mathcal{F}\rightarrow 0
\end{displaymath}
be a finite resolution of a coherent sheaf by coherent sheaves. Assume
that we have a commutative diagram
  $$\begin{array}{ccccccccccc}
    && \vdots && \vdots && \vdots && &&\\
    && \downarrow && \downarrow && \downarrow && &&\\
    && \overline{E}_{1,n} & \rightarrow &
    \ldots & \rightarrow &
    \overline{E}_{1,0} && &&\\
    && \downarrow && \downarrow && \downarrow && &&\\
    && \overline{E}_{0,n} & \rightarrow &
    \ldots & \rightarrow &
    \overline{E}_{0,0} && &&\\
    && \downarrow && \downarrow && \downarrow && &&\\
    0 & \rightarrow & \overline{\mathcal{G}}_{n} & \rightarrow &
    \ldots & \rightarrow &
    \overline{\mathcal{G}}_{0} & \rightarrow & \mathcal{F}
    &\rightarrow & 0 \\
    && \downarrow && \downarrow && \downarrow &&&& \\
    && 0 && 0 && 0 &&
  \end{array}$$
where the columns are exact, the rows are complexes and the $\overline
E_{i,j}$ are hermitian 
vector bundles. The columns of this diagram define
metrized coherent sheaves $\overline{\mathcal{G}}_{i}$. Let $\overline
{\mathcal{F}}$ be the metrized coherent sheaf defined by the
resolution $\Tot(\overline E_{\ast,\ast})\longrightarrow
\mathcal{F}$. 

\begin{proposition}\label{prop:13} With the notations above, let
  $\overline {\varepsilon }$ be the exact sequence of metrized
  coherent sheaves
\begin{displaymath}
 \overline{\varepsilon} \colon 0\rightarrow \overline{\mathcal{G}}_{n}\rightarrow
 \overline{\mathcal{G}}_{n-1}\rightarrow 
\dots \rightarrow
  \overline{\mathcal{G}}_{0} \rightarrow \overline{\mathcal{F}}\rightarrow 0
\end{displaymath}
Then $\widetilde {\ch}(\overline{\varepsilon })=0$.
\end{proposition}
\begin{proof}
  For each $k$, let $\Tot_{k}=\Tot((E_{\ast,j})_{j\ge k})$. There are 
  inclusions $\Tot_{k}\longrightarrow \Tot_{k-1}$. Let 
  $\overline D_{\ast,j}=s(\Tot_{j+1}\to \Tot_{j})$ with the hermitian
  metric induced by $\overline E_{\ast,\ast}$. There are exact sequences of
  complexes
  \begin{equation}\label{eq:77}
    0\longrightarrow \overline E_{\ast,j}\longrightarrow
    \overline D_{\ast,j}\longrightarrow 
    s(\Tot_{j+1}\to \Tot_{j+1})\longrightarrow 0
  \end{equation}
  that are orthogonally split at each degree. The third complex is
  orthogonally split. Therefore, if we denote by $h_{E}$ and $h_{D}$ the metric
  structures of $\mathcal{G}_{j}$ induced respectively by the first
  and second column of diagram \eqref{eq:77}, then
  \begin{equation}
    \label{eq:79}
    \widetilde {\ch}(\mathcal{G}_{j},h_{E},h_{D})=0.
  \end{equation}
  There is a commutative diagram of resolutions
  $$\begin{array}{ccccccccccc}
    && \vdots && \vdots && \vdots && \vdots &&\\
    && \downarrow && \downarrow && \downarrow && \downarrow &&\\
    0 &\rightarrow & \overline{D}_{1,n} & \rightarrow &
    \ldots & \rightarrow &
    \overline{D}_{1,0} &\rightarrow &(\Tot_{0})_{1} & \rightarrow & 0\\
    && \downarrow && \downarrow && \downarrow && \downarrow &&\\
    0 &\rightarrow & \overline{D}_{0,n} & \rightarrow &
    \ldots & \rightarrow &
    \overline{D}_{0,0} &\rightarrow & (\Tot_{0})_{0} &\rightarrow &0\\
    && \downarrow && \downarrow && \downarrow && \downarrow &&\\
    0 & \rightarrow & \mathcal{G}_{n} & \rightarrow &
    \ldots & \rightarrow &
    \mathcal{G}_{0} & \rightarrow & \mathcal{F}
    &\rightarrow & 0 \\
    && \downarrow && \downarrow && \downarrow &&\downarrow && \\
    && 0 && 0 && 0 && 0 &&
  \end{array}$$
  where the rows of degree greater or equal than zero are orthogonally
  split. Hence the result follows from equation (\ref{eq:102}), 
  equation \eqref{eq:79} and proposition \ref{prop:21}.
\end{proof}

\begin{remark} We have only defined the Bott-Chern
   classes associated to the Chern 
  character. Everything applies without change to any additive
  characteristic class. The reader will find no difficulty to adapt
  the previous results to any multiplicative characteristic class like
  the Todd genus or the total Chern class.
\end{remark}

\section{Direct images of Bott-Chern classes}
\label{sec:direct-images-bott}

The aim of this section is to show that certain direct images of
Bott-Chern classes are closed. This result is a generalization of
results of Bismut, Gillet and Soul\'e
\cite{BismutGilletSoule:MR1086887} page 325 and of Mourougane
\cite{Mourougane04:cbcc} proposition 6. The fact that these direct
images of Bott-Chern classes are closed implies that certain relations
between characteristic classes are true at the level of differential
forms (see corollary \ref{cor:3} and corollary \ref{cor:4}). 

In the first part of this section we deal with differential
geometry. Thus all the varieties will be differentiable manifolds.
  
Let $G_{1}$ be a Lie group and let $\pi
\colon N_{2}\longrightarrow M_{2}$ be a 
principal bundle with structure group $G_{2}$ and connection
$\omega_{2} $. Assume 
that there is a left action of $G_{1}$ over $N_{2}$ that commutes with the
right action of $G_{2}$ and such that the connection $\omega_{2} $ is
$G_{1}$-invariant. 

Let $\mathfrak{g}_{1}$ and $\mathfrak{g}_{2}$ be the Lie algebras of
$G_{1}$ and $G_{2}$. Every element $\gamma \in \mathfrak{g}_{1}$
defines a tangent vector field $\gamma ^{\ast}$ over $N_{2}$ given by
\begin{displaymath}
  \gamma ^{\ast}_{p}=\left. \frac{d}{dt}\right |_{t=0}\exp(t\gamma )p. 
\end{displaymath}
Let $(\gamma ^{\ast})^{V}$ be the vertical component of $\gamma
^{\ast}$ with respect to the connection $\omega_{2} $. For every point
$p\in N_{2}$, we denote by $\varphi(\gamma,p )\in \mathfrak{g}_{2}$
the element characterized by $(\gamma ^{\ast})_{p}^{V}=\varphi(\gamma
,p)_{p}^{\ast}$, where $\varphi(\gamma
,p)^{\ast}$ is the fundamental vector field associated to  
$\varphi(\gamma
,p)$. 

The commutativity of the actions of $G_{1}$ and $G_{2}$ and the
invariance of the connection $\omega_{2} $ implies that, for $g\in
G_{1}$ and $\gamma \in \mathfrak{g}_{1}$, the
following equalities hold 
\begin{align}
  L_{g\ast} (\gamma ^{\ast})&= (\ad (g)\gamma ^{\ast}),\\
  L_{g\ast} (\gamma ^{\ast})^{V}&= (\ad (g)\gamma ^{\ast})^{V},\\
  \varphi(\ad(g)\gamma ,p)&=\varphi(\gamma ,g^{-1}p).\label{eq:4}
\end{align}
Let $\mathcal{G}_{2}$ be the vector bundle over $M_{2}$ associated to $N_{2}$
and the adjoint representation of $G_{2}$. That is,
$$\mathcal{G}_{2}=N_{2}\times \mathfrak{g}_{2}\left/\big\langle (pg,v)\sim
(p,\ad(g)v)\big\rangle\right. . $$ Thus, we
can identify smooth sections of $\mathcal{G}_{2}$ with  $
\mathfrak{g}_{2}$-valued functions on $N_{2}$ that are invariant
under the action of $G_{2}$. In this way, 
$\varphi(\gamma ,p)$ determines a section  
\begin{displaymath}
 \varphi(\gamma )\in C^{\infty}(N_{2},\mathfrak{g}_{2})^{G_{2}}=
 C^{\infty}(M_{2},\mathcal{G}_{2}).
\end{displaymath}
Equation \eqref{eq:4} implies that, for $g\in
G_{1}$ and $\gamma \in \mathfrak{g}_{1}$,
\begin{displaymath}
  \varphi(\ad(g)\gamma )=L_{g^{-1}}^{\ast}\varphi(\gamma ).
\end{displaymath}

We denote by $\Omega ^{\omega_{2} }$ the curvature of the connection
$\omega_{2} $. Let $P$ be an invariant function on $\mathfrak{g}_{2}$,
then $P(\Omega ^{\omega_{2}
}+\varphi(\gamma))$ is a well defined differential form on $M_{2}$.

\begin{proposition} \label{prop:2}
  Let $P$ be an invariant function on $\mathfrak{g}_{2}$ and let $\mu
  $ be a current on $M_{2}$ invariant under the action of $G_{1}$. 
  Then $\mu (P(\Omega ^{\omega_{2}
}+\varphi(\gamma)))$ is an invariant function on
$\mathfrak{g}_{1}$. 
\end{proposition}
\begin{proof}
  Let $g\in G_{1}$. Then,  
  \begin{align*}
    \mu (P(\Omega ^{\omega_{2}
    }+\varphi(\ad(g) \gamma)))&=
    \mu (P(\Omega ^{\omega_{2}
    }+L_{g^{-1}}^{\ast} \varphi(\gamma)))\\
    &=\mu (P(L_{g^{-1}}^{\ast}\Omega ^{\omega_{2}
    }+L_{g^{-1}}^{\ast} \varphi(\gamma)))\\
    &=L_{g^{-1}\ast}(\mu)(P(\Omega ^{\omega_{2}
    }+\varphi( \gamma)))\\ 
    &= \mu (P(\Omega ^{\omega_{2}
    }+\varphi( \gamma)))
  \end{align*}
\end{proof}

Let now $N_{1}\longrightarrow M_{1}$ be a principal bundle with
structure group
$G_{1}$ and provided with a connection $\omega _{1}$. Then we can form
the diagram
\begin{displaymath}
  \begin{CD}
  N_{1}\times N_{2} @>\pi_{1}>> N_{1}\underset{G_{1}}
    {\times}N_{2}\\
    @VV\pi' V @VV\pi V\\
    N_{1}\times M_{2} @>\pi_{2}>> N_{1}\underset{G_{1}}
    {\times}M_{2}\\
    @. @VV q V\\
    @. M_{1}
\end{CD}
\end{displaymath}
Then $\pi$ is a principal bundle with structure group $G_{2}$. The
connections $\omega _{1}$ and $\omega _{2}$ induce a connection on the
principal bundle $\pi$. The subbundle of horizontal vectors with
respect to this connection is given by $\pi_{1\ast}(T^{H}N_{1}\oplus
T^{H}N_{2})$. We will denote this connection by $\omega _{1,2}$. We
are interested in computing the curvature $\omega _{1,2}$.

In fact, all the maps in the above diagram are fiber bundles provided
with a connection. When applicable, given a vector field $U$ in any of
these spaces, we will denote by $U^{H,1}$ the horizontal lifting to 
$N_{1}\times N_{2}$, by $U^{H,2}$ the horizontal lifting
to $N_{1}\underset {G_{1}}{\times}N_{2}$ and by $U^{H,3}$ the
horizontal lifting to $N_{1}\underset {G_{1}}{\times}M_{2}$.

The tangent space $T(N_{1}\times N_{2})$ can be decomposed as direct
sum in the following ways
\begin{align}
  T(N_{1}\times N_{2})&=
  T^{H}N_{1}\oplus T^{V}N_{1}\oplus T^{H}N_{2}\oplus T^{V}N_{2}\notag \\
  &=T^{H}N_{1}\oplus T^{V}N_{1}\oplus T^{H}N_{2}\oplus \Ker
  \pi_{1\ast},\label{eq:3}   
\end{align}
For every point $(x,y)\in N_{1}\times N_{2}$ we have that
$(\Ker \pi_{1\ast})_{(x,y)}\subset T^{V}_{x}N_{1}\oplus
T_{y}N_{2}$. Moreover, there is an isomorphism
$\mathfrak{g}_{1}\longrightarrow (\Ker \pi_{1\ast})_{(x,y)}$ that
sends an element $\gamma \in \mathfrak{g}_{1}$ to 
the element $(\gamma ^{\ast}_{x},-\gamma ^{\ast}_{y})\in T^{V}_{x}N_{1}\oplus
T_{y}N_{2}$.

The tangent space to $N_{1}\underset {G_{1}}{\times}M_{2}$ can be
decomposed as the sum of the subbundle of vertical vectors with
respect to $q$ and the subbundle of horizontal vectors defined by the
connection $\omega _{1}$. The horizontal lifting to $N_{1}\times
N_{2}$ of a vertical vector lies in $T^{H}N_{2}$ and the
horizontal lifting of  a horizontal vector lies in $T^{H}N_{1}$.

Let $U$, $V$ be two vector fields on $M_{1}$ and let $U^{H,3}$,
$V^{H,3}$ be the horizontal liftings to $N_{1}\underset{G_{1}}{\times
}M_{2}$. Then 
\begin{align*}
  \Omega ^{\omega _{1,2}}(U^{H,3},&V^{H,3})=[U^{H,3},V^{H,3}]^{H,2}-
  [U^{H,2},V^{H,2}]\\
  &=\pi_{1\ast}([U^{H,3},V^{H,3}]^{H,1}-
  [U^{H,1},V^{H,1}])\\
  &= \pi_{1\ast}([U^{H,3},V^{H,3}]^{H,1}-
  [U,V]^{H,1}+ [U,V]^{H,1}-
  [U^{H,1},V^{H,1}])\\
  &=\pi_{1,\ast}([U^{H,3},V^{H,3}]^{H,1}-
  [U,V]^{H,1}+\Omega ^{\omega _{1}}(U,V)).
\end{align*}
But, we have
\begin{align*}
  \Omega ^{\omega _{1,2}}(U^{H,3},V^{H,3})&\in T^{V}N_{2},\\
  \Omega ^{\omega _{1}}(U,V)&\in T^{V}N_{1},\\
  [U^{H,3},V^{H,3}]^{H,1}- [U,V]^{H,1}&\in T^{H}N_{2}. 
\end{align*}
Therefore, by the direct sum decomposition \eqref{eq:3} we obtain that 
\begin{displaymath}
  \Omega ^{\omega _{1,2}}(U^{H,3},V^{H,3})=((\pi_{1\ast} \Omega ^{\omega
    _{1}}(U,V)))^{V},
\end{displaymath}
where the vertical part is taken with respect to the fib re bundle
$\pi$.

If $U$ is a horizontal vector field over
$N_{1}\underset{G_{1}}{\times}M_{2}$ and $V$ is a vertical vector field, a
similar argument shows that $\Omega ^{\omega _{1,2}}(U,V)=0$. Finally,
if $U$ and $V$ are vector fields on $M_{2}$, they determine vertical
vector fields on $N_{1}\underset{G_{1}}{\times}M_{2}$. Then the
horizontal liftings $U^{H,1}$ and $V^{H,1}$ are induced by
horizontal liftings of $U$ and $V$ to $N_{2}$. Therefore, reasoning as
before we see that
\begin{displaymath}
  \Omega^{\omega _{1,2}}(U,V)=\Omega ^{\omega _{2}}(U,V).
\end{displaymath}

\begin{proposition} \label{prop:3}
  Let $G_{1}$ and $G_{2}$ be Lie groups, with Lie algebras
  $\mathfrak{g}_{1}$ and $\mathfrak{g}_{2}$. For $i=1,2$, let
  $N_{i}\longrightarrow 
  M_{i}$ be a principal bundle with structure group $G_{i}$, provided
  with a
  connection $\omega _{i}$. Assume that there is a left action of
  $G_{1}$ over $N_{2}$ that commutes with the right action of $G_{2}$
  and that the connection $\omega _{2}$ is invariant under the
  $G_{1}$-action. We form the $G_{2}$-principal bundle
  $\pi\colon N_{1}\underset{G_{1}}{\times} N_{2}\longrightarrow 
  N_{1}\underset{G_{1}}{\times} M_{2}$ with the induced connection
  $\omega _{1,2}$ and curvature $\Omega ^{\omega _{1,2}}$.  Let $P$ be
  any invariant function on $\mathfrak{g}_{2}$. Thus $P(\Omega
  ^{\omega _{1,2}})$ is a well defined closed differential form on
  $N_{1}\underset{G_{1}}{\times} M_{2}$. Let $\mu $ be a current on
  $M_{2}$ invariant under the $G_{1}$-action. Being $G_{1}$ invariant,
  the
  current $\mu $ 
  induces a current on $N_{1}\underset{G_{1}}{\times} M_{2}$, that we
  denote also by $\mu$. Let $q\colon N_{1}\underset{G_{1}}{\times} M_{2}
  \longrightarrow M_{1}$ be the projection. Then $q_{\ast}(P(\Omega
  ^{\omega _{1,2}})\land \mu)$ is a closed differential form on
  $M_{1}$.
\end{proposition}
\begin{proof}
  Let $U\subset M_{1}$ be a trivializing open subset for $N_{1}$ and
  choose a trivialization of $N_{1}\mid _{U}\cong U\times G_{1}$. With this
  trivialization, we can identify $\Omega ^{\omega _{1}}\mid _{U}$ with a
  2-form on $U$ with values in $\mathfrak{g}_{1}$. 

  For $\gamma \in \mathfrak{g}_{1}$, we denote by 
  \begin{displaymath}
    \psi _{\mu}(\gamma )=\mu (P(\Omega ^{\omega_{2}
}+\varphi(\gamma))) 
  \end{displaymath}
the invariant function provided by proposition \ref{prop:2}. 

Then 
\begin{displaymath}
  q_{\ast}(P(\Omega
  ^{\omega _{1,2}})\land \mu)= \psi _{\mu}(\Omega ^{\omega _{1}}).
\end{displaymath}
Therefore, the result follows from the usual Chern-Weil theory.
\end{proof}

We go back now to complex geometry and analytic real Deligne
cohomology and to the notations \ref{def:19}, in particular
\eqref{eq:36}.  

\begin{corollary}\label{cor:3}
  Let $X$ be a complex manifold and let $\overline E=(E,h^{E})$ be a
  rank $r$ hermitian holomorphic 
  vector bundle on $X$. Let $\pi\colon \mathbb{P}(E)\longrightarrow X$ be
  the associated projective bundle. On $\mathbb{P}(E)$ we consider the
  tautological exact sequence
  \begin{displaymath}
    \overline \xi\colon
    0\longrightarrow \overline {\mathcal{O}(-1)}\longrightarrow \pi
    ^{\ast}\overline E
    \longrightarrow \overline Q \longrightarrow 0
  \end{displaymath}
  where all the vector bundles have the induced metric. Let $P_{1}$,
  $P_{2}$ and $P_{3}$ be invariant power series in $1$, $r-1$ and $r$
  variables respectively with coefficients in $\mathbb{D}$. Let
  $P_{1}(\overline {\mathcal{O}(-1)})$ and 
  $P_{2}(\overline Q)$ be the associated Chern forms and let
  $\widetilde {P}_{3}(\overline {\xi})$ the associated Bott-Chern
  class. Then 
  $$ \pi_{\ast}(P_{1}(\overline {\mathcal{O}(-1)})\bullet
  P_{2}(\overline Q)\bullet
  \widetilde {P}_{3}(\overline {\xi}))\in
  \bigoplus_{k}\widetilde{\mathcal{D}}^{2k-1}(X,k) 
  $$ 
  is closed. Hence it defines a class in analytic real Deligne
  cohomology. This class does not depend on the hermitian metric of
  $E$.   
\end{corollary}
\begin{proof}
  We consider $\mathbb{C}^{r}$ with the standard hermitian metric. On
  the space $\mathbb{P}(\mathbb{C}^{r})$ we have the tautological
  exact sequence   
  \begin{displaymath}
    0\longrightarrow \mathcal{O}_{\mathbb{P}(\mathbb{C}^{r})}(-1)
    \overset{f}{\longrightarrow }\mathbb{C}^{r}\longrightarrow
    Q\longrightarrow 0. 
  \end{displaymath}
  Let $(x:y)$ be homogeneous coordinates on $\mathbb{P}^{1}$ and let
  $t=x/y$ be the absolute coordinate. Let $p_{1}$ and $p_{2}$ be the
  two projections of $M_{2}=\mathbb{P}(\mathbb{C}^{r})\times
  \mathbb{P}^{1}$. Let $\widetilde E$ be the cokernel of the map 
  \begin{displaymath}
    \begin{matrix}
    p_{1}^{\ast}\mathcal{O}_{\mathbb{P}(\mathbb{C}^{r})}(-1)&
    \longrightarrow       &
    p_{1}^{\ast}\mathcal{O}_{\mathbb{P}(\mathbb{C}^{r})}(-1)\otimes 
    p_{2}^{\ast}\mathcal{O}_{\mathbb{P}^{1}}(1)\oplus
    p_{1}^{\ast} \mathbb{C}^{r}\otimes
    p_{2}^{\ast}\mathcal{O}_{\mathbb{P}^{1}}(1) \\
    s & \longmapsto & s\otimes y + f(s)\otimes x
    \end{matrix}
  \end{displaymath}
  with the metric induced by the standard metric of $\mathbb{C}^{r}$
  and the Fubini-Study metric of $\mathcal{O}_{\mathbb{P}(1)}(1)$. 
  
  Let $N_{2}$ be the principal bundle over $M_{2}$ formed by the
  triples $(e_{1},e_{2},e_{3})$, where $e_{1}$, $e_{2}$ and $e_{3}$
  are unitary frames of 
  $p_{1}^{\ast}\mathcal{O}_{\mathbb{P}(\mathbb{C}^{r})}(-1)$,
  $p_{1}^{\ast} Q$ and $\widetilde E$ respectively. The structure
  group of this principal bundle is $G_{2}=U(1)\times U(r-1)\times
  U(r)$. Let $\omega _{2}$ be the connection induced by the hermitian
  holomorphic connections on the vector bundles
  $p_{1}^{\ast}\mathcal{O}_{\mathbb{P}(\mathbb{C}^{r})}(-1)$, 
  $p_{1}^{\ast} Q$ and $\widetilde E$. 

  Now we denote $M_{1}=X$, and let $N_{1}$ be the bundle of unitary
  frames of $\overline E$. This is a principal bundle over $M_{1}$
  with structure
  group $G_{1}=U(r)$.

  The group $G_{1}$ acts on the left on $N_{2}$. This action commutes
  with the right action of $G_{2}$ and the connection $\omega _{2}$ is
  invariant under this action.

  Let $\mu=[-\log(|t|)]$ be the current on $M_{2}$ associated to the
  locally integrable function $-\log(|t|)$. 
  This current is invariant under the action  of $G_{1}$ because this
  group acts trivially on the factor $\mathbb{P}^{1}$.

  The invariant power series $P_{1}$, $P_{2}$ and $P_{3}$ determine an
  invariant function $P$ on $\mathfrak{g}_{2}$, the Lie algebra of
  $G_{2}$. 

  Let $\omega _{1}$ be the connection induced in $N_{1}$ by the
  holomorphic hermitian 
  connection on $\overline E$. As before let $\omega _{1,2}$ be the connection
  on $N_{1}\underset{G_{1}}{\times} N_{2}$ induced by $\omega _{1}$
  and $\omega _{2}$ and let
  $q\colon N_{1}\underset{G_{1}}{\times} M_{2}\longrightarrow
  M_{1}$ be the projection.
  Observe that $N_{1}\underset{G_{1}}{\times}
  M_{2}=\mathbb{P}(E)\times \mathbb{P}^{1}$ and $q=\pi\circ p_{1}$.

  By the projection formula and the definition of
  Bott-Chern classes we have
  \begin{displaymath}
    \pi_{\ast}(P_{1}(\overline {\mathcal{O}(-1)})\land
    P_{2}(\overline Q)\land
    \widetilde {P}_{3}(\overline {\xi}))=
    q_{\ast}(\mu \bullet P(\Omega
    ^{\omega _{1,2}})),
  \end{displaymath}
  Therefore the fact that it is closed follows from
  \ref{prop:3}. Since, for fixed $P_{1}$, $P_{2}$ and $P_{3}$, the
  construction is functorial on $(X.\overline E)$, the fact that
  the class in analytic real Deligne 
  cohomology does not depend on the choice of the hermitian metric
  follows from proposition \ref{prop:22}. 
\end{proof}    

\begin{corollary} \label{cor:4}
  Let $\overline E=(E,h^{E})$ be a hermitian holomorphic vector bundle
  on a complex 
  manifold $X$. We consider the projective bundle $\pi
  \colon\mathbb{P}(E\oplus \mathbb{C})\longrightarrow X$. Let $\overline
  {Q}$ be the 
  universal quotient bundle on the space 
  $\mathbb{P}(E\oplus \mathbb{C})$ with the induced metric. Then the
  following equality of differential forms holds 
  \begin{displaymath}
    \pi_{\ast}\sum_{i}(-1)^{i}\ch(\bigwedge ^{i}\overline
    {Q}^{\vee})=
    \pi_{\ast}(c_{r}(\overline {Q})\Td^{-1}(\overline
    {Q}))=\Td^{-1}(\overline {E}). 
  \end{displaymath}
\end{corollary}
\begin{proof}
  Let $\overline {\xi}$ be the tautological exact sequence with
  induced metrics. We first prove that
  \begin{displaymath}
    \pi_{\ast}(c_{r}(\overline Q)\Td(\overline {\mathcal{O}(-1)}))=1. 
  \end{displaymath}
We can write $\Td(\overline {\mathcal{O}(-1)})=1+c_{1}(\overline
{\mathcal{O}(-1)}) \phi (\overline
{\mathcal{O}(-1)})$ for certain power series $\phi $. Since
$c_{r+1}(\overline E\oplus \mathbb{C})=0$ we have 
\begin{displaymath}
  c_{r}(\overline
  Q)c_{1}(\overline{\mathcal{O}(-1)})=\dd_{\mathcal{D}}\widetilde{c}_{r+1}(\overline
  {\xi}). 
\end{displaymath}
Therefore, by corollary \ref{cor:3}, we have 
\begin{align*}
  \pi_{\ast}(c_{r}(\overline Q)\Td(\overline {\mathcal{O}(-1)}))&=
  \pi_{\ast}(c_{r}(\overline Q)) + \pi_{\ast}(c_{r}(\overline Q) c_{1}(\overline 
  {\mathcal{O}(-1)}) \phi (\overline
  {\mathcal{O}(-1)}))\\
  &= 1+\dd_{\mathcal{D}} \pi_{\ast}(\widetilde c_{r+1}(\overline
  \xi) \phi (\overline 
  {\mathcal{O}(-1)}))\\ 
  &=1. 
\end{align*}

Then the corollary follows from corollary \ref{cor:3} by using the identity 
\begin{multline*}
  \pi_{\ast}(c_{r}(\overline {Q})\Td^{-1}(\overline {Q}))=
  \pi_{\ast}(c_{r}(\overline Q)\Td(\overline {\mathcal{O}(-1)})
  \pi^{\ast}\Td^{-1}(\overline E))\\
  +\dd_{\mathcal{D}}
  \pi_{\ast}(c_{r}(\overline Q)\Td(\overline
  {\mathcal{O}(-1)})\widetilde {\Td^{-1}}(\overline \xi)).
\end{multline*}
\end{proof}

The following generalization of corollary \ref{cor:3} provides many
relations between integrals of Bott-Chern classes and is left to the
reader.

\begin{corollary}\label{cor:11}
  Let $X$ be a complex manifold and let $\overline E=(E,h^{E})$ be a
  rank $r$ hermitian holomorphic 
  vector bundle on $X$. Let $\pi\colon\mathbb{P}(E)\longrightarrow X$ be
  the associated projective bundle. On $\mathbb{P}(E)$ we consider the
  tautological exact sequence
  \begin{displaymath}
    \overline \xi\colon
    0\longrightarrow \overline {\mathcal{O}(-1)}\longrightarrow \pi
    ^{\ast}\overline E
    \longrightarrow \overline Q \longrightarrow 0
  \end{displaymath}
  where all the vector bundles have the induced metric. Let $P_{1}$ and
  $P_{2}$ be invariant power series in $1$ and $r-1$
  variables respectively with coefficients in $\mathbb{D}$ and let
  $P_{3},\dots ,P_{k}$ be  invariant power series in $r$
  variables with coefficients in $\mathbb{D}$.
  Let
  $P_{1}(\overline {\mathcal{O}(-1)})$ and 
  $P_{2}(\overline Q)$ be the associated Chern forms and let
  $\widetilde {P}_{3}(\overline {\xi}),\dots ,\widetilde
  {P}_{k}(\overline {\xi})$ be the associated Bott-Chern
  classes. Then 
  $$\pi_{\ast}(P_{1}(\overline {\mathcal{O}(-1)})\bullet
  P_{2}(\overline Q)\bullet
  \widetilde {P}_{3}(\overline {\xi})\bullet \dots \bullet
  \widetilde {P}_{k}(\overline {\xi}))$$ 
  is a closed differential form
  on $X$ for any choice of the ordering in computing the non
  associative product under the integral. 
\end{corollary}

\section{Cohomology of currents and wave front sets}
\label{sec:cohom-curr-wave}

The aim of this section is to prove the Poincar\'e lemma for the
complex of currents with fixed wave front set. This implies in
particular a certain $\partial\bar\partial$-lemma (corollary
\ref{cor:1}) that will allow us to control the singularities of
singular Bott-Chern classes.

Let $X$ be a complex manifold of dimension $n$.
Following notation \ref{def:19} recall that there is a canonical
isomorphism   
\begin{displaymath}
  H^{\ast}_{\mathcal{D}^{\an}}(X,\mathbb{R}(p))\cong
  H^{\ast}(\mathcal{D}^{\ast}_{D}(X,p)).
\end{displaymath}

A current $\eta$ can be viewed as a generalized
section of a vector bundle and, as such, has a wave front set that is
denoted by
$\WF(\eta)$. The theory of wave front sets of distributions is
developed in \cite{Hormander:MR1065993} chap. VIII. For the theory of
wave front 
sets of generalized sections, the reader can consult
\cite{GuilleminStenberg:MR0516965} chap. VI. Although we will work
with currents and 
hence with generalized sections of vector bundles, we will follow
\cite{Hormander:MR1065993}. 

The wave front set
of $\eta$ is a closed conical subset of the cotangent bundle of $X$
minus the 
zero section 
$T^{\ast}X_{0}=T^{\ast}X\setminus \{0\}$. This set describes the
points and directions of the singularities of $\eta$ and it allows us
to define certain products and inverse images of currents.

Let $S\subset T^{\ast}X_{0}$ be a closed conical subset, we will denote by
$\mathscr{D}^{\ast}_{X,S}$ the subsheaf of currents whose wave front set
is contained in $S$. We will denote by  $D^{\ast}(X,S)$ its 
complex of global sections. 

For every open set $U\subset X$ there is an appropriate notion of
convergence in $\mathscr{D}^{\ast}_{X,S}(U)$ (see
\cite{Hormander:MR1065993} VIII  Definition 8.2.2). All references to
continuity below are with respect to this notion of convergence. 

We next summarize the basic properties of wave front sets.

\begin{proposition}\label{prop:4}
  Let $u$ be a generalized section of a vector bundle and let $P$ be a
  differential operator with smooth coefficients. Then
  \begin{displaymath}
    \WF(Pu)\subseteq \WF(u).
  \end{displaymath}
\end{proposition}
\begin{proof}
  This is \cite{Hormander:MR1065993} VIII (8.1.11).
\end{proof}

\begin{corollary}
  The sheaf $\mathscr{D}^{\ast}_{X,S}$ is closed under $\partial$ and $\bar
  \partial$. Therefore it is a sheaf of Dolbeault complexes.   
\end{corollary}

Let $f\colon X\longrightarrow Y$ be a morphism of complex
manifolds. The \emph{set of normal directions} of $f$ is
\begin{displaymath}
  N_{f}=\{(f(x),v)\in T^{\ast}Y \mid df(x)^{t}v=0\}.
\end{displaymath}

This set measures the singularities of $f$. For instance, if $f$ is a
smooth map then $N_{f}=0$ whereas, if $f$ is a closed immersion,
$N_{f}$ is the conormal bundle of $f(X)$. Let $S\subset T^{\ast}Y_{0}$ be a
closed conical subset. We will say that $f$ is transverse to $S$ if
$N_{f}\cap S=\emptyset$. We will denote
\begin{displaymath}
  f^{\ast}S=\{(x,df(x)^{t}v)\in T^{\ast}X_{0}\mid (f(x),v)\in
  S\}. 
\end{displaymath}

\begin{theorem}\label{thm:1}
Let $f\colon X\longrightarrow Y$ be a morphism of complex manifolds
that is transverse to $S$. Then
there exists one and only one extension of the pull-back morphism
$f^{\ast}\colon\mathscr{E}^{\ast}_{Y}\longrightarrow
\mathscr{E}^{\ast}_{X}$ to a continuous morphism
\begin{displaymath}
  f^{\ast}\colon\mathscr{D}^{\ast}_{Y,S}\longrightarrow
  \mathscr{D}^{\ast}_{X,f^{\ast}S}. 
\end{displaymath}
In particular there is a continuous morphism of complexes
\begin{displaymath}
  D^{\ast}(Y,S)\longrightarrow D^{\ast}(X,f^{\ast}S).
\end{displaymath}
\end{theorem}
\begin{proof}
  This follows from \cite{Hormander:MR1065993} theorem 8.2.4.  
\end{proof}

We now recall the effect of correspondences on the wave front sets.

Let $K\in D^{\ast}(X\times Y)$, and let $S$ be a conical subset of
$T^{\ast}Y_{0}$.  We will write
\begin{align*}
  \WF(K)_{X}&=\{(x,\xi)\in T^{\ast}X_{0}\mid \exists y\in Y,
  (x,y,\xi,0)\in \WF(K)\}\\
  \WF'(K)_{Y}&=\{(y,\eta)\in T^{\ast}Y_{0}\mid \exists x\in X,
  (x,y,0,-\eta)\in \WF(K)\}\\
  \WF'(K)\circ S&=\{(x,\xi)\in T^{\ast}X_{0}\mid \exists (y,\eta)\in S,
  (x,y,\xi,-\eta)\in \WF(K)\}.
\end{align*}

\begin{theorem}\label{thm:4} The image of the correspondence map
  \begin{displaymath}
    \begin{matrix}
      E^{\ast}_{c}(Y) & \longrightarrow & D^{\ast}(X)\\
      \eta & \longmapsto & p_{1 \ast}(K\land p_{2}^{\ast}(\eta))
    \end{matrix}
  \end{displaymath}
is contained in $ D^{\ast}(X,WF(K)_{X})$. Moreover, if $S\cap
\WF'(K)_{Y}=\emptyset$, then  
  there exists one and only one extension 
  to a continuous map
  \begin{displaymath}
    D^{\ast}_{c}(Y,S) \longrightarrow D^{\ast}(X,S'),
  \end{displaymath}
  where $S'=\WF(K)_{X}\cup \WF'(K)\circ S$.
\end{theorem}
\begin{proof}
  This is \cite{Hormander:MR1065993} theorem 8.2.13.
\end{proof}

We are now in a position to state and prove the Poincar\'e lemma for
currents with fixed wave front set.
As usual, we will
denote by $F$ the Hodge filtration of any Dolbeault complex.
  
\begin{theorem}[Poincar\'e lemma]\label{thm:2} Let $S$ be any conical
  subset of $T^{\ast}X_{0}$. Then the natural morphism 
  \begin{displaymath}
    \iota\colon (E^{\ast}(X), F)\longrightarrow
    (D^{\ast}(X,S), F) 
  \end{displaymath}
  is a filtered quasi-isomorphism.
\end{theorem}
\begin{proof}
  Let $K$ be the Bochner-Martinelli integral operator on $\mathbb{C}^{n}\times
  \mathbb{C}^{n}$. 
  It is the operator
  \begin{displaymath}
    \begin{matrix}
      E_{c}^{p,q}(\mathbb{C}^{n})&\longrightarrow &
      E^{p,q-1}(\mathbb{C}^{n})\\
       \varphi&\longmapsto & \int_{w\in \mathbb{C}^{n}}k(z,w)\land \varphi(w),
    \end{matrix}
  \end{displaymath}
  where $k$ is the Bochner-Martinelli kernel
  (\cite{GriffithsHarris:pag} pag. 383). Thus $k$ is a differential
  form on $\mathbb{C}^{n}\times 
  \mathbb{C}^{n}$ with singularities only along the diagonal.

  Using the explicit description of $k$ in
  \cite{GriffithsHarris:pag}, it can be seen that
  $WF(k)=N^{\ast}\Delta _{0}$, the conormal bundle of the diagonal. By
  theorem \ref{thm:4}, 
  the operator $K$ defines a continuous linear map
  from $\Gamma _{c}(\mathbb{C}^{n},\mathscr{D}^{\ast}_{\mathbb{C}^{n},S})$ to 
  $\Gamma
  (\mathbb{C}^{n},\mathscr{D}^{\ast}_{\mathbb{C}^{n},S})$. This is the
  key fact that allows us to adapt the proof of the Poincar\'e Lemma
  for arbitrary currents to the case of currents with fixed wave front
  set.   

  We will prove that the sheaf inclusion
  \begin{displaymath}
    (\mathscr{E}_{X},F)\longrightarrow (\mathscr{D}_{X,S},F)
  \end{displaymath}
  is a filtered quasi-isomorphism. Then the theorem will follow from
  the fact that both are fine sheaves.
  
  The previous statement is equivalent to the fact that, for any
  integer $p\ge 0$, the inclusion 
  \begin{displaymath}
    \iota\colon \mathscr{E}^{p,*}_{X}\longrightarrow \mathscr{D}^{p,*}_{X,S}
  \end{displaymath}
  is a quasi-isomorphism.

  Let $x\in X$, since exactness can be checked at the level of stalks,
  we need to show that
  \begin{displaymath}
    \iota_{x}\colon \mathscr{E}^{p,*}_{X,x}\longrightarrow \mathscr{D}^{p,*}_{X,S,x}
  \end{displaymath}
is a quasi-isomorphism.
let $U$ be a coordinate neighborhood around $x$ and let
  $x\in V\subset U$ be a relatively compact open subset.

  Let $\rho \in C_{c}^{\infty}(U)$ be a function with compact support
  such that $\rho \mid _{V}=1$. We define an operator
  \begin{displaymath}
    K\rho \colon\mathscr{D}^{p,q}_{X,S}(U)\longrightarrow
    \mathscr{D}^{p,q-1}_{X,S}(V).
  \end{displaymath}
  If $T\in \mathscr{D}^{p,q}_{X,S}(U)$ and $\varphi \in E_{c}^{\ast}(V)$
  is a test form, then
  \begin{displaymath}
    K\rho (T)(\varphi)=(-1)^{p+q}T(\rho K(\varphi)).
  \end{displaymath}
  Hence, using that $\bar \partial K(\varphi)+K(\bar \partial
  \varphi)=\varphi$,  and that
  $\varphi=\rho \varphi$, we have
  \begin{displaymath}
    (\bar \partial K\rho T + K\rho \bar \partial T +T)(\varphi)=
    -T(\bar \partial (\rho )\land K(\varphi)).
  \end{displaymath}
  Observe that, even if the support of $\varphi$ is contained in $V$,
  the support of $K(\varphi)$ can be $\mathbb{C}^{n}$; therefore the
  right hand side of the above equation may be non zero.

  We compute
  \begin{align*}
    T(\bar \partial (\rho )\land K(\varphi))&=
    T\left(\bar \partial (\rho)\land \int_{w\in
        \mathbb{C}^{n}}k(w,z)\land \varphi(w)
      \right)\\
      &=
      T\left(\int_{w\in
        \mathbb{C}^{n}}\bar \partial (\rho)\land k(w,z)\land
      \varphi(w) \right).
  \end{align*}

  Since $\supp(\varphi)\subset V$ and $\bar \partial (\rho)|_{V}\equiv
  0$, we can find a number $\epsilon >0$ such that, if
  $\|z-w\|<\epsilon $, then $\bar \partial (\rho)\land k(w,z)\land
      \varphi(w) =0$. Since the singularities of $k(w,z)$ are
      concentrated on the diagonal, it follows that 
  the differential form $\bar \partial (\rho)\land k(w,z)\land
  \varphi(w)$ is smooth. Therefore, the current in $V$ given by
  \begin{displaymath}
    \varphi\longmapsto T\left(\int_{w\in
        \mathbb{C}^{n}}\bar \partial (\rho)\land k(w,z)\land
      \varphi(w) \right),
  \end{displaymath}
  is the current associated to the smooth differential form 
  $T_{z}\left(\bar \partial (\rho)\land
    k(w,z)\right)$, where the subindex $z$ means that $T$ only acts on the
 $z$ variable, being $w\in V$ a parameter. This smooth form will be denoted by $\Psi (T)$.
  
  Summing up, we have shown that, for any current $T\in
  \mathscr{D}^{p,q}_{X,S}(U)$ there exists a smooth differential form
  $\Psi (T)\in \mathscr{E}^{p,q}_{X}(V)$ such that
  \begin{displaymath}
    T\mid _{V}=-\bar \partial K\rho T -K\rho \bar \partial T -\Psi (T).
  \end{displaymath}
  Observe that we can not say that $\Psi $ is a quasi-inverse of
  $\iota _{x}$ because it depends on the choice of $\rho $ and it is
  not possible to choose a single $\rho $ that can be applied to all
  $T$. Hence it 
  is not a well defined operator at the level of stalks. 
  Let now $T\in \mathscr{D}^{p,*}_{X,S,x}$ be closed. It is defined in
  some neighborhood of $x$, say $U'$. Applying the above procedure we
  find a smooth differential form $\Psi (T)$ defined on a relatively
  compact subset of $U'$, say $V'$, that is cohomologous to $T$. Hence
  the map induced by $\iota _{x}$ in cohomology is surjective. Let
  $\omega \in \mathscr{E}^{p,*}_{X,x}$ be closed and such that $\iota
  _{x}\omega =\bar \partial T$ for some $T\in
  \mathscr{D}^{p,*-1}_{X,S,x}$. We may assume that $\omega $ and $T$
  are defined is some neighborhood $U''$ of $x$. Then, on some
  relatively compact subset $V''\subset U''$, we have
  \begin{displaymath}
    \omega \mid _{V''} = \bar \partial T\mid _{V''}=-\bar \partial K\rho \omega
    - \bar \partial \Psi (T).
  \end{displaymath}
  Since $K\rho \omega$ and $\Psi (T)$ are smooth differential forms we
  conclude that the map induced by $\iota_{x}$ in cohomology is injective.
\end{proof}

We will denote by  
$\mathcal{D}^{\ast}_{D}(X,S,p)$ the Deligne complex associated to
$D^{\ast}(X,S)$.

The following two results are direct consequences of theorem \ref{thm:2}.
\begin{corollary} \label{cor:2} The inclusion
  $\mathcal{D}^{\ast}_{D}(X,S,p)\longrightarrow
  \mathcal{D}^{\ast}_{D}(X,p)$ induces 
  an isomorphism  
\begin{displaymath}
  H^{\ast}(\mathcal{D}^{\ast}_{D}(X,S,p)) \cong
H^{\ast}_{\mathcal{D}^{\an}}(X,\mathbb{R}(p)).
\end{displaymath}
\end{corollary}

\begin{corollary} \label{cor:1}
  \begin{enumerate}
  \item Let $\eta\in \mathcal{D}^{n}_{D}(X,p)$ be a current such
    that $$\dd_{\mathcal{D}}\eta \in \mathcal{D}^{n+1}_{D}(X,S,p),$$ then
    there is a current $a\in \mathcal{D}^{n-1}_{D}(X,p)$ such that
    $\eta + \dd_{\mathcal{D}} a\in \mathcal{D}^{n}_{D}(X,S,p)$.
  \item Let $\eta\in \mathcal{D}^{n}_{D}(X,S,p)$ be a current such
    that there is a current $a\in \mathcal{D}^{n-1}_{D}(X,p)$ with 
    $\eta = \dd_{\mathcal{D}} a$, then there is a current $b\in
    \mathcal{D}^{n-1}_{D}(X,S,p)$ such that $\eta = \dd_{\mathcal{D}}
    b$.  
  \end{enumerate}
\hfill $\square$
\end{corollary}

\section{Deformation of resolutions}
\label{sec:deform-resol}

In this section we will recall  the deformation of resolutions based
on the Grassmannian 
graph construction of \cite{BaumFultonMacPherson:RRsv}.  We will
also recall the Koszul resolution associated to a section of a vector
bundle.

The main theme is that given a bounded 
complex $E_{\ast}$ of locally free sheaves (with some properties) on a
complex manifold $X$, one can
construct a bounded complex 
$\tr_{1}(E_{\ast})_{\ast}$ 
over a certain manifold  
$W$. This new manifold has a birational map $\pi \colon W\longrightarrow X\times
\mathbb{P}^{1}$, that is an isomorphism over $X\times \mathbb{P}^{1}\setminus
\{\infty\}$.  
The complex
$\tr_{1}(E_{\ast})_{\ast}$ agrees with the original 
complex over $X\times \{0\} $ and is particularly simple over $\pi^{-1}(X\times
\{\infty\})$. Thus $\tr_{1}(E_{\ast})_{\ast}$ is a
deformation of the original complex to a simpler one. The two
examples we are interested in are: first, when
the original complex is exact, then $W$ 
agrees with $X\times \mathbb{P}^{1}$ and  $\tr_{1}(E_{\ast})_{\ast}$ was
defined in \ref{def:12}. Its restriction
to $\pi^{-1}(X\times \{\infty\})$ is split; second, when $i\colon
Y\longrightarrow 
X$ is a closed immersion of complex manifolds, and $E_{\ast}$ is a bounded
resolution of $i_{\ast}\mathcal{O}_{Y}$, then $W$ agrees with the
deformation to the normal cone of $Y$ and the restriction of
$\tr_{1}(E_{\ast})_{\ast}$ to $\pi^{-1}(X\times \{\infty\})$ is an extension
  of a Koszul resolution by a split complex. Note that, if we allow
  singularities, then the Grassmannian graph construction is much more
  general. 

The deformation
of resolutions is based on the Grassmannian 
graph construction of \cite{BaumFultonMacPherson:RRsv}, and, in the
form that we present here, has been developed in  
\cite{BismutGilletSoule:MR1086887} and \cite{GilletSoule:aRRt}.

In order to fix notations we first recall the deformation to the
normal cone and the Koszul resolution associated to the
zero section of a vector bundle.

Let $Y\hookrightarrow X$ be a closed immersion of complex manifolds,
with $Y$ of pure codimension $n$. In the sequel we will use 
notation \ref{def:3}.
Let $W=W_{Y/X}$ be the blow-up of $X\times \mathbb{P}^{1}$ along
$Y\times \{\infty\}$. Since $Y$ and $X\times \mathbb{P}^{1}$ are
manifolds, $W$ is also a manifold. The map $\pi \colon W\longrightarrow
X\times \mathbb{P}^{1}$ is an isomorphism away from $Y\times \{\infty\}$; we
will write $P$ for the exceptional divisor of the blow-up. Then
\begin{displaymath}
  P=
  \mathbb{P}(N_{Y/X}\otimes N^{-1}_{\infty/\mathbb{P}^{1}}\oplus \mathbb{C}).
\end{displaymath}
 Thus $P$ can be seen 
as the projective completion of the vector bundle $N_{Y/X}\otimes
N^{-1}_{\infty/\mathbb{P}^{1}}$. Note that $N_{\infty/\mathbb{P}^{1}}$
is trivial although not canonically trivial. Nevertheless we can
choose to trivialize it by means of the
section $y\in \mathcal{O}_{\mathbb{P}^{1}}(1)$.
Sometimes we will tacitly assume this trivialization and omit
$N_{\infty/\mathbb{P}^{1}}$ from the formulae. 

The map $q_{W}\colon W\longrightarrow \mathbb{P}^{1}$, obtained by composing
$\pi $ with the projection $q\colon X\times \mathbb{P}^{1}\longrightarrow
\mathbb{P}^{1}$, is flat and, 
for $t\in \mathbb{P}^{1}$, we have
\begin{displaymath}
  q_{W}^{-1}(t) \cong
  \begin{cases}
    X\times\{t\},&$\text{ if } $t\not = \infty, \\
    P\cup \widetilde X, &\text{ if } t=\infty,
  \end{cases}
\end{displaymath}
where $\widetilde X$ is the blow-up of $X$ along $Y$, and $P\cap
\widetilde X$ is, at the same time, the divisor at $\infty$ of $P$ and the
exceptional divisor of $\widetilde X$. 

Following \cite{BismutGilletSoule:MR1086887} we will use the following notations
\begin{displaymath}
  \xymatrix{
  P \ar[r]^{f} \ar[d]_{\pi _{P}}& W\ar[d]^{\pi}\\
Y\times \{\infty\} \ar[r]^{i_{\infty}} & X\times \mathbb{P}^{1}}
\end{displaymath}
\begin{displaymath}
  \begin{array}{cr}
    i\colon Y\longrightarrow X,&\\
    W_{\infty}=\pi ^{-1}(\infty)=P\cup\widetilde X,&\\
    q\colon X\times \mathbb{P}^{1}\longrightarrow \mathbb{P}^{1},& \text{the
      projection,}\\
    p\colon X\times \mathbb{P}^{1}\longrightarrow X,& \text{the
      projection,}\\
    q_{W}=q \circ \pi &\\
    p_{W}=p \circ \pi &\\
    q_{Y}\colon Y\times \mathbb{P}^{1}\longrightarrow \mathbb{P}^{1},& \text{the
      projection,}\\
    p_{Y}\colon Y\times \mathbb{P}^{1}\longrightarrow Y,& \text{the
      projection,}\\
    j\colon Y\times \mathbb{P}^{1}\longrightarrow W& \text{the induced map,}\\ 
    j_{\infty}\colon  Y\times \{\infty\} \longrightarrow P. &\\
  \end{array}
\end{displaymath}
Given any map $g\colon Z\longrightarrow X\times \mathbb{P}^{1}$, we will
denote $p_{Z}=p\circ g$ and $q_{Z}=q\circ g$. For instance
$p_{P}=p\circ \pi\circ f=p_{W}\circ f=i\circ \pi _{P}$, where, in the
last equality, we are identifying $Y$ with $Y\times \{\infty\}$.

We next recall the construction of the Koszul resolution. Let $Y$ be a
complex manifold and let $N$ be a rank $n$ vector bundle. Let
$P=\mathbb{P}(N\oplus \mathbb{C})$ be the projective bundle of lines
in $N\oplus \mathbb{C}$. It is obtained by 
completing $N$ with the divisor at infinity. Let $\pi
_{P}\colon P\longrightarrow Y$ be the projection 
and let $s\colon Y\longrightarrow P$ be the zero section. On $P$ there
is a tautological short exact sequence 
\begin{equation}
  \label{eq:5}
  0\longrightarrow \mathcal{O}(-1) \longrightarrow 
  \pi _{P}^{\ast}(N\oplus \mathbb{C})\longrightarrow 
  Q\longrightarrow 0.
\end{equation}

The above exact sequence and the inclusion $\mathbb{C}\longrightarrow
\pi _{P}^{\ast}(N\oplus \mathbb{C})$ induce a section $\sigma
\colon\mathcal{O}_{P}\longrightarrow Q$ that vanishes along the zero section
$s(Y)$. By duality we obtain a morphism $Q^{\vee}\longrightarrow
\mathcal{O}_{P}$ that induces a long exact sequence
\begin{displaymath}
  0\longrightarrow \bigwedge^{n}Q^{\vee}\longrightarrow \dots
  \longrightarrow  \bigwedge^{1} Q^{\vee}\longrightarrow
  \mathcal{O}_{P}\longrightarrow  s_{\ast}
  \mathcal{O}_{Y}\longrightarrow 0. 
\end{displaymath}

If $F$ is another vector bundle over $Y$, we obtain an exact
sequence,  
\begin{equation}
\label{eq:6}
  0\longrightarrow \bigwedge^{n}Q^{\vee}\otimes \pi
  _{P}^{\ast}F\longrightarrow \dots 
  \longrightarrow  \bigwedge^{1} Q^{\vee}\otimes \pi _{P}^{\ast}F\longrightarrow
  \pi _{P}^{\ast}F\longrightarrow  s_{\ast}
  F\longrightarrow 0. 
\end{equation}

\begin{definition} \label{def:2}
  The \emph{Koszul resolution} of $s_{\ast}(F)$ is the resolution
  \eqref{eq:6}. The complex
\begin{displaymath}
  0\longrightarrow \bigwedge^{n}Q^{\vee}\otimes \pi
  _{P}^{\ast}F\longrightarrow \dots 
  \longrightarrow  \bigwedge^{1} Q^{\vee}\otimes \pi _{P}^{\ast}F\longrightarrow
  \pi _{P}^{\ast}F\longrightarrow  0
\end{displaymath}
will be  denoted by $K(F,N)$. 
When $\overline N$ is a hermitian vector bundle, the exact sequence
\eqref{eq:5} 
induces a hermitian metric on $Q$. If, moreover, $\overline F$ is
also a hermitian
vector bundle, all the vector bundles that appear in the Koszul resolution
have an induced hermitian metric. We will denote by $K(\overline
F,\overline N)$ the corresponding complex of hermitian vector
bundles. 
\end{definition}

In particular, we shall write $K(\overline
{\mathcal{O}_Y},\overline N)$ if $F=\mathcal{O}_Y$ is endowed with
the trivial metric $\|1\|=1$, unless expressly stated otherwise.

We finish this section by recalling the results about deformation
of resolutions that will be used in the sequel. For more details see 
\cite{BaumFultonMacPherson:RRsv} II.1,
\cite{BismutGilletSoule:MR1086887} Section 4 (c) 
and \cite{GilletSoule:aRRt} Section 1.

\begin{theorem} \label{thm:5}
  Let $i:Y\hookrightarrow X$ be a closed immersion of complex
  manifolds, where $Y$ may be empty. Let $U=X\setminus Y$. Let $F$ be a
  vector bundle over $Y$ and
    $E_{\ast}\longrightarrow i_{\ast} F\longrightarrow 0$ be a
    resolution of $i_{\ast}F$.
    Then
    there exists a complex manifold $W=W(E_{\ast})$, called the Grassmannian
  graph construction, 
  with a birational map $\pi \colon W\longrightarrow X\times
  \mathbb{P}^{1}$ and 
  a complex of vector bundles, $\tr_{1}(E_{\ast})_{\ast}$,
  over $W$ such that 
  \begin{enumerate}
  \item The map $\pi $ is an isomorphism away from $Y\times
    \{\infty\}$. The restriction of $\tr_{1}(E_{\ast})_{\ast}$ to $X\times
    (\mathbb{P}^{1}\setminus \{\infty\})$ is isomorphic to
    $p_{W}^{\ast} E_{\ast}$ restricted to $X\times
    (\mathbb{P}^{1}\setminus \{\infty\})$. Moreover, If $\widetilde X$
    is the Zariski closure of $U\times \{\infty\}$ inside $W$,  the
    restriction 
    of $\tr_{1}(E_{\ast})_{\ast}$ to $\widetilde X$ is split acyclic. In
    particular, if $Y$ is empty or $F$ is the zero vector bundle, hence
    $E_{\ast}$ is acyclic in the whole $X$, then
    $W=X\times 
    \mathbb{P}^{1}$ and $\tr_{1}(E_{\ast})_{\ast}$ is the first
    transgression exact sequence introduced in \ref{def:12}.
  \item \label{item:7} When $Y$ is non-empty and $F$ is a non-zero
    vector bundle over $Y$, then $W(E_{\ast})$ agrees with $W_{Y/X}$, the
    deformation to the 
    normal cone of 
    $Y$. Moreover, there is an exact sequence of resolutions on~$P$ 
    \begin{displaymath}
      \xymatrix{
        0 \ar[r]&
        A_{\ast} \ar[r] \ar[d]&
        \tr_{1}(E_{\ast})_{\ast}\mid _{P}\ar[r] \ar[d]&
        K(F,N_{Y/X}\otimes N^{-1}_{\infty/\mathbb{P}^{1}})\ar[r]
        \ar[d] & 0\\ 
        & 0 \ar[r]& (j_{\infty})_{\ast}F\ar[r]^{=} &
        (j_{\infty})_{\ast}F 
        &  
      },
    \end{displaymath}
    where $A_{\ast}$ is split acyclic and $K(F,N_{Y/X}\otimes
    N^{-1}_{\infty/\mathbb{P}^{1}})$ is the Koszul resolution.
  \item \label{item:18} Let $f\colon X'\longrightarrow X$ be a morphism of
    complex
    manifolds and assume that we are in one of the following cases:
    \begin{enumerate}
    \item The map $f$ is smooth.
    \item The map $f$ is arbitrary and $E_{\ast}$
      is acyclic. 
    \item \label{item:11}  $f$ is transverse to $Y$.
    \end{enumerate}
    Then $E'_{\ast}:=f^{\ast}(E_{\ast})$ is exact over $f^{-1}(U)$, 
    \begin{displaymath}
      W':=W(E'_{\ast})=W\underset{X}{\times}X',
    \end{displaymath}
    with $f_{W}\colon W'\longrightarrow W$ the induced map, 
    and we have
    $f_{W}^{\ast}(\tr_{1}(E_{\ast})_{\ast})=\tr_{1}(f^{\ast}(E_{\ast}))_{\ast}$.  
  \item \label{item:6} If the vector bundles $E_{i}$ are provided with
    hermitian 
    metrics, then one can choose a hermitian metric on
    $\tr_{1}(E_{\ast})_{\ast}$ such that its restriction to $X\times \{0\}$
    is 
    isometric to $E_{\ast}$ and the restriction to $U\times
    \{\infty\}$ is orthogonally split. We will denote by
    $\tr_{1}(\overline {E}_{\ast})_{\ast}$ the complex $\tr_{1}(E_{\ast})_{\ast}$
    with such a choice of hermitian metrics. 
    Moreover, this choice of metrics can be made functorial. That is, if $f$ is
    a map as in item \ref{item:18}, then 
    \begin{displaymath}
      f_{W}^{\ast}(\tr_{1}(\overline
      E_{\ast})_{\ast})=\tr_{1}(f^{\ast}(\overline E_{\ast}))_{\ast}
    \end{displaymath}
  \end{enumerate}
\end{theorem}

\begin{proof}
  The case when $E_{\ast}$ is acyclic has already been treated. For
  the case when $Y$ is non-empty and $F$ is non zero, 
  we first recall the construction  of the
  Grassmannian graph of an arbitrary complex from
  \cite{GilletSoule:aRRt}, which 
  is more general than what we need here. If
  $E$ is a vector bundle over $X$ we will denote by $E(i)$ the vector
  bundle over $X\times \mathbb{P}^{1}$ given by $E(i)=p^{\ast}E\otimes
  q^{\ast} \mathcal{O}(i)$.

  Let
  $\widetilde C_{\ast}$ be the complex of locally free sheaves given
  by $\widetilde C_{i}=E_{i}(i)\oplus E_{i-1}(i-1)$ with differential
  given by $\dd(a,b)=(b,0)$. On $X\times (\mathbb{P}^{1}\setminus
  \{\infty\})$ we consider, for each $i$, the inclusion of vector
  bundles $\gamma _{i}\colon E_{i}\hookrightarrow \widetilde C_{i}$ given by
  $s\longmapsto (s\otimes y^{i},\dd s\otimes y^{i-1})$. Let $G$ be the
  product of the Grassmann bundles $Gr(n_{i},\widetilde C_{i})$
  that parametrize rank $n_{i}=\rk E_{i}$ subbundles of $\widetilde
  C_{i}$ over $X\times \mathbb{P}^{1}$. The inclusion $\gamma_{\ast }
  \colon \bigoplus E_{i}\longrightarrow \bigoplus \widetilde C_{i}$ induces
  a section $s$ of $G$ over $X\times \mathbb{A}^{1}$.

  Then $W(E_{\ast})$ is defined to be the closure of $s(X\times
  \mathbb{A}^{1})$ in $G$. Since the projection from $G$ to $X\times
  \mathbb{P}^{1}$ is proper, the same is true for the induced map
  $\pi\colon W\longrightarrow X\times 
  \mathbb{P}^{1}$.  For each $i$,
  the induced map $W\longrightarrow
  Gr(n_{i},\widetilde C_{i})$ defines a subbundle
  $\tr_{1}(E_{\ast})_{i}$ of $\pi ^{\ast} \widetilde C_{i}$. This subbundle
  agrees with $E_{i}$ over $X\times \mathbb{A}^{1}$. The differential
  of $\widetilde C_{\ast}$ induces a differential on
  $\tr_{1}(E_{\ast})_{\ast}$.  

  Assume now that the bundles $E_{i}$ are provided with hermitian
  metrics. Using the Fubini-Study metric of $\mathcal{O}(1)$ we obtain
  induced metrics on $\widetilde C_{i}$. Over $\pi ^{-1}(X\times
  (\mathbb{P}^{1}\setminus \{\infty\}))$ we induce a metric on
  $\tr_{1}(E_{\ast})_{i}$ by means of the identification with
  $E_{i}$. Over $\pi ^{-1}(X\times
  (\mathbb{P}^{1}\setminus \{0\}))$ we consider on
  $\tr_{1}(E_{\ast})_{i}$ the metric induced by $\widetilde C_{i}$. We glue
  together 
  both metrics with the partition of unity $\{\sigma _{0},\sigma
  _{\infty}\}$ of notation \ref{def:3}. 

  In the case we are interested there is a more explicit description of
  $\tr_{1}(E_{\ast})_{\ast}$ given in \cite{BismutGilletSoule:MR1086887}
  Section 4 (c). Namely, $\tr_{1}(E_{\ast})_{i}$ is the kernel of the
  morphism
  \begin{equation}\label{eq:32}
    \phi \colon  p_{W}^{\ast}\widetilde C_{i}=
    p_{W}^{\ast}E_{i}(i)\oplus p_{W}^{\ast}E_{i-1}(i-1)\longrightarrow 
    p_{W}^{\ast}E_{i-1}(i)\oplus p_{W}^{\ast}E_{i-2}(i-1)
  \end{equation}
  given by $\phi(s,t)=(\dd s-t\otimes y,\dd t)$.

  The only statements that are not explicitly proved in
  \cite{BismutGilletSoule:MR1086887} 
  Section 4 (c) 
  or \cite{GilletSoule:aRRt} Section 1 are the functoriality when $f$
  is not smooth and the properties of the  explicit choice of metrics. 

  If the complex
  $E_{\ast}$ is acyclic,
  then 
  the same is true for
  $E'_{\ast}=f^{\ast}E_{\ast}$. In this case $W=X\times
  \mathbb{P}^{1}$ and $W'=X'\times \mathbb{P}^{1}$. Then the
  functoriality follows from the definition of $\tr_{1}(E_{\ast})_{\ast}$.

  Assume now that we are in case \ref{item:11}. We can form the
  Cartesian square
  \begin{displaymath}
    \xymatrix{
      Y' \ar[r]^{i'} \ar[d]_{g}&X'\ar[d]^{f}\\
      Y \ar[r]^{i} &X
    }
  \end{displaymath}
  where $i'$ is also a closed immersion of complex
  manifolds. Then we have that $E'_{\ast}$ is a resolution of
  $i'_{\ast}g^{\ast} F$. Hence $W'=W(E'_{\ast})$ is the
  deformation to the normal cone of $Y'$ and therefore
  $W'=W\underset{X}{\times}X'$. Again the functoriality of
  $\tr_{1}(E_{\ast})_{\ast}$ can be checked using the 
  explicit construction of \cite{GilletSoule:aRRt} Section 1 that we
  have recalled above.
\end{proof}

\begin{remark}
  \begin{enumerate}
  \item The definition of $\tr_{1}(E_{\ast})$ can be extended to any bounded
    chain complex
    over a integral scheme (see \cite{GilletSoule:aRRt}). 
  \item There is a sign difference in the definition of the
    inclusion $\gamma $ used in \cite{GilletSoule:aRRt} and the one
    used in  \cite{BismutGilletSoule:MR1086887}. We have followed the
    signs of the 
    first reference.
  \end{enumerate}
\end{remark}

\section{Singular Bott-Chern classes}
\label{sec:singular-bott-chern}

Throughout this section we will use notation \ref{def:19}.
In particular we will write 
\begin{align*}
  \widetilde
  {\mathcal{D}}^{n}_{D}(X,p)&=\left. \mathcal{D}^{n}_{D}(X,p)\right/
  \dd_{\mathcal{D}}\mathcal{D}^{n-1}_{D}(X,p),\\
  \widetilde
  {\mathcal{D}}^{n}_{D}(X,S,p)&=\left. \mathcal{D}^{n}_{D}(X,S,p)\right/
  \dd_{\mathcal{D}}\mathcal{D}^{n-1}_{D}(X,S,p).
\end{align*}

A particularly important current is $W_{1}\in
\mathcal{D}^{1}_{D}(\mathbb{P}^{1},1)$ given by
\begin{equation}
  \label{eq:60}
  W_{1}=[\frac{-1}{2}\log \|t\|^{2}].
\end{equation}
With the above convention, this means that
\begin{equation}
  \label{eq:61}
  W_{1}(\eta)=\frac{1}{2\pi i}\int_{\mathbb{P}^{1}}
  \frac{-1}{2}\log \|t\|^{2}\bullet \eta.
\end{equation}
By the Poincar\'e-Lelong equation
\begin{equation}
  \label{eq:62}
  \dd_{\mathcal{D}}W_{1}=\delta _{\infty}-\delta _{0}.
\end{equation}

Note that the current $W_{1}$ was used in the construction of
Bott-Chern classes (definition \ref{def:1}) and will also have a
role in the definition of singular Bott-Chern classes.

Before defining singular Bott-Chern classes we need to
define the objects that give rise to them.

\begin{definition} \label{def:8}
  Let $i\colon Y\longrightarrow X$ be a closed immersion of complex
  manifolds. Let $N$ be the normal bundle of
  $Y$ and  let $h_{N}$ be a 
  hermitian metric on $N$. We denote $\overline N=(N,h_{N})$. Let
  $r_{N}$ be the rank of $N$, that agrees 
  with the codimension of $Y$ in $X$. Let
  $\overline F=(F,h_{F})$ be a hermitian vector bundle on $Y$ of rank
  $r_{F}$. Let  $\overline{E}_{\ast}\to i_{\ast}F$ be a metric on the
  coherent sheaf $i_{\ast}F$.
  The four-tuple 
  \begin{equation}
    \label{eq:40}
    \overline {\xi}=(i,\overline N,\overline
    F, \overline E_{\ast}). 
  \end{equation}  
  is called a \emph{hermitian embedded vector bundle}. The number
  $r_{F}$ will be called the \emph{rank} of $\overline {\xi}$ and the
  number $r_{N}$ will be called the \emph{codimension} of $\overline {\xi}$.

  By convention, any exact complex of hermitian vector bundles on $X$
  will be considered a hermitian embedded vector bundle of any rank
  and codimension.
\end{definition}

Obviously, to any hermitian embedded vector bundle we can associate the
metrized coherent sheaf $(i_{\ast}F,\overline{E}_{\ast}\to
i_{\ast}F)$.

\begin{definition}
  A \emph{singular Bott-Chern class} for a hermitian embedded vector
  bundle $\overline {\xi}$ is a class
  $\widetilde \eta\in \bigoplus_{p} \widetilde
  {\mathcal{D}}^{2p-1}_{D}(X,p)$ such that 
  \begin{equation}\label{eq:8}
    \dd_{\mathcal{D}} \eta=
    \sum_{i=0}^{n}(-1)^{i}[\ch(\overline E_{i})]-i_{\ast}
    ([\Td^{-1}(\overline N)\ch(\overline F)]) 
  \end{equation}
  for any current $\eta\in\tilde{\eta}$.  
\end{definition}

The existence of this class is guaranteed by the
Grothendieck-Riemann-Roch theorem, which implies that the two currents
in the right hand side of equation \eqref{eq:8} are cohomologous.

Even if we have defined singular Bott-Chern classes as classes of
currents with arbitrary singularities, it is an important observation
that in each singular Bott-Chern class we can find representatives
with controlled singularities. 
Let $N^{\ast}_{Y,0}$ be the conormal bundle of $Y$ with the zero
section deleted. It is a closed conical subset of $T^{\ast}_{0}(X)$. 
Since the current
\begin{multline*}
  \sum_{i=0}^{n}(-1)^{i}[\ch(\overline E_{i})] -
i_{\ast} ([\Td^{-1}(\overline N)\ch(\overline F)])\\
= \sum_{i=0}^{n}(-1)^{i}[\ch(\overline E_{i})] -
\Td^{-1}(\overline N)\ch(\overline F)\delta _{Y}
\end{multline*}
belongs to
$\mathcal{D}^{\ast}_{D}(X,N^{\ast}_{Y,0},p)$, by corollary
\ref{cor:1}, we obtain 

\begin{proposition} \label{prop:1} Let $\overline {\xi}=(i,\overline
  N,\overline F, \overline E_{\ast})$ be a
  hermitian embedded vector bundle as before. Then
any
singular Bott-Chern class for $\overline {\xi}$ belongs to the subset
\begin{displaymath}
  \bigoplus_{p}  \widetilde
  {\mathcal{D}}^{2p-1}_{D}(X,N^{\ast}_{Y,0},p)\subset
  \bigoplus_{p} \widetilde
{\mathcal{D}}^{2p-1}_{D}(X,p). 
\end{displaymath}
\hfill $\square$ 
\end{proposition}

This result will allow us to define inverse images of singular
Bott-Chern classes for certain maps.

Let $f\colon X'\longrightarrow X$ be a morphism of complex
manifolds that is transverse to $Y$. We form the
Cartesian square
\begin{displaymath}
  \xymatrix{Y'\ar[r]^{i'}\ar[d]^{g}& X' \ar[d]^{f}\\
    Y\ar[r]^{i}& X}.
\end{displaymath}
Observe that, by the transversality hypothesis, the normal bundle to
$Y'$ on $X'$ is the inverse image of the normal bundle to $Y$ on $X$
and $f^{\ast} E_{\ast}$ is a resolution of
$i'_{\ast}g^{\ast}F$.
Thus we write $f^{\ast}\overline \xi=(i',f^{\ast}\overline
N,g^{\ast} \overline F, f^{\ast} \overline E_{\ast} )$, which is a
hermitian embedded vector bundle.

By proposition \ref{prop:1}, given any singular Bott-Chern class
$\widetilde 
\eta$ for 
$\xi$, we can find a representative $\eta \in  \bigoplus_{p}
  \mathcal{D}^{2p-1}_{D}(X,N^{\ast}_{Y,0},p)$. By theorem \ref{thm:1},
  there is a well defined current $f^{\ast}\eta$ and it
is a singular Bott-Chern current for $f^{\ast}\xi$. Therefore we can
define $f^{\ast}(\widetilde \eta)=\widetilde{f^{\ast}(\eta)}$.  Again
by theorem 
\ref{thm:1}, this class does not depend on the choice of the
representative $\eta$. 

Our next objective is to study
the possible definitions of functorial singular Bott-Chern
classes.    

\begin{definition}
  \label{def:7}
  Let $r_{F}$ and $r_{N}$ be two integers.
  A \emph{theory of singular Bott-Chern classes of rank $r_{F}$ and
  codimension $r_{N}$} is an assignment 
  which, to each hermitian embedded vector bundle $\overline {\xi}=
  (i\colon Y\longrightarrow 
  X,\overline N, \overline 
  F, \overline E_{\ast}) $ of rank $r_{F}$ and codimension $r_{N}$,
  assigns a class of currents 
  \begin{displaymath}
    T(\overline \xi)\in \bigoplus_{p} \widetilde
     {\mathcal{D}}^{2p-1}_{D}(X,p) 
  \end{displaymath}
  satisfying the following properties
  \begin{enumerate}
  \item \label{item:15} (Differential equation) The following equality
    holds 
    \begin{equation}\label{eq:42}
      \dd_{\mathcal{D}} T(\overline \xi)=
      \sum_{i}(-1)^{i}[\ch(\overline
      E_{i})]-i_{\ast}([\Td^{-1}(\overline N)\ch(\overline F)]). 
    \end{equation}
  \item \label{item:16}(Functoriality) For every morphism
    $f\colon X'\longrightarrow X$ of 
    complex manifolds that is transverse to $Y$, then 
    \begin{displaymath}
      f^{\ast} T(\overline \xi)=T(f^{\ast} \overline
      \xi). 
    \end{displaymath}
  \item \label{item:17}(Normalization)  Let $\overline A=(
    A_{\ast},g_{\ast})$ be a non-negatively graded orthogonally split
    complex of vector 
    bundles. Write $\overline{\xi}\oplus \overline
    A=(i\colon Y\longrightarrow 
    X,\overline N, \overline 
    F, \overline E_{\ast}\oplus \overline A_{\ast})$.
    Then
    $T(\overline \xi)=T(\overline \xi\oplus \overline A)$. Moreover,
    if $X=\Spec \mathbb{C}$ is one point, $Y=\emptyset$ and $\overline
    E_{\ast}=0$, then $T(\overline \xi)=0$.
  \end{enumerate}
  
  A \emph{theory of singular Bott-Chern classes} is an assignment as
  before, for all positive integers $r_{F}$ and $r_{M}$. When the
  inclusion $i$ and the bundles $F$ and $N$ are clear from the context,
  we will denote $T(\overline \xi)$ by $T(\overline
  E_{\ast})$. Sometimes we will have to restrict ourselves to complex
  algebraic manifolds and algebraic vector bundles. In this case we
  will talk of \emph{theory of singular Bott-Chern classes for
    algebraic vector bundles}.  
\end{definition}

\begin{remark}
  \label{rem:1}
  \begin{enumerate}
  \item Recall that the case when $Y=\emptyset$ and $\overline
    E_{\ast}$ is any bounded exact 
    sequence of hermitian vector bundles is considered a hermitian
    embedded vector bundle of arbitrary rank. In this case, the
    properties above imply that
    \begin{displaymath}
      T(\overline \xi)=[\widetilde{\ch}(\overline E_{\ast})],
    \end{displaymath}
    where $\widetilde {\ch}$ is the Bott-Chern class associated to the
    Chern character.
    That is, for acyclic complexes, any theory of singular Bott-Chern
    classes agrees with the Bott-Chern
    classes associated to the Chern character.
  \item If the map $f$ is transverse to $Y$, then either $f^{-1}(Y)$ is
    empty or it has the same codimension as $Y$. Moreover, it is clear
    that $f^{\ast}F$ has the same rank as $F$. Therefore, the
    properties of singular Bott-Chern classes do not mix rank or
    codimension. This 
    is why we have defined singular Bott-Chern classes for a particular
    rank and codimension.
  \item By contrast with the case of Bott-Chern classes, the
    properties above are not enough to characterize singular
    Bott-Chern classes.
  \end{enumerate}
\end{remark}

For the rest of this section we will assume the existence of a theory
of singular Bott-Chern classes and we will obtain some consequences
of the definition. 

We start with the compatibility of singular Bott-Chern classes with
exact sequences and Bott-Chern classes. 

Let
\begin{equation}\label{eq:46}
  \overline \chi\colon  
  0\longrightarrow \overline F_{n}\longrightarrow \dots
  \longrightarrow \overline F_{1} 
  \longrightarrow \overline F_{0}
  \longrightarrow 0
\end{equation}
be a bounded exact sequence of hermitian vector bundles on $Y$. For
$j=0,\dots ,n$,  let $\overline E_{j,\ast}\longrightarrow
i_{\ast}F_{j}$ be a 
resolution, and assume that they fit in a commutative diagram
\begin{displaymath}
  \xymatrix{0 \ar[r] & \overline E_{n,\ast} \ar[r] \ar[d]&\dots \ar[r] 
    & \overline E_{1,\ast} \ar[r] \ar[d]
    & \overline E_{0,\ast} \ar[r] \ar[d]
    & 0\\
    0 \ar[r] & i_{\ast}F_{n} \ar[r]&\dots \ar[r] 
    & i_{\ast}F_{1} \ar[r]
    & i_{\ast}F_{0} \ar[r]
    & 0},
\end{displaymath}
with exact rows.
We write $\overline {\xi_{j}}= (i\colon Y\longrightarrow
  X,\overline N, \overline 
  F_{j}, \overline E_{j,\ast}) $. 
For each $k$, we denote by $\overline \eta_{k}$ the exact sequence
\begin{displaymath}
  0\longrightarrow \overline E_{n,k}\longrightarrow \dots
  \longrightarrow \overline E_{1,k} 
  \longrightarrow \overline E_{0,k}
  \longrightarrow 0.
\end{displaymath}

\begin{proposition} \label{prop:5}
With the above notations, the following equation holds: 
\begin{displaymath}
  T(\bigoplus_{j \text{ even} }\overline{\xi}_{j}) - T(\bigoplus
  _{j\text{ odd} }\overline {\xi}_j)=
  \sum_{k}(-1)^{k} [\widetilde{\ch}(\overline \eta_{k})]-
i_{\ast}([\Td^{-1}(\overline N)\widetilde{\ch}(\overline{\chi})]).
\end{displaymath}
\end{proposition}
Here the direct sum of hermitian embedded vector bundles, involving
the same embedding and the same hermitian normal bundle, is defined
in the obvious manner.
\begin{proof} We consider the construction of theorem \ref{thm:5} for
  each of the exact sequences $\overline \eta_{k}$ and the exact
  sequence $\overline \chi$. For each $k$, we have $W_{X}:=W(\overline
  \eta_{k})= X\times \mathbb{P}^{1}$ and we denote $W_{Y}:=W(\overline
  \chi)= Y\times \mathbb{P}^{1}$. 
  On $W_{Y}$ we consider the transgression exact
  sequence 
  $\tr_{1}(\overline {\chi})_{\ast}$ and on $W_{X}$ we
  consider the transgression exact sequences $\tr_{1}(\overline
  {\eta_{k}})_{\ast}$. We denote by $j\colon W_{Y}\longrightarrow
  W_{X}$ the induced morphism. Then there is an exact
  sequence (of exact sequences)
  \begin{displaymath}
    \dots \longrightarrow \tr_{1}(\overline \eta_{1})_{\ast }
    \longrightarrow \tr_{1}(\overline \eta_{0})_{\ast }
    \longrightarrow j_{\ast} \tr_{1}(\overline \chi)_{\ast
    }\longrightarrow 0. 
  \end{displaymath}
  We denote
  \begin{alignat*}{2}
    \tr_{1}(\overline \chi)_{+}&=\bigoplus_{j\text{ even} }
    \tr_{1}(\overline \chi)_{j},\quad&
    \tr_{1}(\overline \chi)_{-}&=\bigoplus_{j\text{ odd} }
    \tr_{1}(\overline \chi)_{j},\\
    \tr_{1}(\overline \eta_{k})_{+}&=\bigoplus_{j\text{ even} }
    \tr_{1}(\overline \eta_{k})_{j},\quad &
    \tr_{1}(\overline \eta_{k})_{-}&=\bigoplus_{j\text{ odd} }
    \tr_{1}(\overline \eta_{k})_{j},
  \end{alignat*}
  and
  \begin{align*}
    \tr_{1}(\overline {\xi})_{+}&=(j\colon W_{Y}\longrightarrow 
  W_{X},p_{Y}^{\ast}\overline N,  
  \tr_{1}(\overline \chi)_{+}, \tr_{1}(\overline \eta_{\ast})_{+} ),\\
    \tr_{1}(\overline {\xi})_{-}&=(j\colon W_{Y}\longrightarrow 
  W_{X},p_{Y}^{\ast}\overline N,  
  \tr_{1}(\overline \chi)_{-}, \tr_{1}(\overline \eta_{\ast})_{-} ),
  \end{align*}
where here $p_{Y}\colon W_{Y}\longrightarrow Y$ denotes the projection.

  We consider the current on $X\times \mathbb{P}^{1}$ given by
  $W_{1} \bullet\left(T(\tr_{1}(\overline \xi)_{+})-T(\tr_{1}(\overline
  \xi)_{-})\right)$. This current is well 
  defined because the wave front set of $W_{1}$ is the conormal bundle
  of $(X\times \{0\})\cup (X\times \{\infty\})$, whereas the wave front
  set of $T(\tr_{1}(\overline \xi)_{\pm})$ is the conormal bundle of
  $Y\times \mathbb{P}^{1}$. 

  By the functoriality of the transgression exact sequences, we obtain
  that 
  \begin{displaymath}
    \tr_{1}(\overline \xi)_{+}\mid_{X\times \{0\}}=\bigoplus_{j\text{
      even} }\overline \xi_{j},\quad
    \tr_{1}(\overline \xi)_{-}\mid_{X\times \{0\}}=\bigoplus_{j\text{
      odd} }\overline \xi_{j}.
  \end{displaymath}
  Moreover, using the fact that, for any bounded acyclic complex of
  hermitian vector bundles $\overline E_{\ast}$, the exact sequence
  $\tr_{1}(\overline E_{\ast})\mid_{X\times \{\infty\}}$ is orthogonally
  split, we have an isometry 
  \begin{displaymath}
    \tr_{1}(\overline \xi)_{+}\mid_{X\times \{\infty\}}\cong
    \tr_{1}(\overline \xi)_{-}\mid_{X\times \{\infty\}}.
\end{displaymath}

We now denote by $p_{X}\colon W_{X}\longrightarrow X$ the projection. 
Using the properties that define a theory of singular Bott-Chern classes, in
the group $\bigoplus_{p}  \widetilde
{\mathcal{D}}^{2p-1}_{D}(X,N^{\ast}_{Y,0},p)$, the following holds
\begin{align*}
  0&=\dd_{\mathcal{D}} (p_{X})_{\ast} \left(W_{1} \bullet T(\tr_{1}(\overline
    \xi)_{+})-W_{1} \bullet T(\tr_{1}(\overline 
    \xi)_{-})\right)\\
  &=\left(T(\tr_{1}(\overline \xi)_{+})-T(\tr_{1}(\overline 
  \xi)_{-})\right)\mid _{X\times \{\infty\}}-
  \left(T(\tr_{1}(\overline
  \xi)_{+})-T(\tr_{1}(\overline 
  \xi)_{-})\right)\mid _{X\times \{0\}}\\
  &-(p_{X})_{\ast}\sum_{k}(-1)^{k}W_{1}\bullet
  \left(\ch(\tr_{1}(\overline \eta_{k})_{+})- \ch(\tr_{1}(\overline
    \eta_{k})_{-})\right)\\
  &+(p_{X})_{\ast}\left(W_{1}\bullet j_{\ast}\left[\Td^{-1}(p_{Y}^{\ast}\overline
    N)\ch(\tr_{1}(\overline \chi)_{+})-\Td^{-1}(p_{Y}^{\ast} \overline
    N)\ch(\tr_{1}(\overline \chi)_{-})\right]\right)\\
  &=  -T(\bigoplus_{j \text{ even} }\overline{\xi}_{j}) + T(\bigoplus
  _{j\text{ odd} }\overline {\xi}_j)+
\sum(-1)^{k} [\widetilde{\ch}(\overline \eta_{k})]-
  i_{\ast}[\Td^{-1}(\overline N)\bullet \widetilde{\ch}(\overline{\chi})],
\end{align*}
which implies the proposition.
\end{proof}

The following result is a consequence of proposition \ref{prop:5} and
theorem \ref{zhabc}.

\begin{corollary}\label{comparThevb} Let $Y\longrightarrow X$ be a closed immersion of
  complex manifolds.
  Let $\overline {\chi}$ be an exact sequence of hermitian vector
  bundles on $Y$ as \eqref{eq:46}. For each $j$, let
  $\xi_{j}=(i\colon Y\longrightarrow X,\overline N,\overline F_{j},\overline
  E_{j,\ast})$ be a hermitian embedded vector bundle. We denote by
  $\overline {\varepsilon }$ the induced exact sequence of metrized
  coherent sheaves. Then
  \begin{displaymath}
    T(\bigoplus_{j \text{ even} }\overline{\xi}_{j}) - T(\bigoplus
    _{j\text{ odd} }\overline {\xi}_j)=
    [\widetilde{\ch}(\overline \varepsilon )]-
    i_{\ast}([\Td^{-1}(\overline N)\widetilde{\ch}(\overline{\chi})]).
  \end{displaymath} 
\end{corollary}
\hfill$\square$

We now study the effect of changing the metric of the normal bundle
$N$.

\begin{proposition}\label{prop:15}
  Let $\overline {\xi}_{0}=(i,\overline
  N_{0},\overline F, \overline E_{\ast})$ be a hermitian embedded vector
  bundle, where $\overline N_{0}=(N,h_{0})$. Let $h_{1}$ be another
  metric in the vector bundle $N$ and 
  write $\overline N_{1}=(N,h_{1})$,  $\overline {\xi}_{1}=(i,\overline
  N_{1},\overline F, \overline E_{\ast})$. Then
  \begin{displaymath}
    T(\overline {\xi}_{0})-T(\overline
    {\xi}_{1})=-i_{\ast}[\widetilde{\Td^{-1}}(N,h_{0},h_{1})\ch(\overline
    F)].
  \end{displaymath}
\end{proposition}
\begin{proof}
  The proof is completely analogous to the proof of proposition \ref{prop:5}.  
\end{proof}

We now study the case when $Y$ is the zero section of a completed
vector bundle. Let $\overline F$ and $\overline N$ be hermitian
vector bundles over $Y$. We denote $P=\mathbb{P}(N\oplus
\mathbb{C})$, the projective bundle of lines in $N\oplus
\mathcal{O}_{Y}$. Let $s\colon Y\longrightarrow P$ denote the zero section
and let $\pi _{P}\colon P\longrightarrow Y$ denote the projection. Let
$K(\overline F,\overline N)$ be the Koszul resolution of definition
\ref{def:2}. We will use the notations before this definition.

The following result is due to Bismut, Gillet and Soul\'e for the
particular choice of singular Bott-Chern classes defined in
\cite{BismutGilletSoule:MR1086887}. 

\begin{theorem} \label{thm:8} Let $T$ be a theory of singular
  Bott-Chern classes of
  rank $r_{F}$ and codimension $r_{N}$. Let $Y$ be a complex
  manifold 
  and let $\overline F$ and $\overline N$ be hermitian vector bundles
  of rank $r_{F}$ and 
  $r_{N}$ respectively. Then 
  the current $(\pi _{P})_{\ast}(T(K(\overline F,\overline N)))$ is
  closed. Moreover the 
  cohomology class that it represents does not depend on the metric of
  $N$ and $F$ and determines a characteristic class for pairs of
  vector bundles of rank $r_{F}$ and $r_{N}$. We denote this class by
  $C_{T}(F,N)$.
\end{theorem}
\begin{proof}
  We have that
  \begin{align*}
    \dd_{\mathcal{D}}(\pi _{P})_{\ast}&(T(K(\overline
    F,\overline N)))\\
    &=(\pi _{P})_{\ast}(\dd_{\mathcal{D}}T(K(\overline
    F,\overline N)))\\
    &=(\pi _{P})_{\ast}\left(\sum_{k=0}^{r}(-1)^{k}[\ch(
    \bigwedge^{k}\overline Q^{\vee})\pi _{P}^{\ast}\ch(\overline
    F)]-s_{\ast}[\Td^{-1}(\overline N)\ch(\overline F)]\right)\\
    &=\left((\pi _{P})_{\ast}[c_{r}(\overline
      Q)\Td^{-1}(\overline Q)]-[\Td^{-1}(\overline N)]
    \right)\ch(\overline F)).
  \end{align*}
  Therefore, the fact that the current $(\pi _{P})_{\ast}(T(K(\overline
  F,\overline N)))$ is closed follows from corollary \ref{cor:4}.
  The fact that this class is functorial on $(Y,\overline N,\overline
  F)$ is clear from the construction Thus, the fact that it does not
  depend on the hermitian metrics of $N$ and $F$ follows from
  proposition \ref{prop:22}.
\end{proof}

\begin{remark}
  By theorem \ref{thm:14} we know that, if we restrict ourselves to
  the algebraic category, $C_{T}(F,N)$ is given by a
  power series on the Chern classes with coefficients in
  $\mathbb{D}$. By degree reasons
  \begin{displaymath}
    C_{T}(F,N)\in
    \bigoplus_{p}H_{\mathcal{D}^{\an}}^{2p-1}(Y,\mathbb{R}(p)). 
  \end{displaymath}
  Let ${\bf 1}_{1}\in H^{1}_{\mathcal{D}}(\ast,\mathbb{R}(1))$ be the
  element determined by the constant function with value 1 in
  $\mathcal{D}^{1}(\ast,1)$. Then $C_{T}(F,N)/{\bf 1}_{1}$ is a power
  series in  the Chern
  classes of $N$ and $F$ with real coefficients.   
\end{remark}

\section{Classification of theories of singular Bott-Chern classes}
\label{sec:class-theor-sing}

The aim of this section is to give a complete classification of the
possible theories of singular Bott-Chern classes. This classification
is given in terms of the characteristic class $C_{T}$ introduced in
the previous section.

\begin{theorem} \label{thm:6}
  Let $r_{F}$ and $r_{N}$ be two positive integers. Let $C$ be a
  characteristic 
  class for pairs of vector bundles of rank $r_{F}$ and $r_{N}$.
  Then there exists a unique theory $T_{C}$ of singular Bott-Chern
  classes of rank $r_{F}$ and codimension $r_{N}$ such that
  $C_{T_{C}}=C$.  
\end{theorem}
\begin{proof}
  We first prove the uniqueness. Assume that $T$ is a theory of
  singular Bott-Chern classes such that $C_{T}=C$. Let $\overline
  {\xi}=(i\colon Y\longrightarrow X,\overline N, \overline F, \overline
  E_{\ast})$ be a hermitian embedded vector bundle as in section
  \ref{sec:singular-bott-chern}. Let $W$ be the deformation to the
  normal cone of $Y$. We will use all the notations of section
  \ref{sec:deform-resol}. In particular, we will denote by $p_{\widetilde
    X}\colon \widetilde X\longrightarrow X$ and $p_{P}\colon P\longrightarrow X$
  the morphisms induced by restricting $p_{W}$. Recall that $p_{P}$
  can be factored as
  \begin{displaymath}
    P\overset {\pi _{P}}{\longrightarrow }Y\overset
    {i}{\longrightarrow }X. 
  \end{displaymath}

  The normal vector bundle to the inclusion $j\colon Y\times
  \mathbb{P}^{1}\longrightarrow W$ is isomorphic to
  $p_{Y}^{\ast}N\otimes q_{Y}^{\ast}\mathcal{O}(-1)$. We provide it
  with the hermitian metric induced by the metric of $N$ and the
  Fubini-Study metric of $\mathcal{O}(-1)$ and we denote it by
  $\overline N'$. 

  By theorem \ref{thm:5} we have a complex of hermitian vector
  bundles, $\tr_{1}(E_{\ast})_{\ast }$ such that the restriction $\tr_{1}(E_{\ast})_{\ast
  }|_{X\times \{0\}}$ is isometric to $E_{\ast }$, the restriction
  $\tr_{1}(E_{\ast})_{\ast 
  }|_{\widetilde X}$ is orthogonally split and there is an exact
  sequence on $P$
  \begin{displaymath}
    0\longrightarrow A_{\ast }\longrightarrow \tr_{1}(E_{\ast})_{\ast
  }|_{P} \longrightarrow K(F,N)\longrightarrow 0, 
  \end{displaymath}
  where $A_{\ast}$ is split acyclic and $K(F,N)$ is the Koszul
  resolution. Recall that we have trivialized
  $N^{-1}_{\infty/\mathbb{P}^{1}}$ 
  by means of the section $y$ of $\mathcal{O}_{\mathbb{P}^{1}}(1)$.
  We choose a hermitian
  metric in every bundle of $A_{\ast}$ such that it becomes orthogonally
  split. For each $k$ we will
  denote by $\overline {\eta}_{k}$ the exact sequence of hermitian
  vector bundles
  \begin{equation}\label{eq:33}
     0\longrightarrow \overline A_{k}\longrightarrow \tr_{1}(\overline
     E_{\ast})_{k}|_{P} 
     \longrightarrow K(\overline F,\overline N)_{k}\longrightarrow 0. 
  \end{equation}
  Observe that the current $W_{1}$ is defined as the current
  associated to a locally integrable differential form. The pull-back
  of this form to $W$ is also  locally integrable. Therefore it
  defines a current on $W$ that we also denote by $W_{1}$. Moreover,
  since the wave front sets of $W_{1}$ and of $T(\tr_{1}(\overline E_{\ast})_{\ast})$
  are disjoint, there is a well defined current $W_{1}\bullet
  T(\tr_{1}(\overline E_{\ast})_{\ast})$. 
  Then, using the properties of singular Bott-Chern classes in
  definition \ref{def:7}, the equality 
  \begin{align*}
    0&=\dd_{\mathcal{D}}
    (p_{W})_{\ast}\left(W_{1} \bullet T(\tr_{1}(\overline E_{\ast})_{\ast})\right)\\
    &=(p_{\widetilde X})_{\ast}(T(\tr_{1}(\overline
    E_{\ast})_{\ast})|_{\widetilde 
      X})+(p_{P})_{\ast}(T(\tr_{1}(\overline E_{\ast})_{\ast})|_{P})-T(\overline \xi)\\
    &-(p_{W})_{\ast}\left(W_{1} \bullet
      \left(\sum_{k}(-1)^{k}\ch(\tr_{1}(\overline E_{\ast})_{\ast})- 
        (j_{\ast}(\ch(p_{Y}^{\ast}\overline F)\Td^{-1}(\overline N'))
      \right)\right)
  \end{align*}
  holds 
  in the group
  $\bigoplus_{k}\widetilde{\mathcal{D}}^{2k-1}(X,k)$.
  By properties   \ref{def:7}\ref{item:16} and
  \ref{def:7}\ref{item:17}, $T(\tr_{1}(\overline 
  E_{\ast})_{\ast})|_{\widetilde 
    X}=T(\tr_{1}(\overline 
  E_{\ast})_{\ast}|_{\widetilde 
    X})=0$. 
    
  By proposition \ref{prop:5} we have
  \begin{displaymath}
    T(\tr_{1}(\overline E_{\ast})_{\ast}|_{P})= T(K(\overline F,\overline
    N))-
    \sum_{k}(-1)^{k}[\widetilde{\ch}(\overline \eta_{k})]. 
  \end{displaymath}
  Moreover, we have
  \begin{displaymath}
    (p_{P})_{\ast}(T(K(\overline F,\overline N)))=i_{\ast}(\pi
    _{P})_{\ast}(T(K(\overline F,\overline N)))=i_{\ast}C_{T}(F,N). 
  \end{displaymath}

  By the definition of $N'$ and the choice of its metric,
  there are two differential forms $a,b$ on $Y$, such that
  \begin{displaymath}
    \ch(p_{Y}^{\ast}\overline F)\Td^{-1}(\overline
    N')=p_{Y}^{\ast}(a)+p_{Y}^{\ast 
    }(b)\land q_{Y}^{\ast}(c_{1}(\mathcal{O}(-1))).
  \end{displaymath}
  We denote $\omega =-c_{1}(\mathcal{O}(-1))$. By the properties of
  the Fubini-Study metric,
  $\omega$ is invariant under the involution of $\mathbb{P}^{1}$ that
  sends $t$ to $1/t$. Then
  \begin{displaymath}
    (p_{W})_{\ast}\left(W_{1} \bullet 
        (j_{\ast}(\ch(p_{Y}^{\ast}\overline F)\Td^{-1}(\overline N'))
    \right)=i_{\ast}(p_{Y})_{\ast}(W_{1} \bullet
    (p_{Y}^{\ast}a+p_{Y}^{\ast}b\omega
    )) =0
  \end{displaymath}
  because the current $W_{1}$ changes sign under the involution
  $t\longmapsto 1/t$.

  Summing up, we have obtained the equation
  \begin{multline} 
    T(\overline \xi)=
    -(p_{W})_{\ast}
    \left(\sum_{k}(-1)^{k}W_{1} \bullet \ch(\tr_{1}(\overline E_{\ast})_{k})
    \right)\\
    -\sum_{k}(-1)^{k}(p_{P})_{\ast}[\widetilde
    \ch(\overline \eta_{k})]+i_{\ast}C_{T}(F,N).\label{eq:10}
  \end{multline}
  Hence the singular Bott-Chern class is characterized by the
  properties of definition \ref{def:7} and the
  characteristic class $C_{T}$.

  In order to prove the existence of a theory of singular Bott-Chern
  classes, we use equation \eqref{eq:10} to define a class
  $T_{C}(\xi)$ as follows. 
  \begin{definition}\label{def:5}
    Let $C$ be a characteristic class for pairs of vector bundles of
    rank $r_{F}$ and $r_{N}$ as in theorem \ref{thm:6}. Let $\overline
    {\xi}= (i\colon Y\longrightarrow X,\overline N, \overline  
  F, \overline E_{\ast}) $ be as in definition \ref{def:7}. Let
  $\overline A_{\ast}$, $\tr_{1}(\overline
  E_{\ast})_{\ast}$ and  $\overline {\eta}_{\ast}$ be as in
  \eqref{eq:33}. Then we define
  \begin{multline} \label{eq:68}
    T_{C}(\overline \xi)=
    -(p_{W})_{\ast}
    \left(\sum_{k}(-1)^{k}W_{1} \bullet \ch(\tr_{1}(\overline E_{\ast})_{k}) 
    \right)\\
    -\sum_{k}(-1)^{k}(p_{P})_{\ast}[\widetilde
    \ch(\overline \eta_{k})]+i_{\ast}C(F,N).
  \end{multline}
  \end{definition}

  We have to 
  prove that this definition does not depend on the choice of the metric
  of $\tr_{1}(\overline E_{\ast})_{\ast }$ or the metric of $\overline A_{\ast }$,
  that $T_{C}$ satisfies the properties of 
  definition \ref{def:7} and that the characteristic class $C_{T_{C}}$
  agrees with $C$.
  
  First we prove the independence from the metrics. We denote by $h_{k}$
  the hermitian metric on $\tr_{1}(\overline E_{\ast})_{k}$ and by $g_{k}$
  the hermitian metric on $A_{k}$. Let $h'_{k}$ and $g'_{k}$ be another
  choice of metrics satisfying also that $(A_{\ast},g'_{\ast})$
  is orthogonally split, that $(\tr_{1}(E_{\ast})_{k},h'_{k})|_{X\times \{0\}}$ is
  isometric to $\overline E_{k}$ and that $(\tr_{1}(E_{\ast})_{k},h'_{k})|_{\widetilde
    X}$ is orthogonally split. We denote by $\overline \eta'_{k}$ the exact
  sequence $\eta_{k}$ provided with the metrics $g'$ and $h'$. Then,
  in the group 
  $\bigoplus_{p}\widetilde{\mathcal{D}}^{2p-1}(X,p)$, we have
  \begin{multline} \label{eq:11}
    \sum_{k}(-1)^{k}(p_{P})_{\ast}[\widetilde
    \ch(\overline \eta_{k})]-  \sum_{k}(-1)^{k}(p_{P})_{\ast}[\widetilde
    \ch(\overline \eta'_{k})]=\\
    \sum_{k}(-1)^{k}(p_{P})_{\ast}
    \left[\widetilde{\ch}(A_{k},g_{k},g'_{k}) \right]- 
    \sum_{k}(-1)^{k}(p_{P})_{\ast}
    \left[
    \widetilde{\ch}(\tr_{1}(E_{\ast})_{k}|_{P},h_{k},h'_{k})\right].
  \end{multline}
  Observe that the first term of the right hand side vanishes due to the
  hypothesis of $A_{\ast}$ being orthogonally split for both metrics.
  
  Moreover, we also have,
  \begin{multline} \label{eq:63}
    (p_{W})_{\ast}
    \left(\sum_{k}(-1)^{k}W_{1} \bullet \ch(\tr_{1}(E_{\ast})_{k},h_{k}) 
    \right)- \\(p_{W})_{\ast}
    \left(\sum_{k}(-1)^{k}W_{1} \bullet \ch(\tr_{1}(E_{\ast})_{k},h'_{k})
    \right)=\\
    (p_{W})_{\ast}
    \left(\sum_{k}(-1)^{k}W_{1} \bullet \dd_{\mathcal{D}}\widetilde
      {\ch}(\tr_{1}(E_{\ast})_{k},h_{k},h'_{k})  
    \right).
  \end{multline}
  But, in the group $\bigoplus_{p}\widetilde{\mathcal{D}}^{2p-1}(X,p)$,
  \begin{multline}\label{eq:12}
    (p_{W})_{\ast}
    \left(\sum_{k}(-1)^{k}W_{1} \bullet \dd_{\mathcal{D}}\widetilde
      {\ch}(\tr_{1}(E_{\ast})_{k},h_{k},h'_{k})
    \right)=\\\sum_{k}(-1)^{k}(p_{\widetilde X})_{\ast}[\widetilde
    {\ch}(\tr_{1}(E_{\ast})_{k},h_{k},h'_{k})]|_{\widetilde X}\\
    +
    \sum_{k}(-1)^{k}(p_{P})_{\ast}[\widetilde
    {\ch}(\tr_{1}(E_{\ast})_{k},h_{k},h'_{k})]|_{P})\\
    -\sum_{k}(-1)^{k}[\widetilde
    {\ch}(\tr_{1}(E_{\ast})_{k},h_{k},h'_{k})]|_{X\times
      \{0\}}.
  \end{multline}
  The last term of the right hand side vanishes because the metrics $h_{k}$
  and $h'_{k}$ agree when restricted to $X\times \{0\}$ and the first
  term vanishes by the hypothesis that $\tr_{1}(E_{\ast})_{\ast}|_{\widetilde
    X}$ is orthogonally split with both metrics. Combining equations
  \eqref{eq:11}, \eqref{eq:63} and \eqref{eq:12} we obtain that the right hand
  side of equation \eqref{eq:68} does not depend on the 
  choice of metrics.
  
  We next prove the property \ref{item:15} of definition \ref{def:7}.
  We compute
  \begin{multline*}
    \dd_{\mathcal{D}} T_{C}(\overline{\xi})=
    -\sum_{k}(-1)^{k}\left((p_{\widetilde
        X})_{\ast}\ch(\tr_{1}(\overline E_{\ast})_{k}|_{\widetilde X})+
      (p_{P})_{\ast}\ch(\tr_{1}(\overline E_{\ast})_{k}|_{P})\right)\\
    +\sum_{k}(-1)^{k}\ch(\tr_{1}(\overline E_{\ast})_{k}|_{X\times \{0\}})\\
    -\sum_{k}(-1)^{k}(p_{P})_{\ast}\left(\ch(\overline
      A_{k})+\ch(K(\overline F,\overline N)_{k})- 
      \ch(\tr_{1}(\overline E_{\ast})_{k}|_{P}) 
    \right).
  \end{multline*}
  Using that $\overline A_{\ast}$ and that $\tr_{1}(\overline
  E_{\ast})_{\ast}|_{\widetilde X}$ are 
  orthogonally split and corollary \ref{cor:4} we obtain 
  \begin{align*}
    \dd_{\mathcal{D}}
    T_{C}(\overline{\xi})&=\sum_{k}(-1)^{k}\ch(\overline E_{k})- 
    \sum_{k}(-1)^{k}(p_{P})_{\ast}\ch(K(\overline F,\overline N)_{k}) 
    \\
    &=\sum_{k}(-1)^{k}[\ch(\overline E_{k})]-
    (p_{P})_{\ast} [c_{r}(\overline Q)\Td^{-1}(\overline Q)]
    \\
    &= \sum_{k}(-1)^{k}[\ch(\overline E_{k})]-
    i_{\ast}[\ch(\overline F) \Td^{-1}(\overline N)].
  \end{align*}
  
  We now prove the normalization property. We consider first the case
  when $Y=\emptyset$ and $\overline E_{\ast}$ is a non-negatively graded
  orthogonally split complex. We denote by
  \begin{displaymath}
    \overline K_{i}=\Ker(\dd_{i}\colon E_{i}\longrightarrow E_{i-1})
  \end{displaymath}
  with the induced metric. By hypothesis there are isometries
  \begin{displaymath}
    \overline E_{i}=\overline K_{i}\oplus \overline K_{i-1}.
  \end{displaymath}
  Under these isometries, the differential is
  $\dd(s,t)=(t,0)$. Following the explicit construction of
  $\tr_{1}(E_{\ast})$ given in \cite{GilletSoule:aRRt}, recalled in definition
  \ref{def:12}, we see that
  \begin{displaymath}
    \tr_{1}(E_{\ast})_{i}=p^{\ast}K_{i}\otimes q^{\ast}\mathcal{O}(i)
    \oplus p^{\ast}K_{i-1}\otimes q^{\ast}\mathcal{O}(i-1)=
    K_{i}(i)
    \oplus K_{i-1}(i-1).
  \end{displaymath}
  Moreover, we can induce a metric on $ \tr_{1}(E_{\ast})_{\ast}$ satisfying
  the hypothesis of definition \ref{def:13} by means of
  the metric of the bundles $K_{i}$ and the Fubini-Study metric on the
  bundles $\mathcal{O}(i)$. It is clear that the second and third
  terms of the right hand side of equation \eqref{eq:10} are zero. For
  the first term we have
  \begin{align*}
    \sum_{k}(-1)^{k}&(p_{W})_{\ast}W_{1} \bullet 
    \left(\ch(\tr_{1}(\overline E_{\ast})_{k})
    \right)\\
    &=
    (p_{W})_{\ast}
    \left(\sum_{k}(-1)^{k}W_{1} \bullet \ch(\overline
      K_{k}(k)\overset{\perp}{\oplus}\overline K_{k-1}(k-1)) 
    \right)\\
    &=(p_{W})_{\ast}
    \left(W_{1} \bullet (a+b\land \omega ) 
    \right),
  \end{align*}
  where $\omega $ is the Fubini-Study $(1,1)$-form on $\mathbb{P}^{1}$
  and $a,b$ are inverse images of differential forms on $X$. Therefore
  we obtain that 
  $T_{C}(\overline E_{\ast})=0$. 

  Now let $\overline \xi=(i\colon Y\longrightarrow X,\overline N,
  \overline F, \overline 
  E_{\ast})$ and let $\overline B_{\ast}$ be a non-negatively graded orthogonally
  split complex 
  of vector bundles. By \cite{GilletSoule:aRRt} section 1.1, we have
  that $W(E_{\ast}\oplus B_{\ast})=W(E_{\ast})$ and that
  \begin{displaymath}
    \tr_{1}(E_{\ast}\oplus B_{\ast})= \tr_{1}(E_{\ast})\oplus \pi
    ^{\ast} \tr_{1}(B_{\ast}).
  \end{displaymath}

  In order to compute $T_{C}(\overline \xi)$, we have to
  consider the exact sequences of hermitian vector bundles over $P$ 
  \begin{displaymath}
    \overline {\eta}_{k}\colon 0\longrightarrow \overline A_{k}\longrightarrow
    \tr_{1}(\overline E_{\ast})_{k}|_{P}\longrightarrow  K(\overline
    F,\overline N)_{k}\longrightarrow 0,
  \end{displaymath}
  whereas, in order to compute  $T_{C}(\overline \xi\oplus \overline B_{\ast})$,
  we consider 
  the sequences
 \begin{multline*}
   \overline {\eta}'_{k}\colon \\
    0\longrightarrow \overline A_{k}\oplus \pi
    ^{\ast}(\tr_{1}(\overline B)_{k})|_{P}\longrightarrow 
    \tr_{1}(\overline E_{\ast})_{k}\oplus \pi
    ^{\ast}(\tr_{1}(\overline B)_{k})|_{P}\longrightarrow
    K(\overline F,\overline N)_{k}\longrightarrow 0. 
  \end{multline*}
  By the additivity of Bott-Chern classes, we have that $\widetilde
  {\ch}(\overline {\eta}_{k})=\widetilde
  {\ch}(\overline {\eta}'_{k})$. Therefore
  \begin{align*}
    T_{C}(\overline {\xi}\oplus \bar B_{\ast})-T_{C}(\overline {\xi})&=
    -(p_{W})_{\ast}
    \left(\sum_{k}(-1)^{k}W_{1} \bullet \ch(\tr_{1}(\overline E_{\ast}\oplus \overline B_{\ast})_{k})
    \right)\\ &\phantom{AAA}
    +(p_{W})_{\ast}
    \left(\sum_{k}(-1)^{k}W_{1} \bullet \ch(\tr_{1}(\overline E_{\ast})_{k})
    \right)\\
    &=-(p_{W})_{\ast}
    \left(\sum_{k}(-1)^{k}W_{1} \bullet \ch(\tr_{1}(\overline B_{\ast})_{k})
    \right)\\
    &=0.
  \end{align*}

  The proof of the functoriality is  left to the reader.
  
  Finally we prove that $C_{T_{C}}=C$. Let $Y$ be a complex 
  manifold and let $\overline F$ and $\overline N$ be two hermitian
  vector bundles. We write $X=\mathbb{P}(N\oplus \mathbb{C})$. Let
  $i\colon Y\longrightarrow X$ be the inclusion given by the zero
  section and let
  $\pi _{X}\colon X\longrightarrow Y$ be the projection. On $X$ we have the
  tautological exact sequence
  \begin{displaymath}
    0\longrightarrow \mathcal{O}(-1)\longrightarrow 
    \pi _{X}^{\ast}(N\oplus \mathbb{C})\longrightarrow
    Q\longrightarrow 0 
  \end{displaymath}
  and the Koszul resolution, denoted $K(\overline F,\overline N)$.
  We denote $$\overline {\xi}= (i\colon Y\longrightarrow
  X,\overline N, \overline 
  F,K(\overline F,\overline N)).$$
  Using the definition of $T_{C}$, that is, equation \eqref{eq:68},
  and the fact that $T_{C}$ satisfies the properties of definition
  \ref{def:7}, hence equation \eqref{eq:10} is satisfied, we obtain
  that
  \begin{displaymath}
    i_{\ast}C(F,N)=i_{\ast}C_{T_{C}}(F,N)
  \end{displaymath}
  Applying $(\pi _{X})_{\ast}$ we obtain that
  $C(F,N)=C_{T_{C}}(F,N)$ which finishes the proof of theorem
\ref{thm:6}. 
\end{proof}

\section{Transitivity and projection formula}
\label{sec:trans-mult}

We now investigate how different properties of the characteristic
class $C_{T}$ are reflected in the corresponding theory of singular
Bott-Chern classes.

\begin{proposition} \label{prop:6}
  Let $i\colon Y\hookrightarrow X$ be a closed immersion of complex
  manifolds. Let $\overline F$ be a hermitian vector bundle on $Y$ and
  $\overline G$ a hermitian vector bundle
  on $X$. Let $\overline N$ denote the normal bundle to $Y$ provided
  with a hermitian metric. Let $\overline E_{\ast}$ be a finite
  resolution of $i_{\ast}F$ by hermitian vector bundles. We denote
  $\overline {\xi}=(i\colon Y\longrightarrow X,\overline N,
  \overline F, \overline 
  E_{\ast})$ and $\overline {\xi}\otimes \overline
  G=(i\colon Y\longrightarrow X,\overline N, 
  \overline F\otimes i^{\ast}\overline G, \overline 
  E_{\ast}\otimes \overline G)$. Then
  \begin{displaymath}
    T(\overline {\xi}\otimes \overline G)-T(\overline {\xi})\bullet
    \ch(\overline G)=i_{\ast}(C_{T}(F\otimes i^{\ast}
    G,N))-i_{\ast}(C_{T}(F,N))\bullet \ch(\overline G). 
  \end{displaymath}
\end{proposition}
\begin{proof}
  Since the construction of $\tr_{1}(E_{\ast})_{\ast}$ is local on $X$ and $Y$ and
  compatible with finite sums, we have that
  \begin{displaymath}
    W(E_{\ast})=W(E_{\ast}\otimes G), \qquad
    \tr_{1}(\overline E_{\ast}\otimes \overline
    G)_{\ast}=\tr_{1}(\overline E_{\ast})_{\ast}\otimes 
    p_{W}^{\ast}\overline G.
  \end{displaymath}

  We first compute
  \begin{multline}\label{eq:15}
    (p_{W})_{\ast}\left(\sum_{k}(-1)^{k}W_{1} \bullet \ch(\tr_{1}(\overline E_{\ast}\otimes
      \overline G)_{\ast })
      \right)\\
      =(p_{W})_{\ast}\left(\sum_{k}(-1)^{k}W_{1} \bullet \ch(\tr_{1}(\overline
        E_{\ast})_{\ast })p^{\ast}_{W}\ch(
      \overline G)
      \right)\\
      =(p_{W})_{\ast}\left(\sum_{k}(-1)^{k}W_{1} \bullet \ch(\tr_{1}(\overline
        E_{\ast})_{\ast }) 
      \right)\ch(
      \overline G). 
  \end{multline}

  The Koszul resolution of $i_{\ast}(F\otimes i^{\ast} G)$ is given by
  \begin{displaymath}
    K(F\otimes i^{\ast} G,N)= K(F,N)\otimes p_{P}^{\ast} G.
  \end{displaymath}
  For each $k\ge 0$, we will denote by $\overline \eta_{k}\otimes  p_{P}^{\ast}
  \overline G$ the exact sequence
  \begin{displaymath}
    0\longrightarrow \overline A_{k}\otimes  p_{P}^{\ast}
  \overline G \longrightarrow \tr_{1}(\overline
     E_{\ast}\otimes  \overline G)_{k}|_{P} 
     \longrightarrow K(\overline F,\overline N)_{k}\otimes  p_{P}^{\ast}
  \overline G \longrightarrow 0. 
  \end{displaymath}
  Then, we have 
  \begin{equation}
    \label{eq:14}
    (p_{P})_{\ast}[\widetilde{\ch}(\overline \eta_{k}\otimes  p_{P}^{\ast}
  \overline G)]=(p_{P})_{\ast}[\widetilde{\ch}(\overline
  \eta_{k})\bullet  p_{P}^{\ast} 
  \ch(\overline G)]=
  (p_{P})_{\ast}[\widetilde{\ch}(\overline \eta_{k})]\bullet 
  \ch(\overline G)
  \end{equation}
  Thus the proposition follows from equation \eqref{eq:15},  equation
  \eqref{eq:14} and formula \eqref{eq:10}.
\end{proof}

\begin{definition}
  We will say that a theory of singular Bott-Chern classes is
  \emph{compatible with the projection formula} if, whenever we are in
  the situation of proposition \ref{prop:6}, the following equality holds:
  \begin{displaymath}
    T(\overline {\xi}\otimes \overline G)=T(\overline {\xi})\bullet
    \ch(\overline G).
  \end{displaymath}
  We will say that a characteristic class $C$ (of pairs of vector bundles)
  is \emph{compatible with the projection formula} if it satisfies
  \begin{displaymath}
     C(F,N)=C(\mathcal{O}_{Y},N)\bullet \ch(F).
  \end{displaymath}
\end{definition}

\begin{corollary} \label{cor:7}
  A theory of singular Bott-Chern classes $T$ is compatible with the
  projection formula if and only if it is the case for the associated characteristic
  class $C_{T}$.
\end{corollary}
\begin{proof}
  Assume that $C_{T}$ is  compatible with the
  projection formula  and that we are in the situation
  of proposition \ref{prop:6}. Then
  \begin{align*}
    i_{\ast}C_{T}(F\otimes i^{\ast}
    G,N))&=i_{\ast}(C_{T}(\mathcal{O}_{Y},N)\bullet
    \ch(F\otimes i^{\ast} 
    G))\\
    &=i_{\ast}(C_{T}(\mathcal{O}_{Y},N)\bullet \ch(F)  i^{\ast}
    \ch(G))\\
    &=i_{\ast}(C_{T}(\mathcal{O}_{Y},N)\bullet \ch(F))
    \ch(G)\\
    &=i_{\ast}(C_{T}(F,N))\bullet \ch(G).
  \end{align*}
  Thus, by proposition \ref{prop:6}, $T$ is compatible with the
  projection formula.

  Assume that $T$ is compatible with the projection formula. Let
  $\overline F$ and $\overline N$ be hermitian vector bundles over a
  complex manifold 
  $Y$. Let $s\colon Y\hookrightarrow P:=\mathbb{P}(N\oplus \mathbb{C})$ be
  the zero section and let $\pi \colon P\longrightarrow Y$ be the
  projection. Then
  \begin{align*}
    C_{T}(F,N)&=\pi _{\ast}(T(K(\overline F,\overline N)))\\
    &=\pi _{\ast}(T(K(\overline {\mathcal{O}}_{Y},\overline N)\otimes
    \pi ^{\ast} \overline F))\\ 
    &=\pi _{\ast}(T(K(\overline {\mathcal{O}}_{Y},\overline N))\bullet
    \pi ^{\ast} 
    \ch(F))\\
    &=\pi _{\ast}(T(K(\overline {\mathcal{O}}_{Y},\overline
    N)))\bullet \ch(F)\\ 
    &=C_{T}(\mathcal{O}_{Y},N)\bullet \ch(F).
  \end{align*}
\end{proof}

We will next investigate the relationship between singular Bott-Chern
classes and compositions of closed 
immersions. Thus,  let
\begin{displaymath}
\xymatrix{
  Y\ \ar@{^{(}->}[r]^{i_{Y/X}} \ar@/_17pt/@{^{(}->}[rr]_{i_{Y/M}} &
  \,X\  \ar@{^{(}->}[r]^{i_{X/M}} & 
  \,M}
\end{displaymath}
be a composition of closed immersions. Assume that the normal bundles
$N_{Y/X}$, $N_{X/M}$ and $N_{Y/M}$ are provided with hermitian
metrics. We will denote by $\overline {\varepsilon }$ the exact sequence
\begin{equation}\label{eq:84}
  \overline {\varepsilon }\colon 
  0\rightarrow \overline N_{Y/X}
  \rightarrow \overline N_{Y/M}
  \rightarrow i_{Y/X}^{\ast}\overline N_{X/M}
  \rightarrow 0.
\end{equation}
Let $P_{X/M}=\mathbb{P}(N_{X/M}\oplus \mathbb{C})$ be
the projective completion of the normal cone to $X$ in $M$. Then there is an
isomorphism
\begin{equation}
  \label{eq:83}
  N_{Y/P_{X/M}}\cong N_{Y/X}\oplus i^{\ast}_{Y/X}N_{X/M}.
\end{equation}
We denote by $\overline N_{Y/P_{X/M}}$ the vector bundle on the left
hand side with the
hermitian metric induced by the isomorphism \eqref{eq:83}.

 Let $\overline F$ be a
hermitian vector bundle over
$Y$, let $\overline E_{\ast }\longrightarrow 
(i_{Y/X})_{\ast}F$ be a resolution by hermitian
vector bundles. Let $\overline
E'_{\ast,\ast}$ be a complex of  complexes of vector bundles over $M$,
such that, for each $k\ge 0$, $\overline
E'_{k,\ast}\longrightarrow (i_{X/M})_{\ast}E_{k}$ is a resolution,
and there is a commutative diagram of resolutions
\begin{displaymath}
  \xymatrix{\dots \ar[r] & E'_{k+1,\ast} \ar[r]\ar[d]
& E'_{k,\ast} \ar[r]\ar[d] & E'_{k-1,\ast} \ar[r]\ar[d] & \dots\\
\dots \ar[r] & (i_{X/M})_{\ast}E_{k+1} \ar[r] &
(i_{X/M})_{\ast}E_{k} \ar[r]& (i_{X/M})_{\ast}E_{k-1} \ar[r]& \dots
}.
\end{displaymath}
It follows that we have a resolution $\Tot(\overline
E'_{\ast,\ast})\longrightarrow (i_{Y/M})_{\ast} F$ of
$(i_{Y/M})_{\ast} F$ by hermitian vector bundles.

\begin{notation}\label{def:4}
We will denote
\begin{align*}
  \overline \xi_{Y\hookrightarrow X}&=(i_{Y/X},\overline
  N_{Y/X},\overline F,\overline E_{\ast}),\\ 
  \overline \xi_{Y\hookrightarrow M}&=(i_{Y/M},\overline
  N_{Y/M},\overline F,\Tot (\overline E'_{\ast,\ast})),\\ 
  \overline \xi_{X\hookrightarrow M,k}&=(i_{X/M},\overline
  N_{X/M},\overline E_{k},\overline E'_{k,\ast}). 
\end{align*}
We will also denote by
$\overline \xi_{Y\hookrightarrow P_{X/M}}$ the hermitian embedded
vector bundle
\begin{displaymath}
  \left(Y\hookrightarrow P_{X/M},\overline N_{Y/P_{X/M}},\overline
    F,\Tot(\pi _{P_{X/M}}^{\ast}\overline E_{\ast} \otimes 
    K(\mathcal{O}_{X},\overline N_{X/M}))\right).
\end{displaymath}  
\end{notation}

Let $T$ be a theory of singular Bott-Chern classes, and let $C_{T}$ be
its associated characteristic class. 
Our aim now is to relate $T(\overline \xi_{Y\hookrightarrow X})$,
$T(\overline \xi_{Y\hookrightarrow M})$ 
and $T(\overline \xi_{X\hookrightarrow M,k})$.

Let $W_{X}$ be the deformation to the normal cone of $X$ in $M$. As
before we
denote by $j_{X}\colon X\times \mathbb{P}^{1}\longrightarrow W_{X}$
the inclusion.

We denote by $W$ the deformation to the normal cone of $j_{X}(Y\times
\mathbb{P}^{1})$ in $W_{X}$.
\begin{figure}[h]
  \centering
  \ifpdf
  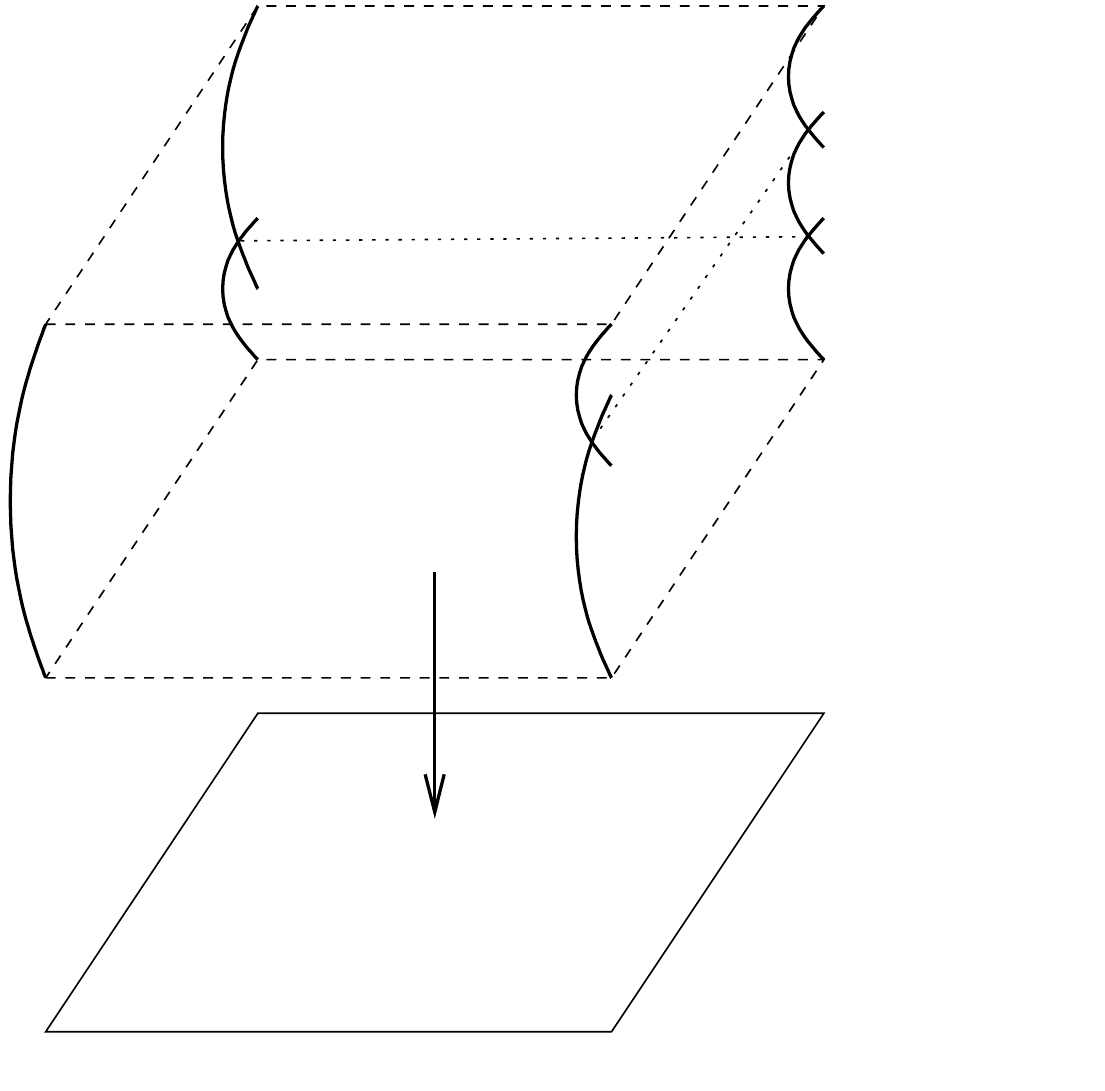
  \else
  \input doubledeformation.pstex_t
  \fi
  \caption{Double deformation}
  \label{fig:df}
\end{figure}

This double deformation is represented in figure \ref{fig:df}. There
is a proper map $q_{W}\colon W\longrightarrow \mathbb{P}^{1}\times
\mathbb{P}^{1}$. The fibers of $q_{W}$ over the corners of $\mathbb{P}^{1}\times
\mathbb{P}^{1}$ are as follows:
\begin{align*}
  q_{W} ^{-1}(0,0)&=M,\\
  q_{W} ^{-1}(\infty,0)&=\widetilde M_{X}\times \{0\} \cup P_{X/M},\\
  q_{W} ^{-1}(0,\infty)&=\widetilde M_{Y} \cup P_{Y/M},\\
  q_{W} ^{-1}(\infty,\infty)&=\widetilde M_{X} \times \{\infty\} \cup
  \widetilde 
  P_{X/M}\cup P_{Y/P_{X/M}},\\
\end{align*}
where $\widetilde M_{X}$ and $\widetilde M_{Y}$ are the blow-up of $M$
along $X$ and $Y$ respectively,
$P_{Y/M}=\mathbb{P}(N_{Y/M}\oplus \mathbb{C})$ is the projective completion of
the 
normal cone to $Y$ in $M$, $P_{Y/P_{X/M}}$ of the normal cone to
 $Y$ in $P_{X/M}$ and $\widetilde 
  P_{X/M}$ is the blow-up of $P_{X/M}$ along $Y$. 
The preimages by $\pi$ of the different faces of
$\mathbb{P}^{1}\times 
\mathbb{P}^{1}$  are as follows:
\begin{align*}
  q_{W} ^{-1}(\mathbb{P}^{1}\times \{0\})&= W_{X},\\
  q_{W} ^{-1}(\{0\}\times \mathbb{P}^{1})&= W_{Y},\\
  q_{W} ^{-1}(\mathbb{P}^{1}\times \{\infty\})&= \widetilde W_{X}\cup
  P_{Y\times \mathbb{P}^{1}},\\
  q_{W} ^{-1}(\{\infty\}\times \mathbb{P}^{1})&= \widetilde M_{X}\times
  \mathbb{P}^{1}\cup W_{Y/P},
\end{align*}
where $W_{Y}$ is the deformation to the normal cone of $Y$ in $M$,
the component $\widetilde W_{X}$ is the blow-up of $W_{X}$ along $j_{X}(Y\times
\mathbb{P}^{1})$, while 
$ P_{Y\times \mathbb{P}^{1}}=\mathbb{P}(N_{Y\times
    \mathbb{P}^{1}/W_{X}}\oplus \mathbb{C})$ is the
  projective completion of the normal cone to $j_{X}(Y\times 
\mathbb{P}^{1})$ in $W_{X}$ and $W_{Y/P} $ is the deformation to the
normal cone of $Y$ inside $P_{X/M}$. All the above subvarieties will
be called boundary components of $W$. 

We will use the following notations for the different maps.
\begin{align*}
  &p_{X}\colon X\times \mathbb{P}^{1}\longrightarrow X &
  &p_{Y}\colon Y\times \mathbb{P}^{1}\longrightarrow Y\\
  &p_{Y\times \mathbb{P}^{1}}\colon Y\times \mathbb{P}^{1}\times
  \mathbb{P}^{1}\longrightarrow Y\times \mathbb{P}^{1} & 
  &p_{\widetilde M_{X}\times
  \mathbb{P}^{1}}\colon \widetilde M_{X}\times
  \mathbb{P}^{1}\longrightarrow M\\
  & p_{W_{Y/P}}\colon W_{Y/P}\longrightarrow M &
  & p_{W_{Y}}\colon W_{Y}\longrightarrow M\\
  & p_{W_{X}}\colon W_{X}\longrightarrow M &
  & p_{P_{Y\times \mathbb{P}^{1}}}\colon  P_{Y\times
    \mathbb{P}^{1}}\longrightarrow M\\
  & p_{\widetilde W_{X}}\colon \widetilde W_{X}\longrightarrow M &
  &p_{P_{Y/P_{X/M}}}\colon P_{Y/P_{X/M}}\longrightarrow M\\
  &p_{P_{X/M}}\colon P_{X/M}\longrightarrow M &
  &p_{\widetilde P_{X/M}}\colon \widetilde P_{X/M}\longrightarrow M \\
  &p_{P_{Y/M}}\colon P_{Y/M}\longrightarrow M &
  &p_{W}\colon W\longrightarrow M\\
  &j_{Y}\colon Y\times \mathbb{P}^{1}\longrightarrow W_{Y} &
  &j'_{Y}\colon Y\times \mathbb{P}^{1}\longrightarrow W_{X} \\
  &j_{Y\times \mathbb{P}^{1}}\colon Y\times \mathbb{P}^{1}\times
  \mathbb{P}^{1}\longrightarrow W&
  &i_{Y/P_{X/M}}\colon Y\longrightarrow P_{X/M}\\
  &\pi _{P_{X/M}}\colon  P_{X/M}\longrightarrow X &
  &\pi _{P_{Y/M}}\colon  P_{Y/M}\longrightarrow Y\\
  &\pi _{P_{Y/P}}\colon  P_{Y/P_{X/M}}\longrightarrow Y &
  &\pi _{P_{Y\times \mathbb{P}^{1}}}\colon  P_{Y\times
    \mathbb{P}^{1}}\longrightarrow Y\times \mathbb{P}^{1}\\
  &\pi _{\widetilde M_{X}}\colon \widetilde M_{X}\longrightarrow M&
  &\pi _{\widetilde M_{Y}}\colon \widetilde M_{Y}\longrightarrow M
\end{align*}
Note that the map $p_{\widetilde M_{X}\times \mathbb{P}^{1}}$
factors through the blow-up $\widetilde M_{X}\longrightarrow M$ and
the map $p_{\widetilde W_{X}}$ factors through the blow-up $\widetilde
M_{Y}\longrightarrow M$, whereas 
the maps $p_{W_{Y/P}}$, $p_{P_{X/M}}$ and $p_{\widetilde P_{X/M}}$
factor through the inclusion 
$X\hookrightarrow 
M$ and the maps $p_{P_{Y\times \mathbb{P}^{1}}}$, $p_{P_{Y/M}}$ and
$p_{P_{Y/P_{X/M}}}$ factor through the 
inclusion $Y\hookrightarrow M$.

The normal bundle to $X\times \mathbb{P}^{1}$ in $W_{X}$ is isomorphic
to $p_{X}^{\ast}N_{X/M}\otimes q_{X}^{\ast}\mathcal{O}(-1)$ and we
consider on it the metric induced by the metric on $\overline N_{X/M}$
and the Fubini-Study metric on $\mathcal{O}(-1)$. We denote it by 
$\overline N_{X\times \mathbb{P}^{1}/W_{X}}$. The normal bundle to
$Y\times \mathbb{P}^{1}$ in $W_{X }$ satisfies
\begin{align*}
  N_{Y\times \mathbb{P}^{1}/W_{X}}|_{Y\times \{0\}}&\cong N_{Y/M}\\
  N_{Y\times \mathbb{P}^{1}/W_{X}}|_{Y\times \{\infty \}} &
  \cong N_{Y/X}\oplus i^{\ast}_{Y/X}N_{X/M}.
\end{align*}
On $N_{Y\times \mathbb{P}^{1}/W_{X}}$ we choose a hermitian metric
such that the above isomorphisms are isometries. Finally, on the normal
bundle to $Y\times \mathbb{P}^{1}\times \mathbb{P}^{1}$ in $W$, we
define a metric using the same procedure as the definition of the
metric of $\overline N_{X\times \mathbb{P}^{1}/W_{X}}$.

On $W_{X}$ we obtain a sequence of resolutions $\tr_{1}(
\overline E')_{n,\ast}\longrightarrow (j_{X})_{\ast}p_{X}^{\ast}E_{n}$. They form
a complex of complexes $\tr_{1}(\overline E')_{\ast,\ast}$ and the
associated total complex $\Tot(\tr_{1}(\overline E')_{\ast,\ast})$ provides us
with a resolution
\begin{equation}\label{eq:38}
  \Tot(\tr_{1}(\overline E')_{\ast,\ast })_{\ast }\longrightarrow
  (j'_{Y})_{\ast}p_{Y}^{\ast}F.
\end{equation}
The restriction of $\Tot(\tr_{1}(\overline E')_{\ast,\ast})$ to $M$ is
$\Tot(\overline E'_{\ast,\ast})$. The restriction of each complex $\tr_{1}(
\overline E')_{n,\ast}$ to $\widetilde M_{X}\times \{0\}$ is orthogonally
split. Therefore the restriction of $\Tot(\tr_{1}(\overline E'))$ to
$\widetilde M_{X}\times \{0\}$ is the total complex of a complex of orthogonally
split complexes. So it is acyclic although not necessarily
orthogonally split. The restriction of each complex $\tr_{1}(
\overline E')_{n,\ast}$ to $P_{X/M}$ fits in an exact sequence
\begin{displaymath}
  0\longrightarrow \overline A_{n,\ast}\longrightarrow
  \tr_{1}(\overline E')_{n,\ast}|_{P_{X/M}}\longrightarrow
   \pi _{P_{X/M}}^{\ast}\overline E_{n} \otimes 
   K(\overline {\mathcal{O}}_{X},\overline N_{X/M})_{\ast}\longrightarrow
  0.
\end{displaymath}
These exact sequences glue together giving a commutative diagram
\begin{displaymath}
  \xymatrix@C-1pt{
  \Tot(\overline A_{\ast,\ast})\ar@{^{(}->}[r] \ar[d]
  &\Tot(\tr_{1}(\overline E')_{\ast,\ast}|_{P_{X/M}})\ar@{->>}[r] \ar[d]
  &\Tot(\pi _{P_{X/M}}^{\ast}\overline E_{\ast}\otimes
  K(\overline {\mathcal{O}}_{X},\overline N_{X/M})_{\ast})\ar[d]
  \\
   0 \ar@{^{(}->}[r] & (i_{Y/P_{X/M}})_{\ast} F \ar@{->>}[r] &
   (i_{Y/P_{X/M}})_{\ast} F 
}
\end{displaymath}
where the rows are short exact sequences.
Even if the complexes $(\overline A_{n})_{\ast}$ are orthogonally split, this is not
necessarily the case for $\Tot(\overline A_{\ast,\ast})$. To ease the notation
we will denote $\overline A_{\ast}=\Tot(\overline A_{\ast,\ast})$.

Applying theorem \ref{thm:5} to the resolution \eqref{eq:38}, we obtain a complex
of hermitian vector bundles $\widetilde
E'_{\ast}=\tr_{1}(\Tot(\tr_{1}(\overline E')_{\ast,\ast}))$ 
which is a resolution of the coherent sheaf $(j_{Y\times
  \mathbb{P}^{1}})_{\ast}p_{Y\times 
  \mathbb{P}^{1}}^{\ast}p_{Y}^{\ast}F$.

We now study the restriction of $\widetilde
E'_{\ast}$
to each of the boundary components of $W$. 

\begin{itemize}
\item The restriction of $\widetilde E'_{\ast}$ to $W_{X}$ is just
  $\Tot(\tr_{1}(\overline E'))$ which has already been described. For
  each $k\ge 0$, we
  will denote by $\eta^{1}_{k}$ the short exact sequence of hermitian
  vector bundles on $P_{X/M}$
  \begin{displaymath}
    \xymatrix{
      \overline A_{k}\ar@{^{(}->}[r]
      &\Tot(\tr_{1}(\overline E')_{\ast,\ast}|_{P_{X/M}})_{k}\ar@{->>}[r]
      &\Tot(\pi _{P_{X/M}}^{\ast}\overline E\otimes
      K(\mathcal{O}_{X},\overline N_{X/M}))_{k}
    },
  \end{displaymath}
  whereas, for each $n,k\ge 0$ we will denote by $\eta^{1}_{n,k}$ the 
  short exact sequence
  \begin{displaymath}
    \xymatrix{
      \overline A_{n,k}\ar@{^{(}->}[r]
      &\tr_{1}(\overline E')_{n,k}|_{P_{X/M}}\ar@{->>}[r]
      &\pi _{P_{X/M}}^{\ast}\overline E_{n}\otimes
      K(\mathcal{O}_{X},\overline N_{X/M})_{k}
    }.
  \end{displaymath}

\item Its restriction to $W_{Y}$ is $\tr_{1}(\Tot(\overline E'))$. It is a
  resolution of 
$(j_{Y})_{\ast}p_{Y}^{\ast} F$. Its restriction to $\widetilde M_{Y}$
is orthogonally split, whereas its restriction to $P_{Y/M}$ fits in an
exact sequence
\begin{displaymath}
  0\longrightarrow \overline B_{\ast}\longrightarrow 
  \tr_{1}(\Tot(\overline E'))_{\ast}|_{P_{Y/M}}\longrightarrow 
  \pi _{P_{Y/M}}^{\ast}\overline F\otimes K(\overline
  {\mathcal{O}}_{Y},\overline N_{Y/M}) 
  \longrightarrow 0.
\end{displaymath}
For each $k\ge 0$ we will denote by $\eta^{2}_{k}$ the degree $k$
piece of the above exact sequence. 
\item  Its restriction to $\widetilde M_{X}\times \mathbb{P}^{1}$ is an acyclic
complex, such that its further restriction to $\widetilde M_{X}\times
\{0\}$ is acyclic and its restriction to $\widetilde M_{X}\times
\{\infty\}$ is orthogonally split.
\item Its restriction to $W_{Y/P}$ fits in
  a short exact sequence
  \begin{displaymath}
 0\rightarrow \tr_{1}(\overline A_{\ast})\rightarrow 
 \widetilde E'_{\ast}|_{W_{Y/P}}\rightarrow 
 \tr_{1}(\Tot (\pi _{P_{X/M}}^{\ast}\overline E\otimes
  K(\overline {\mathcal{O}}_{X},\overline N_{X/M})))\rightarrow 0.
  \end{displaymath}
  For each $k\ge 0$, we will denote by $\mu^{1} _{k}$ the
  exact sequence 
  of hermitian vector bundles over $W_{Y/P}$ given by the piece of
  degree $k$ of this exact sequence.
  The three terms of the above exact sequence become orthogonally split when
  restricted to $\widetilde P_{X/M}$. By contrast, when restricted to
  $P_{Y/P_{X/M}}$ they fit in a commutative diagram
  \begin{displaymath}
    \xymatrix{
      \overline C^{1}_{\ast}\ar@{^{(}->}[r]\ar[d]&
      \overline C^{2}_{\ast}\ar@{->>}[r]\ar[d]&\overline C^{3}\ar[d]\\
      \tr_{1}(\overline A)_{\ast}|_{P_{Y/P_{X/M}}} \ar@{^{(}->}[r]\ar[d]&
      \widetilde E'_{\ast}|_{P_{Y/P_{X/M}}} \ar@{->>}[r]\ar[d]&
      \overline D^{2}_{\ast}\ar[d]\\
      0 \ar@{^{(}->}[r] & \overline D^{1}_{\ast} \ar@{->>}[r]
      & \overline D^{1}_{\ast}
    }
  \end{displaymath}
where the complexes $\overline C^{i}_{\ast}$ are orthogonally split, and 
\begin{align*}
 \overline D^{1}_{\ast}&=\pi ^{\ast}_{P_{Y/P}} \overline F\otimes
  K(\overline {\mathcal{O}}_{Y},\overline N_{Y/P_{X/M}}),\\
\overline D^{2}_{\ast}&= \tr_{1}(\Tot (\pi _{P_{X/M}}^{\ast}\overline
E\otimes 
  K(\overline {\mathcal{O}}_{X},\overline N_{X/M})))|_{P_{Y/P_{X/M}}}.
\end{align*}
For each $k\ge 0$, we will denote by $\eta^{3}_{k}$ the exact sequence 
corresponding to the piece of degree $k$ of the second row of the above
diagram, by $\eta^{4}_{k}$
that of the second column and by $\eta^{5}_{k}$ that of the third
column. Notice that the map in the third row is an isometry. We assume
that the metric on $C^{1}_{\ast}$ is chosen in such a way that the
first column is an isometry. Since the complexes $\overline C^{i}_{\ast}$ are
orthogonally 
split, by lemma \ref{lemm:1} we obtain
\begin{equation}
  \label{eq:25}
  \sum_{k}(-1)^{k}\left(\widetilde{\ch}(\eta^{3}_{k})-
  \widetilde{\ch}(\eta^{4}_{k}) + \widetilde{\ch}(\eta^{5}_{k})\right)=0.
\end{equation}

Note that the restriction of $\mu ^{1}_{k}$ to $P_{X/M}$ agrees with
$\eta^{1}_{k}$, whereas its restriction to $P_{Y/P_{X/M}}$ agrees with
$\eta^{3}_{k}$. 

\item Its restriction to $\widetilde W_{X}$ is orthogonally split.
\item Finally its restriction to $P_{Y\times \mathbb{P}^{1}}$ fits in
  an exact sequence
  \begin{displaymath}
    \xymatrix{
    \overline D_{\ast}\ar@{^{(}->}[r] & \widetilde E'_{\ast}|_{P_{Y\times
        \mathbb{P}^{1}}} \ar@{->>}[r]& \pi^{\ast}  _{P_{Y\times
        \mathbb{P}^{1}}} p^{\ast}_{Y\times \mathbb{P}^{1}} \overline F\otimes
    K(\mathcal{O}_{Y\times \mathbb{P}^{1}},\overline N_{Y\times \mathbb{P}^{1}/W_{X}})
},
\end{displaymath}
where $\overline D_{\ast}$ is orthogonally split.
For each $k\ge 0$ we will denote by $\mu ^{2}_{k}$ the piece of
degree $k$ of this exact sequence. Note that the restriction of
$\mu^{2}_{k}$ to $P_{Y/M}$ agrees with $\eta^{2}_{k}$ and the
restriction of $\mu^{2}_{k}$ to $P_{Y/P_{X/M}}$ agrees with
$\eta^{4}_{k}$. 
\end{itemize}

On $\mathbb{P}^{1}\times \mathbb{P}^{1}$ we denote the two projections
by $p_{1}$ and 
$p_{2}$. Since the currents $p_{1}^{\ast}W_{1}$
and $p_{2}^{\ast}W_{1}$ have disjoint wave front sets we can define the
current 
$W_{2}=p_{1}^{\ast}W_{1}\bullet p_{2}^{\ast}W_{1}\in
\mathcal{D}^{2}_{D}(\mathbb{P}^{1}\times \mathbb{P}^{1},2)$ which
satisfies
\begin{equation} \label{eq:16}
  \dd_{\mathcal{D}}W_{2}=(\delta _{\{\infty\}\times \mathbb{P}^{1}}-
  \delta _{\{0\}\times \mathbb{P}^{1}})\bullet p_{2}^{\ast}W_{1}-
  p_{1}^{\ast}W_{1}\bullet(\delta _{\mathbb{P}^{1}\times \{\infty\}}-
  \delta _{\mathbb{P}^{1}\times \{0\}} ).
\end{equation}

The key point in order to study the compatibility of singular
Bott-Chern classes and composition of closed immersions is that, in
the group $\bigoplus_{p}\widetilde{\mathcal{D}}^{2p-1}(M,p)$, we have
$$\dd_{\mathcal{D}}(p_{W})_{\ast}\left(\sum_{k}(-1)^{k}W_{2} \bullet 
\ch(\widetilde E'_{k})\right)=0.$$ 

We compute this class using the
equation \eqref{eq:16}. It can be decomposed as follows.

\begin{align*}
  \dd_{\mathcal{D}}(p_{W})_{\ast}&\left(\sum_{k}(-1)^{k}W_{2} \bullet \ch(\widetilde
    E'_{k})\right)=\\
  & (p_{\widetilde M_{X}\times \mathbb{P}^{1}})_{\ast}
  \left(\sum_{k}(-1)^{k}W_{1} \bullet \ch(\widetilde E'_{k}|_{\widetilde M_{X}\times
      \mathbb{P}^{1}})\right) \tag{a}\label{eq:18}\\
  & + (p_{W_{Y/P}})_{\ast} \left(\sum_{k}(-1)^{k}W_{1} \bullet \ch(\widetilde
    E'_{k}|_{W_{Y/P}})
    \right) \tag{b}\label{eq:19}\\
  & - (p_{W_{Y}})_{\ast}\left(\sum_{k}(-1)^{k}W_{1} \bullet 
    \ch(\widetilde E'_{k}|_{W_{Y}})\right)
  \tag{c}\label{eq:20}\\ 
  & - (p_{\widetilde W_{X}})_{\ast}\left(\sum_{k}(-1)^{k}W_{1} \bullet 
    \ch(\widetilde E'_{k}|_{\widetilde W_{X}})\right)
  \tag{d}\label{eq:21}\\ 
  & - (p_{P_{Y\times \mathbb{P}^{1}}})_{\ast}\left(\sum_{k}(-1)^{k}W_{1} \bullet 
    \ch(\widetilde E'_{k}|_{P_{Y\times \mathbb{P}^{1}}})
    \right) \tag{e}\label{eq:22}\\ 
  & +(p_{W_{X}})_{\ast}\left(\sum_{k}(-1)^{k}W_{1} \bullet 
    \ch(\widetilde E'_{k}|_{W_{X}})\right)
  \tag{f}\label{eq:23}\\ 
  &\phantom{AAAA}=\colon  I_{a}+I_{b}-I_{c}-I_{d}-I_{e}+I_{f}
\end{align*}

We compute each of the above terms. 

\eqref{eq:18} Since the restriction $\widetilde E'|_{\widetilde
  M_{X}\times \{\infty\}}$ is orthogonally split, we have 
\begin{displaymath}
  I_{a}=-(\pi _{\widetilde M_{X}})_{\ast}\widetilde {\ch}(\widetilde
  E'|_{\widetilde M_{X}\times \{0\}}).  
\end{displaymath}
But, using lemma \ref{lemm:1} and the fact, for each $k$, the
complexes $\tr_{1}(\overline E')_{k,\ast}|_{\widetilde M_{X}}$ are
orthogonally split, we obtain that $I_{a}=0$.

\eqref{eq:19} We compute 
\begin{align*}
  I_{b}=&(p_{W_{Y/P}})_{\ast} \left(\sum_{k}(-1)^{k}W_{1} \bullet \ch(\widetilde
    E'_{k}|_{W_{Y/P}})\right)\\
  = &(p_{W_{Y/P}})_{\ast} \left(W_{1} \bullet \sum_{k}(-1)^{k}(-\dd_{\mathcal{D}}
    \widetilde{\ch} (\mu^{1} _{k})+ \ch(\tr_{1}(\overline A_{\ast})_{k})\right.\\
  &\left.\phantom{(p_{W_{Y/P}})_{\ast} \sum_{k}(-1)^{k}}
    +\ch(\tr_{1}(\Tot (\pi _{P_{X/M}}^{\ast}\overline E \otimes
    K(\mathcal{O}_{X},\overline N_{X/M})))_{k}))\right)\\
  =&\sum_{k}(-1)^{k}
  (-(p_{P_{Y/P_{X/M}}})_{\ast}\widetilde{\ch}(\eta^{3}_{k})-
  (p_{\widetilde P_{X/M}})_{\ast}\widetilde {\ch}(\mu^{1} _{k}|_{\widetilde P_{X/M}})
  +(p_{P_{X/M}})_{\ast}\widetilde{\ch}(\eta^{1}_{k}))\\
  & -\widetilde {\ch}(\overline A)\\
  &-(i_{X/M})_{\ast}(\pi _{P_{X/M}})_{\ast}
  T(\overline \xi_{Y\hookrightarrow P_{X/M}})+(i_{Y/M})_{\ast}C_{T}(F,N_{Y/P_{X/M}})\\
  &-\sum_{k}(-1)^{k} (p_{P_{Y/P_{X/M}}}) \widetilde {\ch}(\eta^{5}_{k}),
\end{align*}
where $\xi_{Y\hookrightarrow P_{X/M}}$ is as in notation \ref{def:4}.

By corollary \ref{cor:9} and the fact that the exact sequences
$\overline A_{k,\ast}$ are orthogonally split, the term $\widetilde
{\ch}(\overline A)$ vanishes. 

Also by corollary \ref{cor:9} we can see that $$ 
\sum_{k}(-1)^{k}
  (p_{\widetilde P_{X/M}})_{\ast}\widetilde {\ch}(\mu^{1} _{k}|_{\widetilde P_{X/M}})
$$
vanishes. 

Therefore we conclude
\begin{align*}
  I_{b}
  =&\sum_{k}(-1)^{k}
  (-(p_{P_{Y/P_{X/M}}})_{\ast}\widetilde{\ch}(\eta^{3}_{k})
  +(p_{P_{X/M}})_{\ast}\widetilde{\ch}(\eta^{1}_{k}))
  -(p_{P_{Y/P_{X/M}}}) \widetilde {\ch}(\eta^{5}_{k})\\
  &-(i_{X/M})_{\ast}(\pi _{P_{X/M}})_{\ast}
  T(\overline \xi_{Y\hookrightarrow P_{X/M}})+(i_{Y/M})_{\ast}C_{T}(F,N_{Y/P_{X/M}}).\\
\end{align*}

\eqref{eq:20} By the definition of singular Bott-Chern forms we have
\begin{displaymath}
  I_{c}=-T(\overline \xi_{Y\hookrightarrow M})+(i_{Y/M})_{\ast}C_{T}(F,N_{Y/M})
  -\sum_{k}(-1)^{k} (p_{P_{Y/M}})_{\ast}\widetilde{\ch}(\eta^{2}_{k}),
\end{displaymath}

\eqref{eq:21} Since the restriction of $\widetilde E'_{\ast}$ to
$\widetilde W_{X}$ is orthogonally split, we have $I_{d}=0$.

\eqref{eq:22} We compute
\begin{align*}
  I_{e}=&  (p_{P_{Y\times \mathbb{P}^{1}}})_{\ast}\left(\sum_{k}(-1)^{k}
    W_{1} \bullet \ch(\widetilde E'_{k}|_{P_{Y\times \mathbb{P}^{1}}})\right)\\
  =&(p_{P_{Y\times \mathbb{P}^{1}}})_{\ast}\left(W_{1} \bullet
    \sum_{k}(-1)^{k}\big( 
    -\dd_{\mathcal{D}} \widetilde {ch}(\mu
    ^{2}_{k})+\ch(\overline D_{k})\right.\\ 
  &\left.\phantom{(p_{P_{Y\times \mathbb{P}^{1}}})_{\ast}\sum_{k}}
  +\ch(\pi^{\ast}  _{P_{Y\times
        \mathbb{P}^{1}}} p^{\ast}_{Y}\overline  F\otimes
    K(\mathcal{O}_{Y\times \mathbb{P}^{1}},\overline N_{Y\times
      \mathbb{P}^{1}/W_{X}})_{k})
    \big)\right).
\end{align*}
The term $\sum(-1)^{k}\ch(\overline D_{k}) $ vanishes because the complex
$D_{\ast}$ is orthogonally split. We have
\begin{multline}
  \sum_{k}(-1)^{k}(p_{P_{Y\times \mathbb{P}^{1}}})_{\ast}(W_{1}
  \bullet \ch(\pi^{\ast}  _{P_{Y\times 
      \mathbb{P}^{1}}} p^{\ast}_{Y}\overline  F\otimes
    K(\overline {\mathcal{O}}_{Y\times \mathbb{P}^{1}},\overline N_{Y\times
      \mathbb{P}^{1}/W_{X}})_{k}))\\
    =(i_{Y/M})_{\ast}\ch(\overline F)\bullet
    (p_{Y})_{\ast}\left(W_{1}\bullet\pi ^{\ast}_{P_{Y\times \mathbb{P}^{1}}}
      \sum_{k}(-1)^{k}\ch(K(\overline {\mathcal{O}}_{Y\times
        \mathbb{P}^{1}},\overline N_{Y\times \mathbb{P}^{1}/W_{X}})_{k})
    \right)\\
    = (i_{Y/M})_{\ast}\ch(\overline F)\bullet
    (p_{Y})_{\ast}\left(W_{1}\bullet \Td^{-1}(\overline N_{Y\times \mathbb{P}^{1}/W_{X}})
    \right)\\
    = (i_{Y/M})_{\ast}\ch(\overline F)\bullet\widetilde
    {\Td^{-1}}(\overline {\varepsilon }_{N}),
\end{multline}
where $\overline{\varepsilon}_{N} $ is the exact sequence \eqref{eq:84}.

Therefore we obtain
\begin{multline*}
  I_{e}=-\sum_{k}(-1)^{k}(p_{P_{Y/P_{X/M}}})_{\ast}\widetilde {\ch}(\eta_{k}^{4})+
  \sum_{k}(-1)^{k}(p_{P_{Y/M}})_{\ast}\widetilde {\ch}(\eta_{k}^{2})\\
  +(i_{Y/M})_{\ast}\ch(\overline F)\bullet\widetilde
    {\Td^{-1}}(\overline {\varepsilon }_{N}).
\end{multline*}

\eqref{eq:23} Finally we have
\begin{align*}
  I_{f}=&-\sum_{k}(-1)^{k} T(\overline \xi_{X\hookrightarrow M,k})
  + \sum_{k}(-1)^{k} (i_{X/M})_{\ast} C_{T}(E_{k},N_{X/M})\\
  &- \sum_{k,l}(-1)^{k+l} (p_{P_{X/M}})_{\ast}\widetilde
  {\ch}(\eta^{1}_{k,l}). 
\end{align*}
By corollary \ref{cor:9} we have that
\begin{displaymath}
  \sum_{m,l}(-1)^{m+l} (p_{P_{X/M}})_{\ast}\widetilde
  {\ch}(\eta^{1}_{m,l})=\sum_{k}(-1)^{k} (p_{P_{X/M}})_{\ast}\widetilde
  {\ch}(\eta^{1}_{k}).
\end{displaymath}
Thus 
\begin{align*}
  I_{f}=&-\sum_{k}(-1)^{k} T(\overline \xi_{X\hookrightarrow M,k})
  + \sum_{k}(-1)^{k} (i_{X/M})_{\ast} C_{T}(E_{k},N_{X/M})\\
  &- \sum_{k}(-1)^{k} (p_{P_{X/M}})_{\ast}\widetilde
  {\ch}(\eta^{1}_{k}). 
\end{align*}

Summing up all the terms we have computed, and taking into account
equation \eqref{eq:25} and the fact that
\begin{displaymath}
  C_{T}(F,N_{Y/M})=C_{T}(F,N_{Y/P_{X/M}})  
\end{displaymath}
we have obtained the following partial result.

\begin{lemma} \label{lemm:2}
  Let $i_{Y/M}=i_{X/M}\circ i_{Y/X}$ be a composition of closed
  immersions of complex manifolds. Let $T$ be a theory of singular
  Bott-Chern classes with $C_{T}$ its associated characteristic
  class. Let $\overline \xi_{Y\hookrightarrow M}$, 
  $\overline \xi _{X\hookrightarrow M,k}$ 
  and $\overline \xi_{Y\hookrightarrow P_{X/M}}$
  be as in notation \ref{def:4}, and let $\overline {\varepsilon }$ be
  as in \eqref{eq:84}.
  Then, in the group $\bigoplus_{p}\widetilde{\mathcal{D}}^{2p-1}(M,p)$,
  the equation
  \begin{multline}\label{eq:50}
    T(\overline \xi_{Y\hookrightarrow M})=
    \sum_{k}(-1)^{k}T(\overline \xi _{X\hookrightarrow M,k})-\sum_{k}(-1)^{k}
    (i_{X/M})_{\ast}C_{T}(E_{k},N_{X/M})\\+
    (i_{X/M})_{\ast}(\pi _{P_{X/M}})_{\ast} 
    T(\overline \xi_{Y\hookrightarrow P_{X/M}})+
    (i_{Y/M})_{\ast}\ch(\overline F)\bullet\widetilde
    {\Td^{-1}}(\overline {\varepsilon }_{N})
  \end{multline}
  holds.
\end{lemma}

In order to compute the third term of the right hand side of equation
\eqref{eq:50} we consider the following situation
\begin{displaymath}
  \xymatrix{
Y\times _{X}P_{X/M} \ar[r]^{j}\ar@/_/[d]_{\pi} &
P_{X/M}\ar@/_/[d]_{\pi }\\ 
Y\ar@/_/[u]_{s}\ar[r]^{i} & X\ar@/_/[u]_{s}}.
\end{displaymath}
To ease the notation, we denote  $P_{X/M}$ by $P$, $Y\underset{X}{\times} P_{X/M}$ by
$X'$ and we 
denote by $P'$ the projective completion of the normal cone to $X'$ in
$P$ and by $\pi 
_{P'}\colon P'\longrightarrow X'$, $\pi _{X'/Y}\colon X'\longrightarrow Y$ and $\pi
_{P'/Y}\colon P'\longrightarrow Y$ the 
projections. Observe that $X$
and $X'$ intersect transversely along $Y$. Moreover,
$N_{Y/X'}=i^{\ast}_{Y/X}N_{X/M}$, $N_{X'/P}=\pi ^{\ast}_{X'/Y}N_{Y/X}$
and $N_{Y/P}=N_{Y/X}\oplus N_{Y/X'}$. We use these identifications to
define metrics on $N_{Y/X'}$, $N_{X'/P}$ and $N_{Y/P}$. Therefore the
exact sequence
\begin{displaymath}
  0\longrightarrow \overline N_{Y/X'}
  \longrightarrow \overline N_{Y/P}
  \longrightarrow i^{\ast}_{Y/X'}\overline N_{X'/P}
  \longrightarrow 0
\end{displaymath}
is orthogonally split.

We apply the previous lemma to the composition of closed
inclusions  
\begin{displaymath}
  Y\hookrightarrow X'\hookrightarrow P,
\end{displaymath}
the vector bundle $\overline F$ over $Y$ and the resolutions
\begin{gather*}
  \pi^{\ast}\overline F\otimes j^{\ast}K(\overline {\mathcal{O}}_{X},\overline
  N_{X/M})_{\ast}\longrightarrow
  s_{\ast}F\\
  \pi^{\ast}\overline E_{\ast}\otimes K(\overline
  {\mathcal{O}}_{X},\overline N_{X/M})_{k} 
  \longrightarrow j_{\ast}(\pi^{\ast}F\otimes
  j^{\ast}K(\mathcal{O}_{X},N_{X/M})_{k}).
\end{gather*}

We denote by $\overline \xi _{Y\hookrightarrow P}$ and  $\overline \xi
_{X'\hookrightarrow P,k}$ the
hermitian embedded vector bundles corresponding to the above
resolutions. If $i_{Y/P'}\colon Y\hookrightarrow P'$ is the induced
inclusion, we denote by 
$\overline \xi _{Y\hookrightarrow P'}$ the hermitian embedded vector bundle
\begin{displaymath}
  \left(i_{Y/P'},\overline N_{Y/P'}, \overline F,\Tot(\pi
  ^{\ast}_{P'}j^{\ast}K(\overline {\mathcal{O}}_{X},\overline
  N_{X/M})\otimes 
  K(\overline {\mathcal{O}}_{X'},\overline N_{X'/P})\otimes (\pi
  _{P'/Y})^{\ast}\overline F)\right).
\end{displaymath}
Note that the hermitian embedded vector bundle $\overline
\xi_{Y\hookrightarrow P}$ agrees with the 
hermitian embedded vector bundle denoted $\overline
\xi_{Y\hookrightarrow P_{X/M}}$ in lemma 
\ref{lemm:2}. Moreover, we have that
\begin{displaymath}
  \overline \xi_{X'\hookrightarrow P,k}=\pi ^{\ast}\overline \xi_{Y\hookrightarrow
    X}\otimes K(\overline {\mathcal{O}}_{X},\overline N_{X/M})_{k}.
\end{displaymath}

Applying lemma \ref{lemm:2}, we obtain
\begin{multline}
  T(\overline \xi_{Y\hookrightarrow P_{X/M}})=  \sum_{k}(-1)^{k}T(\overline \xi
_{X'\hookrightarrow P_{X/M},k})\\
- \sum_{k}(-1)^{k}j_{\ast}C_{T}(\pi ^{\ast} F\otimes
j^{\ast}K(\mathcal{O}_{X},N_{X/M})_{k},N_{X'/P}) \\
+j_{\ast}(\pi _{P'})_{\ast}T(\overline \xi _{Y\hookrightarrow P'})\label{eq:27} 
\end{multline}

By proposition \ref{prop:6}, 
\begin{multline} \label{eq:26}
  \sum_{k}(-1)^{k}T(\overline \xi
  _{X'\hookrightarrow P_{X/M},k})=\sum_{k}(-1)^{k}
  T(\pi ^{\ast}\overline \xi_{Y\hookrightarrow
    X}\otimes K(\overline {\mathcal{O}}_{X},\overline N_{X/M})_{k})\\
  =T(\pi ^{\ast}\overline \xi_{Y\hookrightarrow
    X})\bullet \sum_{k}(-1)^{k}\ch(K(\overline
  {\mathcal{O}}_{X},\overline N_{X/M})_{k})\\ 
  +\sum_{k}(-1)^{k} j_{\ast}C_{T}(\pi ^{\ast} F\otimes
  j^{\ast}K(\mathcal{O}_{X},N_{X/M})_{k}, N_{X'/P})\\
  -\sum_{k}(-1)^{k} j_{\ast}C_{T}(\pi ^{\ast} F, N_{X'/P})\bullet 
  \ch(K(\mathcal{O}_{X},N_{X/M})_{k}) 
\end{multline}

We now want to compute the term $ (i_{X/M})_{\ast}(\pi
_{P_{X/M}})_{\ast}j_{\ast}(\pi 
_{P'})_{\ast}T(\overline \xi _{Y\hookrightarrow P'})$. 

Observe that we can identify
\begin{displaymath}
  P'= \mathbb{P}(i_{Y/X}^{\ast}N_{X/M}\oplus
  \mathbb{C})\underset{Y}{\times} 
  \mathbb{P}(s^{\ast}N_{X'/P}\oplus
  \mathbb{C}),
\end{displaymath}
where $s^{\ast}N_{X'/P}$ is canonically isomorphic to $N_{Y/X}$.
 
Moreover 
\begin{displaymath}
  (i_{X/M})_{\ast}(\pi
_{P_{X/M}})_{\ast}j_{\ast}(\pi 
_{P'})_{\ast}T(\overline \xi _{Y\hookrightarrow P'})=
(i_{Y/M})_{\ast}(\pi _{P'/Y})_{\ast}T(\overline \xi _{Y\hookrightarrow P'}).
\end{displaymath}

\begin{definition}\label{def:10}
  We denote
  \begin{displaymath}
    C_{T}^{\ad}(F,N_{Y/X},i_{Y/X}^{\ast}N_{X/M})=(\pi _{P'/Y})_{\ast}T(\overline \xi
    _{Y\hookrightarrow P'}) 
  \end{displaymath}
  and we define
  \begin{equation}
    \label{eq:47}
    \rho (F,N_{Y/X},i_{Y/X}^{\ast}N_{X/M})=
    C_{T}(F,N_{Y/M})-C_{T}^{\ad}(F,N_{Y/X},i_{Y/X}^{\ast}N_{X/M}).
  \end{equation}
\end{definition}

\begin{lemma}\label{lemm:3}
  The current $ C_{T}^{\ad}(F,N_{Y/X},i_{Y/X}^{\ast}N_{X/M})$ is
  closed and defines a 
  characteristic class of triples of vector bundles. Therefore $\rho$
  is also a characteristic class.  Moreover the class $\rho$ does not
  depend on the theory of singular Bott-Chern classes $T$.
\end{lemma}
\begin{proof}
The fact that $C_{T}^{\ad}(F,N_{Y/X},i_{Y/X}^{\ast}N_{X/M})$ is closed
and determines a 
characteristic class is proved as in \ref{thm:8}. The independence of
$\rho $ 
from to $T$ is seen as follows. We denote by $\overline
K'_{\ast}$ the complex 
\begin{displaymath}
  \Tot(\pi
  ^{\ast}_{P'}j^{\ast}K(\overline {\mathcal{O}}_{X},\overline
  N_{X/M})\otimes 
  K(\overline {\mathcal{O}}_{X'},\overline N_{X'/P}))\otimes (\pi
  _{P'/Y})^{\ast}\overline F. 
\end{displaymath}
This complex is a resolution of $(i_{Y/P'})_{\ast}\overline F$  

Let $W$ be the blow-up of
$P'\times \mathbb{P}^{1}$ along $Y\times \infty$, and let
$\tr_{1}(\overline K')_{\ast}$ be the deformation of complexes on $W$ given by
theorem \ref{thm:5}. Just by
looking at the rank of the different vector bundles we see that the
restriction of $\tr_{1}(\overline K')_{\ast}$ to $P_{Y/P'}$, the exceptional
divisor of this blow-up, is isomorphic (although not necessarily
isometric) to the Koszul complex 
 $K(\overline F,\overline N_{X/M})_{\ast}$. Then, by equation \eqref{eq:10}
\begin{multline*}
  T(\overline \xi
    _{Y\hookrightarrow P'})-(i_{Y/P'})_{\ast}C_{T}(F,N_{Y/M})=\\
    -(p_{W})_{\ast}\left(W_{1} \bullet
      \sum_{k}(-1)^{k}\ch(\tr_{1}(\overline K')_{k}) 
    \right)\\
        -\sum_{k}(-1)^{k}(p_{P})_{\ast}\widetilde
    \ch(\tr_{1}(\overline K')_{k}|_{P_{Y/P'}},K(\overline F,\overline N_{X/M})_{k}). 
\end{multline*}
Since the right hand side of this equation
does not depend on the theory $T$, the result is proved. 
\end{proof}

Using equations
\eqref{eq:27}, \eqref{eq:26}, lemma \ref{lemm:3} and the projection
formula, we obtain 
\begin{align}
  (\pi _{P_{X/M}})_{\ast} T(\overline \xi_{Y\hookrightarrow P_{X/M}})=&
   \left(T(\overline \xi_{Y\hookrightarrow X})-
    (i_{Y/X})_{\ast}C_{T}(F,N_{Y/X})\right)\notag \\
  &\phantom{AA}\bullet (\pi _{P_{X/M}})_{\ast}
  \sum_{k}(-1)^{k} \ch(K(\mathcal{O}_{X},\overline N_{X/M})_{k})\notag\\
  &+(\pi _{P_{X/M}})_{\ast}j_{\ast}(\pi
  _{P'})_{\ast}T(\overline \xi _{Y\hookrightarrow P'})\notag\\
  =&\left(T(\overline \xi_{Y\hookrightarrow X})-
    (i_{Y/X})_{\ast}C_{T}(F,N_{Y/X})\right)\bullet \Td^{-1}(\overline N_{X/M})\notag\\
  &+(i_{Y/X})_{\ast}C^{\ad}_{T}(F,N_{Y/X},i_{Y/X}^{\ast}N_{X/M})\notag\\
  =&\left(T(\overline \xi_{Y\hookrightarrow X})-
    (i_{Y/X})_{\ast}C_{T}(F,N_{Y/X})\right)\bullet \Td^{-1}(\overline N_{X/M})\notag\\
  &+(i_{Y/X})_{\ast}C_{T}(F,N_{Y/M})-\rho (F,N_{Y/X},i_{Y/X}^{\ast}N_{X/M}).
\label{eq:28}  
\end{align}

Joining this equation and lemma \ref{lemm:2} we obtain the main
relationship between singular Bott-Chern classes and composition of
closed immersions.

\begin{proposition} \label{prop:7}
  Let $i_{Y/M}=i_{X/M}\circ i_{Y/X}$ be a composition of closed
  immersions of complex manifolds. Let $T$ be a theory of singular
  Bott-Chern classes with $C_{T}$ its associated characteristic
  class. Let $\overline \xi_{Y\hookrightarrow M}$, 
  $\overline \xi _{X\hookrightarrow M,k}$ and $\overline
  \xi_{Y\hookrightarrow P_{X/M}}$
  be as in notation \ref{def:4} and let $\overline {\varepsilon }$ be
  as in \eqref{eq:84}.
  Then, in the group $\bigoplus_{p}\widetilde{\mathcal{D}}^{2p-1}(M,p)$,
  we have the equation
  \begin{multline*}
    T(\overline \xi_{Y\hookrightarrow M})=
    \sum_{k}(-1)^{k}T(\overline \xi _{X\hookrightarrow
      M,k})+(i_{X/M})_{\ast}(T(\overline \xi_{Y\hookrightarrow X})\bullet
    \Td^{-1}(\overline N_{X/M}))\\
    +
    (i_{Y/M})_{\ast}\ch(\overline F)\bullet\widetilde
    {\Td^{-1}}(\overline {\varepsilon }_{N})\\
    +(i_{Y/M})_{\ast}C^{\ad}_{T}(F,N_{Y/X},i_{Y/X}^{\ast}N_{X/M})\\
    -(i_{X/M})_{\ast}((i_{Y/X})_{\ast}C_{T}(F,N_{Y/X})\bullet\Td^{-1}(N_{X/M}))\\ 
    -(i_{X/M})_{\ast}\sum_{k}(-1)^{k}
    C_{T}(E_{k},N_{X/M})
  \end{multline*}  
\end{proposition}

We can simplify the formula of proposition \ref{prop:7} if we assume
that our theory of singular Bott-Chern classes is compatible with the
projection formula.  

\begin{corollary} \label{cor:5}
  With the hypothesis of proposition \ref{prop:7}, assume furthermore
  that $T$ is compatible with the projection formula. Then
  \begin{multline*}
    T(\overline \xi_{Y\hookrightarrow M})=
    \sum_{k}(-1)^{k}T(\overline \xi _{X\hookrightarrow
      M,k})+(i_{X/M})_{\ast}(T(\overline \xi_{Y\hookrightarrow X})\bullet
    \Td^{-1}(\overline N_{X/M}))\\ 
    +
    (i_{Y/M})_{\ast}\ch(\overline F)\bullet\widetilde
    {\Td^{-1}}(\overline {\varepsilon }_{N})\\
    +(i_{Y/M})_{\ast}\left[C^{\ad}_{T}(F,N_{Y/X},i_{Y/X}^{\ast}N_{X/M})-
    C_{T}(F,N_{Y/X})\bullet\Td^{-1}(i_{Y/X}^{\ast}N_{X/M}))\right.\\\left. 
    -C_{T}(F,i_{Y/X}^{\ast}N_{X/M})\bullet\Td^{-1}(N_{Y/X})\right]
  \end{multline*}    
\end{corollary}
\begin{proof}
  Since $T$ is compatible with the projection formula, then $C_{T}$ is also.
 Therefore, using the Grothendieck-Riemann-Roch
  theorem for
  closed immersions 
  at the level of analytic Deligne cohomology classes, we have 
  \begin{align*}
    \sum_{k}(-1)^{k}
    C_{T}(E_{k},&N_{X/M})=C_{T}(\mathcal{O}_{X},N_{X/M})\bullet
    \sum_{k}(-1)^{k} \ch(E_{k})\\
    &=C_{T}(\mathcal{O}_{X},N_{X/M})\bullet
    (i_{Y/X})_{\ast}(\ch(F)\bullet \Td^{-1}(N_{Y/X}))\\
    &=(i_{Y/X})_{\ast}(i_{Y/X}^{\ast}C_{T}(\mathcal{O}_{X},N_{X/M})\bullet
    \ch(F)\bullet \Td^{-1}(N_{Y/X}))\\
    &=(i_{Y/X})_{\ast}(C_{T}(F,i_{Y/X}^{\ast}N_{X/M})\bullet
    \Td^{-1}(N_{Y/X})),
  \end{align*}
  which implies the result.
\end{proof}

\begin{definition} 
  Let $T$ be a theory of singular Bott-Chern classes. We will say that
  $T$ is \emph{transitive} if the equation
  \begin{multline}\label{eq:49}
    T(\overline \xi_{Y\hookrightarrow M})=
    \sum_{k}(-1)^{k}T(\overline \xi _{X\hookrightarrow
      M,k})+(i_{X/M})_{\ast}(T(\overline \xi_{Y\hookrightarrow X})\bullet
    \Td^{-1}(\overline N_{X/M}))\\
    +
    (i_{Y/M})_{\ast}\ch(\overline F)\bullet\widetilde
    {\Td^{-1}}(\overline {\varepsilon }_{N})
  \end{multline}
  holds.
  When equation \eqref{eq:49} is satisfied for a particular choice of
  complex immersions and resolutions, we say that the theory $T$ is
  \emph{transitive with respect to this particular choice}.
\end{definition}

We now introduce an abstract version of definition \ref{def:10}.

\begin{definition}\label{def:15}
  Given any characteristic class $C$ of pairs of vector bundles, we will
  denote
  \begin{displaymath}
    C^{\rho}(F,N_{1},N_{2}):= C(F,N_{1}\oplus N_{2})-\rho (F,N_{1},N_{2}),
  \end{displaymath}
  where $\rho $ is the characteristic class of definition \ref{def:10}.  
\end{definition}
Note that, when $T$ is a theory of singular Bott-Chern classes we have
\begin{displaymath}
  C^{\rho}_{T}(F,N_{1},N_{2})=C^{\ad}_{T}(F,N_{1},N_{2}).
\end{displaymath}

\begin{definition} 
  We will say that a characteristic class $C$ (of pairs of vector bundles)
  is \emph{$\rho$-Todd additive} (in the second variable) if it satisfies
  \begin{equation*}
    C(F,N_{1}\oplus N_{2})=
     C(F,N_{1})\bullet \Td^{-1}(N_{2})+
     C(F,N_{2})\bullet \Td^{-1}(N_{1})+ \rho (F,N_{1},N_{2})
  \end{equation*}
or, equivalently, 
  \begin{equation*}
    C^{\rho}(F,N_{1}, N_{2})=
     C(F,N_{1})\bullet \Td^{-1}(N_{2})+
     C(F,N_{2})\bullet \Td^{-1}(N_{1}).
  \end{equation*}
\end{definition}

A direct consequence of corollary \ref{cor:5} is
\begin{corollary} \label{cor:6}
  Let $T$ be a theory of singular Bott-Chern classes that is
  compatible with the projection formula. Then it is transitive if and
  only if the associated characteristic class $C_{T}$ is $\rho $-Todd
  additive. 
\end{corollary}

Since we are mainly interested in singular Bott-Chern classes that are
transitive and compatible with the projection formula, we will
study characteristic classes that are compatible with the
  projection formula and $\rho $-Todd-additive in the second
  variable. Since we want to express any characteristic class in terms
  of a power series we will restrict ourselves to the algebraic category. 

\begin{proposition} \label{prop:11}
  Let $C$ be a class that is compatible with the
  projection formula and
  $\rho $-Todd additive in the second variable. Then $C$
  determines a power series $\phi _{C}(x)$ given by
  \begin{equation}\label{eq:48}
    C(\mathcal{O}_{Y},L)=\phi_{C} (c_{1}(L)),
  \end{equation}
  for every complex algebraic manifold $Y$ and algebraic line bundle
  $Y$. 
  Conversely, given any power series in one variable $\phi (x)$, there
  exists a unique characteristic class for algebraic vector bundles
  that is compatible with the 
  projection formula and $\rho
  $-Todd additive in the second variable such that equation
  \eqref{eq:48} holds.
\end{proposition}
\begin{proof}
  This result follows directly from the splitting principle and
  theorem \ref{thm:14}.
\end{proof}

\begin{remark} \label{rem:2}
The utility of corollary \ref{cor:6} and proposition \ref{prop:11} is
limited by the fact that 
we do not know an explicit formula for the class $\rho
(\mathcal{O}_{Y},N_{1},N_{2})$. This class is related with the
arithmetic difference between $\mathbb{P}_{Y}(N_{1}\oplus N_{2}\oplus
\mathbb{C})$ and $\mathbb{P}_{Y}(N_{1}\oplus
\mathbb{C})\underset{Y}{\times }\mathbb{P}_{Y}(N_{2}\oplus
\mathbb{C})$, the second space being simpler than the first. The main
ingredients needed to compute this class are the Bott-Chern classes of
the tautological exact sequence. Therefore the work of Mourougane
\cite{Mourougane04:cbcc} might be useful for computing this class.  
\end{remark}

Recall that an additive genus is a characteristic class for algebraic
vector bundles $S$ such that
\begin{displaymath}
  S(N_{1}\oplus N_{2})=S(N_{1})+S( N_{2}).
\end{displaymath}

Let $\phi (x)=\sum_{i=0}^{\infty}a_{i}x^{i}$ be a power series in one
variable. There is a one to one correspondence between additive genus
and power series characterized by the condition that
$S(L)=\phi (c_{1}(L))$, for each line bundle $L$.

Since the class $\rho $ does not depend on the theory $T$ it
cancels out when considering the difference between two different
theories of singular Bott-Chern classes.

\begin{proposition}\label{prop:9}
  Let $C_{1}$ and $C_{2}$ be two characteristic classes for pairs of
  algebraic vector bundles that 
  are compatible with the
  projection formula and $\rho $-Todd-additive in the
  second  variable. Then there is a unique additive genus $S_{12}$ such that
  \begin{equation}\label{eq:39}
    C_{1}(F,N)-C_{2}(F,N)=\ch(F)\bullet \Td(N)^{-1}\bullet S_{12}(N).
  \end{equation}
\end{proposition}

We can summarize the results of this section in the following theorem.

\begin{theorem}\label{thm:7}
  There is a one to one correspondence between theories of singular
  Bott-Chern classes for complex algebraic manifolds that are
  transitive and compatible with the
  projection formula, and formal power series $\phi (x)\in
  \mathbb{R}[[x]]$. To each theory of singular Bott-Chern classes
  corresponds the power series $\phi $ such that
  \begin{equation}\label{eq:64}
    C_{T}(\mathcal{O}_{Y},L)={\bf 1}_{1}\bullet \phi (c_{1}(L)),
  \end{equation}
  for every complex algebraic manifold $Y$ and every algebraic line
  bundle $L$.
  To each power series $\phi $ it corresponds a unique class $C$,
  compatible with the
  projection formula and $\rho $-Todd-additive in
  the second  variable, characterized by equation \eqref{eq:64} and a
  theory of singular Bott-Chern given by definition \ref{def:5}. 
\end{theorem}

Even if we do not know the exact value of the class $\rho $ another
consequence of corollary \ref{cor:6} is that, in order to prove the
transitivity of a theory of singular Bott-Chern classes it is enough
to check it for a particular class of compositions.

\begin{corollary} \label{cor:8}
  Let $T$ be a theory of singular Bott-Chern classes compatible with
  the projection formula. Then $T$ is
  transitive if and only if for any compact complex manifold $Y$ and
  vector bundles $N_{1}$, $N_{2}$, the theory $T$ is transitive with
  respect to the composition of inclusions
  \begin{displaymath}
    Y\hookrightarrow \mathbb{P}_{Y}(N_{1}\oplus \mathbb{C})
    \hookrightarrow \mathbb{P}_{Y}(N_{1}\oplus \mathbb{C})
    \times_{Y}
    \mathbb{P}_{Y}(N_{2}\oplus \mathbb{C}) 
  \end{displaymath}
  and the Koszul resolutions. 
\hfill $\square$
\end{corollary}

We can make the previous corollary a little more explicit. Let $\pi
_{1}$ and $\pi _{2}$ be the projections from
$P:=\mathbb{P}_{Y}(N_{1}\oplus \mathbb{C}) 
\times_{Y}
\mathbb{P}_{Y}(N_{2}\oplus \mathbb{C})$ to
$P_{1}:=\mathbb{P}_{Y}(N_{1}\oplus \mathbb{C})$ and
$P_{2}:=\mathbb{P}_{Y}(N_{2}\oplus \mathbb{C})$ respectively. Let
$\overline K_{1}=K(\overline{\mathcal{O}}_{Y},\overline N_{1})$ and
$\overline K_{2}=K(\overline{\mathcal{O}}_{Y},\overline N_{2})$ be the
Koszul resolutions in 
$P_{1}$ and $P_{2}$ respectively. Then, $$\overline K=\pi
_{1}^{\ast}K_{1}\otimes
\pi
_{2}^{\ast}K_{2} 
$$  
is a resolution of $\mathcal{O}_{Y}$ in $P$. Then the theory $T$ is
transitive in this case if
\begin{displaymath}
  T(\overline K)=\pi _{2}^{\ast}T(\overline K_{2})\bullet
  \pi _{1}^{\ast}(c_{r_{1}}(\overline Q_{1})\bullet \Td^{-1}(\overline
  Q_{1})) +
  (i_{1})_{\ast}(T(\overline K_{1})\bullet p_{1}^{\ast}
  \Td^{-1}(\overline N_{2})),
\end{displaymath}
where $r_{1}$ is the rank of $N_{1}$, $\overline
Q_{1}$
is the tautological quotient bundle in $P_{1}$ with the induced
metric, $i_{1}\colon P_{1}\longrightarrow P$ is the inclusion
and $p_{1}\colon P_{1}\longrightarrow Y$ is the projection.

The singular Bott-Chern classes that we have defined depend on the
choice of a hermitian metric on the normal bundle and behave well with
respect inverse images. Nevertheless, when
one is interested in covariant functorial properties and, in particular, in a
composition of closed immersions, it might be interesting to consider a
variant of singular Bott-Chern classes that depend on the choice of
metrics on the tangent bundles to $Y$ and $X$.

\begin{notation} \label{def:14}
Let $\overline {\xi}=(i\colon Y\longrightarrow X,\overline
N,\overline F,\overline E_{\ast}\to i_{\ast}F)$ be a hermitian embedded vector
bundle. Let $\overline T_{X}$ and $\overline T_{Y}$ be the tangent
bundles to $X$ and $Y$ provided with hermitian metrics. As usual we
write $\Td(Y)=\Td(\overline T_{Y})$ and $\Td(X)=\Td(\overline
T_{X})$. We put
\begin{displaymath}
  \overline {\xi}_{c}=(i\colon Y\longrightarrow X,\overline
  T_{X}, \overline T_{Y},\overline F,\overline E_{\ast}\to i_{\ast}F).
\end{displaymath}
By abuse of notation we will also say that $\overline {\xi}_{c}$
is a hermitian embedded vector bundle. In this situation we denote by
$\overline 
{\xi}_{N_{Y/X}}$ the exact sequence of hermitian vector bundles
\begin{displaymath}
  \overline {\xi}_{N_{Y/X}}\colon 
  0\longrightarrow \overline T_{Y}
  \longrightarrow i^{\ast}\overline T_{X}
  \longrightarrow \overline N_{Y/X}
  \longrightarrow 0.
\end{displaymath}
If there is no danger of confusion we will denote $\overline
N=\overline N_{Y/X}$ and therefore $\overline {\xi}_{N}=\overline
{\xi}_{N_{Y/X}}$.  
\end{notation}

\begin{definition} \label{def:18}
  Let $T$ be a theory of singular Bott-Chern classes.
  Then the \emph{covariant singular Bott-Chern class} associated to
  $T$ is given by
  \begin{equation}
    \label{eq:85}
    T_{c}(\overline {\xi}_{c})=
    T(\overline {\xi})+i_{\ast}(\ch(\overline F)\bullet
    \widetilde{\Td^{-1}}(\overline{\xi}_{N_{Y/X}})\Td(Y)) 
  \end{equation}
\end{definition}

\begin{proposition}\label{prop:17}
  The covariant singular Bott-Chern classes satisfy the following
  properties
  \begin{enumerate}
  \item The class $T_{c}(\overline {\xi}_{c})$ does not depend on the
    choice of the metric on $N_{Y/X}$.
  \item The differential equation
    \begin{equation}
      \label{eq:86}
      \dd_{\mathcal{D}}T_{c}(\overline {\xi}_{c})=
      \sum_{k}(-1)^{k}\ch(\overline E_{k})-
      i_{\ast}(\ch(\overline F)\bullet \Td(Y))\bullet \Td^{-1}(X)
    \end{equation}
    holds.
  \item If the theory $T$ is compatible with the projection formula, then
    \begin{displaymath}
      T_{c}(\overline {\xi}_{c}\otimes \overline G)=
      T_{c}(\overline {\xi}_{c})\bullet \ch(\overline G).
    \end{displaymath}
  \item If, moreover, the theory $T$ is transitive, then, using
    notation \ref{def:4} adapted to the current setting, we have
    \begin{multline}
      \label{eq:87}
       T_{c}(\overline \xi_{Y\hookrightarrow M,c})=
    \sum_{k}(-1)^{k}T_{c}(\overline \xi _{X\hookrightarrow
      M,k,c})\\+(i_{X/M})_{\ast}(T_{c}(\overline \xi_{Y\hookrightarrow X,c})\bullet
    \Td(X))\bullet \Td^{-1}(M).
    \end{multline}
  \item With the hypothesis of corollary \ref{comparThevb}, we have
  \begin{equation}\label{eq:88}
    T_{c}(\bigoplus_{j \text{ even} }\overline{\xi}_{j,c}) - T_{c}(\bigoplus
    _{j\text{ odd} }\overline {\xi}_{j,c})=
    [\widetilde{\ch}(\overline \varepsilon )]-
    i_{\ast}([\widetilde{\ch}(\overline{\chi})\bullet \Td(Y)])\bullet 
    \Td^{-1}(X).
  \end{equation} 
  \end{enumerate}
\end{proposition}
\begin{proof}
  All the statements follow from straightforward computations.
\end{proof}

\section{Homogeneous singular Bott-Chern classes}
\label{sec:bismut-gillet-soule}

In this section we will show that, by adding a natural fourth axiom to
definition \ref{def:7}, we obtain a unique theory of singular
Bott-Chern classes that we call homogeneous singular Bott-Chern
classes, and we will compare it with the classes previously defined
by Bismut, Gillet and Soul\'e and by Zha.

In the paper \cite{BismutGilletSoule:MR1086887}, Bismut, Gillet and
Soul\'e introduced a theory of singular Bott-Chern classes that is
the main ingredient in their construction of direct images for
closed immersions.

Strictly speaking, the construction of
\cite{BismutGilletSoule:MR1086887} only produces a theory of  
singular Bott-Chern classes in the sense of this paper when the metrics
involved satisfy a technical condition, called Condition (A) of
Bismut. Nevertheless, there is a unique way to extend the definition of
\cite{BismutGilletSoule:MR1086887} from metrics satisfying Bismut's 
condition (A) to 
general metrics in such a way that one obtains a theory of singular
Bott-Chern classes in the sense of this paper.

In his thesis \cite{zha99:_rieman_roch}, Zha gave another
definition of singular Bott-Chern classes, and he also used them to
define direct 
images for closed immersions in Arakelov theory. 

We will recall the construction of both theories of singular
Bott-Chern classes and we will show that they agree with the theory
of homogeneous singular Bott-Chern classes.

 We warn
the reader that the normalizations we use differ from the
normalizations in \cite{BismutGilletSoule:MR1086887} and
\cite{zha99:_rieman_roch}. The 
main difference is that we insist on using the
algebro-geometric twist in cohomology, whereas in the other two papers
the 
authors 
use cohomology with real coefficients. 

Let $r_{F}$ and $r_{N}$ be two positive integers. Let $Y$ be a
complex manifold and let $\overline F$ and $\overline N$ be two
hermitian 
vector bundles of rank $r_{F}$ and $r_{N}$ respectively. Let
$P=\mathbb{P}(N\oplus \mathbb{C})$ and let $s$ be the zero section. We
will follow the notations of 
definition \ref{def:2}. Then $T(K(\overline F,\overline N))$
satisfies the differential equation
\begin{displaymath}
  \dd_{\mathcal{D}} T(K(\overline
F,\overline N))=c_{r_{N}}(\overline Q)\Td^{-1}(\overline Q)\ch(\pi
_{P}^{\ast}\overline F)-s_{\ast}(\ch(\overline F)\Td^{-1}(\overline
N)). 
\end{displaymath}
Therefore, the class
\begin{displaymath}
  \widetilde e_{T}(\overline F,\overline N):= T(K(\overline
F,\overline N))\bullet \Td(\overline Q)\bullet\ch^{-1}(\pi
_{P}^{\ast}\overline F)
\end{displaymath}
satisfies the simpler equation
\begin{equation}\label{eq:69}
  \dd_{\mathcal{D}}\widetilde e_{T}(\overline F,\overline N)=
  [c_{r_{N}}(\overline Q)]-\delta _{Y}.
\end{equation}
Observe that the right hand side of this equation belongs to
$\mathcal{D}^{2r_{N}}_{D}(P,r_{N})$. Thus it seems natural to
introduce the following definition.
\begin{definition} \label{def:16}
  Let $T$ be a theory of singular Bott-Chern classes of rank $r_{F}>0$
  and codimension $r_{N}$. Then the class
  \begin{displaymath}
  \widetilde e_{T}(\overline F,\overline N)= T(K(\overline
F,\overline N))\bullet \Td(\overline Q)\bullet\ch^{-1}(\pi
_{P}^{\ast}\overline F)    
  \end{displaymath}
  is called the \emph{Euler-Green class associated to} $T$.
The class $T(K(\overline
F,\overline N))$ is said to be \emph{homogeneous} if
  \begin{displaymath}
    \widetilde e_{T}(\overline F,\overline N) \in \widetilde
{\mathcal{D}}^{2r_{N}-1}_{D}(P,r_{N}). 
  \end{displaymath}
  A theory of singular Bott-Chern classes of rank $0$ is said to be
  \emph{homogeneous} if it agrees with the theory of 
  Bott-Chern classes associated to the Chern character.  Finally, a
  theory of singular Bott-Chern classes 
  is said to be \emph{homogeneous} if its restrictions to all
  ranks and codimensions are homogeneous.  
\end{definition}

The main interest of the above definition is the following result.
\begin{theorem} \label{thm:9}
  Given two positive integers $r_{F}$ and $r_{N}$ there exists a
  unique theory of homogeneous singular Bott-Chern classes of
  rank~$r_{F}$ and codimension~$r_{N}$. 
\end{theorem}

\begin{proof}
  The proof of this result is based on the theory of Euler-Green
  classes. 

  Let $P=\mathbb{P}(N\oplus \mathbb{C})$ be as before, and let $s$
  denote the zero section of $P$. 
  Let $D_{\infty}$ be
  the subvariety of $P$ that parametrizes the lines contained in
  $N$. Then $D_{\infty}=\mathbb{P}(N)$.

  \begin{lemma}\label{lemm:4}
    There exists a unique class $\widetilde e(P,\overline Q,s)\in
    \mathcal{D}^{2r_{N}-1}_{D}(P,r_{N})$ such that
    \begin{enumerate}
    \item \label{item:4} It satisfies
      \begin{equation}
        \label{eq:30}
 \dd_{\mathcal{D}}\widetilde e(P,\overline
      Q,s)=[c_{r_{N}}(\overline Q)]-\delta _{Y}.       
      \end{equation}
    \item \label{item:5} The restriction $\widetilde e(P,\overline
      Q,s)|_{D_{\infty}}=0.$
    \end{enumerate}
  \end{lemma}

  \begin{proof}
    We first show the uniqueness. Assume that $\widetilde e$ and
    $\widetilde e'$ are two classes that satisfy the hypothesis of the
    theorem. Then $\widetilde e'-\widetilde e$ is closed. Hence it
    determines a cohomology class in
    $H^{2r_{N}-1}_{\mathcal{D}^{\an}}(P,r_{N})$. Since, by theorem
    \ref{thm:13}, the restriction
    \begin{equation}
      \label{eq:70}
      H^{2r_{N}-1}_{\mathcal{D}^{\an}}(P,r_{N})\longrightarrow
    H^{2r_{N}-1}_{\mathcal{D}^{\an}}(D_{\infty},r_{N})
    \end{equation}
    is an
    isomorphism, condition \ref{item:5} implies that $\widetilde
    e'=\widetilde e$. Now we prove the existence. Since $Y$ is the
    zero locus of the section $s$, that is transversal to the zero
    section of $Q$, we know that the currents $[c_{r_{N}}]$ and
    $\delta _{Y}$ are cohomologous. Therefore there exists an element
    $\widetilde a\in \widetilde {\mathcal{D}}^{2r_{N}-1}_{D}(P,r_{N})$
    such that $\dd_{\mathcal{D}}\widetilde a
    =[c_{r_{N}}(\overline Q)]-\delta _{Y}$.
    Since $\overline Q$ restricted to
    $D_{\infty}$ splits as an orthogonal direct sum
    \begin{equation}
      \label{eq:35}
      \overline
    Q|_{D_{\infty}}=\overline S\oplus \overline {\mathbb{C}}
    \end{equation}
    where
    the metric on the factor $\mathbb{C}$ is trivial, and
    the section $s$ restricts to the constant section $1$, we obtain
    that $([c_{r_{N}}(\overline Q)]-\delta
    _{Y})|_{D_{\infty}}=0$. Therefore $\widetilde a$ determines a
    class in  $H^{2r_{N}-1}_{\mathcal{D}^{\an}}(P,r_{N})$.  Using again that
  \eqref{eq:70} is an isomorphism, we find an element $\widetilde b\in
  H^{2r_{N}-1}_{\mathcal{D}^{\an}}(P,r_{N})$, such that $\widetilde
  e=\widetilde a-\widetilde b$  satisfies the conditions of the lemma.
  \end{proof}

  We continue with the proof of theorem \ref{thm:9}. We first prove
  the uniqueness. Let $T$ be a
  theory of homogeneous singular Bott-Chern classes. The splitting
  \eqref{eq:35} implies easily that
  the restriction of the Koszul resolution $K(\overline F,\overline N)$ 
  to $D_{\infty}$ is orthogonally split. By the functoriality of
  singular Bott-Chern classes, $T(K(\overline
  F,\overline N))|_{D_{\infty}}=0$. Thus the class
  \begin{displaymath}
    \widetilde e_{T}(\overline F,\overline N):= T(K(\overline
    F,\overline N))\bullet \Td(\overline Q)\bullet\ch^{-1}(\pi
    _{P}^{\ast}\overline F)\in \widetilde
    {\mathcal{D}}^{2r_{N}-1}_{D}(P,r_{N})
  \end{displaymath}
  satisfies the two conditions of lemma \ref{lemm:4}. Therefore 
  $\widetilde e_{T}(\overline F,\overline N)=\widetilde e(P,\overline
  Q,s)$ and
  \begin{equation}\label{eq:73}
    T(K(\overline
    F,\overline N))= \widetilde e(P,\overline
  Q,s)\bullet \Td^{-1}(\overline Q)\bullet\ch(\pi
    _{P}^{\ast}\overline F),
  \end{equation}
  where the right hand side does not depend on the theory $T$. In
  consequence we have that
  \begin{equation}
    \label{eq:71}
    C_{T}(F,N)=(\pi _{P})_{\ast}T(K(\overline
    F,\overline N))
  \end{equation}
  does not depend on the theory $T$. Thus by the uniqueness in theorem
  \ref{thm:6} we obtain the uniqueness here.

  For the existence we observe 
  \begin{lemma}\label{lemm:6} The current 
      \begin{displaymath}
    C(F,N)=(\pi _{P})_{\ast}(\widetilde e(P,\overline
  Q,s)\bullet \Td^{-1}(\overline Q)) \bullet \ch(\overline F)
  \end{displaymath}
  is a characteristic class for pairs of vector bundles of rank
  $r_{F}$ and $r_{N}$.
  \end{lemma}
  \begin{proof}
    We first compute, using equation \eqref{eq:30} and corollary
    \ref{cor:4}, 
    \begin{align*}
      \dd_{\mathcal{D}} C(F,N)&=(\pi
      _{P})_{\ast}\left(\dd_{\mathcal{D}}\widetilde e(P,\overline 
      Q,s)\bullet \Td^{-1}(\overline Q)\right)\bullet \ch(\overline F)\\
      &=  (\pi
      _{P})_{\ast}\left(([c_{r_{N}}(\overline Q)]-\delta _{Y})\bullet
      \Td^{-1}(\overline Q)\right)\bullet 
      \ch(\overline F)\\
      &= (\pi
      _{P})_{\ast}\left(c_{r_{N}}(\overline Q)\bullet \Td^{-1}(\overline
      Q)\right) \bullet
      \ch(\overline F)
      -\Td^{-1}(\overline N)\bullet \ch(\overline F)\\
      &=0.
    \end{align*}
    Thus $C(F,N)$ determines a cohomology class. This class is
    functorial by construction. By proposition \ref{prop:22} this
    class does not depend on the metric and defines a characteristic
    class.  
  \end{proof}

  By the existence in theorem \ref{thm:6} we
  obtain a theory of singular Bott-Chern classes $T_{C}$ that is
  easily seen to be homogeneous. 
\end{proof}

A reformulation of theorem \ref{thm:9} is 
\begin{theorem} \label{thm:12}
  There exists a unique way to associate to each 
  hermitian embedded vector bundle $\overline {\xi}=
  (i\colon Y\longrightarrow 
  X,\overline N, \overline 
  F, \overline E_{\ast}) $ a class of currents 
  \begin{displaymath}
    T^{h}(\overline \xi)\in \bigoplus_{p} \widetilde
     {\mathcal{D}}^{2p-1}_{D}(X,N^{\ast}_{Y,0},p) 
  \end{displaymath}
  that we call homogeneous singular Bott-Chern class,
  satisfying the following properties
  \begin{enumerate}
  \item  (Differential equation) The equality 
    \begin{equation}
      \dd_{\mathcal{D}} T^{h}(\overline \xi)=
      \sum_{i}(-1)^{i}[\ch(\overline
      E_{i})]-i_{\ast}([\Td^{-1}(\overline N)\ch(\overline F)]) 
    \end{equation}
    holds.
  \item (Functoriality) For every morphism
    $f\colon X'\longrightarrow X$ of 
    complex manifolds that is transverse to $Y$, 
    \begin{displaymath}
      f^{\ast} T^{h}(\overline \xi)=T^{h}(f^{\ast} \overline
      \xi). 
    \end{displaymath}
  \item (Normalization)  Let $\overline A=(
    A_{\ast},g_{\ast})$ be a non-negatively graded orthogonally split
    complex of vector 
    bundles. Write $\overline{\xi}\oplus \overline
    A=(i\colon Y\longrightarrow 
    X,\overline N, \overline 
    F, \overline E_{\ast}\oplus \overline A_{\ast})$.
    Then
    $T^{h}(\overline \xi)=T^{h}(\overline \xi\oplus \overline A)$. Moreover,
    if $X=\Spec \mathbb{C}$ is one point, $Y=\emptyset$ and $\overline
    E_{\ast}=0$, then $T^{h}(\overline \xi)=0$.
  \item (Homogeneity) If $r_{F}=\rk(F)>0$ and $r_{N}=\rk(N)>0$, then, with
    the notations of 
    definition \ref{def:16},
    \begin{displaymath}
      T^{h}(K(\overline 
      F,\overline N))\bullet \Td(\overline Q)\bullet\ch^{-1}(\pi
      _{P}^{\ast}\overline F)\in \widetilde
      {\mathcal{D}}^{2r_{N}-1}_{D}(P,r_{N}).
    \end{displaymath}
  \end{enumerate}
\hfill $\square$
\end{theorem}

The class $\widetilde e(P,\overline
  Q,s)$ of lemma \ref{lemm:4} is a particular case of the Euler-Green
  classes introduced by 
  Bismut, Gillet and Soul\'e in \cite{BismutGilletSoule:MR1086887}. 
The basic properties of the
Euler-Green classes are summarized in the following results.

\begin{proposition}\label{prop:14}
  Let $X$ be a complex manifold, let $\overline E$ be a  hermitian
  holomorphic 
  vector bundle of rank $r$ and let $s$ be a holomorphic section of
  $E$ that is 
  transverse to the zero section. Denote by $Y$ the zero locus of
  $s$. There is a unique way to assign to each $(X,\overline E,s)$ as
  before a class of currents
  \begin{displaymath}
    \widetilde e(X,\overline E,s)\in \widetilde
  {\mathcal{D}}^{2r-1}_{D}(X,N^{\ast}_{Y,0},r)
  \end{displaymath}
  satisfying the following properties
  \begin{enumerate}
  \item \label{item:8}(Differential equation)
    \begin{equation}
      \label{eq:51}
      \dd_{\mathcal{D}}\widetilde e(X,\overline E,s)=
      c_{r}(\overline E)-\delta _{Y}.
    \end{equation}
  \item \label{item:9}(Functoriality) If $f\colon X'\longrightarrow X$ is a morphism
    transverse to $Y$ then
    \begin{equation}
      \label{eq:52}
      \widetilde e(X',f^{\ast}\overline E,f^{\ast}s)=
      f^{\ast}\widetilde e(X,\overline E,s).
    \end{equation}
  \item \label{item:10} (Multiplicativity) Let $\overline E_{1}$ and
    $\overline E_{2}$ be
    hermitian holomorphic vector bundles, and let $s_{1}$ and $s_{2}$ be
    holomorphic sections of $\overline E_{1}$ and $\overline E_{2}$
    respectively that are transverse to the zero section and with zero
    locus $Y_{1}$  and $Y_{2}$. We write $\overline E=\overline
    E_{1}\oplus \overline E_{2}$ and $s=s_{1}\oplus s_{2}$. Assume
    that  $s$ is
    transverse to the zero section; hence $Y_{1}$ and $Y_{2}$ meet
    transversely. With this hypothesis we have
    \begin{multline*}
      \label{eq:53}
      \widetilde e(X,\overline E,s)=\widetilde e(X,\overline
      E_{1},s_{1})\land c_{r_{2}}(\overline E_{2})+
      \delta _{Y_{1}}\land \widetilde e(X,\overline
      E_{2},s_{2})\\
      =\widetilde e(X,\overline
      E_{1},s_{1})\land \delta _{Y_{2}}+
       c_{r_{1}}(\overline E_{1})\land \widetilde e(X,\overline
      E_{2},s_{2}).
    \end{multline*}
  \item \label{item:19} (Line bundles) If $\overline L$ is a hermitian
    line bundle and $s$ is a  
    section of $L$, then
    \begin{equation}
      \label{eq:54}
       \widetilde e(X,\overline L,s)=-\log\|s\|.
    \end{equation}
  \end{enumerate}
\end{proposition}
\begin{proof}
  Bismut, Gillet and Soul\'e prove the existence by constructing explicitly
  an Euler-Green current in the total space of $E$ and pulling it back to $X$
  by the section $s$. For the uniqueness, first we see that properties
  \ref{item:8} and \ref{item:9} imply that, if $h_{0}$ and $h_{1}$ are
  two hermitian metrics in $E$, then
  \begin{equation}
    \label{eq:55}
     \widetilde e(X,(E,h_{0}),s)-\widetilde e(X,(E,h_{1}),s)=
     \widetilde c_{r}(E,h_{0},h_{1}).
  \end{equation}
  We now consider $\pi \colon P=\mathbb{P}(E\oplus
  \mathbb{C})\longrightarrow X$,  with the
  tautological exact sequence
  \begin{equation*}
    0\longrightarrow \mathcal{O}(-1)\longrightarrow \pi ^{\ast}E\oplus
    \mathbb{C} 
    \longrightarrow Q \longrightarrow 0
  \end{equation*}
  On $Q$ we consider the metric induced by the metric of $\overline E$
  and the 
  trivial metric on the factor $\mathbb{C}$, and let $s_{Q}$ the
  section of $Q$ induced by the section $1$ of
  $\mathbb{C}$. Let $D_{\infty}$ be as in lemma \ref{lemm:4}.
  Then properties \ref{item:9} to \ref{item:19} imply that 
  $\widetilde e(P,\overline Q,s_{Q})|_{D_{\infty}}=0$. Hence by lemma
  \ref{lemm:4} $\widetilde e$ is uniquely determined. Finally, let 
  $f\colon X\longrightarrow P$ be the map given by $x\longmapsto
  (s(x):-1)$. Then $f^{\ast}Q\cong E$, although they are not
  necessarily isometric,
  and $f^{\ast}s_{Q}=s$. Therefore, the functoriality and  equation
  \eqref{eq:55} determine $ \widetilde e(X,\overline E,s)$.

  To prove the existence, we use lemma \ref{lemm:4}, functoriality and
  equation \eqref{eq:55} to define the Euler-Green classes. It is easy
  to show that they are well defined and satisfy properties
  \ref{item:8} to \ref{item:19}. 
\end{proof}

Equation \eqref{eq:73} relating homogeneous singular Bott-Chern classes and
Euler-Green classes in a particular case can be generalized to
arbitrary vector bundles.

\begin{proposition} \label{prop:10}
  Let $X$ be a complex manifold, $\overline E$ a hermitian
  vector bundle over $X$, $s$ a section of $E$ transversal to the zero
  section and $i\colon Y\longrightarrow X$ the zero locus of $s$. Let
  $K(\overline E)$ be the 
  Koszul resolution of $i_{\ast}\mathcal{O}_{Y}$ determined by
  $\overline E$ and $s$. We can identify $N_{Y/X}$ with
  $i^{\ast}E$. We denote by $\overline N_{Y/X}$ the vector bundle with
  the metric induced by the above identification. Then 
  \begin{displaymath}
    T^{h}(i,\overline {\mathcal{O}}_{Y},\overline N_{Y/X},K(\overline
    E))=\widetilde e(X,\overline E,s)\bullet 
    \Td^{-1}(\overline E).
  \end{displaymath}
\end{proposition}
\begin{proof}
  Let $P=\mathbb{P}(E\oplus \mathbb{C})$. We follow the notation of
  proposition \ref{prop:14}. We denote by $h_{0}$ the original metric
  of $\overline E$ and by $h_{1}$ the metric induced by the
  isomorphism $E\cong f^{\ast}Q$. Observe that $h_{0}$ and $h_{1}$
  agree when restricted to $Y$, because the preimage of
  $\overline Q$ by the zero section agrees with $\overline E$. Hence
  there is an isometry
  $\overline N_{Y/X}\cong i^{\ast}f^{\ast}\overline Q$. We
  denote $T^{h}(K(\overline E))=T^{h}(i,\overline
  {\mathcal{O}}_{Y},\overline N_{Y/X},K(\overline E))$.   
  Then we have
  \begin{align*}
    T^{h}(K(\overline E))&=f^{\ast}T^{h}(K(\overline
    {\mathcal{O}_{X}},\overline E))+\sum_{i}(-1)^{i}\widetilde
    {\ch}(\bigwedge^{i}E^{\vee},h_{0},h_{1})\\
    &=f^{\ast}(\widetilde e(P,\overline Q,s_{Q})\bullet
    \Td^{-1}(\overline Q))+\widetilde c_{r}(E,h_{0},h_{1})\bullet
    \Td^{-1}(E,h_{1})\\
    &\phantom{AAAA}
    +
    c_{r}(E,h_{0})\bullet \widetilde{\Td^{-1}}(E,h_{0},h_{1})\\
    &=\widetilde e(X,\overline E,s)\bullet \Td^{-1}(E,h_{1})
    -\widetilde c_{r}(E,h_{0},h_{1})\bullet \Td^{-1}(E,h_{1})\\
    &\phantom{AAAA}
    +
    \widetilde c_{r}(E,h_{0},h_{1})\bullet \Td^{-1}(E,h_{1})+
    c_{r}(E,h_{0})\bullet \widetilde{\Td^{-1}}(E,h_{0},h_{1})\\
    &=\widetilde e(X,\overline E,s)\bullet \Td^{-1}(E,h_{0})-
    \widetilde e(X,\overline E,s)\bullet \dd_{\mathcal{D}}\widetilde
    {\Td^{-1}}(E,h_{0},h_{1}) \\
    &\phantom{AAAA}
    +c_{r}(E,h_{0})\bullet
    \widetilde{\Td^{-1}}(E,h_{0},h_{1})\\ 
    &=\widetilde e(X,\overline E,s)\bullet \Td^{-1}(E,h_{0})-
    \dd_{\mathcal{D}} \widetilde e(X,\overline E,s)\bullet \widetilde
    {\Td^{-1}}(E,h_{0},h_{1}) \\
    &\phantom{AAAA}
    +c_{r}(E,h_{0})\bullet
    \widetilde{\Td^{-1}}(E,h_{0},h_{1})\\ 
    &=\widetilde e(X,\overline E,s)\bullet \Td^{-1}(E,h_{0})
    +i_{\ast}\widetilde{\Td^{-1}}(E,h_{0},h_{1})|_{Y}\\
    &=\widetilde e(X,\overline E,s)\bullet \Td^{-1}(\overline E),
  \end{align*}
which concludes the proof.
\end{proof}

\begin{theorem}\label{thm:10}
  The theory of homogeneous singular Bott-Chern classes is compatible
  with the projection formula and transitive.
\end{theorem}
\begin{proof}
  We have
  \begin{align*}
    \label{eq:72}
    C_{T^{h}}(F,N)&=(\pi _{P})_{\ast}T^{h}(K(\overline F,\overline
    N))\\
    &=(\pi _{P})_{\ast}(\widetilde e(P,\overline
  Q,s)\bullet \Td^{-1}(\overline Q)\bullet\ch(\pi
    _{P}^{\ast}\overline F))\\
    &=(\pi _{P})_{\ast}(\widetilde e(P,\overline
  Q,s)\bullet \Td^{-1}(\overline Q))\bullet\ch(\overline F)\\
  &=C_{T^{h}}(\mathcal{O}_{Y},N)\bullet \ch(F).
  \end{align*}
  Thus $C_{T^{h}}$ is compatible with the projection formula.
  
  We now prove the transitivity. Let $Y$, $N_{1}$ and $N_{2}$ be as in
  corollary \ref{cor:8}. We follow the notation after this
  corollary. Then applying proposition \ref{prop:10} we obtain
  \begin{equation}\label{eq:74}
    T^{h}(\overline K)=\widetilde e(P,\pi _{1}^{\ast}\overline
    Q_{1}\oplus \pi _{2}^{\ast}\overline Q_{2},s_{1}+s_{2})\bullet
    \Td^{-1}(\pi _{1}^{\ast}\overline
    Q_{1}\oplus \pi _{2}^{\ast}\overline Q_{2}),
  \end{equation}
  where $s_{i}$ denote the tautological section of
  $\overline Q_{i}$ or its preimage by $\pi _{i}$.

  Then, by proposition \ref{prop:14} \ref{item:10}, taking into
  account that $Y_{1}=P_{2}$,
  \begin{multline}
    \label{eq:75}
    T^{h}(\overline K)=\pi _{1}^{\ast}(c_{r_{1}}(\overline
    Q_{1})\Td^{-1}(\overline Q_{1})) \bullet \pi _{2}^{\ast}(
    \widetilde e(P_{2},\overline Q_{2},s_{2})\Td^{-1}(\overline
    Q_{2})) \\+
    (i_{1})_{\ast}(\widetilde e(P_{1},\overline Q_{1},s_{1})\Td^{-1}(\overline
    Q_{1})\bullet p_{1}^{\ast}\Td^{-1}(\overline N_{2})).
  \end{multline}
  Applying again proposition \ref{prop:10} we obtain
  \begin{equation}
    \label{eq:76}
    T^{h}(\overline K)=\pi _{1}^{\ast}(c_{r_{1}}(\overline
    Q_{1})\Td^{-1}(\overline Q_{1})) \bullet \pi _{2}^{\ast}(
    T^{h}(\overline K_{2})) +
    (i_{1})_{\ast}(T^{h}(\overline K_{1})\bullet 
    p_{1}^{\ast}\Td^{-1}(\overline N_{2})).
  \end{equation}
  Thus, by corollary \ref{cor:8} the theory of homogeneous singular
  Bott-Chern classes is transitive.
\end{proof}

We next recall the construction of singular Bott-Chern classes of
Bismut, Gillet and Soul\'e. 
Let $i\colon Y\longrightarrow X$ be a closed immersion of complex manifolds and 
let $\overline{\xi}=(i,\overline N,\overline
F, \overline E_{\ast}) $ be a hermitian embedded vector bundle. We consider the
associated complex of sheaves
\begin{displaymath}
  0\to E_{n}\overset{v}{\to}\dots\overset{v}{\to}
  E_{0}\to 0, 
\end{displaymath}
where we denote by $v$ the differential of this complex.

This complex is exact for all $x\in X\setminus Y$. The cohomology
sheaves of this complex are holomorphic vector bundles on $Y$ which we
denote by
\begin{displaymath}
  H_{n}=\mathcal{H}_{n}(E_{\ast}|_{Y}),\quad
  H=\bigoplus _{n}H_{n}.
\end{displaymath}
For each $x\in Y$ and $U\in T_{x}X$ we denote by $\partial_{U}v(x)$
the derivative of the map $v$ calculated in any holomorphic
trivialization of $E$ near $x$. Then $\partial_{U}v(x)$ acts on
$H_{x}$. Moreover, this action only depends on the class $y$ of $U$
in $N_{x}$. We denote it by $\partial_{y}v(x)$. Moreover
$(\partial_{y}v(x))^{2}=0$; therefore the pull-back of $H$ to the
total space of $N$ together with $\partial_{y}v$ is a complex that
we denote by $(H,\partial_{y}v)$.

On the total space of $N$, the interior multiplication by $y\in N$
turns $\bigwedge N^{\vee}$ into a Koszul complex. By abuse of notation
we denote also by $\iota_{y}$ the operator $\iota_{y}\otimes 1$ acting
on $\bigwedge N^{\vee}\otimes F$. There is a canonical isomorphism
between the complexes $(H,\partial_{y}v)$
and $(\bigwedge N^{\vee}\otimes F,\iota_{y})$. An explicit description
of this isomorphism can be found in \cite{bismut90:SCCI} \S1.

Let $v^{\ast}$ be the adjoint of the operator $v$ with respect to the
metrics of $\overline E_{\ast}$. Then we have an identification of
vector bundles over $Y$
\begin{displaymath}
  H_{k}=\{f\in E_{k}\mid vf=v^{\ast}f=0\}.
\end{displaymath}
This identification induces a hermitian metric on $H_{k}$, and hence
on $H$. Note that the metrics on $N$ and $F$ also induce a hermitian
metric on $\bigwedge N^{\vee}\otimes F$.

\begin{definition}
We say that $\overline {\xi}=(i,\overline N,\overline
F, \overline {E}_{\ast})$ satisfies Bismut assumption (A) if the
canonical isomorphism between  $(H,\partial_{y}v)$ and 
$(\bigwedge N^{\vee}\otimes F,\iota_{y})$ is an isometry.
\end{definition}

\begin{proposition} \label{prop:8}
Let $\overline {\xi}=(i,\overline N,\overline
F, \overline E_{\ast})$ be as before, with $\overline N=(N,h_{N})$ and
$\overline F=(F,h_{F})$. Then there exist metrics $h'_{E_{k}}$ over
$E_{k}$ such that the hermitian embedded vector bundle $\overline
{\xi}'=(i,\overline N,\overline F, 
(E_{\ast},h'_{E_{\ast}}))$ satisfies Bismut assumption (A).
\end{proposition}
\begin{proof}
  This is \cite{bismut90:SCCI} proposition 1.6.
\end{proof}

Let $\nabla^{E}$ be the canonical hermitian holomorphic connection on
$E$ and let $V=v+v^{\ast}$. Then
\begin{displaymath}
  A_{u}=\nabla^{E}+\sqrt{u}V
\end{displaymath}
is a superconnection on $E$. 

Let $\nabla ^{H}$ be the canonical hermitian connection on $H$. Then
\begin{displaymath}
  B=\nabla ^{H}+ \partial_{y}v+(\partial_{y}v)^{\ast}
\end{displaymath}
is a superconnection on $H$.

Let $N_{H}$ be the number operator on the complex $(E,v)$, that is,
$N_{H}$ acts on $E_{k}$ by multiplication by $k$, and let $\Tr_{s}$
denote the supertrace. Recall that here we are using the symbol $\left[\
\right]$ to 
denote the current associated to a locally integrable differential
form and the symbol $\delta _{Y}$ to denote the current integration
along a subvariety, both with the normalizations of
notation \ref{def:19}.  

For $0<Re(s)\le 1/2$ let $\zeta_{E}(s)$ be the current on $X$ given by
the formula
\begin{multline}
  \label{eq:41}
  \zeta_{E}(s)=\frac{1}{\Gamma (s)}\left.\int_{0}^{\infty}u^{s-1}
  \right\{\left[\Tr_{s}\left(N_{H}\exp(-A_{u}^{2})\right)\right]\\
  -\left.i_{\ast}\left[\int_{N}\Tr_{s}\left(N_{H}\exp(-B^{2})\right)\right] 
    \right\}\dd u.
\end{multline}
This current is well defined and extends to a current that depends
holomorphically on $s$ near $0$.

\begin{definition} \label{def:9}
  Assume that $\overline {\xi}=(i,\overline N,\overline
F, \overline E_{\ast})$ satisfies Bismut assumption (A). Then we denote
\begin{displaymath}
  T^{BGS}(\overline {\xi})=-\frac{1}{2}\zeta '_{E}(0). 
\end{displaymath}
By abuse of notation we will denote also by $T^{BGS}(\overline {\xi})$
its class in $\widetilde
  \bigoplus_{p} \widetilde
{\mathcal{D}}^{2p-1}_{D}(X,p). $

Let now
  $\overline {\xi}=(i,\overline N,\overline
F, (E_{\ast},h_{E_{\ast}}))$ be general and let $\overline
{\xi}'=(i,\overline N,\overline 
F, (E_{\ast},h'_{E_{\ast}}))$ be any hermitian embedded vector bundle
satisfying
assumption (A) provided by proposition \ref{prop:8}. 
Then we denote
\begin{displaymath}
  T^{BGS}(\overline{\xi})=T^{BGS}(\overline{\xi}')+
  \sum_{i}(-1)^{i}\widetilde 
{\ch}(E_{i},h_{E_{i}},h'_{E_{i}}), 
\end{displaymath}
where $\widetilde 
{\ch}(E_{i},h_{E_{i}},h'_{E_{i}})$ is as in definition \ref{def:11}. 
\end{definition}

\begin{remark}
  This definition only agrees (up to a normalization factor) with the
  definition in \cite{BismutGilletSoule:MR1086887} 
  for hermitian embedded vector bundles that satisfy assumption (A).
\end{remark}

\begin{theorem}
  The assignment that, to each hermitian embedded vector bundle
  $\overline {\xi}$,  associates the current
  $T^{BGS}(\overline{\xi})$, is a 
  theory of singular Bott-Chern classes that agrees with $T^{h}$.
\end{theorem}

\begin{proof}
  First we have to show that, when $\overline {\xi}$ does not satisfy
  assumption (A) then $T^{BGS}(\overline {\xi})$ is well
  defined. Assume that $\overline {\xi}''=(i,\overline N,\overline 
  F, (E_{\ast},h'_{E_{\ast}}))$ is another choice of hermitian
  embedded vector bundle satisfying assumption (A). By lemma
  \ref{lemm:1} we have that
  \begin{displaymath}
    \widetilde{\ch}(E_{i},h_{i},h'_{i})+
    \widetilde{\ch}(E_{i},h'_{i},h''_{i})  
    +\widetilde{\ch}(E_{i},h''_{i},h_{i})=0.
  \end{displaymath}
  By \cite{BismutGilletSoule:MR1086887} theorem 2.5 we have that
  \begin{displaymath}
    T^{BGS}(\overline{\xi}')-T^{BGS}(\overline{\xi}'')
    =\sum_{i}(-1)^{i}\widetilde
    {\ch}(E_{i},h'_{E_{i}},h''_{E_{i}}).  
  \end{displaymath}
  Summing up we obtain that $T^{BGS}(\overline {\xi})$ is well
  defined. 

  If the hermitian embedded vector bundle $\overline {\xi}$ satisfies
  Bismut assumption 
  (A) then, by \cite{BismutGilletSoule:MR1086887} theorem 1.9,
  $T^{BGS}(\overline 
  {\xi})$ satisfies 
  equation \eqref{eq:42}. If $\overline {\xi}$ does not satisfy
  assumption (A) then, combining \cite{BismutGilletSoule:MR1086887}
  theorem 1.9 and 
  equation \eqref{eq:13}, we also obtain that $T^{BGS}(\overline
  {\xi})$ satisfies 
  equation \eqref{eq:42}. 

  The functoriality property is \cite{BismutGilletSoule:MR1086887}
  theorem 1.10.  

  In order to prove the normalization property, let $\overline
  {\xi}=(i\colon Y\longrightarrow  
    X,\overline N, \overline 
    F, \overline E_{\ast})$
  be a hermitian embedded vector bundle that satisfies assumption (A)
  and let $\overline A$ be a non-negatively graded orthogonally split
  complex of vector bundles on $X$. Observe that $\overline A$ is also a
  (trivial) hermitian embedded vector bundle. Then $\overline A$ and
  $\overline {\xi}\oplus \overline A$ also satisfy assumption (A). By 
  \cite{BismutGilletSoule:MR1086887} theorem 2.9
  \begin{displaymath}
    T^{BGS}(\overline {\xi}\oplus \overline A)=
    T^{BGS}(\overline {\xi})+T^{BGS}(\overline A).
  \end{displaymath}
  But by \cite{BismutGilletSoule:MR1047123} remark 2.3, $T^{BGS}(\overline
  A)$ agrees with the 
  Bott-Chern class associated to the Chern character and  the exact
  complex $\overline A$. Since $A$ is 
  orthogonally split we have $T^{BGS}(\overline A)=0$. Now the case when $\xi$
  does not satisfy assumption (A) follows from the definition.

  By \cite{BismutGilletSoule:MR1086887} theorem 3.17, with the
  hypothesis of proposition \ref{prop:10}, we have that
  \begin{align*}
    T^{BGS}(i,\overline {\mathcal{O}}_{Y},\overline N_{Y/X},K(\overline
    E))&=\widetilde e(X,\overline E,s)\bullet 
    \Td^{-1}(\overline E)\\
    &=T^{h}(i,\overline {\mathcal{O}}_{Y},\overline N_{Y/X},K(\overline
    E)).
  \end{align*}
  From this it follows that $C_{T^{BGS}}=C_{T^{h}}$ and by theorem
  \ref{thm:6}, $T^{BGS}=T^{h}$.
\end{proof}

We now recall Zha's construction. Note that, in order to
obtain a theory of singular Bott-Chern classes, we have changed the
normalization convention from the one
used by Zha. Note also that Zha does not define explicitly  a
singular Bott-Chern 
class, but such a definition is implicit in his definition of direct
images for closed immersions. Let $Y$ be a 
complex 
manifold and let $\overline N=(N,h)$ be a hermitian vector bundle. 
We denote $P=\mathbb{P}(N\oplus \mathbb{C})$. Let $\pi
\colon P\longrightarrow Y$ denote the projection and let
$\iota\colon Y\longrightarrow P$ denote the inclusion as the zero
section. On $P$ we consider the 
tautological exact sequence
\begin{displaymath}
  0\longrightarrow \mathcal{O}(-1)\longrightarrow \pi ^{\ast}N\oplus
  \mathcal{O}_{P} \longrightarrow Q\longrightarrow 0.
\end{displaymath}
Let $h_{1}$ denote the hermitian metric on $Q^{\vee}$ induced by the
metric of $N$ and the trivial metric on $\mathcal{O}_{P}$ and let
$h_{0}$ denote the semi-definite hermitian form on $Q^{\vee}$ induced
by the map $Q^{\vee}\longrightarrow \mathcal{O}_{P}$ obtained from the
above exact sequence and the trivial metric on $\mathcal{O}_{P}$.
Let $h_{t}=(1-t^{2})h_{0}+t^{2}h_{1}$. It is a hermitian metric on
$Q^{\vee}$. We will denote $\overline Q_{t}^{\vee}
=(Q^{\vee},h_{t})$. Let $\nabla_{t}$ be the associated hermitian holomorphic
connection and let $N_{t}$ denote  the endomorphism defined by
\begin{displaymath}
  \frac{\dd}{\dd t}\left<v,w\right>_{t}=\left<N_{t}v,w\right>.
\end{displaymath}

For each $n\ge 1$, let $\Det$ denote the alternate $n$-linear form on
the space of $n$ by $n$ matrices such that
\begin{displaymath}
  \det(A)=\Det(A,\dots ,A).
\end{displaymath}
We denote $\det(B;A)=\Det(B,A,\dots ,A)$.

Zha introduced the differential form 
\begin{equation}
  \label{eq:43}
  \widetilde e_{Z}(\overline Q^{\vee})=\frac{-1}{2}\lim_{s\rightarrow 0}
 \int_{s}^{1}\det(N_{t},\nabla_{t}^{2})\dd t
\end{equation}
which is a smooth form on $P\setminus \iota(Y)$, locally integrable on
$P$. Hence it defines a current, also denoted by $ \widetilde
e_{Z}(\overline Q^{\vee})$ on $P$. 
The important property of this current is that it satisfies
\begin{equation}
  \label{eq:44}
  \dd_{\mathcal{D}}\overline e_{Z}(Q^{\vee})=c_{n}(\overline 
  Q_{1})-\delta
  _{Y}.  
\end{equation}

In \cite{zha99:_rieman_roch}, Zha denotes by $C(\overline Q^{\vee})$ a
form that differs from $\widetilde e_{Z}$ by the normalization factor
and the sign. We denote it by $\widetilde e_{Z}$ because it agrees with
the Euler-Green current introduced in
\cite{BismutGilletSoule:MR1086887}.  

\begin{proposition}
  The equality
  \begin{displaymath}
     \widetilde e_{Z}(Q^{\vee})=\widetilde e(P,\overline 
  Q_{1}, s_{Q})
  \end{displaymath}
  holds.
\end{proposition}
\begin{proof}
  With the notations of lemma  \ref{lemm:4}, 
  both classes satisfy equation \eqref{eq:44} and their restriction to
  $D_{\infty}$ is zero. By lemma \ref{lemm:4} they agree.
\end{proof}

\begin{definition}
   Let $\overline {\xi}= (i\colon Y\longrightarrow
  X,\overline N, \overline 
  F, \overline E_{\ast}) $ be as in definition \ref{def:7}. Let
  $\overline A_{\ast}$, $\tr_{1}(\overline
  E)_{\ast}$ and  $\overline {\eta}_{\ast}$ be as in
  \eqref{eq:33}. Then we define
  \begin{multline} 
    T^{Z}(\overline \xi)=
    -(p_{W})_{\ast}
    \left(\sum_{k}(-1)^{k}W_{1} \bullet \ch(\tr_{1}(\overline E)_{k}) 
    \right)\\
    -\sum_{k}(-1)^{k}(p_{P})_{\ast}[\widetilde
    \ch(\overline \eta_{k})]+(p_{P})_{\ast}(\ch(\pi
    _{p}^{\ast}\overline F)\Td^{-1}(\overline Q_{1})\widetilde
    e_{Z}(\overline 
    Q_{1}^{\vee})).
  \end{multline}
\end{definition}

It follows directly from the definition that $T^{Z}$ is the theory
of singular Bott-Chern classes associated to the class
\begin{equation}
  \label{eq:45}
  C_{Z}(F,N)=(p_{P})_{\ast}(\ch(\pi
  _{p}^{\ast}\overline F)\Td^{-1}(\overline Q_{1})\widetilde
  e_{Z}(\overline 
  Q_{1}^{\vee})).
\end{equation}

\begin{theorem}
  The theory of singular Bott-Chern classes $T^{Z}$ agrees with the
  theory of homogeneous singular Bott-Chern classes $T^{h}$.
\end{theorem}
\begin{proof}
  The result follows directly from theorem \ref{thm:6}, equation
  \eqref{eq:45} and proposition \ref{prop:10}.
\end{proof}

Next we want to use \ref{thm:7} to give another characterization of
$T^{h}$. To this end 
we only need to compute the characteristic class
$C_{T^{h}}(\mathcal{O}_{Y},L)$ for a line bundle $L$ as a power
series in $c_{1}(L)$. 

\begin{theorem} \label{thm:11}
  The theory of homogeneous singular Bott-Chern classes of algebraic
  vector bundles is the unique
  theory of singular Bott-Chern classes of algebraic vector bundles that 
  is compatible with the projection formula and transitive and that
  satisfies
  \begin{displaymath}
    C_{T^{h}}(\mathcal{O}_{Y},L)={\bf 1}_{1}\bullet \phi(c_{1}(L)), 
  \end{displaymath}
  where $\phi $ is the power series
  \begin{displaymath}
    \phi(x)=\frac{1}{2}
    \sum_{n=0}^{\infty}\frac{(-1)^{n+1}H_{n+1}}{(n+2)!}x^{n},
  \end{displaymath}
  and where $H_{n}=1+\frac{1}{2}+\frac{1}{3}+\dots +\frac{1}{n}$, $n\ge 1$
  are the harmonic 
 numbers.
\end{theorem}

We already know that
$T^{h}$ is compatible with the projection
formula and transitive. Thus it only remains to compute the power
series $\phi $.

Let $\overline L=(L,h_{L})$ be a hermitian line bundle over a complex  manifold
$Y$. Let $z$ be a system of holomorphic coordinates of $Y$. Let $e$ be a 
local section of $L$ and let $h(z)=h(e_{z},e_{z})$. Let
$P=\mathbb{P}(L\oplus 
\mathbb{C})$, with $\pi \colon P\longrightarrow Y$ the projection and
$\iota\colon Y\longrightarrow P$ the zero section. We choose homogeneous
coordinates on $P$ given by 
$(z,(x:y))$, here $(x:y)$ represents
the line of $L_{z}\oplus \mathbb{C}$ generated by $xe(z)+y\mathbf{1}$,
where $\mathbf{1}$ is a generator of $\mathbb{C}$ of norm 1. On the
open set $y\not =0$ we will use the absolute coordinate $t=x/y$. Let
\begin{displaymath}
  0\longrightarrow \mathcal{O}(-1)\longrightarrow 
  \pi ^{\ast}(L\oplus \mathbb{C})\longrightarrow Q\longrightarrow 0
\end{displaymath}
be the tautological exact sequence. The section $s=\{\mathbf{1}\}$ is a
global section of $Q$ that vanishes along the zero section. Moreover
we have
\begin{displaymath}
  \|s\|^{2}_{(z,(x:y))}=\frac{x\bar x h(z)}{y\bar y+
x\bar x h(z)}=\frac{t\bar t h}{1+
t\bar t h}.
\end{displaymath}
Then (recall that we are using the algebro-geometric normalization)
\begin{align}
  c_{1}(\overline Q)&=\partial\bar {\partial}\log  \|s\|^{2}\\
  &=\partial\bar {\partial}\log \frac{t\bar t h}{1+
t\bar t h}\\
&=\partial\left(
  \frac{1+t\bar th}{t\bar th}
  \frac{t\bar\partial(\bar th)(1+t\bar t h)-t^{2}\bar t
    h\bar \partial(\bar t h)}{(1+t\bar t h)^{2}}
  \right)\\
&=\partial\left(\frac{t\bar\partial (\bar t h)}{t\bar t h(1+t\bar t h)}
  \right)\\
&=\partial\left(\frac{\bar \partial(\bar t h)}{\bar t h}
\right)\frac{1}{1+t\bar t h}-\frac{\bar t \partial(ht)\land
  \bar \partial(\bar t h)}{\bar t h(1+t\bar t h)^{2}}\\
&=\frac{\pi ^{\ast}c_{1}(\overline L)}{1+t\bar t h}-
\frac{\partial(th)\land
  \bar \partial(\bar t h)}{h(1+t\bar t h)^{2}}.
\end{align}

We now consider the Koszul resolution
\begin{displaymath}
  \overline K\colon  0\longrightarrow Q^{\vee}\overset{s}{\longrightarrow}
  \mathcal{O}_{p} \longrightarrow \iota_{\ast}\mathcal{O}_{X}
  \longrightarrow 0.
\end{displaymath}
 We denote by $T^{h}(\overline K)$ the singular Bott-Chern class
 associated to this Koszul complex. Then, by proposition \ref{prop:14}
 and proposition \ref{prop:10},
 \begin{displaymath}
   T^{h}(\overline K)=-\frac{1}{2}\Td^{-1}(\overline
   Q)\log\|s\|^{2}. 
 \end{displaymath}
 In order to compute $\pi _{\ast}T^{h}(\overline K)$ we have to
 compute first $\pi _{\ast}c_{1}(\overline Q)^{n}\log\|s\|^{2}$.
 But
 \begin{displaymath}
   c_{1}(\overline Q)^{n}=
   \frac{\pi^{\ast}c_{1}(\overline L)^{n}}{(1+t\bar t h)^{n}}-
   n\left(\frac{\pi^{\ast}c_{1}(\overline L)}{(1+t\bar t h)}
   \right)^{n-1}\frac{\partial (t h)\land \bar \partial(\bar t
     h)}{h(1+t\bar t h)^{2}}.
 \end{displaymath}
Therefore 
 \begin{align*}
   \pi _{\ast}c_{1}(\overline
   Q)^{n}\log\|s\|^{2}&=-nc_{1}(\overline L)^{n-1} 
   \frac{1}{2\pi i}\int_{\mathbb{P}^{1}}\frac{\partial (t h)\land
     \bar \partial(\bar t 
     h)}{h(1+t\bar t h)^{n+1}}\log \frac{t\bar t h}{1+t\bar t h}\\
   &=-n c_{1}(\overline L)^{n-1} 
   \frac{1}{2\pi i}\int_{0}^{2\pi}\int _{0}^{\infty}
   \log \frac{r^{2}}{1+r^{2}}\frac{-2ir\dd \theta \dd
     r}{(1+r^{2})^{n+1}}\\ 
   &=n c_{1}(\overline
   L)^{n-1}\int_{0}^{1}\log(1-w)w^{n-1}\dd w\\
   &=-c_{1}(\overline
   L)^{n-1} H_{n},
 \end{align*}
 where $H_{n}$, $n\ge 1$ are the harmonic numbers. 
 Since
 \begin{displaymath}
   Td^{-1}(\overline Q)=\frac{1-\exp(-c_{1}(\overline
     Q))}{c_{1}(\overline
     Q)}=\sum_{n=0}^{\infty}\frac{(-1)^{n}}{(n+1)!}
   c_{1}(\overline Q)^{n},
 \end{displaymath}
we obtain
\begin{align*}
  C_{T^{h}}(\mathcal{O}_{Y},L)=\pi _{\ast}T^{h}(\overline
  K)=\frac{1}{2}\sum_{n=0}^{\infty} 
  \frac{(-1)^{n+1}H_{n+1}}{(n+2)!}c_{1}(\overline L)^{n}{\bf 1}_{1}.
\end{align*}

Then, a reformulation of proposition \ref{prop:9} is
\begin{corollary} \label{cor:10}
  Let $T$ be a theory of singular Bott-Chern classes for algebraic
  vector bundles that is
  compatible with the projection formula and transitive. Then there is
  a unique additive genus $S_{T}$ such that
  \begin{equation}\label{eq:17}
    C_{T}(F,N)-C_{T^{h}}(F,N)=\ch(F)\bullet \Td(N)^{-1}\bullet S_{T}(N).
  \end{equation}
  Conversely, any additive genus determines a theory of singular
  Bott-Chern classes by the formula (\ref{eq:17}).
\end{corollary}

\section{The arithmetic Riemann-Roch theorem for regular closed  
  immersions}

In this section we recall the definition of arithmetic Chow groups and
arithmetic $K$-groups. We see that each choice of an additive theory
of singular 
Bott-Chern classes allows us to define direct images for closed
immersions in arithmetic
$K$-theory. Once the direct images for closed immersions are defined,
we prove the arithmetic Grothendieck-Riemann-Roch theorem for closed
immersions. A
version of this theorem was proved earlier by Bismut, Gillet and
Soul\'e 
\cite{BismutGilletSoule:MR1086887} when there is a commutative diagram 
\begin{displaymath}
  \xymatrix{
    \mathcal{Y}\ar[r]^{i}\ar[dr]^{f}& \mathcal{X}\ar[d]^{g}\\
    &\mathcal{Z}
  },
\end{displaymath}
where $i$ is a closed immersion and $f$ and $g$ are
smooth  over $\mathbb{C}$.
The version of this theorem given in this paper is due to Zha
\cite{zha99:_rieman_roch}, but still unpublished. 
 The theorem of Bismut, Gillet and Soul\'e
compares $g_{\ast}\chh(i_{\ast}\overline E)$ with
$f_{\ast}\chh(\overline E)$, whereas the theorem of Zha compares
directly $\chh(i_{\ast}\overline E)$ with $i_{\ast}\chh(\overline
E)$. The main difference between the theorem of Bismut, Gillet and 
Soul\'e and that of Zha is the kind of arithmetic Chow groups they
use. In 
the first 
case these groups are only covariant for proper
morphisms that are smooth over $\mathbb{C}$; thus the
Grothendieck-Riemann-Roch can only be stated for a diagram as above,
while in the second case a version of these groups that are covariant
for arbitrary proper morphisms is used.

Since each choice of a theory of singular Bott-Chern classes gives
rise to a 
different definition of direct images for closed immersions, the
arithmetic Grothendieck-Riemann-Roch theorem will have a correction
term that 
depends on the theory of singular Bott-Chern classes used. In the
particular case of the homogeneous singular Bott-Chern classes, which
are the theories used by  Bismut, Gillet and Soul\'e and by Zha, this
correction term vanishes and we obtain the simplest formula. In this
case the arithmetic Grothendieck-Riemann-Roch theorem is formally
identical to the classical one.

Let $(A,\Sigma ,F_{\infty})$ be an arithmetic ring
\cite{GilletSoule:ait}. Since we will allow the arithmetic varieties
to be non regular and we will use Chow groups indexed by dimension,
following \cite{GilletSoule:aRRt} we will assume that the ring $A$ is
equidimensional and Jacobson. Let $F$ be the field of fractions of A. 
An \textit{arithmetic variety} $\mathcal{X}$ is a scheme  
flat and quasi-projective over $A$ such that
$\mathcal{X}_{F}=\mathcal{X}\times \Spec F$ 
is smooth. 
 Then 
$X:=\mathcal{X}_{\Sigma }
$ is a
complex algebraic manifold, which is endowed with an
anti-holomorphic automorphism $F_{\infty}$. One also associates to
$\mathcal{X}$ the real variety 
$X_{\mathbb{R}}=(X,F_{\infty})$.  

Following \cite{BurgosKramerKuehn:cacg},  
to each regular arithmetic variety we can associate different kinds
of arithmetic Chow groups.
Concerning arithmetic Chow groups, we shall use the terminology and
notation in
op. cit.  \S 4 and \S 6.

Let $\mathcal{D}_{\log}$ be the Deligne complex of sheaves defined
in \cite{BurgosKramerKuehn:cacg} section 5.3; we refer to op. cit.
for the precise definition and properties. A
\textit{$\mathcal{D}_{\log}$-arithmetic variety} is a pair
$(\mathcal{X},\mathcal{C})$
consisting of an arithmetic variety $\mathcal{X}$ and a complex of
sheaves $\mathcal{C}$ on $X_{\mathbb{R}}$ which is a
$\mathcal{D}_{\log}$-complex (see op. cit. section 3.1).

We are interested in the following $\mathcal{D}_{\log}$-complexes of
sheaves:
\begin{enumerate}
\item The Deligne complex $\mathcal{D}_{\las,X}$ of
differential forms on $X$ with logarithmic and arbitrary
singularities. That is, for every Zariski open subset $U$ of $X$, we
write
\begin{displaymath}
E^{\ast}_{\las,X}(U)=\lim_{\substack{\longrightarrow \\
\overline{U}}}\Gamma(\overline{U},\mathscr{E}^{\ast}_{\overline{U}}
(\log B)),
\end{displaymath}
where the limit is taken over all diagrams
\begin{displaymath}
\xymatrix{
U\ar[r]^{\overline{\iota}}\ar[rd]^{\iota}
&\overline{U}\ar[d]^{\beta} \\
&X
}
\end{displaymath}
such that $\overline{\iota}$ is an open immersion, $\beta $ is a
proper morphism, $B=\overline{U}\setminus U$, 
is a normal crossing divisor and $\mathscr{E}^{\ast}_{\overline{U}}
(\log B)$ denotes the sheaf of smooth differential forms on $U$ with
logarithmic singularities along $B$ introduced in \cite{Burgos:CDc} .

For any Zariski open subset $U\subseteq X$, we put
\begin{displaymath}
\mathcal{D}^{\ast}_{\las,X}(U,p)=(\mathcal{D}^{\ast}_
{\las,X}(U,p),\dd_{\mathcal{D}})=(\mathcal{D}^{\ast}
(E_{\las,X}(U),p),\dd_{\mathcal{D}}).
\end{displaymath}
If $U$ is now a Zariski open subset of $X_{\mathbb{R}}$, then 
we write
\begin{displaymath}
\mathcal{D}^{\ast}_{\las,X}(U,p)=(\mathcal{D}^{\ast}_
{\las,X}(U,p),\dd_{\mathcal{D}})=
(\mathcal{D}^{\ast}_{\las,X}(U_{\mathbb{C}},p)^{\sigma},  
\dd_{\mathcal  {D}}),
\end{displaymath}
where $\sigma$ is the involution $\sigma (\eta)=\overline
{F_{\infty}^{\ast}\eta}$ as in \cite{BurgosKramerKuehn:cacg}
notation 5.65.

Note that the
sections of $\mathcal{D}^{\ast}_{\las,X}$ over an open set $U\subset X$ are
differential forms on $U$ with logarithmic singularities along
$X\setminus U$ and arbitrary singularities along
$\overline{X}\setminus X$, where $\overline{X}$ is an arbitrary
compactification of $X$. Therefore the complex of global sections
satisfy
\begin{displaymath}
  \mathcal{D}^{\ast}_{\las,X}(X,*)=\mathcal{D}^{\ast}(X,\ast),
\end{displaymath}
where the right hand side complex has been introduced in section \S
1. The complex $\mathcal{D}^{\ast}_{\las,X}$ is a particular case of
the construction 
of \cite{BurgosKramerKuehn:accavb} section 3.6. 

\item The Deligne complex $\mathcal{D}_{\D, X}$ 
of currents on $X$. This is the complex introduced in
\cite{BurgosKramerKuehn:cacg} definition 6.30.
\end{enumerate}

When $\mathcal{X}$ is regular, applying the theory of
\cite{BurgosKramerKuehn:cacg} we can define 
the arithmetic Chow groups
$\cha^{\ast}(\mathcal{X},\mathcal{D}_{\las,X})$ and
$\cha^{\ast}(\mathcal{X},\mathcal{D}_{\D,X})$. These groups satisfy
the following properties
\begin{enumerate}
\item There are natural morphisms
  \begin{displaymath}
    \cha^{\ast}(\mathcal{X},\mathcal{D}_{\las,X}) \longrightarrow 
    \cha^{\ast}(\mathcal{X},\mathcal{D}_{\D,X})
  \end{displaymath}
  and, when applicable, all properties below will be compatible with these
  morphisms.  

\item There is a product structure that turns
  $\cha^{\ast}(\mathcal{X},\mathcal{D}_{\las,X})_{\QQ}$ into an
  associative and commutative algebra. Moreover, it turns
  $\cha^{\ast}(\mathcal{X},\mathcal{D}_{\D,X})_{\QQ}$ into a 
  $\cha^{\ast}(\mathcal{X},\mathcal{D}_{\las,X})_{\QQ}$-module.
\item  If $f\colon \mathcal{Y}\longrightarrow \mathcal{X}$ is a map 
  of regular arithmetic varieties, there
  are pull-back morphisms  
  \begin{displaymath}
    f^{\ast}\colon \cha^{\ast}(\mathcal{X},\mathcal{D}_{\las,X})\longrightarrow 
    \cha^{\ast}(\mathcal{Y},\mathcal{D}_{\las,Y}).    
  \end{displaymath}
  If moreover, $f$ is smooth over $F$, there are pull-back morphisms 
  \begin{displaymath}
    f^{\ast}\colon \cha^{\ast}(\mathcal{X},\mathcal{D}_{\D,X})\longrightarrow 
    \cha^{\ast}(\mathcal{Y},\mathcal{D}_{\D,Y}).    
  \end{displaymath}
  The inverse image is compatible with the product structure.
\item  If $f\colon \mathcal{Y}\longrightarrow \mathcal{X}$ is a proper map  
  of regular arithmetic varieties of relative dimension $d$, there are
  push-forward morphisms 
  \begin{displaymath}
    f_{\ast}\colon \cha^{\ast}(\mathcal{Y},\mathcal{D}_{\D,Y})\longrightarrow 
    \cha^{\ast-d}(\mathcal{X},\mathcal{D}_{\D,X}).    
  \end{displaymath}
  If moreover, $f$ is smooth over $F$, there are push-forward  morphisms 
  \begin{displaymath}
    f_{\ast}\colon \cha^{\ast}(\mathcal{Y},\mathcal{D}_{\las,Y})\longrightarrow 
    \cha^{\ast-d}(\mathcal{X},\mathcal{D}_{\las,X}).    
  \end{displaymath}
  The push-forward morphism satisfies the projection
  formula and is compatible with base change.
\item The groups $\cha^{\ast}(\mathcal{X},\mathcal{D}_{\las,X})$ are naturally
  isomorphic to the groups defined by Gillet and Soul\'e in
  \cite{GilletSoule:ait} (see \cite{BurgosKramerKuehn:accavb} theorem
  3.33).
  When $X$ is
  generically projective, the groups
  $\cha^{\ast}(\mathcal{X},\mathcal{D}_{\D,X})$ are isomorphic to analogous
  groups introduced by Kawaguchi and Moriwaki
  \cite{KawaguchiMoriwaki:isfav} and are very similar to the weak
  arithmetic Chow groups introduced by Zha (see
  \cite{Burgos:MR2384539}).
\item There are well-defined maps
\begin{align*}
\zeta&\colon \cha^{p}(\mathcal{X},\cc)\longrightarrow\CH^{p}(\mathcal{X}),\\  
\amap&\colon \widetilde{\cc}^{2p-1}(X_{\mathbb{R}},p)\longrightarrow\cha^{p}
(\mathcal{X},\cc),\\
\omega&\colon \cha^{p}(\mathcal{X},\cc)\longrightarrow{\rm
  Z}\cc^{2p}(X_{\mathbb{R}},p),
\end{align*}  
where $\cc$ is either $\mathcal{D}_{\las,X}$ or $\mathcal{D}_{\D,X}$.
For the precise definition of these maps see
\cite{BurgosKramerKuehn:cacg} notation 4.12. 
\end{enumerate}

When $\mathcal{X}$ is not necessarily regular, following 
\cite{GilletSoule:aRRt} and combining with the definition of
\cite{BurgosKramerKuehn:cacg} we can define the arithmetic Chow
groups indexed by dimension
$\cha_{\ast}(\mathcal{X},\mathcal{D}_{\las,X})$ and 
$\cha_{\ast}(\mathcal{X},\mathcal{D}_{\D,X})$ (see
\cite{BurgosKramerKuehn:accavb} section 5.3). 

They have the following properties (see \cite{GilletSoule:aRRt}).
\begin{enumerate}
\item If $\mathcal{X}$ is regular and equidimensional of dimension
  $n$ then there are isomorphisms
  \begin{align*}
    \cha_{\ast}(\mathcal{X},\mathcal{D}_{\las,X})
    &\cong \cha^{n-\ast}(\mathcal{X},\mathcal{D}_{\las,X}),\\
    \cha_{\ast}(\mathcal{X},\mathcal{D}_{\D,X})
    &\cong \cha^{n-\ast}(\mathcal{X},\mathcal{D}_{\D,X}).
  \end{align*}
\item If $f\colon \mathcal{Y}\longrightarrow \mathcal{X}$ is a proper map
  between arithmetic varieties then there is a push-forward map
  \begin{displaymath}
    f_{\ast}\colon \cha_{\ast}(\mathcal{Y},\mathcal{D}_{\D,Y})\longrightarrow 
    \cha_{\ast}(\mathcal{X},\mathcal{D}_{\D,X}).
  \end{displaymath}
  If $f$ is smooth over $F$ then there is a push-forward map
  \begin{displaymath}
    f_{\ast}\colon \cha_{\ast}(\mathcal{Y},\mathcal{D}_{\las,Y})\longrightarrow 
    \cha_{\ast}(\mathcal{X},\mathcal{D}_{\las,X}).
  \end{displaymath}
\item If $f\colon \mathcal{Y}\longrightarrow \mathcal{X}$ is a flat map or,
  more generally, a
  local complete intersection (l.c.i) map of relative dimension $d$,
  there 
  are pull-back morphisms  
  \begin{displaymath}
    f^{\ast}\colon \cha_{\ast}(\mathcal{X},\mathcal{D}_{\las,X})\longrightarrow 
    \cha_{\ast+d}(\mathcal{Y},\mathcal{D}_{\las,Y}).    
  \end{displaymath}
  If moreover, $f$ is smooth over $F$, there are pull-back morphisms 
  \begin{displaymath}
    f^{\ast}\colon \cha_{\ast}(\mathcal{X},\mathcal{D}_{\D,X})\longrightarrow 
    \cha_{\ast+d}(\mathcal{Y},\mathcal{D}_{\D,Y}).    
  \end{displaymath}
\item If $f\colon \mathcal{Y}\longrightarrow \mathcal{X}$ is a morphism of
  arithmetic varieties with $\mathcal{X}$ regular, then there is a cap
  product
  \begin{displaymath}
    \cha^{p}(\mathcal{X},\mathcal{D}_{\las,X})\otimes 
    \cha_{d}(\mathcal{Y},\mathcal{D}_{\las,Y})\longrightarrow  
    \cha_{d-p}(\mathcal{Y},\mathcal{D}_{\las,Y})_{\QQ},
  \end{displaymath}
  and a similar cap-product with the groups
  $\cha_{d}(\mathcal{Y},\mathcal{D}_{\D,Y})$. This product is denoted
  by $y\otimes x\mapsto y._{f}x$, 

\end{enumerate}
For more properties of these groups see \cite{GilletSoule:aRRt}.

We will define now the arithmetic $K$-groups in this context.
As a matter of convention, in the sequel
we will
use slanted letters to denote a object defined over $A$ and the
same letter in roman type for the corresponding object defined over
$\mathbb{C}$. For instance we will denote a vector bundle over
$\mathcal{X}$ by $\mathcal{E}$ and the corresponding vector bundle
over $X$ by $E$.

\begin{definition}
  A \emph{hermitian vector bundle} on an arithmetic
  variety $\mathcal{X}$, $\overline{\mathcal{E}}$, is a locally free
  sheaf $\mathcal{E}$ with a hermitian metric $h_E$ on the vector
  bundle $E$ induced on $X$, that is invariant under
  $F_{\infty}$. A sequence of hermitian vector bundles on
  $\mathcal{X}$
  $$(\overline{\varepsilon})\qquad \ldots \longrightarrow
  \overline{\mathcal{E}}_{n+1} \longrightarrow
  \overline{\mathcal{E}}_n \longrightarrow
  \overline{\mathcal{E}}_{n-1} \longrightarrow \ldots$$ is said to be
  exact if it is exact as a sequence of vector bundles.

  A \emph{metrized coherent sheaf} is a pair
  $\overline{\mathcal{F}}=(\mathcal{F},\overline E_{\ast}\to
  F)$, where $\mathcal{F}$ is a coherent sheaf on
  $\mathcal{X}$ and $\overline E_{\ast}\to
  F$ is a resolution of the coherent sheaf
  $F=\mathcal{F}_{\CC}$ by hermitian vector bundles, that is defined over
  $\mathbb{R}$, hence is invariant under $F_{\infty}$. We assume that
  the hermitian metrics are also invariant under $F_{\infty}$.
\end{definition}

Recall that to every hermitian vector bundle 
we can associate a collection of Chern
forms, denoted by $c_{p}$. Moreover, the invariance of the hermitian
metric under $F_{\infty}$ implies that 
the Chern forms will be invariant under the involution $\sigma $. Thus 
$$c_p(\overline
{\mathcal{E}})\in\mathcal{D}^{2p}_{\las,X}(X_{\mathbb{R}},p)=
\mathcal{D}^{2p}(X,p)^{\sigma}.$$  
 We will denote also by
$c_{p}(\overline {\mathcal{E}})$ its image in 
$\mathcal{D}^{2p}_{\D,X}(X_{\mathbb{R}},p)$. In particular we have
defined the 
Chern character $\ch(\overline {\mathcal{E}})$ in either of the
groups $\bigoplus _{p}\mathcal{D}^{2p}_{\las,X}(X_{\mathbb{R}},p)$ or 
$\bigoplus_{p}\mathcal{D}^{2p}_{\D,X}(X_{\mathbb{R}},p)$. Moreover, to
each 
finite exact sequence $(\overline{\varepsilon})$ of hermitian vector
bundles on $\mathcal{X}$ we can attach a secondary Bott-Chern class
$\widetilde{\ch}(\overline{\varepsilon})$. Again, the fact that the
sequence is defined over $A$ and the invariance of the metrics with
respect to $F_{\infty}$ imply that
$$
\widetilde{\ch}(\overline{\varepsilon})\in 
\bigoplus
_{p}\widetilde{\mathcal{D}}^{2p-1}_{\las,X}(X_{\mathbb{R}},p)= 
\bigoplus _{p}\widetilde{\mathcal{D}}^{2p-1}(X,p)^{\sigma}.
$$
We will denote also by $\widetilde{\ch}(\overline{\varepsilon})$ its
image in $\bigoplus
_{p}\widetilde{\mathcal{D}}^{2p-1}_{\D,X}(X_{\mathbb{R}},p)$. The
Bott-Chern 
classes associated to exact sequences of metrized coherent sheaves
enjoy the same properties.

\begin{definition}
Let $\mathcal{X}$ be an arithmetic variety and
let $\mathcal{C}^{\ast}(\ast)$ be one of the two $\mathcal{D}_{\log}$-complexes
$\mathcal{D}_{\las,X}$ or $\mathcal{D}_{\D, X}$. The
\textit{arithmetic $K$-group} associated to the 
$\mathcal{D}_{\log}$-arithmetic variety $(\mathcal{X},
\mathcal{C})$ is the abelian group
$\widehat{K}(\mathcal{X},\mathcal{C})$ generated by pairs
$(\overline{\mathcal{E}},\eta)$, where $\overline{\mathcal{E}}$ is a
hermitian vector bundle on $\mathcal{X}$ and $\eta\in\bigoplus_{p\geq
0} \widetilde{\mathcal{C}}^{2p-1}(X_{\mathbb{R}},p)$, modulo relations
\begin{equation}\label{equivKgr}
(\overline{\mathcal{E}}_1,\eta_1)+(\overline{\mathcal{E}}_2,\eta_2)=
(\overline{\mathcal{E}},\tilde{\ch}(\overline{\varepsilon})+
\eta_1+\eta_2)  
\end{equation}
for each short exact sequence
$$(\overline{\varepsilon})\qquad 0 \longrightarrow
\overline{\mathcal{E}}_{1} \longrightarrow \overline{\mathcal{E}}
\longrightarrow \overline{\mathcal{E}}_{2} \longrightarrow 0\ .$$

The \textit{arithmetic $K'$-group} associated to the
$\mathcal{D}_{\log}$-arithmetic variety $(\mathcal{X},
\mathcal{C})$ is 
the abelian group
$\widehat{K}'(\mathcal{X},\mathcal{C})$ generated by pairs
$(\overline{\mathcal{F}},\eta)$, where $\overline{\mathcal{F}}$ is a
metrized coherent sheaf on $\mathcal{X}$ and $\eta\in\bigoplus_{p\geq
0} \widetilde {\mathcal{C}}^{2p-1}(X_{\mathbb{R}},p)$, modulo
relations 
\begin{equation}\label{equivKprim}
(\overline{\mathcal{F}}_1,\eta_1)+(\overline{\mathcal{F}}_2,\eta_2)=
(\overline{\mathcal{F}},\tilde{\ch}(\overline{\varepsilon})+\eta_1+\eta_2)
\end{equation}
for each short exact sequence of metrized coherent sheaves
$$(\overline{\varepsilon})\qquad 0 \longrightarrow
\overline{\mathcal{F}}_{1} \longrightarrow \overline{\mathcal{F}}
\longrightarrow \overline{\mathcal{F}}_{2} \longrightarrow 0\ .$$
\end{definition}

We now give some properties of the arithmetic $K$-groups. As their
proofs are similar, in the essential
points, to those of analogous statements in, for example,
\cite{GilletSoule:ait} in the regular case and
\cite{GilletSoule:aRRt} in the singular case,
we omit them.
\begin{enumerate}
\item We have natural morphisms
$$\widehat{K}(\mathcal{X},\mathcal{D}_{\las,X}) \longrightarrow
\widehat{K}(\mathcal{X},\mathcal{D}_{\D,X})\text{ and }
\widehat{K}'(\mathcal{X},\mathcal{D}_{\las,X}) \longrightarrow
\widehat{K}'(\mathcal{X},\mathcal{D}_{\D,X}).
$$
  When applicable, all properties below will be compatible with these
  morphisms.  
\item $\widehat{K}(\mathcal{X},\mathcal{D}_{\las,X})$ is a ring. The
  product structure is given by
  \begin{equation}\label{eq:78}
    (\overline{\mathcal{F}}_{1},\eta_{1})\cdot
    (\overline{\mathcal{F}}_{2},\eta_{2})=
    (\overline{\mathcal{F}}_{1}\otimes
    \overline{\mathcal{F}}_{2},\ch(\overline{\mathcal{F}}_{1})\bullet
    \eta_{2}+\eta_{1}\bullet\ch(\overline{\mathcal{F}}_{2})+
 \dd_{\mathcal{D}}\eta_{1}\bullet \eta_{2})
  \end{equation}

\item $\widehat{K}(\mathcal{X},\mathcal{D}_{\D,X})$ is a 
$\widehat{K}(\mathcal{X},\mathcal{D}_{\las,X})$-module.

\item There are natural maps 
$$
\widehat{K}(\mathcal{X},\mathcal{C}) \longrightarrow 
\widehat{K}'(\mathcal{X},\mathcal{C}) 
$$
that, when $\mathcal{X}$ is regular, are isomorphisms.

\item The groups $\widehat{K}'(\mathcal{X},\mathcal{D}_{\las,X})$ and $
\widehat{K}'(\mathcal{X},\mathcal{D}_{\D,X})$ are  
$\widehat{K}(\mathcal{X},\mathcal{D}_{\las,X})$-modules.

\item There are natural maps 
$$
\omega \colon \widehat K'(\mathcal{X},\mathcal{C})
\longrightarrow \bigoplus _{p}Z\mathcal{C}^{2p}(p)
$$
that send the class of a pair $(\overline {\mathcal{F}},\eta)$ with  
$\overline {\mathcal{F}}=(\mathcal{F},\overline
E_{\ast}\to\mathcal{F}_{\CC})$ to the form (or current)
\begin{displaymath}
  \omega (\overline {\mathcal{F}},\eta)=\sum_{i}(-1)^{i}\ch(\overline
  E_{i}) +\dd_{\mathcal{D}}\eta.
\end{displaymath}

\item When $\mathcal{X}$ is regular, there exists a Chern character,
$$\widehat{\ch}\colon \widehat{K}(\mathcal{X},\mathcal{C})_{\QQ}
\longrightarrow 
\bigoplus_{p}\widehat{\CH}^{p}(\mathcal{X},\mathcal{C})_{\QQ},$$
that is an isomorphism. Moreover, if
$\mathcal{C}=\mathcal{D}_{\las,X}$ this isomorphism is compatible with
the product 
structure. If $\mathcal{X}$ is not regular, there is a biadditive
pairing
\begin{displaymath}
  \widehat K(\mathcal{X},\mathcal{D}_{\las,X})\otimes
  \cha_{\ast}(\mathcal{X},\mathcal{D}_{\las,X}) \longrightarrow
    \cha_{\ast}(\mathcal{X},\mathcal{D}_{\las,X})_{\QQ},
\end{displaymath}
and a similar pairing with the groups $
\cha_{\ast}(\mathcal{X},\mathcal{D}_{\D,X})$, which is denoted in both
cases by $\alpha \otimes x\mapsto \widehat{\ch}(\alpha )\cap x$. For
the properties of this product see \cite{GilletSoule:aRRt} pg. 496.
\item If $\mathcal{Y}$ and $\mathcal{X}$ are
arithmetic varieties and $f\colon \mathcal{Y}\to \mathcal{X}$ is a
morphism of arithmetic varieties, $f$ induces a morphism of rings:
\begin{equation*}
f^*\colon \widehat{K}(\mathcal{X},\mathcal{D}_{\las,X})\rightarrow
\widehat{K}(\mathcal{Y},\mathcal{D}_{\las,Y}).
\end{equation*}
When $f$ is flat, the inverse image is also defined for the groups
$\widehat{K}'(\mathcal{X},\mathcal{D}_{\las,X})$. Moreover, if
$f_{\mathbb{C}}$ is 
smooth, the inverse image can be defined for the groups
$\widehat{K}(\mathcal{X},\mathcal{D}_{\D,X})$ and, when in addition
$f$ is flat, for the groups 
$\widehat{K}'(\mathcal{X},\mathcal{D}_{\D,X})$. 
\end{enumerate}

In what follows we will be interested in direct images for closed
immersions. Since the direct images in arithmetic $K$-theory will
depend on the choice of a metric, we have the following

\begin{definition}
  A \textit{metrized arithmetic variety} is a pair
$(\mathcal{X},h_X)$
consisting of an arithmetic variety $\mathcal{X}$ and a hermitian
metric on the complex tangent bundle $T_{X}$ that is invariant
under $F_{\infty}$. 
\end{definition}

Let $(\mathcal{X},h_X)$ and $(\mathcal{Y},h_Y)$ be metrized
arithmetic varieties and let $i\colon \mathcal{Y}\longrightarrow
\mathcal{X}$  be a closed immersion. Over the complex numbers, we are
in the situation of notation \ref{def:14}. In particular we have a
canonical exact 
sequence of 
hermitian vector bundles
\begin{equation}
  \label{eq:90}
\overline{\xi}_N \colon  0\longrightarrow
\overline{T}_{Y}\longrightarrow i^*\overline{T}_{X} \longrightarrow
\overline{N}_{Y/X} \longrightarrow 0  
\end{equation}
where the tangent bundles
$T_{Y}$, $T_{X}$ are endowed with the hermitian metrics $h_Y$, $h_X$
respectively and the normal bundle $N_{Y/X}$ is endowed with an
arbitrary hermitian metric $h_{N}$. We will follow the conventions of
notation \ref{def:14}.

We next define push-forward maps, via a closed immersion, for
the elements of the arithmetic $K$-group of a metrized arithmetic
variety.
 We will define two kinds of push-forward maps. One will
depend only on a metric on the complex normal bundle $N_{Y/X}$. By
contrast, the 
second will depend on the choice of metrics on the complex tangent
bundles $T_{X}$ and $T_{Y}$. The second definition allows us to see
$K'(\underline {\phantom{A}},\mathcal{D}_{\D,Y})$ as a functor from
the category whose objects are  metrized arithmetic varieties and
whose morphisms are closed immersions to the category of abelian
groups.   

 As we deal with hermitian vector bundles and metrized coherent
sheaves, both definitions will involve the choice of a theory of
singular Bott-Chern classes. In order for the push forward to be well
defined in $K$-theory we need a minimal additivity property for the
singular Bott-Chern classes.

\begin{definition}\label{defadtheory}
A theory of singular Bott-Chern classes $T$ is called
\emph{additive} if for any closed embedding of complex manifolds
$i\colon Y\hookrightarrow X$ and any hermitian embedded vector bundles
$\overline{\xi}_1=(i,\overline{N},\overline{F}_1,\overline{E}_{1,\ast})$,
$\overline{\xi}_2=(i,\overline{N},\overline{F}_2,\overline{E}_{2,\ast})$
the equation
$$T(\overline{\xi}_1\oplus
\overline{\xi}_2)=T(\overline{\xi}_1)+T(\overline{\xi}_2)$$
is satisfied.

Let $C$ be a characteristic class for pairs of vector bundles. We
say that it is \emph{additive} (in the first variable) if
$$C(F_1\oplus F_2,N)=C(F_1,N)+C(F_2,N)$$
for any vector bundles $F_1,F_2,N$ on a complex manifold $X$.
\end{definition}

The following statement follows directly from equation \ref{eq:68}:
\begin{proposition}
A theory of singular Bott-Chern classes $T$ is additive if and only
if the corresponding characteristic class $C_T$ is additive in the
first variable.
\end{proposition}

Note that a theory of singular Bott-Chern classes consists in joining
theories of singular Bott-Chern classes in arbitrary rank and
codimension (definition \ref{def:7}). The property of being additive
gives a compatibility condition for these theories, by respect to the
hermitian vector bundles $\overline{F}$ (with the notation used in
definition \ref{def:7}). Note also that if a theory of singular
Bott-Chern classes is compatible with the projection formula then it
is additive.

\begin{definition}\label{def:17} Let $T$ be an additive theory of
  singular Bott-Chern 
  classes, and let $T_{c}$ be the associated covariant class as in
  definition \ref{def:18}. 
  Let $i\colon (\mathcal{Y},h_{Y})\longrightarrow
(\mathcal{X},h_{X})$ be a closed immersion of metrized arithmetic
varieties and let $\overline N=\overline N_{Y/X}=(N_{Y/X},h_{N})$ be a
choice of a
hermitian metric on the complex normal bundle. The 
\emph{push-forward maps}
\begin{displaymath}
  i^{T_{c}}_{\ast}, i^{T}_{\ast}\colon \widehat
  K(\mathcal{Y},\mathcal{D}_{\D,Y})\longrightarrow  
  \widehat K'(\mathcal{X},\mathcal{D}_{\D,X})
\end{displaymath}
are defined by
\begin{align}
i^{T_{c}}_{\ast} (\overline{\mathcal{F}},\eta)
&= [((i_{\ast}\mathcal{F},\overline{E}_{\ast}\to
(i_{\ast}\mathcal{F})_{\CC}) ,0)]-[(0,T_{c}(\overline{\xi}_{c}))]\notag\\
&\phantom{AA}
+[(0,i_{\ast}(\eta\Td(Y) i^*\Td^{-1}(X)))]\label{eq:66}\\
i^{T}_{\ast}(\overline{\mathcal{F}},\eta)&=
[((i_{\ast}\mathcal{F},\overline{E}_{\ast}\to
(i_{\ast}\mathcal{F})_{\CC}) ,0)]-[(0,T(\overline{\xi}))]\notag
\\
&\phantom{AA}+[(0,i_{\ast}(\eta\Td^{-1}(\overline N_{Y/X})))].
\label{eq:91}
\end{align}
Here
$$0\rightarrow\overline{E}_n\rightarrow\ldots\rightarrow
\overline{E}_1\rightarrow\overline{E}_0\rightarrow
(i_*\mathcal{F})_{\mathbb{C}}\rightarrow 0$$ is a finite resolution of
the coherent 
sheaf $(i_*\mathcal{F})_{\CC}$ by hermitian vector bundles, 
$\overline{\xi}=
(i,\overline{N}_{X/Y},
\overline{\mathcal{F}}_{\mathbb{C}},\overline{E}_*)$  
is the induced hermitian embedded vector bundle on $X$, and
$\overline{\xi}_{c}=
(i,\overline T_{X}, \overline T_{Y},
\overline{\mathcal{F}}_{\mathbb{C}},\overline{E}_*)$ as in definition
\ref{def:18}.

We can extend this definition to push-forward maps
\begin{displaymath}
  i^{T_{c}}_{\ast},i^{T}_{\ast}\colon \widehat
  K'(\mathcal{Y},\mathcal{D}_{\D,Y})\longrightarrow  
  \widehat K'(\mathcal{X},\mathcal{D}_{\D,X})
\end{displaymath}
by the rule
\begin{align}
i^{T_{c}}_{\ast} (\overline{\mathcal{F}},\eta) &=
[((i_{\ast}\mathcal{F},\Tot(\overline E_{\ast,\ast})\to
(i_{\ast}\mathcal{F})_{\CC}),0)]-
\sum_{i}(-1)^{i}[(0,T_{c}(\overline{\xi}_{i,c}))]
\notag \\
&\phantom{AAA}+[(0,i_{\ast}(\eta\Td(Y) i^*\Td^{-1}(X)))],
  \label{eq:67}\\
i^{T}_{\ast} (\overline{\mathcal{F}},\eta) &=
[((i_{\ast}\mathcal{F},\Tot(\overline E_{\ast,\ast})\to
(i_{\ast}\mathcal{F})_{\CC}),0)]-
\sum_{i}(-1)^{i}[(0,T(\overline{\xi}_{i}))]
\notag \\
&\phantom{AAA}+[(0,i_{\ast}(\eta\Td^{-1}(\overline N_{Y/X})))],
  \label{eq:92}
\end{align}
where $0\to \overline {E}_{n}\to\dots\to \overline {E}_{0}\to
\mathcal{F}_{\CC}\to 0$ is a resolution of $\mathcal{F}_{\CC}$ by
hermitian vector bundles, $\overline E_{\ast,\ast}$ is a complex of
complexes of vector bundles over $X$, such that, for each $i\ge 0$,
$\overline {E}_{i,\ast}\to i_{\ast}E_{i}$ is also a resolution by
hermitian vector bundles and $\overline{\xi}_{i}=
(i,\overline{N}_{X/Y},\overline{E}_{i},\overline{E}_{i,*})$ is the
induced hermitian embedded vector bundle and $\overline{\xi}_{i,c}$ is
as in definition \ref{def:18}. We suppose that there is a
commutative diagram of resolutions
\begin{displaymath}
  \xymatrix{\dots \ar[r] & E_{k+1,\ast} \ar[r]\ar[d]
& E_{k,\ast} \ar[r]\ar[d] & E_{k-1,\ast} \ar[r]\ar[d] & \dots\\
\dots \ar[r] & i_{\ast}E_{k+1} \ar[r] & i_{\ast}E_{k} \ar[r]&
i_{\ast}E_{k-1} \ar[r]& \dots }.
\end{displaymath}
hence a resolution $\Tot(\overline
E_{\ast,\ast})\longrightarrow (i_{\ast} \mathcal{F})_{\CC}$ by
hermitian vector bundles.
\end{definition}

Note that, whenever the push-forward $i^{T}_{\ast}$ appears, we will
assume that we have chosen a metric on $N_{Y/X}$.

The two push-forward maps are related by the equation
\begin{equation}
  \label{eq:93}
  i^{T_{c}}_{\ast} (\overline{\mathcal{F}},\eta) =
  i^{T}_{\ast} (\overline{\mathcal{F}},\eta)
  -\left[\left(0,i_{\ast}\left( \omega (\overline{\mathcal{F}},\eta) 
      \widetilde{\Td^{-1}} 
(\overline{\xi}_N)\Td(Y)\right)\right)\right],
\end{equation}
where $\overline{\xi}_{N}$ is the exact sequence \eqref{eq:90}.

\begin{proposition}
  The push-forward maps $i^{T}_{\ast}$, $i^{T_{c}}_{\ast}$ are well
  defined. That is, they do 
  not depend on the choice of a representative  of a class in
  $\widehat K$, nor
  on the choice of metrics on the coherent sheaf
  $(i_{\ast}\mathcal{F})_{\CC}$. The first one does not depend on the
  choice of metrics on $T_{X}$ nor on $T_{Y}$, whereas the second one
  does  
  not depend on the choice of a metric on the normal bundle 
  $N_{Y/X}$. Moreover, if $i$ is a regular closed immersion or
  $\mathcal{X}$ is a regular arithmetic variety, then $i^{T_{c}}_{\ast}$
  and $i^{T}_{\ast}$
  can be lifted to maps
  \begin{displaymath}
    i^{T_{c}}_{\ast}, i^{T}_{\ast}\colon \widehat
    K(\mathcal{Y},\mathcal{D}_{\D,Y})\longrightarrow  
    \widehat K(\mathcal{X},\mathcal{D}_{\D,Y}).
  \end{displaymath}
\end{proposition}

\begin{proof}
  The fact that $i^{T}_{\ast}$ only depends on the metric on
  $\overline N$ and not on the metrics on $T_{X}$ and $T_{Y}$ and that
  for $i^{T_{c}}_{\ast}$ is the opposite, follows directly from the
  definition in the first case and from proposition \ref{prop:17} in
  the second.

We will only prove the other statements for $i^{T_{c}}_{\ast}$, as
the other case is analogous.
We first prove the independence from the metric chosen on the coherent
sheaf
$(i_{\ast}\mathcal{F})_{\CC}$. If $\overline{E}_{\ast}\to
(i_{\ast}\mathcal{F})_{\CC}$, $\overline{E}'_{\ast}\to
(i_{\ast}\mathcal{F})_{\CC}$ are two such metrics, inducing the
hermitian embedded vector bundles $\overline{\xi}$ respectively
$\overline{\xi}'$, then, using corollary \ref{comparThevb}
\begin{displaymath}
T_{c}(\overline{\xi}_{c}')-T_{c}(\overline{\xi}_{c})=
T(\overline{\xi}')-T(\overline{\xi})=
\widetilde{\ch}(\overline{\varepsilon}), 
\end{displaymath}
where $\overline {\varepsilon }$ is the exact complex of hermitian
embedded vector bundles
\begin{displaymath}
\overline{\varepsilon}\colon 
0\longrightarrow\overline{\xi}\longrightarrow
\overline{\xi}'\longrightarrow 0,
\end{displaymath}
where $\overline{\xi}'$ sits in degree zero.

Therefore, by equation \ref{equivKprim},
\begin{multline*}
  [((i_{\ast}\mathcal{F},\overline{E}_{\ast}\to
(i_{\ast}\mathcal{F})_{\CC}) ,0)]-[(0,T_{c}(\overline{\xi_{c}}))]\\=
[((i_{\ast}\mathcal{F},\overline{E}'_{\ast}\to
(i_{\ast}\mathcal{F})_{\CC}) ,0)]-[(0,T_{c}(\overline{\xi}_{c}'))].
\end{multline*}

Since the last term of equation \ref{eq:66} does not depend on the
metric on $(i_{\ast}\mathcal{F})_{\CC}$, we obtain that $i^{T_{c}}_{\ast}$
does not depend on this 
metric.

For proving that the push-forward map $i^{T_{c}}_{\ast}$ is well defined
it remains to show the independence from the choice of a
representative  of a class in $\widehat
K(\mathcal{Y},\mathcal{D}_{\D,Y})$. We consider an exact sequence of
hermitian vector bundles on $\mathcal{Y}$
\begin{displaymath}
\overline{\varepsilon}\colon 0\longrightarrow \overline{\mathcal{F}}_1
\longrightarrow \overline{\mathcal{F}} \longrightarrow
\overline{\mathcal{F}}_2\longrightarrow 0
\end{displaymath}
and two classes $\eta_1, \eta_2\in\bigoplus_{p\geq 0} \widetilde
{\mathcal{D}}_{\D}^{2p-1}(Y,p)$. We also denote $\overline{\varepsilon}$
the induced exact sequence of hermitian vector bundles on $Y$. We
have to prove
\begin{equation}\label{compatrepr}
i^{T_{c}}_{\ast}([(\overline{\mathcal{F}},\eta_1+\eta_2+\widetilde{\ch}
(\overline{\varepsilon})])=i^{T_{c}}_{\ast}([(\overline{\mathcal{F}_1},\eta_1)])
+i^{T_{c}}_{\ast}([(\overline{\mathcal{F}_2},\eta_2)]).
\end{equation}
Since it is clear that $i^{T_{c}}_{\ast}(0,\eta_1+\eta_2)
=i^{T_{c}}_{\ast}(0,\eta_1)
+i^{T_{c}}_{\ast}(0,\eta_2)$, we are led to prove
\begin{equation}\label{compatrepr2}
i^{T_{c}}_{\ast}([(\overline{\mathcal{F}},\widetilde{\ch}
(\overline{\varepsilon})])=i^{T_{c}}_{\ast}([(\overline{\mathcal{F}_1},0)])
+i^{T_{c}}_{\ast}([(\overline{\mathcal{F}_2},0)]).
\end{equation}
We choose metrics on the coherent sheaves
$(i_{\ast}\mathcal{F}_1)_{\CC}$, $(i_{\ast}\mathcal{F}_2)_{\CC}$ and
$(i_{\ast}\mathcal{F})_{\CC}$ respectively:
\begin{displaymath}
\overline{E}_{1,\ast}\longrightarrow (i_{\ast}\mathcal{F}_1)_{\CC}\
,\ \overline{E}_{2,\ast}\longrightarrow
(i_{\ast}\mathcal{F}_2)_{\CC}\ ,\ \overline{E}_{\ast}\longrightarrow
(i_{\ast}\mathcal{F})_{\CC}.
\end{displaymath}
We denote $\overline{\xi}_1$, $\overline{\xi}_2$, $\overline{\xi}$
the induced hermitian embedded vector bundles. We obtain an exact
sequence of metrized coherent sheaves on $\mathcal{X}$:
\begin{displaymath}
\overline{\nu}\colon 0\longrightarrow \overline{i_{\ast}\mathcal{F}_{1}}
\longrightarrow \overline{i_{\ast}\mathcal{F}} \longrightarrow
\overline{i_{\ast}\mathcal{F}_{2}} \longrightarrow 0.
\end{displaymath}

Then, using the fact that the theory $T$ is additive and equation
\eqref{eq:88} we have
\begin{equation}\label{equ1}
T_{c}(\overline{\xi}_{1,c})+
T_{c}(\overline{\xi}_{2,c})-T_{c}(\overline{\xi}_{c})=
[\widetilde{\ch}(\overline{\nu})]-
i_{\ast}([
\widetilde{\ch}(\overline{\varepsilon})\bullet \Td(Y)])\bullet
\Td^{-1}(X).  
\end{equation}
Moreover, by the relation \eqref{equivKprim},
\begin{equation}\label{equ2}
[(\overline{i_{\ast}\mathcal{F}_{1}},0)]+
[(\overline{i_{\ast}\mathcal{F}_{2}},0)] 
=[(\overline{i_{\ast}\mathcal{F}},\widetilde{\ch}(\overline{\nu}))].
\end{equation}
Hence, we compute,
\begin{align*}
  i^{T_{c}}_{\ast}([(\overline{\mathcal{F}},\widetilde{\ch}
(\overline{\varepsilon})])&-i^{T_{c}}_{\ast}([(\overline{\mathcal{F}_1},0)])
-i^{T_{c}}_{\ast}([(\overline{\mathcal{F}_2},0)])\\
&=[(i_{\ast}\overline{\mathcal{F}},0)]-[(i_{\ast}\overline{\mathcal{F}_{1}},0)]-
[(i_{\ast}\overline{\mathcal{F}_{2}},0)]\\
&\phantom{A}-[(0,T_{c}(\overline{\xi}_{c}))]+
[(0,T_{c}(\overline{\xi_{1,c}}))]
+[(0,T_{c}(\overline{\xi_{2,c}}))]\\
&\phantom{A}+[(0,i_{\ast}([\widetilde{\ch}(\overline {\varepsilon
})]\bullet \Td(Y)\bullet i^{\ast}\Td^{-1}(X)))]\\
&=-[(0,i_{\ast}([\widetilde{\ch}(\overline {\varepsilon
})]\bullet \Td(Y)\bullet i^{\ast}\Td^{-1}(X))  ))]\\
&\phantom{AA}+[(0,i_{\ast}([\widetilde{\ch}(\overline {\varepsilon
})]\bullet \Td(Y)\bullet i^{\ast}\Td^{-1}(X))  ))]
\\&=0.
\end{align*}

The proof that $i^{T_{c}}_{\ast}$ for metrized coherent
sheaves is well defined is similar. The proof of its independence from 
choice of a metric 
on $N_{Y/X}$ or from the choice of the resolutions and metrics in $X$
is the same as before. Now let
\begin{displaymath}
  0\longrightarrow \overline{\mathcal{F}}'\longrightarrow 
  \overline{\mathcal{F}} \longrightarrow \overline
  {\mathcal{F}}''\longrightarrow 0 
\end{displaymath}
be a short exact sequence of metrized coherent sheaves on
$\mathcal{Y}$. This means that we have resolutions $\overline
E'_{\ast}\to \mathcal{F}'_{\mathbb{C}} $, $\overline
E_{\ast}\to \mathcal{F}_{\mathbb{C}} $ and $\overline
E''_{\ast}\to \mathcal{F}''_{\mathbb{C}} $. Using theorem \ref{zhabc}
we can suppose that 
there is a commutative diagram of resolutions
\begin{equation}\label{eq:80}
  \begin{array}[h]{ccccccccc}
    0&\rightarrow&\overline E'_{\ast}&\rightarrow&\overline E_{\ast}
    &\rightarrow&\overline E''_{\ast}&\rightarrow&0\\
    &&\downarrow && \downarrow && \downarrow &&\\
    0&\rightarrow& \mathcal{F}'_{\mathbb{C}}
    &\rightarrow& \mathcal{F}_{\mathbb{C}}
    &\rightarrow& \mathcal{F}''_{\mathbb{C}}
    &\rightarrow& 0,
  \end{array}
\end{equation}
with exact rows. Moreover, we can assume 
that the complexes of complexes $\overline
E'_{\ast,\ast}$,
$\overline E_{\ast,\ast}$, $\overline E''_{\ast,\ast}$ used in
definition \ref{def:17} are chosen compatible with diagram
\eqref{eq:80}. Thus we obtain a commutative diagram
\begin{equation}\label{eq:81}
  \begin{array}[h]{ccccccccc}
    0&\rightarrow&\Tot \overline E'_{\ast,\ast}&\rightarrow&\Tot \overline
    E_{\ast, \ast}
    &\rightarrow&\Tot \overline E''_{\ast,\ast}&\rightarrow&0\\
    &&\downarrow && \downarrow && \downarrow &&\\
    0&\rightarrow& i_{\ast}\mathcal{F}'_{\mathbb{C}}
    &\rightarrow& i_{\ast}\mathcal{F}_{\mathbb{C}}
    &\rightarrow& i_{\ast}\mathcal{F}''_{\mathbb{C}}
    &\rightarrow& 0.
  \end{array}
\end{equation}
We denote by $\overline {\nu }$ the exact sequence of metrized
coherent sheaves on $X$ defined by diagram \eqref{eq:81}. We denote
$\overline {\chi}_{i}$ the exact sequence of hermitian vector bundles
on $Y$
\begin{displaymath}
  \overline {\chi}_{i}\colon  0\longrightarrow \overline E'_{i}
  \longrightarrow \overline E_{i}
  \longrightarrow \overline E''_{i}
  \longrightarrow 0,
\end{displaymath}
and by $\overline \varepsilon $ the exact sequence of metrized
coherent sheaves on $X$
\begin{displaymath}
  \overline {\varepsilon }_{i}\colon  0\longrightarrow \overline
  {i_{\ast}E}'_{i} 
  \longrightarrow \overline {i_{\ast} E}_{i}
  \longrightarrow \overline {i_{\ast}E}''_{i}
  \longrightarrow 0.
\end{displaymath}
Moreover, let $\overline {\xi}_{i}$, $\overline {\xi}'_{i}$ and
$\overline {\xi}''_{i}$ denote the 
hermitian embedded 
vector bundles defined by the above resolutions and $\overline E_{i}$,
$\overline E'_{i}$ and 
$\overline E''_{i}$ respectively and let  $\overline {\xi}_{i,c}$,
$\overline {\xi}'_{i,c}$ and 
$\overline {\xi}''_{i,c}$ be as in definition \ref{def:18}.  
Then, using proposition \ref{prop:13} and equation \eqref{eq:88}
we obtain 
\begin{align}
  \widetilde{\ch}(\overline {\nu })&=\sum_{i}(-1)^{i}\widetilde{\ch} 
  (\overline {\varepsilon })\notag\\
  &=\sum_{i}(-1)^{i}(T_{c}(\overline {\xi}'_{i,c})+
  T_{c}(\overline {\xi}''_{i,c})-
  T_{c}(\overline {\xi}_{i,c}))\label{eq:82}\\  
  &\phantom{AA}+\sum_{i}(-1)^{i}i_{\ast}(\widetilde{\ch}
  (\overline{\chi}_{i})\bullet \Td(Y))\bullet \Td^{-1}(X)\notag  
\end{align}
Now the proof follows as before, but using equation \eqref{eq:82}
instead of equation \eqref{equ1}. 

If $\mathcal{X}$ is a regular arithmetic variety, the lifting
property follows from the isomorphism between the $\widehat K$-groups
and the 
$\widehat K'$-groups.

Suppose now that  $i\colon \mathcal{Y}\longrightarrow \mathcal{X}$ is a
regular closed immersion and let $[\overline{\mathcal{F}},\eta]\in
\widehat K(\mathcal{Y},\mathcal{D}_{\D,Y})$. Then it follows from
\cite{BerthelotGrothendieck:SGA6} III that the
coherent sheaf  
$i_{\ast}\mathcal{F}$ can be resolved
\begin{displaymath}
0\longrightarrow \mathcal{E}_n\longrightarrow\ldots\longrightarrow
\mathcal{E}_0\longrightarrow i_{\ast}\mathcal{F}\longrightarrow 0
\end{displaymath}
with $\mathcal{E}_i$ locally free sheaves on $\mathcal{X}$. Moreover
we endow the 
vector bundles $E_i$ induced on $X$ with hermitian metrics and so
we obtain a metric on the coherent sheaf $i_{\ast}\mathcal{F}$ and the
corresponding hermitian embedded vector bundle $\overline{\xi}$.
Using the independence from the resolutions and on the metrics we see
that the equation \ref{eq:66} defines an element in
$\widehat K(\mathcal{X},\mathcal{D}_{\D,X})$.
\end{proof}

\begin{proposition}
  For any element $\alpha \in \widehat
  K'(\mathcal{Y},\mathcal{D}_{\D,Y})$ we have 
  \begin{align}
    \label{eq:9}
    \omega (i^{T_{c}}_{\ast}(\alpha ))\Td(X)&=
    i_{\ast}(\omega (\alpha )\Td(Y))\\
    \omega (i^{T}_{\ast}(\alpha ))&=
    i_{\ast}(\omega (\alpha )\Td^{-1}(N_{Y/X}))\label{eq:94}
  \end{align}
\end{proposition}

\begin{proof} We will prove the statement only for $i_{\ast}^{T_{c}}$.
We consider first a class of the form $[\overline{\mathcal{F}},0]$.
Using equation \eqref{eq:85} we obtain, after choosing
a metric $\overline{E}_i\longrightarrow (i_{\ast}\mathcal{F})_{\CC}$,
and considering 
the induced hermitian embedded vector bundle $\overline{\xi_{c}}$:
\begin{align*}
\omega (i^{T_{c}}_{\ast}([\overline{\mathcal{F}},0]))\Td(X) &=\left(
\sum(-1)^i\ch(\overline{E_i})-\dd_{\mathcal{D}}
T_{c}(\overline{\xi_{c}})\right)\Td(X)\\  
 &= i_{\ast}(\ch(\overline{F})\bullet
 \Td(Y)\bullet i^{\ast}\Td^{-1}(X)i^{\ast}(\Td(X)))\\
&=i_{\ast}(\ch(\overline{F})\bullet\Td(Y))\\&= i_{\ast}(\omega
([\overline{\mathcal{F}},0] )\Td(Y))
\end{align*}

Taking now a class of the form $[0,\eta]$ we obtain:
\begin{align*}
\omega (i^{T_{c}}_{\ast}([0,\eta]))\Td(X) &=
\dd_{\mathcal{D}}\left(i_{\ast}(\eta\Td(Y)
i^*\Td^{-1}(X))\right)\Td(X)\\
&=i_{\ast}\dd_{\mathcal{D}}(\eta\Td(Y))\\
&=i_{\ast}(\omega ([0,\eta])\Td(Y))
\end{align*}
and hence the equality \ref{eq:9} is proved.
\end{proof}

The next proposition explains the terminology ``compatible with the
projection formula'' and ``transitive'' that we used for theories of
singular Bott-Chern classes. The second statement is the main reason
to introduce the push-forward $i_{\ast}^{T_{c}}$.

\begin{proposition}
  If the theory of singular Bott-Chern classes is compatible with the
  projection formula, we have that, for $\alpha \in
  \widehat K'(\mathcal{Y},\mathcal{D}_{\D,Y})$ and $\beta \in
  \widehat K(\mathcal{X},\mathcal{D}_{\las,X})$ the following
  equalities hold 
  \begin{align*}
    i^{T_{c}}_{\ast}(\alpha i^{\ast}\beta )&=i^{T_{c}}_{\ast}(\alpha )\beta,\\ 
    i^{T}_{\ast}(\alpha i^{\ast}\beta )&=i^{T}_{\ast}(\alpha )\beta. 
  \end{align*}
  If moreover the theory of singular Bott-Chern classes is transitive
  and $j\colon (\mathcal{Z},h_{Z})\longrightarrow (\mathcal{Y},h_{Y})$ is
  another closed 
  immersion of metrized arithmetic varieties, then 
  \begin{displaymath}
    (i\circ j)_{\ast}^{T_{c}}=i_{\ast}^{T_{c}}\circ j_{\ast}^{T_{c}}.
  \end{displaymath}
\end{proposition}
\begin{proof}
  We prove first the projection formula. For simplicity we will treat
  the case when $\alpha \in \widehat
  K(\mathcal{Y},\mathcal{D}_{\D,Y})$. Let $\alpha
  =(\overline{\mathcal{F}},\eta)$, let $\overline
  {\xi_{c}}=(i,\overline T_{X},\overline T_{Y},\overline {\mathcal
    F}_{\mathbb{C}},\overline E_{\ast})$   be a hermitian embedded vector bundle
  and let
  $\beta =(\overline {\mathcal{E}},\chi)$.   Using equations
  \eqref{eq:66} 
  and \eqref{eq:78}, we obtain
  \begin{align*}
    i^{T_{c}}_{\ast}(\alpha i^{\ast}\beta )-i^{T_{c}}_{\ast}(\alpha )\beta&=
    -\sum_{i}(-1)^{i}\ch(\overline E_{i})\bullet \chi +
    \dd_{\mathcal{D}}(T_{c}(\overline {\xi}_{c}))\bullet \chi\\
    &\phantom{AA}
    +i_{\ast}(\ch((\overline
    {\mathcal{F}})_{\mathbb{C}})
    \bullet \Td(Y)))\bullet\Td^{-1}(X)\bullet \chi \\
    &\phantom{AA}+T_{c}(\overline \xi_{c})\bullet \ch(\overline
    {\mathcal{E}}_{\mathbb{C}})-  
    T_{c}(\overline \xi_{c}\otimes \overline {\mathcal{E}}_{\mathbb{C}})\\
    &=
    T_{c}(\overline \xi_{c}\otimes \overline {\mathcal{E}}_{\mathbb{C}})
    -T_{c}(\overline \xi_{c})\bullet \ch(\overline
    {\mathcal{E}}_{\mathbb{C}}).
  \end{align*}
  Therefore, if $T$ is compatible with the projection formula, then
  the projection formula holds.

  The fact that, if moreover $T$ is transitive then $ (i\circ
  j)_{\ast}^{T_{c}}= i_{\ast}^{T_{c}}\circ j_{\ast}^{T_{c}}$ follows
  directly from 
  the definition and equation \eqref{eq:87}.
\end{proof}

If $i\colon \mathcal{Y}\longrightarrow \mathcal{X}$ is a regular closed immersion
between arithmetic varieties, then the normal cone
$\mathcal{N}_{\mathcal{Y}/\mathcal{X}}$ is a locally free sheaf. The
choice of 
a hermitian metric on $N_{Y/X}$ determines a hermitian vector bundle
$\overline {\mathcal{N}}_{\mathcal{Y}/\mathcal{X}}$. If now
$i\colon (\mathcal{Y},h_{Y})\longrightarrow (\mathcal{X},h_{X})$ is a closed
immersion 
between regular metrized arithmetic varieties, then the tangent
bundles ${\mathcal{T}}_{\mathcal{Y}}$ and
${\mathcal{T}}_{\mathcal{X}}$ are virtual vector bundles. Since over
$\mathbb{C}$ they define vector bundles, we can provide them with
hermitian metrics and denote the
hermitian virtual vector bundles by $\overline {\mathcal{T}}_{\mathcal{X}}$ and 
$\overline {\mathcal{T}}_{\mathcal{Y}}$. There are well defined clases
$\widehat 
{\Td}(\mathcal{Y})= \widehat
{\Td}(\overline {\mathcal{T}}_{\mathcal{Y}})$ and $\widehat
{\Td}(\mathcal{X})= \widehat
{\Td}(\overline {\mathcal{T}}_{\mathcal{X}})$. 
 
The arithmetic Grothendieck-Riemann-Roch theorem for closed immersions
compares the direct images in the arithmetic $K$-groups with
the 
direct images in the arithmetic Chow groups.

\begin{theorem}[\cite{BismutGilletSoule:MR1086887},
  \cite{zha99:_rieman_roch}] \label{thm:15}  
  Let $T$ be a theory of singular Bott-Chern
  classes and let $S_{T}$ be the additive genus of corollary
  \ref{cor:10}. 
  \begin{enumerate}
  \item   Let $i\colon \mathcal{Y}\longrightarrow \mathcal{X}$ be a
    regular closed immersion between arithmetic varieties. Assume that
    we have chosen a hermitian metric on the complex bundle
    $N_{Y/X}$. Then, for any 
    $\alpha=(\overline {\mathcal{F}},\eta) \in
    \widehat K(\mathcal{Y},\mathcal{D}_{\D,Y})$ 
    the 
    equation
    \begin{equation}\label{eq:29}
      \widehat {\ch}(i^{T}_{\ast}(\alpha ))= 
      i_{\ast}(\widehat{\ch}(\alpha
      )\widehat{\Td}^{-1}(\overline{\mathcal{N}}
      _{\mathcal{Y}/\mathcal{X}}))
      -
      \amap(i_{\ast}(\ch(\mathcal{F}_{\CC})
      \Td^{-1}(N_{Y/X})S_{T}(N))
    \end{equation}
    holds.
  \item   Let $i\colon (\mathcal{Y},h_{Y})\longrightarrow (\mathcal{X},h_{X})$ be a
    closed immersion between regular metrized arithmetic
    varieties. Then, for any 
    $\alpha=(\overline {\mathcal{F}},\eta) \in
    \widehat K(\mathcal{Y},\mathcal{D}_{\D,Y})$ the equation
    \begin{equation}\label{eq:24}
      \widehat {\ch}(i^{T_{c}}_{\ast}(\alpha ))\widehat
      {\Td}(\mathcal{X})= 
      i_{\ast}(\widehat{\ch}(\alpha
      )\widehat{\Td}(\mathcal{Y}))-\amap(i_{\ast}(\ch(
      \mathcal{F}_{\CC})\Td(Y)S_{T}(N)))
    \end{equation}  
    holds.
  \end{enumerate}
\end{theorem}
\begin{proof}
  The proof follows the classical pattern of the deformation to the
  normal cone as in \cite{BismutGilletSoule:MR1086887} and
  \cite{zha99:_rieman_roch}. 

  Let $\mathcal{W}$ be the deformation to the normal cone to
  $\mathcal{Y}$ in $\mathcal{X}$. We will follow the notation of
  section \ref{sec:deform-resol}. Since $i$ is a regular closed
  immersion, there is a finite resolution by locally free sheaves
  \begin{displaymath}
    0\to
    \mathcal{E}_{n}\to \dots \to \mathcal{E}_{1}\to \mathcal{E}_{0}
    \to i_{\ast}\mathcal{F}\to 0. 
  \end{displaymath}
  We choose hermitian metrics on the complex bundles
  $E_{i}=(\mathcal{E}_{i})_{\mathbb{C}}$. The immersion
  $j\colon \mathcal{Y}\times \mathbb{P}^{1}\longrightarrow \mathcal{W}$ is
  also a regular immersion. The construction of theorem 
  \ref{thm:5} is valid over the arithmetic ring $A$. Therefore we have
  a resolution by hermitian vector bundles
  \begin{displaymath}
    0\to
    \widetilde {\mathcal{G}}_{n}\to \dots \to \widetilde
    {\mathcal{G}}_{1}\to \widetilde{\mathcal{G}}_{0} 
    \to i_{\ast}\mathcal{F}\to 0.
  \end{displaymath}
  such that its restriction to $\mathcal{X}\times \{0\}$ is isometric
  to 
  $\mathcal{E}_{\ast}$. Its restriction to $\widetilde {\mathcal{X}}$
  is orthogonally split, and its restriction to
  $\mathcal{P}=\mathbb{P}(\mathcal{N}_{\mathcal{Y}/\mathcal{X}}\oplus
  \mathcal{O}_{\mathcal{Y}})$ fits in a short exact sequence
  \begin{displaymath}
    0\longrightarrow \overline{\mathcal{A}}_{\ast}\longrightarrow 
    \widetilde {\mathcal{E}}_{\ast}|_{\mathcal{P}}\longrightarrow 
    K(\overline{\mathcal{F}},
    \overline{\mathcal{N}}_{\mathcal{Y}/\mathcal{X}})\longrightarrow   
    0, 
  \end{displaymath}
  where $\overline{\mathcal{A}}_{\ast}$ is orthogonally split and
  $K(\overline{\mathcal{F}},
  \overline{\mathcal{N}}_{\mathcal{Y}/\mathcal{X}})$ is the Koszul 
  resolution. We denote by $\overline {\eta}_{k}$ the piece of degree
  $k$ of this exact sequence. Let $t$ be the absolute coordinate of
  $\mathbb{P}^{1}$. It defines a rational function in $\mathcal{W}$
  and
  $$\widehat {\dv}(t)=(\mathcal{X}_{0}+\mathcal{P}+\widetilde
  {\mathcal{X}},(0,-\frac{1}{2}\log t\overline t))
  $$ 
  The key point of the proof of the theorem is that, in the group
  $\cha^{\ast}(\mathcal{X},\mathcal{D}_{\D,X} )$, we have
  \begin{displaymath}
    (p_{\mathcal{W}})_{\ast}(\widehat{\ch}(\widetilde
    {\mathcal{E}}_{\ast})\widehat {\dv}(t))=0.
  \end{displaymath}
  Using the definition of the product in the arithmetic Chow rings
  we obtain
  \begin{multline}\label{eq:96}
    (p_{\mathcal{W}})_{\ast}(\widehat{\ch}(\widetilde
    {\mathcal{E}}_{\ast})\widehat {\dv}(t))=\widehat{\ch}(
    \overline{\mathcal{E}}_{\ast})-
    (p_{\widetilde{\mathcal{X}}})_{\ast}\widehat{\ch}(
    \widetilde {\mathcal{E}}_{\ast}|_{\widetilde{\mathcal{X}}})
    -(p_{\widetilde{\mathcal{P}}})_{\ast}\widehat{\ch}(
    \widetilde{\mathcal{E}}_{\ast}|_{\mathcal{P}})\\+
    \amap((p_{W})_{\ast}(\ch((\widetilde{\mathcal{E}}_{\ast})_{\mathbb{C}})
    \bullet W_{1})).
  \end{multline}
  But we have
  \begin{align}
    \widehat{\ch}(
    \overline{\mathcal{E}}_{\ast})&=\widehat{\ch}(i^{T}_{\ast}(\overline
    {\mathcal{F}}))+\amap(T(\overline {\xi})),\label{eq:97}\\
    (p_{\widetilde{\mathcal{X}}})_{\ast}\widehat{\ch}(
    \widetilde {\mathcal{E}}_{\ast}|_{\widetilde{\mathcal{X}}})
    &=0,\label{eq:98}\\
    (p_{\widetilde{\mathcal{P}}})_{\ast}\widehat{\ch}(
    \widetilde{\mathcal{E}}_{\ast}|_{\mathcal{P}})&=
    i_{\ast}(\pi _{\mathcal{P}})_{\ast}(
    \widehat{\ch}(K(\overline{\mathcal{F}},
    \overline{\mathcal{N}}_{\mathcal{Y}/\mathcal{X}}))
    -\sum_{k}(-1)^{k}\amap(\widetilde{\ch}(\overline{\eta}_{k}))).\label{eq:99} 
  \end{align}
  Moreover, by equation \eqref{eq:10},
  \begin{multline}\label{eq:65}
    \amap((p_{W})_{\ast}(\ch((\widetilde{\mathcal{E}}_{\ast})_{\mathbb{C}})
    \bullet W_{1}))=
    -\amap(T(\overline{\xi})) -
    \sum_{k}(-1)^{k}\amap(\widetilde{\ch}(\overline{\eta}_{k})))\\
    +\amap(i_{\ast}C_{T}(\mathcal{F}_{\mathbb{C}},\mathcal{N}_{\mathbb{C}})). 
  \end{multline}
  Thus we are led to compute $ i_{\ast}(\pi _{\mathcal{P}})_{\ast}
    \widehat{\ch}(K(\overline{\mathcal{F}},
    \overline{\mathcal{N}}_{\mathcal{Y}/\mathcal{X}}))$. This is done
    in the following two lemmas.

    \begin{lemma}\label{lemm:5}
      Let $\mathcal{Y}$ be an arithmetic variety, 
      $\overline{\mathcal{N}}$ 
      a rank $r$ hermitian vector bundle over $\mathcal{Y}$ and denote 
      $\mathcal{P}=\mathbb{P}^{1}(\mathcal{N}\oplus
      \mathcal{O}_{\mathcal{Y}})$, and $\overline{\mathcal{Q}}$ the
      tautological quotient bundle. Let $\mathcal{Y}_{0}$ be the cycle
      defined by the zero
      section of $\mathcal{P}$. Then
      \begin{equation}
        \label{eq:95}
        \widehat{c}_{r}(\overline {\mathcal{Q}})=(\mathcal{Y}_{0},
        (c_{r}(\overline {\mathcal{Q}}_{\mathbb{C}}),\widetilde
        e(\mathcal{P}_{\mathbb{C}},\overline
        {\mathcal{Q}}_{\mathbb{C}},s))), 
      \end{equation}
      where $\widetilde e(\mathcal{P}_{\mathbb{C}},\overline
        {\mathcal{Q}}_{\mathbb{C}},s)$ is the Euler-Green current of
        lemma \ref{lemm:4}.
    \end{lemma}    
    \begin{proof}
      We know that
      $\widehat{c}_{r}(\overline {\mathcal{Q}})=(\mathcal{Y}_{0},
        (c_{r}(\overline {\mathcal{Q}}_{\mathbb{C}}),\widetilde
        e))$ for certain Green current $\widetilde e$. By definition
        this Green current satisfies
        \begin{displaymath}
          \dd_{\mathcal{D}}\widetilde e=c_{r}(\overline
          {\mathcal{Q}}_{\mathbb{C}}) -\delta _{\mathcal{Y}_{\mathbb{C}}}.
        \end{displaymath}
        Moreover, since the restriction of  $\overline
        {\mathcal{Q}}_{\mathbb{C}}$ to $D_{\infty}$ has a global section
        of constant norm we have that $\widetilde
        e|_{D_{\infty}}=0$. Therefore, by lemma \ref{lemm:4},  
        $$\widetilde e=\widetilde e(\mathcal{P}_{\mathbb{C}},\overline
        {\mathcal{Q}}_{\mathbb{C}},s).$$
    \end{proof}

    \begin{lemma}
      The following equality hold:
       \begin{multline}\label{eq:56}
        (\pi _{\mathcal{P}})_{\ast}
        \widehat{\ch}(K(\overline{\mathcal{F}},\overline{\mathcal{N}})_{\ast})=\\
        \widehat{\ch}(\overline{\mathcal{F}})
        \widehat{\Td^{-1}}(\overline{\mathcal{N}})
        +\amap(C_{T}(\overline{\mathcal{F}},\overline{\mathcal{N}})-
        \ch(\mathcal{F}_{\CC})
        \Td^{-1}(N_{Y/X})S_{T}(N)). 
      \end{multline}
    \end{lemma}
    \begin{proof}
      We just compute, using lemma \ref{lemm:5},
      \begin{align*}
                (\pi _{\mathcal{P}})_{\ast}
        \widehat{\ch}(K(\overline{\mathcal{F}}&,\overline{\mathcal{N}})_{\ast})=
         (\pi _{\mathcal{P}})_{\ast}\sum_{k}(-1)^{k}
         \widehat {\ch}(\bigwedge^{k}\overline {\mathcal{Q}}^{\vee})
         \widehat {\ch}(\pi _{\mathcal{P}}^{\ast}\overline
         {\mathcal{F}})\\&=
         (\pi _{\mathcal{P}})_{\ast}(\widehat
         c_{r}(\overline{\mathcal{Q}})
         \widehat{\Td^{-1}}(\overline{\mathcal{Q}}))
         \widehat{\ch}(\overline{\mathcal{F}})\\  
         &=
         \widehat{\Td^{-1}}(\overline{\mathcal{N}})
         \widehat{\ch}(\overline{\mathcal{F}})+
         \amap((\pi _{P})_{\ast}(\widetilde e \Td^{-1}(\overline
         {Q}))\ch(\overline F))\\
         &=\widehat{\Td^{-1}}(\overline{\mathcal{N}})
         \widehat{\ch}(\overline{\mathcal{F}})
         +
         \amap((\pi _{P})_{\ast}(T^{h}(K(\overline F,\overline N)))
         \ch(\overline F))\\
         &=\widehat{\Td^{-1}}(\overline{\mathcal{N}})
         \widehat{\ch}(\overline{\mathcal{F}})
         +\amap(C_{T^{h}}(F,N))\\
         &=\widehat{\Td^{-1}}(\overline{\mathcal{N}})
         \widehat{\ch}(\overline{\mathcal{F}})
         +C_{T}(F,N)-\amap(\Td^{-1}(N)\ch(F)S_{T}(N)).
      \end{align*}
    \end{proof}
    The equation \eqref{eq:29} follows by combining equations
    \eqref{eq:96}, \eqref{eq:97}, \eqref{eq:98}, \eqref{eq:99},
    \eqref{eq:65} and \eqref{eq:56}.
    
    The equation \eqref{eq:24} follows from equation \eqref{eq:29} by
    a straightforward computation.
\end{proof}

Since $T$ is homogeneous if and only if $S_{T}=0$,
in view of this result, 
  the theory of homogeneous singular Bott-Chern classes is characterized
  for being the unique theory of singular Bott-Chern classes that
  provides an exact arithmetic Grothendieck-Riemann-Roch theorem for
  closed immersions. 
  By contrast, if one uses a theory of singular Bott-Chern
  classes that is not homogeneous, there is an analogy between the genus
  $S_{T}$ and the $R$-genus that appears in the arithmetic
  Grothendieck-Riemann-Roch theorem for submersions.

Since there is a unique theory of homogeneous singular Bott-Chern
classes, the following definition is natural. 

\begin{definition}
    Let $i\colon (\mathcal{Y},h_{Y})\longrightarrow
(\mathcal{X},h_{X})$ be a closed immersion of metrized arithmetic
varieties, the 
\emph{push-forward map}
\begin{displaymath}
  i_{\ast}\colon \widehat
  K'(\mathcal{Y},\mathcal{D}_{\D,Y})\longrightarrow  
  \widehat K'(\mathcal{X},\mathcal{D}_{\D,Y})
\end{displaymath}
is defined as $i_{\ast}=i^{T_{c}^{h}}_{\ast}$.
\end{definition}

\begin{corollary}
  The push-forward map makes
  $\widehat{K}'(\underline{\phantom{\mathcal{Y}}},\mathcal{D}_{\D,Y})$
  and 
  $\widehat{K}(\underline{\phantom{\mathcal{Y}}},\mathcal{D}_{\D,Y})$
  functors from the category of  regular metrized arithmetic varieties and
  closed immersions to the category of  abelian groups.
\end{corollary}

\begin{corollary}
      Let $i\colon (\mathcal{Y},h_{Y})\longrightarrow
(\mathcal{X},h_{X})$ be a closed immersion of regular metrized arithmetic
varieties, then 
  \begin{equation}\label{eq:89}
    \widehat {\ch}(i^{T}_{\ast}(\alpha ))\widehat {\Td}(\mathcal{X})= 
    i_{\ast}(\widehat{\ch}(\alpha
    )\widehat{\Td}(\mathcal{Y})). 
  \end{equation}  
\end{corollary}

\begin{remark}
  Combining theorem \ref{thm:15} with
  \cite{GilletRoesslerSoule:_arith_rieman_roch_theor_in_higher_degrees}
  we can obtain an arithmetic Grothendieck-Riemann-Roch theorem for
  projective morphisms of regular arithmetic varieties. 

  In a forthcoming paper we will show that the higher torsion forms used to
  define the direct images for submersions can also be characterized
  axiomatically.  
\end{remark}

\newcommand{\noopsort}[1]{} \newcommand{\printfirst}[2]{#1}
  \newcommand{\singleletter}[1]{#1} \newcommand{\switchargs}[2]{#2#1}
\providecommand{\bysame}{\leavevmode\hbox to3em{\hrulefill}\thinspace}
\providecommand{\MR}{\relax\ifhmode\unskip\space\fi MR }
\providecommand{\MRhref}[2]{%
  \href{http://www.ams.org/mathscinet-getitem?mr=#1}{#2}
}
\providecommand{\href}[2]{#2}

\Addresses

\end{document}